\documentclass[12pt]{amsart}

\usepackage[left=2.5cm,right=2.5cm]{geometry}
\usepackage[mathscr]{euscript}
\usepackage{bookmark}
\usepackage{mathabx}
\usepackage{amssymb}
\usepackage{bm}
\usepackage{epsfig}
\usepackage{lscape} 
\usepackage{tikz}
\usetikzlibrary{positioning,calc}

\newtheorem{theorem}{Theorem}[section]
\newtheorem{lemma}[theorem]{Lemma}
\newtheorem{proposition}[theorem]{Proposition}
\newtheorem{corollary}[theorem]{Corollary}

\theoremstyle{definition}
\newtheorem{definition}[theorem]{Definition}
\newtheorem{example}[theorem]{Example}

\theoremstyle{remark}
\newtheorem*{remark}{Remark}

\numberwithin{equation}{section}

\DeclareMathOperator{\sgn}{sgn}
\DeclareMathOperator{\diag}{diag}

\DeclareMathOperator{\wt}{wt}
\DeclareMathOperator{\Ad}{Ad}
\DeclareMathOperator{\tr}{tr}
\DeclareMathOperator{\Hom}{Hom}
\DeclareMathOperator{\nonint}{ni}

\DeclareMathOperator{\HPC}{HPC}
\DeclareMathOperator{\bund}{bund}
\DeclareMathOperator{\odd}{odd}
\DeclareMathOperator{\even}{even}

\begin{document}

\title[Generalized B\"{a}cklund-Darboux transformations for Coxeter-Toda systems]{Generalized B\"{a}cklund-Darboux transformations for Coxeter-Toda systems on simple Lie groups}

\begin{abstract}
We derive the cluster structure on the conjugation quotient Coxeter double Bruhat cells of a simple Lie group from that on the double Bruhat cells of the corresponding adjoint Lie group given by Fock and Goncharov using the notion of amalgamation given by Fock and Goncharov, and Williams, thereby generalizing the construction developed by Gekhtman \emph{et al}. We will then use this cluster structure on the conjugation quotient Coxeter double Bruhat cells to construct generalized B\"{a}cklund-Darboux transformations between two Coxeter-Toda systems on simple Lie groups in terms of cluster mutations, thereby generalizing the construction developed by Gekhtman \emph{et al}. We show that these generalized B\"{a}cklund-Darboux transformations preserve Hamiltonian flows generated by the restriction of the trace function of any representation of the simple Lie group, from which we deduce that the family of Coxeter-Toda systems on a simple Lie group forms a single cluster integrable system. Finally, we also develop network formulations of the Coxeter-Toda Hamiltonians for the classical Lie groups, and use these network formulations to obtain combinatorial formulas for these Coxeter-Toda Hamiltonians.
\end{abstract}


\author{Mingyan Simon Lin}

\maketitle

\section{Introduction}\label{Section 1}

\subsection{Overview}\label{Section 1.1}

There are two main goals in this paper. Our first goal is to extend the work of Gekhtman \emph{et al.} in \cite{GSV11} to the other Lie types, where we will show that the family of Coxeter-Toda systems on a simple Lie group forms a single cluster integrable system, by constructing generalized B\"{a}cklund-Darboux transformations between Coxeter-Toda systems on simple Lie groups in terms of cluster mutations that preserve Hamiltonian flows generated by the restriction of the trace function of any representation of the simple Lie group. Our second goal is to provide explicit combinatorial formulas for the Coxeter-Toda Hamiltonians for all simple classical Lie groups, or equivalently, conserved quantities of the $Q$-systems of affine Dynkin types $A_r^{(1)}, D_{r+1}^{(2)}, A_{2r-1}^{(2)}$ and $D_r^{(1)}$, thereby extending the work of Reshetikin \cite{Reshetikhin00}, Di Francesco-Kedem \cite{DFK09-2,DFK24}, Kruglinskaya-Marshakov \cite{KM15} and Williams \cite{Williams16}.

There are two main notions of interest in this paper, namely double Bruhat cells and $Q$-systems. Double Bruhat cells are geometric objects that lie naturally at the crossroads of cluster algebras and integrable systems. It was shown in \cite{FZ99, BFZ05} that the coordinate rings of double Bruhat cells has a natural upper cluster algebra structure, where the generalized minors can be identified with cluster variables, and the cluster mutations are given by determinantal relations satisfied by these generalized minors. The cluster algebraic formulation of these double Bruhat cells was then used to solve the total positivity problem for semisimple Lie groups.

On the other hand, double Bruhat cells also arise naturally in the context of integrable systems. These double Bruhat cells have a canonical Poisson structure that is inherited from the standard Poisson-Lie structure on a simple Lie group \cite{KZ02, Reshetikhin03, Yakimov02}. A natural question to ask at this point is if the canonical cluster and the Poisson structures on a double Bruhat cell are compatible with each other. This question was answered affirmatively in \cite{GSV03}, where it was shown that there is a canonical structure on a given cluster algebra of geometric type, such that the Poisson brackets between cluster variables of any cluster are log-canonical. The same is true of the conjugation quotient Coxeter double Bruhat cells of adjoint Lie groups as well, where the cluster structures on the conjugation quotient Coxeter double Bruhat cells are obtained from that on the Coxeter double Bruhat cells by amalgamation \cite{FG06, Williams15}. 

With respect to the compatible Poisson structure on a given quotient Coxeter double Bruhat cell, there is a family of integrable systems, called Coxeter-Toda systems, which are dynamical systems corresponding to Hamiltonian flows, called Coxeter-Toda Hamiltonians, that are generated by the restrictions of conjugation-invariant functions of the simple Lie group to the quotient Coxeter double Bruhat cell. These Coxeter-Toda systems are completely integrable \cite{HKKR00}, and there are maximal algebraically independent sets of integrals of motion of these Coxeter-Toda systems \cite{Kostant79}, given by the Coxeter-Toda Hamiltonians arising from the fundamental representations of a simple Lie group $G$ of rank $r$, or the first $r$ exterior powers of the defining representation of $G$.

Explicit formulas and network formulations of Coxeter-Toda Hamiltonians were also given in the literature. They were first given for the Coxeter-Toda Hamiltonian arising from the defining representation of a simple classical Lie group in \cite{Reshetikhin00}. Later, combinatorial formulas for Coxeter-Toda Hamiltonians arising from the fundamental representations of the special linear group were given in \cite{DFK09-2}, while network formulations of these Coxeter-Toda Hamiltonians were given in \cite{Williams16} for the special linear group. Finally, explicit formulas for the Coxeter-Toda Hamiltonians arising from the fundamental representations of other simple Lie groups of types $B_2,B_3,C_3,D_4$ and $G_2$ were given in \cite{KM15}.

The $Q$-systems, which are the other main notion of interest in this paper, are families of recurrence relations that are satisfied by the restricted characters of the Kirillov-Reshetikhin (KR-) modules over the Yangians \cite{KR87, Kirillov89, KNS94} or the quantum affine algebras \cite{Nakajima03, Hernandez06, Hernandez10}. These $Q$-system relations can be realized as cluster algebra mutations \cite{Kedem08, DFK09, Williams15} when the $Q$-system is not of type $A_{2r}^{(2)}$.

When the $Q$-system is of simply-laced or twisted affine type $\neq A_{2r}^{(2)}$, the $Q$-system relations are given by exchange relations in a cluster algebra corresponding to an appropriate conjugation quotient Coxeter double Bruhat cell. This connection was first established by Gekhtman \emph{et al.} in \cite{GSV11}, between the $A_r^{(1)}$ $Q$-system and certain conjugation quotient Coxeter double Bruhat cells of the general linear group $GL_{r+1}$. To establish this connection, they first constructed for each ordered pair of Coxeter elements a weighted directed network representation of elements of the corresponding conjugation quotient Coxeter double Bruhat cell in an annulus. Here, the network representation arises from the factorization of a generic element of the Coxeter double Bruhat cell into elementary bidiagonal factors. Subsequently, they introduced a log-canonical Poisson bracket on the space of face weights associated to the weighted directed network, where this Poisson bracket is induced from the standard Poisson structure on $GL_{r+1}$.

With this log-canonical Poisson bracket in mind, Gekhtman \emph{et al.} constructed an initial cluster seed for each ordered pair of Coxeter elements (or equivalently, conjugation quotient Coxeter double Bruhat cell), such that these seeds are mutation equivalent to each other, and the resulting cluster algebra arising from these mutation equivalent seeds is compatible with the log-canonical Poisson bracket on the space of face weights associated to the weighted directed network. When the Coxeter elements in the ordered pair are equal to each other and the values of the frozen variables in the initial cluster are set to $1$, the exchange matrix is the same as the exchange matrix for the initial cluster seed of the $A_r^{(1)}$ $Q$-system cluster algebra given in \cite[(3.6)]{Kedem08}, up to a reordering of indices. 

Subsequently, Gekhtman \emph{et al.} constructed generalized B\"{a}cklund-Darboux transformations between two conjugation quotient Coxeter double Bruhat cells in terms of cluster transformations of the cluster seeds corresponding to these cells. They then used the compatibility of their constructed cluster algebra with the log-canonical Poisson bracket to show that these generalized B\"{a}cklund-Darboux transformations preserve Hamiltonian flows generated by Coxeter-Toda Hamiltonians arising from the traces of powers of a matrix. In particular, this shows that the Coxeter-Toda systems on any pair of conjugation quotient Coxeter double Bruhat cells of $GL_{r+1}$ are equivalent to each other, or equivalently, the family of Coxeter-Toda systems on $GL_{r+1}$ forms a single cluster integrable system. Using these generalized B\"{a}cklund-Darboux transformations, they showed that the $A_r^{(1)}$ $Q$-system evolution dynamics are equivalent to that of the factorization mapping on a conjugation quotient Coxeter double Bruhat cell arising from an ordered pair of Coxeter elements that are equal to each other.

In further developments, Williams \cite{Williams15} generalized the results of Gekhtman \emph{et al.} \cite{GSV11} to the other finite Dynkin types, where he showed via the amalgamation of Coxeter double Bruhat cells that when the $Q$-system is of simply-laced affine type (respectively twisted affine type $\neq A_{2r}^{(2)}$), the $Q$-system evolution dynamics are equivalent to that of the factorization mapping on the conjugation quotient Coxeter double Bruhat cell of a simply-laced (respectively non-simply laced) simple Lie group, with the conjugation quotient Coxeter double Bruhat cell arising an ordered pair of Coxeter elements that are equal to each other. This relation between the $Q$-system evolution dynamics and that of the factorization mapping is then used to show that $Q$-system is discrete Liouville integrable. 

Outside of the simply-laced or the twisted affine types $\neq A_{2r}^{(2)}$, Vichitkunakorn \cite{Vichitkunakorn18} proved using dimer integrable systems in a cluster variable setting that the $B_r^{(1)}$ $Q$-system is discrete Liouville integrable, and expressed the conserved quantities of the $Q$-system in terms of partition functions of hard particles on a certain graph. More recently, Di Francesco and Kedem \cite{DFK24} proved using Koornwinder-Macdonald theory that the quantum untwisted and twisted $Q$-systems of classical type are discrete Liouville integrable, and provided explicit formulas for the first conserved quantity of the quantum $Q$-system.

In this paper, we will build on the constructions and results obtained by Gekhtman \emph{et al.} \cite{GSV11}, where we will use the tools and results obtained by Fock and Goncharov \cite{FG06} and Williams \cite{Williams15} to describe the cluster structure on the conjugation quotient Coxeter double Bruhat cells of a simple Lie group, and show that the initial cluster seeds that correspond to these cells are mutation equivalent to each other. 

Following a similar strategy employed by Gekhtman \emph{et al.} \cite{GSV11}, we will show that the family of Coxeter-Toda systems on a simple Lie group forms a single cluster integrable system, by constructing generalized B\"{a}cklund-Darboux transformations between two conjugation quotient Coxeter double Bruhat cells in terms of cluster transformations of the cluster seeds corresponding to these cells, and show that these generalized B\"{a}cklund-Darboux transformations preserve the Hamiltonian flows generated by the restriction of the trace function of any representation of the simple Lie group. 

At the same time, we will describe the factorization mapping on conjugation quotients of Coxeter double Bruhat cells. By specializing to the case where the Coxeter elements in the ordered pair are equal to each other, we draw an explicit connection between the discrete dynamics of the factorization mapping and that of the evolutions of the $Q$-systems of simply-laced or twisted affine type $\neq A_{2r}^{(2)}$, thereby allowing us to relate the Coxeter-Toda Hamiltonians for the simple Lie groups to the conserved quantities of $Q$-systems of simply-laced or twisted affine type $\neq A_{2r}^{(2)}$.

Finally, we will develop network formulations of Coxeter-Toda Hamiltonians for simple classical Lie groups, and use the network formulations of these Coxeter-Toda Hamiltonians to obtain combinatorial formulas for Coxeter-Toda Hamiltonians of classical types $A_r, B_r, C_r$ and $D_r$, or equivalently, conserved quantities of certain $Q$-systems of affine types $A_r^{(1)}, D_{r+1}^{(2)}, A_{2r-1}^{(2)}$ and $D_r^{(1)}$ respectively.

\subsection{Outline of the paper}\label{Section 1.2}

The rest of the paper is organized as follows. 

In Section \ref{Section2}, we will review the basic notions that we will need throughout the paper. In particular, we will first review some basic definitions and results concerning finite-dimensional simple Lie groups, as well as the Poisson-Lie structure on these Lie groups. Subsequently, we will review the basic definitions and properties of a skew-symmetrizable cluster algebra of geometric type, where the treatment of the material will follow from that of \cite{FG09, Williams15}. Finally, we will review the combinatorial models of the one-parameter subgroups, as well as elements of the Cartan subgroup of a simple classical Lie group, given in terms of elementary chips by Fomin and Zelevinsky \cite{FG00} for type $A$, and by Yang \cite{Yang12} for types $B$, $C$ and $D$. These elementary chips are the building blocks of network representations \cite{FG00, GSV11, Yang12} of elements of a given simple classical Lie group, which in turn is central to the network formulations of Coxeter-Toda Hamiltonians for the simple classical Lie groups.

In Section \ref{Section3}, we will first review cluster structures on double Bruhat cells. More precisely, we will first recall the construction of rational cluster $\mathcal{X}$-coordinates on double Bruhat cells of adjoint Lie groups given in \cite{FG06}. Using the notion of amalgamation developed in \cite{Williams15}, which is a generalization of that developed in \cite{FG06}, we show that we can obtain a canonical cluster structure, and hence rational cluster $\mathcal{X}$-coordinates, on the conjugation quotient Coxeter double Bruhat cell of an adjoint Lie group from that on the corresponding Coxeter double Bruhat cell by amalgamation. 

Subsequently, we show that the initial cluster seeds arising from any two conjugation quotient Coxeter double Bruhat cells of an adjoint Lie group are mutation-equivalent to each other in Corollary \ref{3.11}. This fact was established in \cite{GSV11} in the case where the Lie group in question is the general/special linear group and mentioned in passing in \cite[Remark 3.5]{Williams15}. In addition, we also prove in Proposition \ref{3.12} that the rational cluster $\mathcal{X}$-coordinates on the conjugation quotient Coxeter double Bruhat cell of an adjoint Lie group can be lifted to rational coordinates on the corresponding conjugation quotient Coxeter double Bruhat cell of a simple Lie group, given in terms of monomials of cluster variables, and prove in Proposition \ref{3.13} that the cluster and Poisson structures on the conjugation quotient Coxeter double Bruhat cell of a simple Lie group are compatible with each other.

In Section \ref{Section4}, we will construct generalized B\"{a}cklund-Darboux transformations between two conjugation quotient Coxeter double Bruhat cells of a simple Lie group. Similar to \cite[Section 6.1]{GSV11}, we show that that the generalized B\"{a}cklund-Darboux transformations are built out of compositions of conjugation maps and refactorization relations, and these generalized B\"{a}cklund-Darboux transformations have a cluster algebraic interpretation from the cluster structures on these conjugation quotient Coxeter double Bruhat cells. We will then show that these transformations preserve Hamiltonian flows generated by the trace function of any representation of the simple Lie group, by showing in Theorem \ref{4.4} that these transformations preserve the Hamiltonians. In particular, we deduce that the family of Coxeter-Toda systems on a simple Lie group forms a single cluster integrable system, which generalizes \cite[Theorem 6.1]{GSV11} for the general/special linear groups to the other simple Lie groups. 

In Theorem \ref{4.5}, we will draw a connection between these transformations and the factorization mapping on conjugation quotient Coxeter double Bruhat cells of a simple Lie group, which generalizes \cite[Proposition 6.5]{GSV11}. We will then specialize to the case where the Coxeter elements in the ordered pair are equal to each other, and recall an explicit connection between the discrete dynamics of the factorization mapping and $Q$-system evolutions, with the exact correspondence given in Section \ref{Section2.5}, which we will need in the computation of Coxeter-Toda Hamiltonians in terms of regular functions of $Q$-system variables in Section \ref{Section5}.

Finally, in Section \ref{Section5}, we will develop network formulations of Coxeter-Toda Hamiltonians arising from trace functions of the fundamental representations or the exterior powers of the defining representation of a simple classical Lie group, using the network formulations of elements that we have developed in Section \ref{Section2.6}. We then use the network formulations of these Coxeter-Toda Hamiltonians, along with the explicit connection between $Q$-system evolutions and factorization mappings on conjugation quotient Coxeter double Bruhat cells recalled in Section \ref{Section4.2}, to obtain combinatorial formulas for Coxeter-Toda Hamiltonians of types $A_r, B_r, C_r$ and $D_r$ in terms of partition functions of hard particles on weighted graphs, or equivalently, conserved quantities of $Q$-systems of affine Dynkin types $A_r^{(1)}, D_{r+1}^{(2)}, A_{2r-1}^{(2)}$ and $D_r^{(1)}$ respectively, given in Theorems \ref{5.8}, \ref{5.11}, \ref{5.17}, \ref{5.23}, \ref{5.24} and \ref{5.27}. 

\subsection*{Acknowledgments}

This paper is largely based on the author's doctorate thesis of the same title while the author was a graduate student at the University of Illinois Urbana-Champaign, with changes made to update and clarify the references therein, and to further improve the exposition in this paper. The author would like to thank Michael Gekhtman for his illuminating course on Cluster Integrable Systems at the Spring School for the thematic program ``Faces of Integrability" at the Centre de Recherche Math\'{e}matique de l'Universit\'{e} de Montr\'{e}al, which provided the motivation for the first part of this paper. The author would also like to thank Rinat Kedem for introducing the problem of looking at the integrability and conserved quantities of $Q$-systems from the point of view of factorization mappings and network formulations to the author, which provided the motivation for the second part of this paper, as well as her guidance throughout the project. Lastly, the author would also like to thank Phillipe Di Francesco for his illuminating discussions throughout the project, and Alexander Yong for his helpful comments on the exposition in this paper. The author was supported by a graduate fellowship from A*STAR (Agency for Science, Technology and Research, Singapore), and this work was also supported in part by the US National Science Foundation (DMS-1802044).

\section{Preliminaries}\label{Section2}

Throughout this paper, we will use the following notation:
\begin{align*}
[m]_+&=\max(m,0),\\
[m,n]&=\{m,m+1,\ldots,n-1,n\},\quad\text{and}\\
\sgn(m)&=
\begin{cases}
0 & \text{if }m=0;\\
1 & \text{if }m>0;\\
-1 & \text{if }m<0.
\end{cases}
\end{align*}

\subsection{Basic Lie theoretic data}\label{Section2.1}

In this subsection, we will review some basic notions and definitions in Lie Theory, following \cite{Knapp02} as a reference. To this end, we will let $G$ be a simple complex Lie group of rank $r$, with Lie algebra $\mathfrak{g}$ and Cartan matrix $C$. We will also let $\mathfrak{h}$ be a Cartan subalgebra of $\mathfrak{g}$, $\mathfrak{g}=\mathfrak{n}_-\oplus\mathfrak{h}\oplus\mathfrak{n}_+$ be the Cartan decomposition of $\mathfrak{g}$. In addition, we let $H$ be the Cartan subgroup of $G$ whose Lie algebra is given by the Cartan subalgebra $\mathfrak{h}$ of $\mathfrak{g}$.

Next, we will fix a set of simple roots $\{\alpha_1,\ldots,\alpha_r\}\subseteq\mathfrak{h}^*$ and a set of corresponding simple coroots $\{\alpha_1^{\vee},\ldots,\alpha_r^{\vee}\}\subseteq\mathfrak{h}^*$ satisfying $\alpha_j(\alpha_i^{\vee})=C_{i,j}$ for all $i\in[1,r]$. We let $d_1',\ldots,d_r'$ be positive coprime integers that satisfy $d_i'C_{i,j}=d_j'C_{j,i}$ for all $i,j\in[1,r]$. Then it follows that for all $i\in[1,r]$, we have
\begin{equation*}
d_i'=
\begin{cases}
1 & \text{if }\alpha_i\text{ is a short root of }\mathfrak{g},\\
2 & \text{if }\mathfrak{g} \text{ is of types }B_r,C_r\text{ or }F_4,\text{ and }\alpha_i\text{ is a long root of }\mathfrak{g},\\
3 & \text{if }\mathfrak{g} \text{ is of type }G_2\text{ and }\alpha_i\text{ is a long root of }\mathfrak{g}.
\end{cases}
\end{equation*}
Finally, we will also fix a set $\{e_1,\ldots,e_r,e_{-1},\ldots,e_{-r}\}$ of positive and negative Chevalley generators for $\mathfrak{g}$. These Chevalley generators generate one-parameter subgroups
\begin{equation}\label{eq:2.1}
E_i(t):=\exp(te_i),\quad E_{-i}(t):=\exp(te_{-i})
\end{equation}
for all $i\in[1,r]$ and $t\in\mathbb{C}$. For latter convenience, we will also denote $E_i=E_i(1)$ and $E_{-i}=E_{-i}(1)$ for all $i\in[1,r]$.

Next, we let $P$ denote the weight lattice of $\mathfrak{g}$, and $P^+\subseteq P$ the set of dominant integral weights. The weight lattice $P$ can be identified with the character lattice $\Hom(H,\mathbb{C}^\times)$ of $H$, with basis given by the set $\{\omega_1,\ldots,\omega_r\}\subseteq\mathfrak{h}^*$ of fundamental weights of $\mathfrak{g}$, defined by $\omega_i(\alpha_j^{\vee})=\delta_{i,j}$ for all $i,j\in[1,r]$. The irreducible highest weight $\mathfrak{g}$-modules are parameterized by $\lambda\in P^+$ and are denoted by $V({\lambda})$. For any given $\lambda=\sum_{i=1}^r\ell_i\omega_i$, the $\mathfrak{g}$-module $V(\lambda)$ is generated by a highest weight vector $v_{\lambda}$, with relations given by
\begin{equation*}
e_iv_{\lambda}=0,\quad hv_{\lambda}=\lambda(h)v_{\lambda},\quad e_{-i}^{\ell_i+1}v_{\lambda}=0
\end{equation*}
for all $i\in[1,r]$ and $h\in\mathfrak{h}$.

The character lattice $\Hom(H,\mathbb{C}^{\times})$ of $H$ has a dual cocharacter lattice $\Hom(\mathbb{C}^{\times},H)$, and can be identified with the coroot lattice of $\mathfrak{g}$, with a basis given by the set $\{\alpha_1^{\vee},\ldots,\alpha_r^{\vee}\}\subseteq\mathfrak{h}^*$ of simple coroots of $\mathfrak{g}$. For each $\alpha^{\vee}$ in the coroot lattice, $\alpha^{\vee}$ defines a cocharacter $t\mapsto t^{\alpha^{\vee}}$. For latter convenience, we will denote $D(t_1,\ldots,t_r)=\prod_{i=1}^r t_i^{\alpha_i^{\vee}}$.

\subsection{Double Bruhat cells}\label{Section2.2}

In this subsection, we will recall some basic facts concerning double Bruhat cells from \cite{FZ99}. To this end, we let $W$ be the Weyl group of $G$ generated by simple reflections $s_1,\ldots,s_r$, and $B_+$ (respectively $B_-$) is the positive (respectively negative) Borel subgroup of $G$ with respect to the Cartan decomposition $\mathfrak{g}=\mathfrak{n}_-\oplus\mathfrak{h}\oplus\mathfrak{n}_+$ of $\mathfrak{g}$. We recall that $G$ admits two Bruhat decompositions 
\begin{equation*}
G=\bigsqcup_{u\in W}B_+\dot{u}B_+=\bigsqcup_{v\in W}B_-\dot{v}B_-
\end{equation*}
with respect to the opposite Borel subgroups $B_+$ and $B_-$ respectively, where $\dot{u}$ and $\dot{v}$ are representatives of $u$ and $v$ respectively in $G$. 

For any $u,v\in W$, the double Bruhat cell $G^{u,v}$ is defined as \cite{FZ99}
\begin{equation*}
G^{u,v}=B_+\dot{u}B_+\cap B_-\dot{v}B_-.
\end{equation*}
By the above Bruhat decompositions of $G$, it follows that that $G$ admits a decomposition
\begin{equation*}
G=\bigsqcup_{u,v\in W}G^{u,v}
\end{equation*}
into double Bruhat cells $G^{u,v}$.

Next, we will discuss parameterizations of these double Bruhat cells $G^{u,v}$. To this end, we will need to recall some definitions. A word for $u\in W$ is a sequence $\mathbf{i}=(i_1,\cdots,i_m)$ of elements in $[1,r]$ that satisfies $u=s_{i_1}s_{i_2}\cdots s_{i_m}$. We say that $\mathbf{i}$ is reduced if there does not exist another word $\mathbf{i}'=(i_1',\cdots,i_n')$ for $u$, such that $n<m$. 

Similarly, a double reduced word $\mathbf{i}=(i_1,\ldots,i_m)$ for $(u,v)\in W\times W$ is a shuffle of a reduced word for $u$ written in the indices $\{-1,\ldots,-r\}$ and a reduced word for $v$ written in the indices $\{1,\ldots,r\}$, that is to say, there exist disjoint sets $\{j_1<\cdots<j_{\ell}\}$ and $\{k_1<\ldots<k_n\}$, such that $\{j_1<\cdots<j_{\ell}\}\cup\{k_1<\ldots<k_n\}=[1,m]$, and $(i_{j_1},\ldots,i_{j_{\ell}})$ and $(i_{k_1},\ldots,i_{k_n})$ are reduced words for $u$ and $v$ written in the indices $\{-1,\ldots,-r\}$ and $\{1,\ldots,r\}$ respectively. We say that $\mathbf{i}$ is unmixed if $\{j_1<\cdots<j_{\ell}\}=[1,\ell]$, and $\{k_1,\ldots,k_n\}=[\ell+1,m]$.

\begin{theorem}\cite[Theorem 1.2]{FZ99}\label{2.1}
Let $u,v\in W$, and $\mathbf{i}=(i_1,\ldots,i_m)$ be a double reduced word for $(u,v)$. Then the map $x_{\mathbf{i}}:H\times\mathbb{C}^m\to G$, given by
\begin{equation*}
(h,a_1,\ldots,a_m)\mapsto hE_{i_1}(a_1)\cdots E_{i_m}(a_m)
\end{equation*}
restricts to a biregular isomorphism between $x_{\mathbf{i}}:H\times(\mathbb{C}^{\times})^m$ and a Zariski open subset of the double Bruhat cell $G^{u,v}$.
\end{theorem}

In light of Theorem \ref{2.1}, it follows that a generic element $g\in G^{u,v}$ can be factorized in the following form \cite{FZ99}
\begin{equation}\label{eq:2.2}
g=D(t_1,\ldots,t_r)E_{i_1}(a_1)\cdots E_{i_m}(a_m)
\end{equation}
for some $a_1\ldots,a_m,t_1,\ldots,t_r\in\mathbb{C}^{\times}$, which we call factorization parameters, and we can view these factorization parameters as rational coordinates on $G^{u,v}$, as $G^{u,v}$ is birationally isomorphic to $(\mathbb{C}^{\times})^{m+r}$.

Next, we recall that the Cartan subgroup $H$ acts on $G$ by conjugation. For any $h\in H$, we denote the action of $h$ on $G$ by conjugation by $C_h$, that is, $C_h(g)=hgh^{-1}$ for all $g\in G$. For any $u,v\in W$, the conjugation action of $H$ on $G$ preserves the double Bruhat cell $G^{u,v}$ of $G$, and so we may define the quotient $G^{u,v}/H$ of $G^{u,v}$ by the conjugation action of $H$ on $G$. 

In the case where $u$ and $v$ are Coxeter elements of $W$, it follows that for each $k\in[1,r]$, there exist unique integers $k_-<k_+$ in $[1,2r]$ that satisfy $|i_{k_-}|=k=|i_{k_+}|$, and the rational functions $\overline{a}_k:=a_{k_-}a_{k_+}$ and $t_k$ on $G^{u,v}$ are preserved under the conjugation action of $H$ on $G$ for all $k\in[1,r]$. In particular, it follows from \eqref{eq:2.2} that in the case where $\mathbf{i}$ is an unmixed double reduced word for $(u,v)$, a generic element $\overline{g}$ of the conjugation quotient Coxeter double Bruhat cell $G^{u,v}/H$ can be written as
\begin{equation}\label{eq:2.3}
\overline{g}=E_{i_1}(1)\cdots E_{i_r}(1)D(t_1,\ldots,t_r)E_{i_{r+1}}(c_{i_{r+1}})\cdots E_{i_{2r}}(c_{i_{2r}}), 
\end{equation}
where $c_1,\ldots,c_r,t_1,\ldots,t_r\in\mathbb{C}^{\times}$, up to conjugation by an element of $H$. 

\subsection{Coxeter-Toda systems}\label{Section2.3}

In this subsection, we will review the standard Poisson-Lie structure on $G$, and the notion of a Coxeter-Toda system, which is a completely integrable system formed by the restriction of a conjugation-invariant function to the conjugation quotient of a Coxeter double Bruhat cell $G^{u,v}/H$. Our treatment of the material here will follow mostly that of \cite{Drinfeld88, CP94, KS96}.

To begin, we recall that for any Poisson manifold $M$ with Poisson bracket $\{\cdot,\cdot\}$, there exists a unique $P\in\bigwedge^2(TM)$, which we call the Poisson bivector associated to the Poisson bracket $\{\cdot,\cdot\}$ on $M$, that satisfies $\{f,g\}=\langle df\otimes dg,P\rangle$ for all $f,g\in C^{\infty}(M)$.

A Poisson-Lie group is a Lie group $G$, equipped with a Poisson bracket $\{\cdot,\cdot\}$, such that the multiplication map $\mu:G\times G\to G$ given by $(g,h)\mapsto gh$, is Poisson, where the Poisson structure on $G\times G$ is given by the product Poisson structure. This implies that if $G$ is a Lie group equipped with a Poisson bracket $\{\cdot,\cdot\}$, then $\{\cdot,\cdot\}$ defines a Poisson-Lie structure on $G$ if and only if we have 
\begin{equation*}
\{f_1,f_2\}(gh)=\{f_1\circ \lambda_g,f_2\circ \lambda_g\}(h)+\{f_1\cdot \rho_h,f_2\cdot \rho_h\}(g)
\end{equation*}
for all $f_1,f_2\in C^{\infty}(G)$ and $g,h\in G$, where $\lambda_g,\rho_g:G\to G$ are the left and right translation maps on $G$ by $g\in G$ defined by $\lambda_g=\mu(g,\cdot)$ and $\rho_g=\mu(\cdot,g)$. If $P$ is the Poisson bivector associated to the Poisson bracket $\{\cdot,\cdot\}$ on $G$, then the above equation is equivalent to 
\begin{equation*}
P_{gh}=(d(\lambda_g)_h\otimes d(\lambda_g)_h)(P_h)+(d(\rho_h)_g\otimes d(\rho_h)_g)(P_g).
\end{equation*}
In this case, we say that the Poisson bivector $P$ is multiplicative. 

Our next step is to recall a sufficient condition for which a Poisson bracket $\{\cdot,\cdot\}$ on a Lie group $G$ defines a Poisson-Lie structure on $G$. To begin, we let $\mathfrak{g}$ be a Lie algebra. We say that an element $r=\sum_{i=1}^n a_i\otimes b_i$ of $\mathfrak{g}\otimes\mathfrak{g}$ is a quasi-triangular $r$-matrix if $r$ satisfies the following classical Yang-Baxter equation:
\begin{equation*}
[r_{12},r_{13}]+[r_{12},r_{23}]+[r_{13},r_{23}]=0,
\end{equation*}
where
\begin{align*}
[r_{12},r_{13}]&=\sum_{i,j=1}^n [a_i,a_j]\otimes b_i\otimes b_j,\\
[r_{12},r_{23}]&=\sum_{i,j=1}^n a_i\otimes [b_i,a_j]\otimes b_j,\quad\text{and}\\
[r_{13},r_{23}]&=\sum_{i,j=1}^n a_i\otimes a_j\otimes [b_i\otimes b_j].
\end{align*}

\begin{proposition}\cite{KS96}\label{2.2}
Let $G$ be a connected Lie group with Lie algebra $\mathfrak{g}$, and let $r\in\mathfrak{g}\otimes\mathfrak{g}$ be a quasi-triangular $r$-matrix. We define an element $P^r\in\bigwedge^2(TG)$ by 
\begin{equation*}
P_g^r=(d(\rho_g)_e\otimes d(\rho_g)_e)(r)-(d(\lambda_g)_e\otimes d(\lambda_g)_e)(r). 
\end{equation*}
Then the complex bilinear map $\{\cdot,\cdot\}:C^{\infty}(G)\times C^{\infty}(G)\to C^{\infty}(G)$ defined by $\{f_1,f_2\}=\langle df_1\otimes df_2,P^r\rangle$ for all $f_1,f_2\in C^{\infty}(G)$ defines a Poisson-Lie structure on $G$.
\end{proposition}

We are now ready to define the standard Poisson-Lie structure on $G$. To this end, we will proceed in steps. Firstly, we let 
\begin{equation*}
SL_2=\left\{\begin{pmatrix}a&b\\c&d\end{pmatrix}:a,b,c,d\in\mathbb{C},ad-bc=1\right\} 
\end{equation*}
be the Lie group of $2\times 2$ complex matrices with determinant $1$, and $\mathfrak{sl}_2$ be the Lie algebra of $SL_2$ with basis $\{e_+,e_-,\alpha^{\vee}\}$. The Lie algebra $\mathfrak{sl}_2$ has a symmetric, invariant bilinear form $(\cdot,\cdot)$ which is unique up to fixing the scalar $k:=\frac{4}{(\alpha^{\vee},\alpha^{\vee})}$. The Casimir element $\Omega_k\in\mathfrak{sl}_2\otimes\mathfrak{sl}_2$ corresponding to this bilinear form $(\cdot,\cdot)$ can be expressed explicitly as 
\begin{equation*}
\Omega_k=\frac{k}{4}(\alpha^{\vee}\otimes\alpha^{\vee}+2e_+\otimes e_-+2e_-\otimes e_+). 
\end{equation*}
By letting $\Omega_0=\frac{k}{4}\alpha^{\vee}\otimes\alpha^{\vee}$, $\Omega_{+-}=\frac{k}{2}e_+\otimes e_-$ and $\Omega_{-+}=\frac{k}{2}e_-\otimes e_+$, it follows that the Casimir element $\Omega_k$ can be written as $\Omega_k=\Omega_{+-}+\Omega_0+\Omega_{-+}$. We let $r=\Omega_0+2\Omega_{+-}=\frac{k}{4}(\alpha^{\vee}\otimes\alpha^{\vee}+4e_+\otimes e_-)$. Then $r$ is a solution of the following classical Yang-Baxter equation, and we call $r$ the standard quasi-triangular $r$-matrix. The Poisson-Lie structure on $SL_2$ that is defined by the multiplicative Poisson bivector 
\begin{equation*}
P_g^r=(d(\rho_g)_e\otimes d(\rho_g)_e)(r)-(d(\lambda_g)_e\otimes d(\lambda_g)_e)(r)   
\end{equation*}
is then called the standard Poisson-Lie structure on $SL_2$. This Poisson-Lie structure on $SL_2$ is uniquely characterized by the following Poisson brackets between the coordinate functions $a,b,c,d$ on $SL_2$:
\begin{align*}
\{b,a\}&=\frac{k}{2}ab,\quad\{c,a\}=\frac{k}{2}ac,\quad\{d,a\}=kbc,\\
\{b,c\}&=0,\quad\{d,b\}=\frac{k}{2}bd,\quad\{d,c\}=\frac{k}{2}cd.
\end{align*}
We will write $SL_2^{(k)}$ to indicate that the Lie group $SL_2$ is equipped with the above Poisson-Lie structure, reflecting the dependence on the scalar $k$. 

For the general case, there exists a unique symmetric, invariant bilinear form $(\cdot,\cdot)$ on $\mathfrak{g}$ that satisfies $d_i'=\frac{4}{(\alpha_i^{\vee},\alpha_i^{\vee})}$ for all $i\in[1,r]$. Let us denote the Casimir element corresponding to this bilinear form $(\cdot,\cdot)$ by $\Omega\in\mathfrak{g}\otimes\mathfrak{g}$, and write $\Omega=\Omega_{+-}+\Omega_0+\Omega_{-+}$, where $\Omega_0\in\mathfrak{h}\otimes\mathfrak{h}$, $\Omega_{+-}\in\mathfrak{n}_+\otimes\mathfrak{n}_-$ and $\Omega_{-+}\in\mathfrak{n}_+\otimes\mathfrak{n}_-$. We then define $r=\Omega_0+2\Omega_{+-}$. Then $r$ is a solution of the classical Yang-Baxter equation, and we call $r$ the standard quasi-triangular $r$-matrix. The Poisson-Lie structure on $G$ that is defined by the multiplicative Poisson bivector $P^r$ is then called the standard Poisson-Lie structure on $G$, and the standard Poisson-Lie structure on $G$ is uniquely characterized by the property that the canonical inclusion map $\phi_i:SL_2^{(d_i')}\to G$ is Poisson for all $i\in[1,r]$. Here, the inclusion map $\phi_i:SL_2^{(d_i')}\to G$ is defined by
\begin{equation*}
\phi_i\begin{pmatrix}1&t\\0&1\end{pmatrix}=E_i(t),\quad\phi_i\begin{pmatrix}1&0\\t&1\end{pmatrix}=E_{-i}(t).
\end{equation*}

Next, we will collect some results concerning the standard Poisson-Lie structure on $G$:

\begin{proposition}\cite{STS85}\label{2.3}
The set $\mathbb{C}[G]^G$ of conjugation-invariant functions on $G$ form a Poisson commutative subalgebra of the algebra of functions on $G$.
\end{proposition}

\begin{proposition}\cite{HKKR00, KZ02, Reshetikhin03, Yakimov02}\label{2.4}
The double Bruhat cells $G^{u,v}$ of $G$ are Poisson subvarieties of $G$, and the Poisson brackets between the factorization parameters $a_1\ldots,a_m,t_1,\ldots,t_r$ in equation \eqref{eq:2.2} are given by
\begin{align}
\{a_j,a_k\}&=\frac{\epsilon_jd_{|i_j|}'C_{|i_j|,|i_k|}}{2}a_ja_k,\quad j<k,\label{eq:2.4}\\
\{a_j,t_k\}&=\frac{d_{|i_j|}'\delta_{|i_j|,k}}{2}a_jt_k,\label{eq:2.5}\\
\{t_j,t_k\}&=0,\label{eq:2.6}
\end{align}
where $\epsilon_j=1$ if $i_j=1$, and $\epsilon_j=-1$ otherwise. 
\end{proposition}

Our next goal is to define Coxeter-Toda systems. To begin, we recall that the conjugation action of the Cartan subgroup $H$ on $G$ is Poisson with respect to the standard Poisson-Lie structure on $G$, that is, the map $C_h:G\to G$ is Poisson for all $h\in H$. Consequently, it follows that the standard Poisson-Lie structure on $G$ induces a Poisson structure on the conjugation quotient double Bruhat cell $G^{u,v}/H$ for any $u,v\in W$. We will write $\{\cdot,\cdot\}_{u,v}$ to denote the Poisson brackets on $G^{u,v}/H$ that is induced from the standard Poisson-Lie structure on $G$.

In general, the restriction of a conjugation-invariant function on $G$ to a given conjugation quotient double Bruhat cell $G^{u,v}/H$ need not form a completely integrable system on $G^{u,v}/H$ \cite{Reshetikhin03}, in that the maximal number of algebraically independent restricted conjugation-invariant functions on $G^{u,v}/H$ is less than half the dimension of $G^{u,v}/H$. However, they do in the case where $u$ and $v$ are Coxeter elements of $W$, in which case the maximal number of algebraically independent restricted conjugation-invariant functions on $G^{u,v}/H$ is equal to $r$. In this case, we call the restriction of a conjugation-invariant function to $G^{u,v}/H$ a Coxeter-Toda Hamiltonian, and the corresponding completely integrable system a Coxeter-Toda system. 

Given a Coxeter-Toda system corresponding to the conjugation quotient Coxeter double Bruhat cell $G^{u,v}/H$, there are full sets of algebraically independent integrals of motion, which are given by Coxeter-Toda Hamiltonians arising from trace functions of representations of $G$ \cite{Kostant79}. One possible set is given by the restrictions $H_k^{u,v}$ of the conjugation-invariant function $H_k(X)=\tr(\pi_k(X))$ to the conjugation quotient Coxeter double Bruhat cell $G^{u,v}/H$ for all $k\in[1,r]$, where $\pi_k:G\to SL(V(\omega_k))$ denotes the $k$-th fundamental representation of $G$. Another set is given by the restrictions $f_k^{u,v}$ of the conjugation-invariant function $f_k(X)=\tr(\bigwedge^i(X))$ to the conjugation quotient Coxeter double Bruhat cell $G^{u,v}/H$.

\subsection{Cluster algebras}\label{Section2.4}

In this subsection, we shall review some basic definitions and properties of a skew-symmetrizable cluster algebra of geometric type. As in \cite{Williams13,Williams15}, the treatment of material given in this subsection follows mostly from \cite{FG09, Williams15}, where while the definitions here are equivalent to that of \cite{BFZ05,FZ07}, the details will appear somewhat differently here.

\begin{definition}[Seed]\label{2.5}
A (skew-symmetrizable) seed $\Sigma$ consists of the following data:
\begin{enumerate}
\item An indexing set $I=I_f\sqcup I_u$ where $I_f$ (respectively $I_u$) is the set of frozen (respectively unfrozen) indices;
\item An skew-symmetrizable exchange matrix $B$ with rows and columns indexed by $I$ and $B_{i,j}\in\mathbb{Z}$ if either $i\in I_u$ or $j\in I_u$, and skew-symmetrizers $d_i\in\mathbb{Z}_{>0}$ that satisfy $d_iB_{i,j}=-d_jB_{j,i}$ for all $i,j\in I$. 
\end{enumerate}
\end{definition}

\begin{definition}[Cluster mutation]\label{2.6}
Let $\Sigma$ be a seed with indexing set $I$. For any unfrozen index $k\in I_u$, the mutation of $\Sigma$ at the index $k$ is the seed $\mu_k(\Sigma)$, defined as follows: it has the same indexing set, frozen and unfrozen sets, and skew-symmetrizers as $\Sigma$. Its exchange matrix $\mu_k(B)$ is given by: 
\begin{equation}\label{eq:2.7} 
\mu_k(B)_{i,j}=
\begin{cases}
-B_{i,j} & \text{if }i=k\text{ or }j=k;\\
B_{i,j}+\sgn(B_{i,k})[B_{i,k}B_{k,j}]_+ & \text{otherwise}.
\end{cases}\\
\end{equation}
We say that two seeds $\Sigma$ and $\Sigma'$ are mutation equivalent if $\Sigma'$ is obtained from $\Sigma$ by a sequence of mutations.
\end{definition}

\begin{definition}[Cluster variables and $\mathcal{X}$-coordinates]\label{2.7}
To a seed $\Sigma$, we associate two Laurent polynomial rings $\mathbb{C}[A_i^{\pm1}]$ and $\mathbb{C}[X_i^{\pm1}]$, whose generators are indexed by $I$ and are referred to as cluster variables and $\mathcal{X}$-coordinates respectively. These are the coordinate rings of two algebraic tori $\mathcal{A}_{\Sigma}$ and $\mathcal{X}_{\Sigma}$ respectively. There is a canonical map $p_{\Sigma}:\mathcal{A}_{\Sigma}\to\mathcal{X}_{\Sigma}$ defined by 
\begin{equation}\label{eq:2.8}
p_{\Sigma}^*(X_i)=\prod_{j\in I}A_j^{B_{j,i}}.
\end{equation}
The torus $\mathcal{X}_{\Sigma}$ has a log-canonical Poisson structure \cite{GSV03}, given by 
\begin{equation}\label{eq:2.9}
\{X_i,X_j\}=d_iB_{i,j}X_iX_j
\end{equation}
for all $i,j\in I$.
\end{definition}

\begin{definition}[Cluster transformations]\label{2.8}
To each mutation $\mu_k:\Sigma\to\Sigma'$ of seeds $\Sigma$ and $\Sigma'$, we associate a pair of rational maps between the corresponding tori, called cluster transformations and also denoted $\mu_k$. These satisfy \footnote{Our notation follows \cite{FZ07} and differs from that in \cite{Williams15}, which uses $B^T$ instead of $B$. This difference corresponds to $B_{k,j}$ instead of $B_{j,k}$ in equation \eqref{eq:2.10}, and $B_{i,k}$ instead of $B_{k,i}$ in equation \eqref{eq:2.11}.}
\begin{center}
\begin{tikzpicture} [node distance=2cm]
  \node (X) {$\mathcal{A}_{\Sigma}$};
  \node (Y) [right of=X] {$\mathcal{A}_{\Sigma'}$};
  \node (W) [below of=X] {$\mathcal{X}_{\Sigma}$};
  \node (Z) [below of=Y] {$\mathcal{X}_{\Sigma'}$};
  \draw[->, dotted] (X) to node [above] {$\mu_k$} (Y);
  \draw[->] (Y) to node [right] {$p_{\Sigma'}$} (Z);
  \draw[->] (X) to node [left] {$p_{\Sigma}$} (W);
  \draw[->, dotted] (W) to node [below] {$\mu_k$} (Z);
\end{tikzpicture}
\end{center}
and are explicitly defined by
\begin{align} 
\mu_k^*(A_i')&=
\begin{cases}
A_i & \text{if } i\neq k,\\
A_k^{-1}\left(\prod_{j=1}^m A_j^{[B_{j,k}]_+}+\prod_{j=1}^m A_j^{[-B_{j,k}]_+}\right) & \text{if } i=k,
\end{cases}\label{eq:2.10}\\
\mu_k^*(X_i')&=
\begin{cases}
X_iX_k^{[B_{k,i}]_+}(1+X_k)^{-B_{k,i}} & \text{if } i\neq k,\\
X_k^{-1} & \text{if } i=k.
\end{cases}\label{eq:2.11}
\end{align}
It is clear that the cluster transformation $\mu_k:\mathcal{X}_{\Sigma}\to\mathcal{X}_{\Sigma'}$ is a Poisson map with respect to the Poisson bracket defined in Definition \ref{2.7}.
\end{definition}

\begin{definition}[$\sigma$-periods]\cite{Williams15}\label{2.9}
Let $\Sigma$ be a seed, $\hat{\mu}=\mu_{i_1}\circ\cdots\circ\mu_{i_k}$ be a sequence of mutations of $\Sigma$, and $\sigma$ be a permutation of $I$. We say that $\hat{\mu}$ is a $\sigma$-period of $\Sigma$ if $\hat{\mu}(B)_{i,j}=B_{\sigma(i),\sigma(j)}$ for all $i,j\in I$. Any given $\sigma$-period $\hat{\mu}$ of $\Sigma$ induces birational maps $\hat{\mu}_{\sigma}:\mathcal{A}_{\Sigma}\to\mathcal{A}_{\Sigma'}$ and $\hat{\mu}_{\sigma}:\mathcal{X}_{\Sigma}\to\mathcal{X}_{\Sigma'}$, given by
\begin{equation*}
\hat{\mu}_{\sigma}^*(A_i')=(\mu_{i_1}\circ\cdots\circ\mu_{i_k})^*(A_{\sigma^{-1}(i)}),\quad\hat{\mu}_{\sigma}^*(X_i')=(\mu_{i_1}\circ\cdots\circ\mu_{i_k})^*(X_{\sigma^{-1}(i)})
\end{equation*}
for all $i\in I$.
\end{definition}

\subsection{\textit{Q}-systems}\label{Section2.5}

In this subsection, we will recall the definition of $Q$-systems, as well as its connection to cluster algebras. While the $Q$-systems of type $\neq A_{2r}^{(2)}$ admit cluster algebraic formulations \cite{Kedem08, DFK09, Williams15}, we will restrict our attention to $Q$-systems of untwisted simply laced types and twisted types $\neq A_{2r}^{(2)}$ for the ease of our following presentation.

To begin, we note that for each simple finite-dimensional Lie algebra $\mathfrak{g}$ of type $Y_r$, there exists an affine Lie algebra $\widehat{\mathfrak{g}}^{\sigma}$ of type $X_m^{(\kappa)}\neq A_{2r}^{(2)}$, such that:
\begin{enumerate}
\item The simple finite-dimensional Lie algebra of type $X_m$ is simply-laced,
\item $\sigma$ is a Dynkin diagram automorphism of $\mathfrak{g}$ of order $\kappa$, and
\item The finite Dynkin type of the Dynkin diagram that is obtained from the affine Dynkin diagram of type $X_m^{(\kappa)}$ by removing the $0$-th vertex is precisely $Y_r$.
\end{enumerate}
Let $\{Q_{\alpha,k}\mid\alpha\in [1,r],k\in\mathbb{Z}\}$ be a family of commuting variables. Then the $Q$-system relations of type $X_m^{(\kappa)}$ are given by \cite{KR87, KNS94, KS95}:
\begin{equation}\label{eq:2.12}
Q_{\alpha,k+1}Q_{\alpha,k-1}=Q_{\alpha,k}^2-\prod_{\beta\neq\alpha}Q_{\beta,k}^{-C_{\beta,\alpha}}
\end{equation}
for all $\alpha\in[1,r]$ and $k\in\mathbb{Z}$. The following table lists the finite Dynkin type $Y_r$ of the simple complex Lie group $G$ and the corresponding affine Dynkin type $X_m^{(\kappa)}$ of the $Q$-system:
\begin{center}
  \begin{tabular}{|c|c|c|c|c|c|c|c|}
  \hline
  $Y_r$ & $A_r$ & $B_r$ & $C_r$ & $D_r$ & $E_{6,7,8}$ & $F_4$ & $G_2$ \\ \hline
  $X_m^{(\kappa)}$ & $A_r^{(1)}$ & $D_{r+1}^{(2)}$ & $A_{2r+1}^{(2)}$ & $D_r^{(1)}$ & $E_{6,7,8}^{(1)}$ & $E_6^{(2)}$ & $D_4^{(3)}$ \\
  \hline
  \end{tabular}
\end{center}

We note that any solution of the $X_m^{(\kappa)}$ $Q$-system \eqref{eq:2.12} could be expressed as a function of initial data consisting of special subsets of $2r$ elements from $\{Q_{\alpha,k}\mid \alpha\in [1,r],k\in\mathbb{Z}\}$. An important example of a valid set of initial data for the $X_m^{(\kappa)}$ $Q$-system is the components of the following vector 
\begin{equation}\label{eq:2.13}
\mathbf{y}_k=(Q_{\alpha,k},Q_{\alpha,k+1})_{\alpha\in [1,r]},
\end{equation} 
and we call $\mathbf{y}_k$ the $k$-th fundamental initial data for the $X_m^{(\kappa)}$ $Q$-system. More generally, if $\vec{m}=(m_i)_{i=1}^r$ is a vector with integer components, then the components of $\mathbf{y}_{\vec{m}}=(Q_{\alpha,m_{\alpha}},Q_{\alpha,m_{\alpha}+1})_{\alpha=1}^r$ constitute a valid set of initial data for the $X_m^{(\kappa)}$ $Q$-system if and only if $\vec{m}$ is a Motzkin path, that is, we have $|m_i-m_{i+1}|\leq 1$ for all $i\in[1,r-1]$. 
 
It was shown in \cite{Kedem08, DFK09, Williams15} that the $X_m^{(\kappa)}$ $Q$-system relations can be realized as cluster algebra mutations. While the cluster algebra mutation relations are written without any subtractions \cite{FZ02}, the RHS of the $Q$-system relations \eqref{eq:2.12} is written with a subtraction. To avoid the use of cluster algebras with coefficients (as in the Appendix of \cite{DFK09}), we would need to normalize our $Q$-system relations of accordingly. Following \cite[Lemma 2.1]{DFK09} and \cite[Proposition 4.3]{Williams15}, we let $\mu_{\alpha}=\sum_{\beta\in [1,r]}(C^{-1})_{\beta,\alpha}$, $\epsilon_{\alpha}=e^{i\pi\mu_{\alpha}}$ and $R_{\alpha,k}=\epsilon_{\alpha}Q_{\alpha,k}$ for all $\alpha\in[1,r]$ and $k\in\mathbb{Z}$. Then it follows that the normalized twisted $Q$-system variables $R_{\alpha,k}$ satisfy the following relations: 
\begin{equation}\label{eq:2.14}
R_{\alpha,k+1}R_{\alpha,k-1}=R_{\alpha,k}^2+\prod_{\beta\neq\alpha}R_{\beta,k}^{-C_{\beta,\alpha}}.
\end{equation}
We will thereby refer to \eqref{eq:2.14} as the normalized $X_m^{(\kappa)}$ $Q$-system relations.

The cluster algebra associated to the normalized $Q$-system \eqref{eq:2.14} of type $X_m^{(\kappa)}$ is defined from the initial seed $\Sigma_C$, where the indexing set $I$ is given by $[1,2r]$, the exchange matrix $B$ of $\Sigma_C$ is given by
\begin{equation}\label{eq:2.15}
B
=\begin{pmatrix}
0 & -C \\
C & 0
\end{pmatrix},
\end{equation}
where the row and column indices are ordered $1,\ldots,2r$, and the skew-symmetrizers $d_1,\ldots,d_{2r}$ are defined by $d_i=d_i'=d_{r+i}$ for all $i\in[1,r]$. In addition, we let $\widehat{\mu}=\mu_1\circ\mu_2\circ\cdots\circ\mu_r$ be the composition of $r$ cluster mutations $\mu_1,\mu_2,\ldots,\mu_r$, and $\sigma$ be the permutation of $[1,2r]$, defined by 
\begin{equation*}
\sigma(i)=
\begin{cases}
i+r & \text{if }i\leq r,\\
i-r & \text{if }i>r
\end{cases}
\end{equation*}
for all $i\in[1,2r]$. It was shown in \cite[Proposition 3.2]{Williams15} that $\hat{\mu}$ is a $\sigma$-period of $\Sigma_C$. In fact, more is true:

\begin{theorem}\cite{Kedem08, DFK09, Williams15}\label{2.10}
Let $\Sigma_C$ be the seed as defined above, and $A_1,\ldots$, $A_{2r}$ be the cluster variables of $\mathcal{A}_{\Sigma_C}$. Also, we let $A_{i,k}=(\widehat{\mu}_{\sigma}^*)^k(A_i)$ for any $k\in\mathbb{Z}$ and $i\in[1,2r]$. Then we have $A_{i,k}=A_{i+r,k-1}$, and 
\begin{equation}\label{eq:2.16}
A_{i,k+1}A_{i,k-1}=A_{i,k}^2+\prod_{j\neq i}A_{j,k}^{-C_{j,i}}
\end{equation}
for all $k\in\mathbb{Z}$ and $i\in[1,2r]$. 
\end{theorem}

Hence, by \eqref{eq:2.14} and \eqref{eq:2.16} there is an identification of the cluster variables $A_{i,k}$ with the normalized $Q$-system variables $R_{i,k}$, given by $A_{i,k}\mapsto R_{i,k}$, where the relationship between the finite Dynkin type $Y_r$ of the simple complex classical Lie group $G$ and the corresponding affine Dynkin type $X_m^{(\kappa)}$ of the $Q$-system is given by the table at the start of Section \ref{Section2.5}.

\subsection{Elementary chips and network representations of the simple classical Lie groups}\label{Section2.6}

In this subsection, we will review the combinatorial models of the one-parameter subgroups, as well as elements of $H$, given in terms of elementary chips. The elementary chips of the one-parameter subgroups and the elements of $H$ first appeared in \cite{FZ99} in the case where $G=SL_{r+1}(\mathbb{C})$, while the elementary chips of the one-parameter subgroups for the other classical types were derived in \cite{Yang12}, in the context of $F$-polynomials in cluster algebras of classical finite types. We will derive the elementary chips of the elements of $H$ for the other classical types. These elementary chips will then serve as building blocks for our network representations of elements in a double Bruhat cell $G^{u,v}$ that admits the factorization scheme \eqref{eq:2.2}, as well as elements in a quotient Coxeter double Bruhat cell $G^{u,v}/H$ that admits the factorization scheme \eqref{eq:2.6}.

\subsubsection*{Type \textit{A}}\label{Section2.6.1}

Let us first consider the case where $G=SL_{r+1}(\mathbb{C})$. By definition, we can take $G$ to be the set of all $(r+1)\times (r+1)$ complex matrices $M$ that satisfy $\det M=1$. Thus, $\mathfrak{g}$ consists of $(r+1)\times (r+1)$ complex matrices $M$ that satisfy $\tr M=0$. Moreover, a set of Chevalley generators $\{e_{\pm i}\}_{i=1}^r$ for $\mathfrak{g}$ given by $e_i=e_{i,i+1}$ and $e_{-i}=e_{i+1,i}$ for all $i\in[1,r]$, which implies that we have $h_i=e_{i,i}-e_{i+1,i+1}$ for all $i\in[1,r]$. Hence, we have 
\begin{align}
E_i(a_i)&=I_{r+1}+a_ie_{i,i+1},\label{eq:2.17}\\
E_{-i}(b_i)&=I_{r+1}+b_ie_{i+1,i},\quad\text{and}\label{eq:2.18}\\ 
D(t_1,\ldots,t_r)&=\diag(t_1,t_2/t_1,\ldots,t_r/t_{r-1},1/t_r)\label{eq:2.19}
\end{align}
for all $i\in[1,r]$. The elements $E_i(a_i)$, $E_{-i}(b_i)$ and $D(t_1,\ldots,t_r)$ could be represented in the form of elementary chips as shown in Figure \ref{Figure2.1}.
\begin{figure}[t]
\caption{The elementary chips of $E_i(a_i)$, $E_{-i}(b_i)$ and $D(t_1,\ldots,t_r)$ in type $A$.}
\label{Figure2.1}
\begin{center}
\begin{tikzpicture}[
       thick,
       acteur/.style={
         circle,
         thick,
         inner sep=1pt,
         minimum size=0.1cm
       }
] 
\node (a1) at (0,0) [acteur,fill=blue]{};
\node (a2) at (0,0.8) [acteur,fill=blue]{}; 
\node (a3) at (0,1.6) [acteur,fill=blue]{}; 
\node (a4) at (0,2.4) [acteur,fill=blue]{}; 
\node (a5) at (0.8,0.8) [acteur,fill=orange]{};
\node (a6) at (1.6,1.6) [acteur,fill=black]{}; 
\node (a7) at (2.4,0) [acteur,fill=red]{};
\node (a8) at (2.4,0.8) [acteur,fill=red]{}; 
\node (a9) at (2.4,1.6) [acteur,fill=red]{}; 
\node (a10) at (2.4,2.4) [acteur,fill=red]{};

\node (b1) at (-0.6,0) {$1$};
\node (b2) at (-0.6,0.8) {$i$}; 
\node (b3) at (-0.6,1.6) {$i+1$}; 
\node (b4) at (-0.6,2.4) {$r+1$}; 
\node (b5) at (3.0,0) {$1$};
\node (b6) at (3.0,0.8) {$i$}; 
\node (b7) at (3.0,1.6) {$i+1$}; 
\node (b8) at (3.0,2.4) {$r+1$}; 

\node (c) at (1.2,-0.8) {$E_i(a_i)$};

\node (d1) at (1.2,0.2) [acteur,fill=black]{};
\node (d2) at (1.2,0.4) [acteur,fill=black]{};
\node (d3) at (1.2,0.6) [acteur,fill=black]{};
\node (d4) at (1.2,1.8) [acteur,fill=black]{};
\node (d5) at (1.2,2.0) [acteur,fill=black]{};
\node (d6) at (1.2,2.2) [acteur,fill=black]{};

\draw[->] (a1) to node {} (a7);
\draw[->] (a2) to node {} (a5);
\draw[->] (a5) to node {} (a8);
\draw[->] (a3) to node {} (a6);
\draw[->] (a6) to node {} (a9);
\draw[->] (a4) to node {} (a10);
\draw[->] (a5) to node [right] {$a_i$} (a6);

\end{tikzpicture}
\begin{tikzpicture}[
       thick,
       acteur/.style={
         circle,
         thick,
         inner sep=1pt,
         minimum size=0.1cm
       }
] 
\node (a1) at (0,0) [acteur,fill=blue]{};
\node (a2) at (0,0.8) [acteur,fill=blue]{}; 
\node (a3) at (0,1.6) [acteur,fill=blue]{}; 
\node (a4) at (0,2.4) [acteur,fill=blue]{}; 
\node (a5) at (0.8,1.6) [acteur,fill=orange]{};
\node (a6) at (1.6,0.8) [acteur,fill=black]{}; 
\node (a7) at (2.4,0) [acteur,fill=red]{};
\node (a8) at (2.4,0.8) [acteur,fill=red]{}; 
\node (a9) at (2.4,1.6) [acteur,fill=red]{}; 
\node (a10) at (2.4,2.4) [acteur,fill=red]{};

\node (b1) at (-0.6,0) {$1$};
\node (b2) at (-0.6,0.8) {$i$}; 
\node (b3) at (-0.6,1.6) {$i+1$}; 
\node (b4) at (-0.6,2.4) {$r+1$}; 
\node (b5) at (3.0,0) {$1$};
\node (b6) at (3.0,0.8) {$i$}; 
\node (b7) at (3.0,1.6) {$i+1$}; 
\node (b8) at (3.0,2.4) {$r+1$};

\node (c) at (1.2,-0.8) {$E_{-i}(b_i)$};

\node (d1) at (1.2,0.2) [acteur,fill=black]{};
\node (d2) at (1.2,0.4) [acteur,fill=black]{};
\node (d3) at (1.2,0.6) [acteur,fill=black]{};
\node (d4) at (1.2,1.8) [acteur,fill=black]{};
\node (d5) at (1.2,2.0) [acteur,fill=black]{};
\node (d6) at (1.2,2.2) [acteur,fill=black]{};

\draw[->] (a1) to node {} (a7);
\draw[->] (a2) to node {} (a6);
\draw[->] (a6) to node {} (a8);
\draw[->] (a3) to node {} (a5);
\draw[->] (a5) to node {} (a9);
\draw[->] (a4) to node {} (a10);
\draw[->] (a5) to node [left] {$b_i$} (a6);

\end{tikzpicture}
\begin{tikzpicture}[
       thick,
       acteur/.style={
         circle,
         thick,
         inner sep=1pt,
         minimum size=0.1cm
       }
] 
\node (a1) at (0,0) [acteur,fill=blue]{};
\node (a2) at (0,0.8) [acteur,fill=blue]{}; 
\node (a3) at (0,1.6) [acteur,fill=blue]{}; 
\node (a4) at (0,2.4) [acteur,fill=blue]{}; 
\node (a5) at (2.4,0) [acteur,fill=red]{};
\node (a6) at (2.4,0.8) [acteur,fill=red]{}; 
\node (a7) at (2.4,1.6) [acteur,fill=red]{}; 
\node (a8) at (2.4,2.4) [acteur,fill=red]{};

\node (b1) at (-0.6,0) {$1$};
\node (b2) at (-0.6,0.8) {$2$}; 
\node (b3) at (-0.6,1.6) {$r$}; 
\node (b4) at (-0.6,2.4) {$r+1$}; 
\node (b5) at (3.0,0) {$1$};
\node (b6) at (3.0,0.8) {$2$}; 
\node (b7) at (3.0,1.6) {$r$}; 
\node (b8) at (3.0,2.4) {$r+1$};

\node (c) at (1.2,-0.8) {$D(t_1,\ldots,t_r)$};

\node (d1) at (1.2,1.0) [acteur,fill=black]{};
\node (d2) at (1.2,1.2) [acteur,fill=black]{};
\node (d3) at (1.2,1.4) [acteur,fill=black]{};

\draw[->] (a1) to node [below] {$t_1$} (a5);
\draw[->] (a2) to node [below] {$t_2/t_1$} (a6);
\draw[->] (a3) to node [above] {$t_r/t_{r-1}$} (a7);
\draw[->] (a4) to node [above] {$1/t_r$} (a8);
\end{tikzpicture} 
\end{center}
\end{figure}
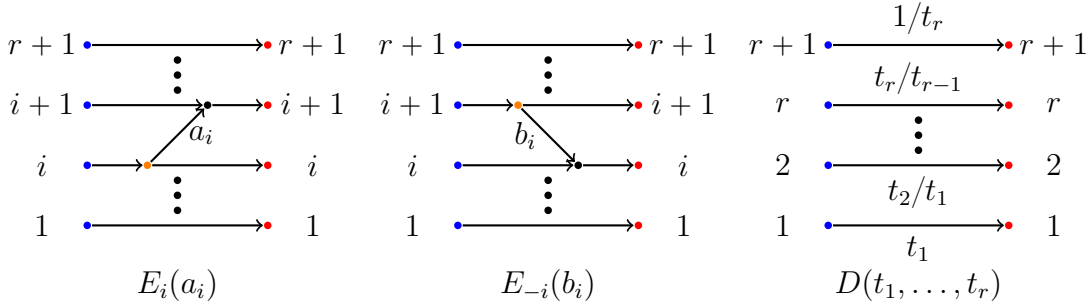

\subsubsection*{Type \textit{B}}\label{Section2.6.2}

Next, let us consider the case where $G=SO_{2r+1}(\mathbb{C})$. By definition, we can take $G$ to be the set of all $(2r+1)\times (2r+1)$ complex matrices $M$ that satisfy $M^TJ_BM=J_B$, where $J_B$ is the symmetric matrix $J_B=\sum_{i=1}^{2r+1}(-1)^{i+1}e_{i,2r+2-i}$, that is, we have 

\begin{equation*}
J_B=
\begin{pmatrix}
0 & 0 & \cdots & 0 & 1 \\
0 & 0 & \cdots & -1 & 0 \\
\vdots & \vdots & \ddots & \vdots & \vdots \\
0 & -1 & \cdots & 0 & 0 \\
1 & 0 & \cdots & 0 & 0
\end{pmatrix}.
\end{equation*}
Thus, $\mathfrak{g}$ consists of the $(2r+1)\times (2r+1)$ complex matrices $M$ that satisfy $M^TJ_B+J_BM=0$, which is then equivalent to the set of $(2r+1)\times (2r+1)$ complex matrices $M$ that satisfy $m_{i,j}=(-1)^{i+j+1}m_{2r+2-j,2r+2-i}$ for all $i,j\in[1,2r+1]$. Moreover, a set of Chevalley generators $\{e_{\pm i}\}_{i=1}^r$ for $\mathfrak{g}$ is given by 
\begin{align*}
e_i&=e_{i,i+1}+e_{2r+1-i,2r+2-i},\quad e_{-i}=e_{i+1,i}+e_{2r+2-i,2r+1-i}, \quad i\in[1,r-1],\\ 
e_r&=\sqrt{2}(e_{r,r+1}+e_{r+1,r+2}),\quad e_{-r}=\sqrt{2}(e_{r+1,r}+e_{r+2,r+1}), 
\end{align*}
\begin{figure}[t]
\caption{The elementary chip of $E_i(a_i)$, $E_{-i}(b_i)$, $i=1,\ldots,r-1$ in type $B$.}
\label{Figure2.2}
\begin{center}
\begin{tikzpicture}[
scale=0.9, transform shape,
       thick,
       acteur/.style={
         circle,
         thick,
         inner sep=1pt,
         minimum size=0.1cm
       }
] 

\node (a1) at (0,0) [acteur,fill=blue]{};
\node (a2) at (0,0.8) [acteur,fill=blue]{};
\node (a3) at (0,1.6) [acteur,fill=blue]{}; 
\node (a4) at (0,2.4) [acteur,fill=blue]{};
\node (a5) at (0,3.2) [acteur,fill=blue]{}; 
\node (a6) at (0,4.0) [acteur,fill=blue]{}; 
\node (a7) at (0,4.8) [acteur,fill=blue]{};
\node (a8) at (0.8,0.8) [acteur,fill=orange]{};
\node (a9) at (1.6,1.6) [acteur,fill=black]{}; 
\node (a10) at (0.8,3.2) [acteur,fill=orange]{}; 
\node (a11) at (1.6,4.0) [acteur,fill=black]{};
\node (a12) at (2.4,0) [acteur,fill=red]{};
\node (a13) at (2.4,0.8) [acteur,fill=red]{};
\node (a14) at (2.4,1.6) [acteur,fill=red]{}; 
\node (a15) at (2.4,2.4) [acteur,fill=red]{}; 
\node (a16) at (2.4,3.2) [acteur,fill=red]{};
\node (a17) at (2.4,4.0) [acteur,fill=red]{};
\node (a18) at (2.4,4.8) [acteur,fill=red]{};

\node (b1) at (-1.0,0) {$1$};
\node (b2) at (-1.0,0.8) {$i$}; 
\node (b3) at (-1.0,1.6) {$i+1$}; 
\node (b4) at (-1.0,2.4) {$r+1$}; 
\node (b5) at (-1.0,3.2) {$2r+1-i$};
\node (b6) at (-1.0,4.0) {$2r+2-i$}; 
\node (b7) at (-1.0,4.8) {$2r+1$}; 
\node (b8) at (3.4,0) {$1$};
\node (b9) at (3.4,0.8) {$i$};
\node (b10) at (3.4,1.6) {$i+1$}; 
\node (b11) at (3.4,2.4) {$r+1$}; 
\node (b12) at (3.4,3.2) {$2r+1-i$}; 
\node (b13) at (3.4,4.0) {$2r+2-i$};
\node (b14) at (3.4,4.8) {$2r+1$}; 

\node (c) at (1.2,-0.8) {$E_i(a_i)$};

\node (d1) at (1.2,0.2) [acteur,fill=black]{};
\node (d2) at (1.2,0.4) [acteur,fill=black]{};
\node (d3) at (1.2,0.6) [acteur,fill=black]{};
\node (d4) at (1.2,1.8) [acteur,fill=black]{};
\node (d5) at (1.2,2.0) [acteur,fill=black]{};
\node (d6) at (1.2,2.2) [acteur,fill=black]{};
\node (d7) at (1.2,2.6) [acteur,fill=black]{};
\node (d8) at (1.2,2.8) [acteur,fill=black]{};
\node (d9) at (1.2,3.0) [acteur,fill=black]{};
\node (d10) at (1.2,4.2) [acteur,fill=black]{};
\node (d11) at (1.2,4.4) [acteur,fill=black]{};
\node (d12) at (1.2,4.6) [acteur,fill=black]{};

\draw[->] (a1) to node {} (a12);
\draw[->] (a2) to node {} (a8);
\draw[->] (a8) to node {} (a13);
\draw[->] (a3) to node {} (a9);
\draw[->] (a9) to node {} (a14);
\draw[->] (a4) to node {} (a15);
\draw[->] (a5) to node {} (a10);
\draw[->] (a10) to node {} (a16);
\draw[->] (a6) to node {} (a11);
\draw[->] (a11) to node {} (a17);
\draw[->] (a7) to node {} (a18);
\draw[->] (a8) to node [left] {$a_i$} (a9);
\draw[->] (a10) to node [left] {$a_i$} (a11);

\end{tikzpicture}
\begin{tikzpicture}[
scale=0.9, transform shape,
       thick,
       acteur/.style={
         circle,
         thick,
         inner sep=1pt,
         minimum size=0.1cm
       }
] 
\node (a1) at (0,0) [acteur,fill=blue]{};
\node (a2) at (0,0.8) [acteur,fill=blue]{};
\node (a3) at (0,1.6) [acteur,fill=blue]{}; 
\node (a4) at (0,2.4) [acteur,fill=blue]{};
\node (a5) at (0,3.2) [acteur,fill=blue]{}; 
\node (a6) at (0,4.0) [acteur,fill=blue]{}; 
\node (a7) at (0,4.8) [acteur,fill=blue]{};
\node (a8) at (0.8,1.6) [acteur,fill=orange]{};
\node (a9) at (1.6,0.8) [acteur,fill=black]{}; 
\node (a10) at (0.8,4.0) [acteur,fill=orange]{}; 
\node (a11) at (1.6,3.2) [acteur,fill=black]{};
\node (a12) at (2.4,0) [acteur,fill=red]{};
\node (a13) at (2.4,0.8) [acteur,fill=red]{};
\node (a14) at (2.4,1.6) [acteur,fill=red]{}; 
\node (a15) at (2.4,2.4) [acteur,fill=red]{}; 
\node (a16) at (2.4,3.2) [acteur,fill=red]{};
\node (a17) at (2.4,4.0) [acteur,fill=red]{};
\node (a18) at (2.4,4.8) [acteur,fill=red]{};

\node (b1) at (-1.0,0) {$1$};
\node (b2) at (-1.0,0.8) {$i$}; 
\node (b3) at (-1.0,1.6) {$i+1$}; 
\node (b4) at (-1.0,2.4) {$r+1$}; 
\node (b5) at (-1.0,3.2) {$2r+1-i$};
\node (b6) at (-1.0,4.0) {$2r+2-i$}; 
\node (b7) at (-1.0,4.8) {$2r+1$}; 
\node (b8) at (3.4,0) {$1$};
\node (b9) at (3.4,0.8) {$i$};
\node (b10) at (3.4,1.6) {$i+1$}; 
\node (b11) at (3.4,2.4) {$r+1$}; 
\node (b12) at (3.4,3.2) {$2r+1-i$}; 
\node (b13) at (3.4,4.0) {$2r+2-i$};
\node (b14) at (3.4,4.8) {$2r+1$}; 

\node (c) at (1.2,-0.8) {$E_{-i}(b_i)$};

\node (d1) at (1.2,0.2) [acteur,fill=black]{};
\node (d2) at (1.2,0.4) [acteur,fill=black]{};
\node (d3) at (1.2,0.6) [acteur,fill=black]{};
\node (d4) at (1.2,1.8) [acteur,fill=black]{};
\node (d5) at (1.2,2.0) [acteur,fill=black]{};
\node (d6) at (1.2,2.2) [acteur,fill=black]{};
\node (d7) at (1.2,2.6) [acteur,fill=black]{};
\node (d8) at (1.2,2.8) [acteur,fill=black]{};
\node (d9) at (1.2,3.0) [acteur,fill=black]{};
\node (d10) at (1.2,4.2) [acteur,fill=black]{};
\node (d11) at (1.2,4.4) [acteur,fill=black]{};
\node (d12) at (1.2,4.6) [acteur,fill=black]{};

\draw[->] (a1) to node {} (a12);
\draw[->] (a2) to node {} (a9);
\draw[->] (a9) to node {} (a13);
\draw[->] (a3) to node {} (a8);
\draw[->] (a8) to node {} (a14);
\draw[->] (a4) to node {} (a15);
\draw[->] (a5) to node {} (a11);
\draw[->] (a11) to node {} (a16);
\draw[->] (a6) to node {} (a10);
\draw[->] (a10) to node {} (a17);
\draw[->] (a7) to node {} (a18);
\draw[->] (a8) to node [left] {$b_i$} (a9);
\draw[->] (a10) to node [left] {$b_i$} (a11);

\end{tikzpicture}
\end{center}
\end{figure}
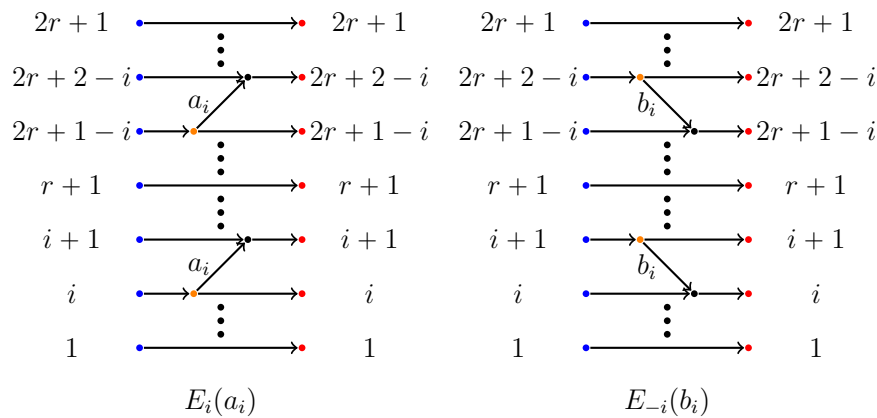
which would imply that we have $h_i=e_{i,i}-e_{i+1,i+1}+e_{2r+1-i,2r+1-i}-e_{2r+2-i,2r+2-i}$ for all $i\in[1,r-1]$ and $h_r=2(e_{r,r}-e_{r+2,r+2})$. Hence, we have
\begin{align}
E_i(a_i)&=I+a_ie_{i,i+1}+a_ie_{2r+1-i,2r+2-i},\label{eq:2.20}\\
E_{-i}(b_i)&=I+b_ie_{i+1,i}+b_ie_{2r+2-i,2r+1-i},\label{eq:2.21}\\
E_r(a_r)&=I+\sqrt{2}a_re_{r,r+1}+\sqrt{2}a_re_{r+1,r+2}+a_r^2e_{r,r+2},\label{eq:2.22}\\
E_{-r}(b_r)&=I+\sqrt{2}b_re_{r+1,r}+\sqrt{2}b_re_{r+2,r+1}+b_r^2e_{r+2,r},\quad\text{and}\label{eq:2.23}\\
D(t_1,\ldots,t_r)&=\diag(t_1,t_2/t_1,\ldots,t_r^2/t_{r-1},1,t_{r-1}/t_r^2,\ldots,t_1/t_2,1/t_1)\label{eq:2.24}
\end{align}
for all $i\in[1,r]$. The elements $E_i(a_i)$, $E_{-i}(b_i)$ and $D(t_1,\ldots,t_r)$ could be represented in the form of elementary chips as shown in Figures \ref{Figure2.2} and \ref{Figure2.3}.
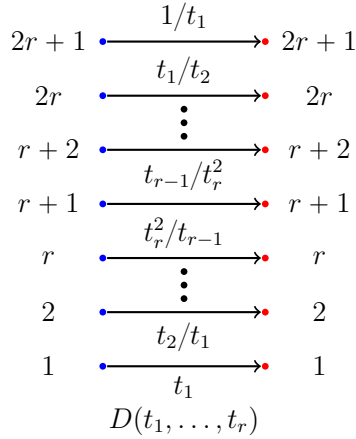
\begin{figure}[t]
\caption{The elementary chips of $E_r(a_r)$, $E_{-r}(b_r)$, and $D(t_1,\ldots,t_r)$ in type $B$.}
\label{Figure2.3}
\begin{center}
\begin{tikzpicture}[
scale=0.9, transform shape,
       thick,
       acteur/.style={
         circle,
         thick,
         inner sep=1pt,
         minimum size=0.1cm
       }
] 

\node (a1) at (0,0) [acteur,fill=blue]{};
\node (a2) at (0,0.8) [acteur,fill=blue]{};
\node (a3) at (0,1.6) [acteur,fill=blue]{}; 
\node (a4) at (0,2.4) [acteur,fill=blue]{};
\node (a5) at (0,3.2) [acteur,fill=blue]{}; 
\node (a6) at (0,4.0) [acteur,fill=blue]{}; 
\node (a7) at (0,4.8) [acteur,fill=blue]{};
\node (a8) at (1.6,1.6) [acteur,fill=orange]{};
\node (a9) at (2.4,2.4) [acteur,fill=black]{}; 
\node (a10) at (1.6,2.4) [acteur,fill=orange]{}; 
\node (a11) at (2.4,3.2) [acteur,fill=black]{};
\node (a12) at (4.0,0) [acteur,fill=red]{};
\node (a13) at (4.0,0.8) [acteur,fill=red]{};
\node (a14) at (4.0,1.6) [acteur,fill=red]{}; 
\node (a15) at (4.0,2.4) [acteur,fill=red]{}; 
\node (a16) at (4.0,3.2) [acteur,fill=red]{};
\node (a17) at (4.0,4.0) [acteur,fill=red]{};
\node (a18) at (4.0,4.8) [acteur,fill=red]{};

\node (b1) at (-0.8,0) {$1$};
\node (b2) at (-0.8,0.8) {$2$}; 
\node (b3) at (-0.8,1.6) {$r$}; 
\node (b4) at (-0.8,2.4) {$r+1$}; 
\node (b5) at (-0.8,3.2) {$r+2$};
\node (b6) at (-0.8,4.0) {$2r$}; 
\node (b7) at (-0.8,4.8) {$2r+1$}; 
\node (b8) at (4.8,0) {$1$};
\node (b9) at (4.8,0.8) {$2$};
\node (b10) at (4.8,1.6) {$r$}; 
\node (b11) at (4.8,2.4) {$r+1$}; 
\node (b12) at (4.8,3.2) {$r+2$}; 
\node (b13) at (4.8,4.0) {$2r$};
\node (b14) at (4.8,4.8) {$2r+1$}; 

\node (c) at (2.0,-0.8) {$E_r(a_r)$};

\node (d1) at (2.0,1.0) [acteur,fill=black]{};
\node (d2) at (2.0,1.2) [acteur,fill=black]{};
\node (d3) at (2.0,1.4) [acteur,fill=black]{};
\node (d4) at (2.0,3.4) [acteur,fill=black]{};
\node (d5) at (2.0,3.6) [acteur,fill=black]{};
\node (d6) at (2.0,3.8) [acteur,fill=black]{};

\draw[->] (a1) to node {} (a12);
\draw[->] (a2) to node {} (a13);
\draw[->] (a3) to node {} (a8);
\draw[->] (a8) to node {} (a14);
\draw[->] (a4) to node {} (a10);
\draw[->] (a10) to node {} (a9);
\draw[->] (a9) to node {} (a15);
\draw[->] (a5) to node {} (a11);
\draw[->] (a11) to node {} (a16);
\draw[->] (a6) to node {} (a17);
\draw[->] (a7) to node {} (a18);
\draw[->] (a8) to node [right] {$\sqrt{2}a_r$} (a9);
\draw[->] (a10) to node [left=0.15cm] {$\sqrt{2}a_r$} (a11);
\draw[->] (a8) to node [above right=0.2cm] {$a_r^2$} (a11);

\end{tikzpicture}
\begin{tikzpicture}[
scale=0.9, transform shape,
       thick,
       acteur/.style={
         circle,
         thick,
         inner sep=1pt,
         minimum size=0.1cm
       }
] 
\node (a1) at (0,0) [acteur,fill=blue]{};
\node (a2) at (0,0.8) [acteur,fill=blue]{};
\node (a3) at (0,1.6) [acteur,fill=blue]{}; 
\node (a4) at (0,2.4) [acteur,fill=blue]{};
\node (a5) at (0,3.2) [acteur,fill=blue]{}; 
\node (a6) at (0,4.0) [acteur,fill=blue]{}; 
\node (a7) at (0,4.8) [acteur,fill=blue]{};
\node (a8) at (1.6,2.4) [acteur,fill=orange]{};
\node (a9) at (2.4,1.6) [acteur,fill=black]{}; 
\node (a10) at (1.6,3.2) [acteur,fill=orange]{}; 
\node (a11) at (2.4,2.4) [acteur,fill=black]{};
\node (a12) at (4.0,0) [acteur,fill=red]{};
\node (a13) at (4.0,0.8) [acteur,fill=red]{};
\node (a14) at (4.0,1.6) [acteur,fill=red]{}; 
\node (a15) at (4.0,2.4) [acteur,fill=red]{}; 
\node (a16) at (4.0,3.2) [acteur,fill=red]{};
\node (a17) at (4.0,4.0) [acteur,fill=red]{};
\node (a18) at (4.0,4.8) [acteur,fill=red]{};

\node (b1) at (-0.8,0) {$1$};
\node (b2) at (-0.8,0.8) {$2$}; 
\node (b3) at (-0.8,1.6) {$r$}; 
\node (b4) at (-0.8,2.4) {$r+1$}; 
\node (b5) at (-0.8,3.2) {$r+2$};
\node (b6) at (-0.8,4.0) {$2r$}; 
\node (b7) at (-0.8,4.8) {$2r+1$}; 
\node (b8) at (4.8,0) {$1$};
\node (b9) at (4.8,0.8) {$2$};
\node (b10) at (4.8,1.6) {$r$}; 
\node (b11) at (4.8,2.4) {$r+1$}; 
\node (b12) at (4.8,3.2) {$r+2$}; 
\node (b13) at (4.8,4.0) {$2r$};
\node (b14) at (4.8,4.8) {$2r+1$}; 

\node (c) at (2.0,-0.8) {$E_{-r}(b_r)$};

\node (d1) at (2.0,1.0) [acteur,fill=black]{};
\node (d2) at (2.0,1.2) [acteur,fill=black]{};
\node (d3) at (2.0,1.4) [acteur,fill=black]{};
\node (d4) at (2.0,3.4) [acteur,fill=black]{};
\node (d5) at (2.0,3.6) [acteur,fill=black]{};
\node (d6) at (2.0,3.8) [acteur,fill=black]{};

\draw[->] (a1) to node {} (a12);
\draw[->] (a2) to node {} (a13);
\draw[->] (a3) to node {} (a9);
\draw[->] (a9) to node {} (a14);
\draw[->] (a4) to node {} (a8);
\draw[->] (a8) to node {} (a11);
\draw[->] (a11) to node {} (a15);
\draw[->] (a5) to node {} (a10);
\draw[->] (a10) to node {} (a16);
\draw[->] (a6) to node {} (a17);
\draw[->] (a7) to node {} (a18);
\draw[->] (a8) to node [left=0.1cm] {$\sqrt{2}b_r$} (a9);
\draw[->] (a10) to node [right] {$\sqrt{2}b_r$} (a11);
\draw[->] (a10) to node [below right=0.2cm] {$b_r^2$} (a9);

\end{tikzpicture}
\end{center}
\begin{center}
\begin{tikzpicture}[
scale=0.9, transform shape,
       thick,
       acteur/.style={
         circle,
         thick,
         inner sep=1pt,
         minimum size=0.1cm
       }
] 
\node (a1) at (0,0) [acteur,fill=blue]{};
\node (a2) at (0,0.8) [acteur,fill=blue]{};
\node (a3) at (0,1.6) [acteur,fill=blue]{}; 
\node (a4) at (0,2.4) [acteur,fill=blue]{};
\node (a5) at (0,3.2) [acteur,fill=blue]{}; 
\node (a6) at (0,4.0) [acteur,fill=blue]{}; 
\node (a7) at (0,4.8) [acteur,fill=blue]{};
\node (a8) at (2.4,0) [acteur,fill=red]{};
\node (a9) at (2.4,0.8) [acteur,fill=red]{};
\node (a10) at (2.4,1.6) [acteur,fill=red]{}; 
\node (a11) at (2.4,2.4) [acteur,fill=red]{}; 
\node (a12) at (2.4,3.2) [acteur,fill=red]{};
\node (a13) at (2.4,4.0) [acteur,fill=red]{};
\node (a14) at (2.4,4.8) [acteur,fill=red]{};

\node (b1) at (-0.8,0) {$1$};
\node (b2) at (-0.8,0.8) {$2$}; 
\node (b3) at (-0.8,1.6) {$r$}; 
\node (b4) at (-0.8,2.4) {$r+1$}; 
\node (b5) at (-0.8,3.2) {$r+2$};
\node (b6) at (-0.8,4.0) {$2r$}; 
\node (b7) at (-0.8,4.8) {$2r+1$}; 
\node (b8) at (3.2,0) {$1$};
\node (b9) at (3.2,0.8) {$2$};
\node (b10) at (3.2,1.6) {$r$}; 
\node (b11) at (3.2,2.4) {$r+1$}; 
\node (b12) at (3.2,3.2) {$r+2$}; 
\node (b13) at (3.2,4.0) {$2r$};
\node (b14) at (3.2,4.8) {$2r+1$}; 

\node (c) at (1.2,-0.8) {$D(t_1,\ldots,t_r)$};

\node (d1) at (1.2,1.0) [acteur,fill=black]{};
\node (d2) at (1.2,1.2) [acteur,fill=black]{};
\node (d3) at (1.2,1.4) [acteur,fill=black]{};
\node (d4) at (1.2,3.4) [acteur,fill=black]{};
\node (d5) at (1.2,3.6) [acteur,fill=black]{};
\node (d6) at (1.2,3.8) [acteur,fill=black]{};

\draw[->] (a1) to node [below] {$t_1$} (a8);
\draw[->] (a2) to node [below] {$t_2/t_1$} (a9);
\draw[->] (a3) to node [above] {$t_r^2/t_{r-1}$} (a10);
\draw[->] (a4) to node {} (a11);
\draw[->] (a5) to node [below] {$t_{r-1}/t_r^2$} (a12);
\draw[->] (a6) to node [above] {$t_1/t_2$} (a13);
\draw[->] (a7) to node [above] {$1/t_1$} (a14);

\end{tikzpicture}
\end{center}
\end{figure}
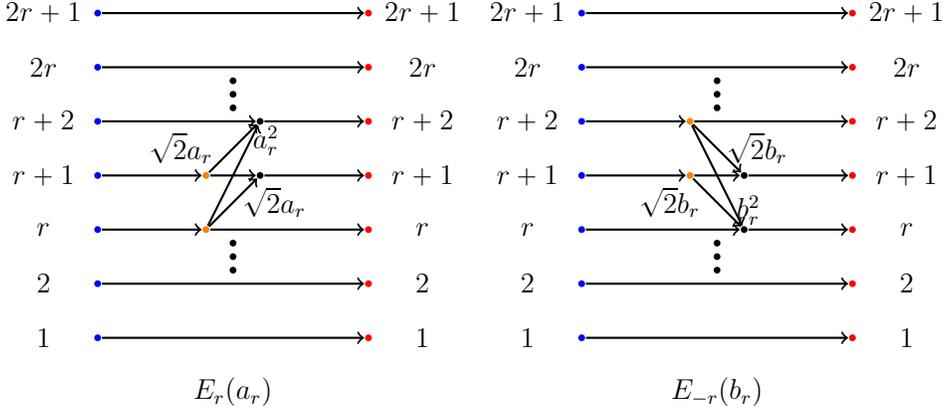

\subsubsection*{Type \textit{C}}\label{Section2.6.3}

Next, let us consider the case where $G=Sp_{2r}(\mathbb{C})$. By definition, we can take $G$ to be the set of all $2r\times 2r$ complex matrices $M$ that satisfy $M^TJ_CM=J_C$, where $J_C$ is the skew-symmetric matrix $\sum_{i=1}^{2r}(-1)^{i+1}e_{i,2r+1-i}$, that is, we have 
\begin{equation*}
J_C=
\begin{pmatrix}
0 & 0 & \cdots & 0 & 1 \\
0 & 0 & \cdots & -1 & 0 \\
\vdots & \vdots & \ddots & \vdots & \vdots \\
0 & 1 & \cdots & 0 & 0 \\
-1 & 0 & \cdots & 0 & 0
\end{pmatrix}.
\end{equation*}
Thus, $\mathfrak{g}$ consists of the $2r\times 2r$ complex matrices $M$ that satisfy $M^TJ_C+J_CM=0$, which is then equivalent to the set of $2r\times 2r$ complex matrices $M$ that satisfy $m_{i,j}=(-1)^{i+j+1}m_{2r+1-j,2r+1-i}$ for all $i,j\in[1,2r]$. Moreover, a set of Chevalley generators $\{e_{\pm i}\}_{i=1}^r$ for $\mathfrak{g}$ is given by 
\begin{align*}
e_i&=e_{i,i+1}+e_{2r-i,2r+1-i},\quad e_{-i}=e_{i+1,i}+e_{2r+1-i,2r-i}, \quad i\in[1,r-1],\\ 
e_r&=e_{r,r+1},\quad e_{-r}=e_{r+1,r}, 
\end{align*}
which would imply that $h_i=e_{i,i}-e_{i+1,i+1}+e_{2r-i,2r-i}-e_{2r+1-i,2r+1-i}$ for all $i\in[1,r]$ and $h_r=e_{r,r}-e_{r+1,r+1}$. Hence, we have
\begin{align}
E_i(a_i)&=I+a_ie_{i,i+1}+a_ie_{2r-i,2r+1-i},\label{eq:2.25}\\
E_{-i}(b_i)&=I+b_ie_{i+1,i}+b_ie_{2r+1-i,2r-i},\label{eq:2.26}\\
E_r(a_r)&=I+a_re_{r,r+1},\label{eq:2.27}\\
E_{-r}(b_r)&=I+b_re_{r+1,r},\quad\text{and}\label{eq:2.28}\\
D(t_1,\ldots,t_r)&=\diag(t_1,t_2/t_1,\ldots,t_r/t_{r-1},t_{r-1}/t_r,\ldots,t_1/t_2,1/t_1)\label{eq:2.29}
\end{align}
for all $i\in[1,r]$. The elements $E_i(a_i)$, $E_{-i}(b_i)$ and $D(t_1,\ldots,t_r)$ could be represented in the form of elementary chips as shown in Figures \ref{Figure2.4} and \ref{Figure2.5}.

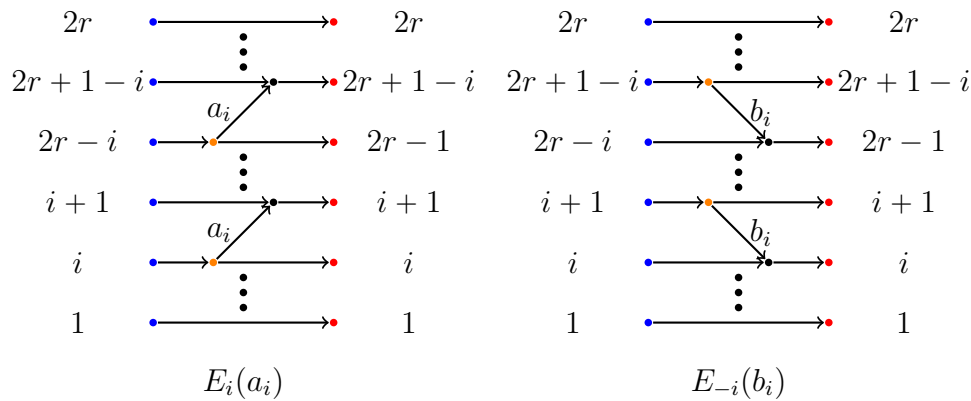
\begin{figure}[t]
\caption{The elementary chips of $E_i(a_i)$, $E_{-i}(b_i)$, $i=1,\ldots,r-1$ in type $C$.}
\label{Figure2.4}
\begin{center}
\begin{tikzpicture}[
       thick,
       acteur/.style={
         circle,
         thick,
         inner sep=1pt,
         minimum size=0.1cm
       }
] 

\node (a1) at (0,0) [acteur,fill=blue]{};
\node (a2) at (0,0.8) [acteur,fill=blue]{};
\node (a3) at (0,1.6) [acteur,fill=blue]{}; 
\node (a4) at (0,2.4) [acteur,fill=blue]{};
\node (a5) at (0,3.2) [acteur,fill=blue]{}; 
\node (a6) at (0,4.0) [acteur,fill=blue]{}; 
\node (a7) at (0.8,0.8) [acteur,fill=orange]{};
\node (a8) at (1.6,1.6) [acteur,fill=black]{}; 
\node (a9) at (0.8,2.4) [acteur,fill=orange]{}; 
\node (a10) at (1.6,3.2) [acteur,fill=black]{};
\node (a11) at (2.4,0) [acteur,fill=red]{};
\node (a12) at (2.4,0.8) [acteur,fill=red]{};
\node (a13) at (2.4,1.6) [acteur,fill=red]{}; 
\node (a14) at (2.4,2.4) [acteur,fill=red]{}; 
\node (a15) at (2.4,3.2) [acteur,fill=red]{};
\node (a16) at (2.4,4.0) [acteur,fill=red]{};

\node (b1) at (-1.0,0) {$1$};
\node (b2) at (-1.0,0.8) {$i$}; 
\node (b3) at (-1.0,1.6) {$i+1$}; 
\node (b4) at (-1.0,2.4) {$2r-i$}; 
\node (b5) at (-1.0,3.2) {$2r+1-i$};
\node (b6) at (-1.0,4.0) {$2r$}; 
\node (b8) at (3.4,0) {$1$};
\node (b9) at (3.4,0.8) {$i$};
\node (b10) at (3.4,1.6) {$i+1$}; 
\node (b11) at (3.4,2.4) {$2r-1$}; 
\node (b12) at (3.4,3.2) {$2r+1-i$}; 
\node (b13) at (3.4,4.0) {$2r$};

\node (c) at (1.2,-0.8) {$E_i(a_i)$};

\node (d1) at (1.2,0.2) [acteur,fill=black]{};
\node (d2) at (1.2,0.4) [acteur,fill=black]{};
\node (d3) at (1.2,0.6) [acteur,fill=black]{};
\node (d4) at (1.2,1.8) [acteur,fill=black]{};
\node (d5) at (1.2,2.0) [acteur,fill=black]{};
\node (d6) at (1.2,2.2) [acteur,fill=black]{};
\node (d7) at (1.2,3.4) [acteur,fill=black]{};
\node (d8) at (1.2,3.6) [acteur,fill=black]{};
\node (d9) at (1.2,3.8) [acteur,fill=black]{};

\draw[->] (a1) to node {} (a11);
\draw[->] (a2) to node {} (a7);
\draw[->] (a7) to node {} (a12);
\draw[->] (a3) to node {} (a8);
\draw[->] (a8) to node {} (a13);
\draw[->] (a4) to node {} (a9);
\draw[->] (a9) to node {} (a14);
\draw[->] (a5) to node {} (a10);
\draw[->] (a10) to node {} (a15);
\draw[->] (a6) to node {} (a16);
\draw[->] (a7) to node [left] {$a_i$} (a8);
\draw[->] (a9) to node [left] {$a_i$} (a10);

\end{tikzpicture}
\begin{tikzpicture}[
       thick,
       acteur/.style={
         circle,
         thick,
         inner sep=1pt,
         minimum size=0.1cm
       }
] 

\node (a1) at (0,0) [acteur,fill=blue]{};
\node (a2) at (0,0.8) [acteur,fill=blue]{};
\node (a3) at (0,1.6) [acteur,fill=blue]{}; 
\node (a4) at (0,2.4) [acteur,fill=blue]{};
\node (a5) at (0,3.2) [acteur,fill=blue]{}; 
\node (a6) at (0,4.0) [acteur,fill=blue]{}; 
\node (a7) at (0.8,1.6) [acteur,fill=orange]{};
\node (a8) at (1.6,0.8) [acteur,fill=black]{}; 
\node (a9) at (0.8,3.2) [acteur,fill=orange]{}; 
\node (a10) at (1.6,2.4) [acteur,fill=black]{};
\node (a11) at (2.4,0) [acteur,fill=red]{};
\node (a12) at (2.4,0.8) [acteur,fill=red]{};
\node (a13) at (2.4,1.6) [acteur,fill=red]{}; 
\node (a14) at (2.4,2.4) [acteur,fill=red]{}; 
\node (a15) at (2.4,3.2) [acteur,fill=red]{};
\node (a16) at (2.4,4.0) [acteur,fill=red]{};

\node (b1) at (-1.0,0) {$1$};
\node (b2) at (-1.0,0.8) {$i$}; 
\node (b3) at (-1.0,1.6) {$i+1$}; 
\node (b4) at (-1.0,2.4) {$2r-i$}; 
\node (b5) at (-1.0,3.2) {$2r+1-i$};
\node (b6) at (-1.0,4.0) {$2r$}; 
\node (b8) at (3.4,0) {$1$};
\node (b9) at (3.4,0.8) {$i$};
\node (b10) at (3.4,1.6) {$i+1$}; 
\node (b11) at (3.4,2.4) {$2r-1$}; 
\node (b12) at (3.4,3.2) {$2r+1-i$}; 
\node (b13) at (3.4,4.0) {$2r$};

\node (c) at (1.2,-0.8) {$E_{-i}(b_i)$};

\node (d1) at (1.2,0.2) [acteur,fill=black]{};
\node (d2) at (1.2,0.4) [acteur,fill=black]{};
\node (d3) at (1.2,0.6) [acteur,fill=black]{};
\node (d4) at (1.2,1.8) [acteur,fill=black]{};
\node (d5) at (1.2,2.0) [acteur,fill=black]{};
\node (d6) at (1.2,2.2) [acteur,fill=black]{};
\node (d7) at (1.2,3.4) [acteur,fill=black]{};
\node (d8) at (1.2,3.6) [acteur,fill=black]{};
\node (d9) at (1.2,3.8) [acteur,fill=black]{};

\draw[->] (a1) to node {} (a11);
\draw[->] (a2) to node {} (a8);
\draw[->] (a8) to node {} (a12);
\draw[->] (a3) to node {} (a7);
\draw[->] (a7) to node {} (a13);
\draw[->] (a4) to node {} (a10);
\draw[->] (a10) to node {} (a14);
\draw[->] (a5) to node {} (a9);
\draw[->] (a9) to node {} (a15);
\draw[->] (a6) to node {} (a16);
\draw[->] (a7) to node [right] {$b_i$} (a8);
\draw[->] (a9) to node [right] {$b_i$} (a10);

\end{tikzpicture}
\end{center}
\end{figure}

\subsubsection*{Type \textit{D}}\label{Section2.6.4}

Finally, let us consider the case where $G=SO_{2r}(\mathbb{C})$. By definition, we can take $G$ to be the set of all $2r\times 2r$ complex matrices $M$ that satisfy $M^TJ_DM=J_D$, where $J_D$ is the symmetric matrix $J_D=\sum_{i=1}^r(-1)^{i+1}(e_{i,2r+1-i}+e_{2r+1-i,i})$, that is, we have 
\begin{equation*}
J_D=
\begin{pmatrix}
0 & 0 & \cdots & 0 & 1 \\
0 & 0 & \cdots & -1 & 0 \\
\vdots & \vdots & \ddots & \vdots & \vdots \\
0 & -1 & \cdots & 0 & 0 \\
1 & 0 & \cdots & 0 & 0
\end{pmatrix}.
\end{equation*}
Thus, $\mathfrak{g}$ consists of the $2r\times 2r$ complex matrices $M$ that satisfy $M^TJ_D+J_DM=0$, which is then equivalent to the set of $2r\times 2r$ complex matrices $M$ that satisfy $m_{i,j}=(-1)^{i+j+1}m_{2r+2-j,2r+2-i}$, $m_{2r+1-i,j}=(-1)^{i+j+1}m_{2r+1-j,i}$ and $m_{i,2r+1-j}=(-1)^{i+j+1}m_{j,2r+1-i}$ for all $i,j\in[1,r]$. Moreover, a set of Chevalley generators $\{e_{\pm i}\}_{i=1}^r$ for $\mathfrak{g}$ is given by 
\begin{align*}
e_i&=e_{i,i+1}+e_{2r-i,2r+1-i},\quad e_{-i}=e_{i+1,i}+e_{2r+1-i,2r-i}, \quad i\in[1,r-1],\\ 
e_r&=e_{r-1,r+1}+e_{r,r+2},\quad e_{-r}=e_{r+1,r-1}+e_{r+2,r}, 
\end{align*}
which implies that $h_i=e_{i,i}-e_{i+1,i+1}+e_{2r-i,2r-i}-e_{2r+1-i,2r+1-i}$ for $i\in[1,r-1]$ and $h_r=e_{r-1,r-1}+e_{r,r}-e_{r+1,r+1}-e_{r+2,r+2}$. Hence, we have
\begin{figure}[t]
\caption{The elementary chips of $E_r(a_r)$, $E_{-r}(b_r)$ and $D(t_1,\ldots,t_r)$ in type $C$.}
\label{Figure2.5}
\begin{center}
\begin{tikzpicture}[
       thick,
       acteur/.style={
         circle,
         thick,
         inner sep=1pt,
         minimum size=0.1cm
       }
] 
\node (a1) at (0,0) [acteur,fill=blue]{};
\node (a2) at (0,0.8) [acteur,fill=blue]{}; 
\node (a3) at (0,1.6) [acteur,fill=blue]{}; 
\node (a4) at (0,2.4) [acteur,fill=blue]{}; 
\node (a5) at (0.8,0.8) [acteur,fill=orange]{};
\node (a6) at (1.6,1.6) [acteur,fill=black]{}; 
\node (a7) at (2.4,0) [acteur,fill=red]{};
\node (a8) at (2.4,0.8) [acteur,fill=red]{}; 
\node (a9) at (2.4,1.6) [acteur,fill=red]{}; 
\node (a10) at (2.4,2.4) [acteur,fill=red]{};

\node (b1) at (-0.6,0) {$1$};
\node (b2) at (-0.6,0.8) {$r$}; 
\node (b3) at (-0.6,1.6) {$r+1$}; 
\node (b4) at (-0.6,2.4) {$2r$}; 
\node (b5) at (3.0,0) {$1$};
\node (b6) at (3.0,0.8) {$r$}; 
\node (b7) at (3.0,1.6) {$r+1$}; 
\node (b8) at (3.0,2.4) {$2r$}; 

\node (c) at (1.2,-0.8) {$E_r(a_r)$};

\node (d1) at (1.2,0.2) [acteur,fill=black]{};
\node (d2) at (1.2,0.4) [acteur,fill=black]{};
\node (d3) at (1.2,0.6) [acteur,fill=black]{};
\node (d4) at (1.2,1.8) [acteur,fill=black]{};
\node (d5) at (1.2,2.0) [acteur,fill=black]{};
\node (d6) at (1.2,2.2) [acteur,fill=black]{};

\draw[->] (a1) to node {} (a7);
\draw[->] (a2) to node {} (a5);
\draw[->] (a5) to node {} (a8);
\draw[->] (a3) to node {} (a6);
\draw[->] (a6) to node {} (a9);
\draw[->] (a4) to node {} (a10);
\draw[->] (a5) to node [left] {$a_r$} (a6);

\end{tikzpicture}
\begin{tikzpicture}[
       thick,
       acteur/.style={
         circle,
         thick,
         inner sep=1pt,
         minimum size=0.1cm
       }
] 
\node (a1) at (0,0) [acteur,fill=blue]{};
\node (a2) at (0,0.8) [acteur,fill=blue]{}; 
\node (a3) at (0,1.6) [acteur,fill=blue]{}; 
\node (a4) at (0,2.4) [acteur,fill=blue]{}; 
\node (a5) at (0.8,1.6) [acteur,fill=orange]{};
\node (a6) at (1.6,0.8) [acteur,fill=black]{}; 
\node (a7) at (2.4,0) [acteur,fill=red]{};
\node (a8) at (2.4,0.8) [acteur,fill=red]{}; 
\node (a9) at (2.4,1.6) [acteur,fill=red]{}; 
\node (a10) at (2.4,2.4) [acteur,fill=red]{};

\node (b1) at (-0.6,0) {$1$};
\node (b2) at (-0.6,0.8) {$r$}; 
\node (b3) at (-0.6,1.6) {$r+1$}; 
\node (b4) at (-0.6,2.4) {$2r$}; 
\node (b5) at (3.0,0) {$1$};
\node (b6) at (3.0,0.8) {$r$}; 
\node (b7) at (3.0,1.6) {$r+1$}; 
\node (b8) at (3.0,2.4) {$2r$}; 

\node (c) at (1.2,-0.8) {$E_{-r}(b_r)$};

\node (d1) at (1.2,0.2) [acteur,fill=black]{};
\node (d2) at (1.2,0.4) [acteur,fill=black]{};
\node (d3) at (1.2,0.6) [acteur,fill=black]{};
\node (d4) at (1.2,1.8) [acteur,fill=black]{};
\node (d5) at (1.2,2.0) [acteur,fill=black]{};
\node (d6) at (1.2,2.2) [acteur,fill=black]{};

\draw[->] (a1) to node {} (a7);
\draw[->] (a2) to node {} (a6);
\draw[->] (a6) to node {} (a8);
\draw[->] (a3) to node {} (a5);
\draw[->] (a5) to node {} (a9);
\draw[->] (a4) to node {} (a10);
\draw[->] (a5) to node [right] {$b_r$} (a6);

\end{tikzpicture}
\begin{tikzpicture}[
       thick,
       acteur/.style={
         circle,
         thick,
         inner sep=1pt,
         minimum size=0.1cm
       }
] 
\node (a1) at (0,0) [acteur,fill=blue]{};
\node (a2) at (0,0.8) [acteur,fill=blue]{}; 
\node (a3) at (0,1.6) [acteur,fill=blue]{}; 
\node (a4) at (0,2.4) [acteur,fill=blue]{}; 
\node (a5) at (2.4,0) [acteur,fill=red]{};
\node (a6) at (2.4,0.8) [acteur,fill=red]{}; 
\node (a7) at (2.4,1.6) [acteur,fill=red]{}; 
\node (a8) at (2.4,2.4) [acteur,fill=red]{};

\node (b1) at (-0.6,0) {$1$};
\node (b2) at (-0.6,0.8) {$r$}; 
\node (b3) at (-0.6,1.6) {$r+1$}; 
\node (b4) at (-0.6,2.4) {$2r$}; 
\node (b5) at (3.0,0) {$1$};
\node (b6) at (3.0,0.8) {$r$}; 
\node (b7) at (3.0,1.6) {$r+1$}; 
\node (b8) at (3.0,2.4) {$2r$}; 

\node (c) at (1.2,-0.8) {$D(t_1,\ldots,t_r)$};

\node (d1) at (1.2,0.2) [acteur,fill=black]{};
\node (d2) at (1.2,0.4) [acteur,fill=black]{};
\node (d3) at (1.2,0.6) [acteur,fill=black]{};
\node (d4) at (1.2,1.8) [acteur,fill=black]{};
\node (d5) at (1.2,2.0) [acteur,fill=black]{};
\node (d6) at (1.2,2.2) [acteur,fill=black]{};

\draw[->] (a1) to node [below right] {$t_1$} (a5);
\draw[->] (a2) to node [below right] {$t_r/t_{r-1}$} (a6);
\draw[->] (a3) to node [above left] {$t_{r-1}/t_r$} (a7);
\draw[->] (a4) to node [above left] {$1/t_1$} (a8);
\end{tikzpicture} 
\end{center}
\end{figure}
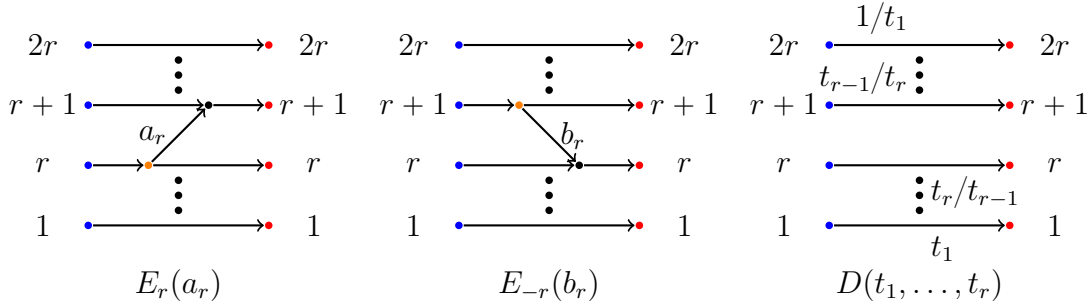
\begin{figure}[t]
\caption{The elementary chips of $E_i(a_i)$, $E_{-i}(b_i)$ and $D(t_1,\ldots,t_r)$ in type $D$.}
\label{Figure2.6}
\begin{center}
\begin{tikzpicture}[
       thick,
       acteur/.style={
         circle,
         thick,
         inner sep=1pt,
         minimum size=0.1cm
       }
] 

\node (a1) at (0,0) [acteur,fill=blue]{};
\node (a2) at (0,0.8) [acteur,fill=blue]{};
\node (a3) at (0,1.6) [acteur,fill=blue]{}; 
\node (a4) at (0,2.4) [acteur,fill=blue]{};
\node (a5) at (0,3.2) [acteur,fill=blue]{}; 
\node (a6) at (0,4.0) [acteur,fill=blue]{}; 
\node (a7) at (0.8,0.8) [acteur,fill=orange]{};
\node (a8) at (1.6,1.6) [acteur,fill=black]{}; 
\node (a9) at (0.8,2.4) [acteur,fill=orange]{}; 
\node (a10) at (1.6,3.2) [acteur,fill=black]{};
\node (a11) at (2.4,0) [acteur,fill=red]{};
\node (a12) at (2.4,0.8) [acteur,fill=red]{};
\node (a13) at (2.4,1.6) [acteur,fill=red]{}; 
\node (a14) at (2.4,2.4) [acteur,fill=red]{}; 
\node (a15) at (2.4,3.2) [acteur,fill=red]{};
\node (a16) at (2.4,4.0) [acteur,fill=red]{};

\node (b1) at (-1.0,0) {$1$};
\node (b2) at (-1.0,0.8) {$i$}; 
\node (b3) at (-1.0,1.6) {$i+1$}; 
\node (b4) at (-1.0,2.4) {$2r-i$}; 
\node (b5) at (-1.0,3.2) {$2r+1-i$};
\node (b6) at (-1.0,4.0) {$2r$}; 
\node (b8) at (3.4,0) {$1$};
\node (b9) at (3.4,0.8) {$i$};
\node (b10) at (3.4,1.6) {$i+1$}; 
\node (b11) at (3.4,2.4) {$2r-1$}; 
\node (b12) at (3.4,3.2) {$2r+1-i$}; 
\node (b13) at (3.4,4.0) {$2r$};

\node (c) at (1.2,-0.8) {$E_i(a_i)$};

\node (d1) at (1.2,0.2) [acteur,fill=black]{};
\node (d2) at (1.2,0.4) [acteur,fill=black]{};
\node (d3) at (1.2,0.6) [acteur,fill=black]{};
\node (d4) at (1.2,1.8) [acteur,fill=black]{};
\node (d5) at (1.2,2.0) [acteur,fill=black]{};
\node (d6) at (1.2,2.2) [acteur,fill=black]{};
\node (d7) at (1.2,3.4) [acteur,fill=black]{};
\node (d8) at (1.2,3.6) [acteur,fill=black]{};
\node (d9) at (1.2,3.8) [acteur,fill=black]{};

\draw[->] (a1) to node {} (a11);
\draw[->] (a2) to node {} (a7);
\draw[->] (a7) to node {} (a12);
\draw[->] (a3) to node {} (a8);
\draw[->] (a8) to node {} (a13);
\draw[->] (a4) to node {} (a9);
\draw[->] (a9) to node {} (a14);
\draw[->] (a5) to node {} (a10);
\draw[->] (a10) to node {} (a15);
\draw[->] (a6) to node {} (a16);
\draw[->] (a7) to node [left] {$a_i$} (a8);
\draw[->] (a9) to node [left] {$a_i$} (a10);

\end{tikzpicture}
\begin{tikzpicture}[
       thick,
       acteur/.style={
         circle,
         thick,
         inner sep=1pt,
         minimum size=0.1cm
       }
] 

\node (a1) at (0,0) [acteur,fill=blue]{};
\node (a2) at (0,0.8) [acteur,fill=blue]{};
\node (a3) at (0,1.6) [acteur,fill=blue]{}; 
\node (a4) at (0,2.4) [acteur,fill=blue]{};
\node (a5) at (0,3.2) [acteur,fill=blue]{}; 
\node (a6) at (0,4.0) [acteur,fill=blue]{}; 
\node (a7) at (0.8,1.6) [acteur,fill=orange]{};
\node (a8) at (1.6,0.8) [acteur,fill=black]{}; 
\node (a9) at (0.8,3.2) [acteur,fill=orange]{}; 
\node (a10) at (1.6,2.4) [acteur,fill=black]{};
\node (a11) at (2.4,0) [acteur,fill=red]{};
\node (a12) at (2.4,0.8) [acteur,fill=red]{};
\node (a13) at (2.4,1.6) [acteur,fill=red]{}; 
\node (a14) at (2.4,2.4) [acteur,fill=red]{}; 
\node (a15) at (2.4,3.2) [acteur,fill=red]{};
\node (a16) at (2.4,4.0) [acteur,fill=red]{};

\node (b1) at (-1.0,0) {$1$};
\node (b2) at (-1.0,0.8) {$i$}; 
\node (b3) at (-1.0,1.6) {$i+1$}; 
\node (b4) at (-1.0,2.4) {$2r-i$}; 
\node (b5) at (-1.0,3.2) {$2r+1-i$};
\node (b6) at (-1.0,4.0) {$2r$}; 
\node (b8) at (3.4,0) {$1$};
\node (b9) at (3.4,0.8) {$i$};
\node (b10) at (3.4,1.6) {$i+1$}; 
\node (b11) at (3.4,2.4) {$2r-1$}; 
\node (b12) at (3.4,3.2) {$2r+1-i$}; 
\node (b13) at (3.4,4.0) {$2r$};

\node (c) at (1.2,-0.8) {$E_{-i}(b_i)$};

\node (d1) at (1.2,0.2) [acteur,fill=black]{};
\node (d2) at (1.2,0.4) [acteur,fill=black]{};
\node (d3) at (1.2,0.6) [acteur,fill=black]{};
\node (d4) at (1.2,1.8) [acteur,fill=black]{};
\node (d5) at (1.2,2.0) [acteur,fill=black]{};
\node (d6) at (1.2,2.2) [acteur,fill=black]{};
\node (d7) at (1.2,3.4) [acteur,fill=black]{};
\node (d8) at (1.2,3.6) [acteur,fill=black]{};
\node (d9) at (1.2,3.8) [acteur,fill=black]{};

\draw[->] (a1) to node {} (a11);
\draw[->] (a2) to node {} (a8);
\draw[->] (a8) to node {} (a12);
\draw[->] (a3) to node {} (a7);
\draw[->] (a7) to node {} (a13);
\draw[->] (a4) to node {} (a10);
\draw[->] (a10) to node {} (a14);
\draw[->] (a5) to node {} (a9);
\draw[->] (a9) to node {} (a15);
\draw[->] (a6) to node {} (a16);
\draw[->] (a7) to node [right] {$b_i$} (a8);
\draw[->] (a9) to node [right] {$b_i$} (a10);

\end{tikzpicture}
\end{center}
\begin{center}
\begin{tikzpicture}[
       thick,
       acteur/.style={
         circle,
         thick,
         inner sep=1pt,
         minimum size=0.1cm
       }
] 

\node (a1) at (0,0) [acteur,fill=blue]{};
\node (a2) at (0,0.8) [acteur,fill=blue]{};
\node (a3) at (0,1.6) [acteur,fill=blue]{}; 
\node (a4) at (0,2.4) [acteur,fill=blue]{};
\node (a5) at (0,3.2) [acteur,fill=blue]{}; 
\node (a6) at (0,4.0) [acteur,fill=blue]{}; 
\node (a7) at (0.8,0.8) [acteur,fill=orange]{};
\node (a8) at (1.6,2.4) [acteur,fill=black]{}; 
\node (a9) at (0.8,1.6) [acteur,fill=orange]{}; 
\node (a10) at (1.6,3.2) [acteur,fill=black]{};
\node (a11) at (2.4,0) [acteur,fill=red]{};
\node (a12) at (2.4,0.8) [acteur,fill=red]{};
\node (a13) at (2.4,1.6) [acteur,fill=red]{}; 
\node (a14) at (2.4,2.4) [acteur,fill=red]{}; 
\node (a15) at (2.4,3.2) [acteur,fill=red]{};
\node (a16) at (2.4,4.0) [acteur,fill=red]{};

\node (b1) at (-0.6,0) {$1$};
\node (b2) at (-0.6,0.8) {$r-1$}; 
\node (b3) at (-0.6,1.6) {$r$}; 
\node (b4) at (-0.6,2.4) {$r+1$}; 
\node (b5) at (-0.6,3.2) {$r+2$};
\node (b6) at (-0.6,4.0) {$2r$}; 
\node (b7) at (3.0,0) {$1$};
\node (b8) at (3.0,0.8) {$r-1$};
\node (b9) at (3.0,1.6) {$r$}; 
\node (b10) at (3.0,2.4) {$r+1$}; 
\node (b11) at (3.0,3.2) {$r+2$}; 
\node (b12) at (3.0,4.0) {$2r$};

\node (c) at (1.2,-0.8) {$E_r(a_r)$};

\node (d1) at (1.2,0.2) [acteur,fill=black]{};
\node (d2) at (1.2,0.4) [acteur,fill=black]{};
\node (d3) at (1.2,0.6) [acteur,fill=black]{};
\node (d4) at (1.2,3.4) [acteur,fill=black]{};
\node (d5) at (1.2,3.6) [acteur,fill=black]{};
\node (d6) at (1.2,3.8) [acteur,fill=black]{};

\draw[->] (a1) to node {} (a11);
\draw[->] (a2) to node {} (a7);
\draw[->] (a7) to node {} (a12);
\draw[->] (a3) to node {} (a9);
\draw[->] (a9) to node {} (a13);
\draw[->] (a4) to node {} (a8);
\draw[->] (a8) to node {} (a14);
\draw[->] (a5) to node {} (a10);
\draw[->] (a10) to node {} (a15);
\draw[->] (a6) to node {} (a16);
\draw[->] (a7) to node [below left=0.15cm] {$a_r$} (a8);
\draw[->] (a9) to node [above right=0.15cm] {$a_r$} (a10);

\end{tikzpicture}
\begin{tikzpicture}[
       thick,
       acteur/.style={
         circle,
         thick,
         inner sep=1pt,
         minimum size=0.1cm
       }
] 

\node (a1) at (0,0) [acteur,fill=blue]{};
\node (a2) at (0,0.8) [acteur,fill=blue]{};
\node (a3) at (0,1.6) [acteur,fill=blue]{}; 
\node (a4) at (0,2.4) [acteur,fill=blue]{};
\node (a5) at (0,3.2) [acteur,fill=blue]{}; 
\node (a6) at (0,4.0) [acteur,fill=blue]{}; 
\node (a7) at (0.8,2.4) [acteur,fill=orange]{};
\node (a8) at (1.6,0.8) [acteur,fill=black]{}; 
\node (a9) at (0.8,3.2) [acteur,fill=black]{}; 
\node (a10) at (1.6,1.6) [acteur,fill=black]{};
\node (a11) at (2.4,0) [acteur,fill=red]{};
\node (a12) at (2.4,0.8) [acteur,fill=red]{};
\node (a13) at (2.4,1.6) [acteur,fill=red]{}; 
\node (a14) at (2.4,2.4) [acteur,fill=red]{}; 
\node (a15) at (2.4,3.2) [acteur,fill=red]{};
\node (a16) at (2.4,4.0) [acteur,fill=red]{};

\node (b1) at (-0.6,0) {$1$};
\node (b2) at (-0.6,0.8) {$r-1$}; 
\node (b3) at (-0.6,1.6) {$r$}; 
\node (b4) at (-0.6,2.4) {$r+1$}; 
\node (b5) at (-0.6,3.2) {$r+2$};
\node (b6) at (-0.6,4.0) {$2r$}; 
\node (b7) at (3.0,0) {$1$};
\node (b8) at (3.0,0.8) {$r-1$};
\node (b9) at (3.0,1.6) {$r$}; 
\node (b10) at (3.0,2.4) {$r+1$}; 
\node (b11) at (3.0,3.2) {$r+2$}; 
\node (b12) at (3.0,4.0) {$2r$};

\node (c) at (1.2,-0.8) {$E_{-r}(b_r)$};

\node (d1) at (1.2,0.2) [acteur,fill=black]{};
\node (d2) at (1.2,0.4) [acteur,fill=black]{};
\node (d3) at (1.2,0.6) [acteur,fill=black]{};
\node (d4) at (1.2,3.4) [acteur,fill=black]{};
\node (d5) at (1.2,3.6) [acteur,fill=black]{};
\node (d6) at (1.2,3.8) [acteur,fill=black]{};

\draw[->] (a1) to node {} (a11);
\draw[->] (a2) to node {} (a8);
\draw[->] (a8) to node {} (a12);
\draw[->] (a3) to node {} (a10);
\draw[->] (a10) to node {} (a13);
\draw[->] (a4) to node {} (a7);
\draw[->] (a7) to node {} (a14);
\draw[->] (a5) to node {} (a9);
\draw[->] (a9) to node {} (a15);
\draw[->] (a6) to node {} (a16);
\draw[->] (a7) to node [below left] {$b_r$} (a8);
\draw[->] (a9) to node [above right] {$b_r$} (a10);

\end{tikzpicture}
\begin{tikzpicture}[
       thick,
       acteur/.style={
         circle,
         thick,
         inner sep=1pt,
         minimum size=0.1cm
       }
] 

\node (a1) at (0,0) [acteur,fill=blue]{};
\node (a2) at (0,0.8) [acteur,fill=blue]{};
\node (a3) at (0,1.6) [acteur,fill=blue]{}; 
\node (a4) at (0,2.4) [acteur,fill=blue]{};
\node (a5) at (0,3.2) [acteur,fill=blue]{}; 
\node (a6) at (0,4.0) [acteur,fill=blue]{}; 
\node (a7) at (2.4,0) [acteur,fill=red]{};
\node (a8) at (2.4,0.8) [acteur,fill=red]{};
\node (a9) at (2.4,1.6) [acteur,fill=red]{}; 
\node (a10) at (2.4,2.4) [acteur,fill=red]{}; 
\node (a11) at (2.4,3.2) [acteur,fill=red]{};
\node (a12) at (2.4,4.0) [acteur,fill=red]{};

\node (b1) at (-0.6,0) {$1$};
\node (b2) at (-0.6,0.8) {$r-1$}; 
\node (b3) at (-0.6,1.6) {$r$}; 
\node (b4) at (-0.6,2.4) {$r+1$}; 
\node (b5) at (-0.6,3.2) {$r+2$};
\node (b6) at (-0.6,4.0) {$2r$}; 
\node (b7) at (3.0,0) {$1$};
\node (b8) at (3.0,0.8) {$r-1$};
\node (b9) at (3.0,1.6) {$r$}; 
\node (b10) at (3.0,2.4) {$r+1$}; 
\node (b11) at (3.0,3.2) {$r+2$}; 
\node (b12) at (3.0,4.0) {$2r$};

\node (c) at (1.2,-0.8) {$D(t_1,\ldots,t_r)$};

\node (d1) at (1.2,0.2) [acteur,fill=black]{};
\node (d2) at (1.2,0.4) [acteur,fill=black]{};
\node (d3) at (1.2,0.6) [acteur,fill=black]{};
\node (d4) at (1.2,3.4) [acteur,fill=black]{};
\node (d5) at (1.2,3.6) [acteur,fill=black]{};
\node (d6) at (1.2,3.8) [acteur,fill=black]{};

\draw[->] (a1) to node [below] {$t_1$} (a7);
\draw[->] (a2) to node [above] {$t_{r-1}t_r/t_{r-2}$} (a8);
\draw[->] (a3) to node [above left] {$t_r/t_{r-1}$} (a9);
\draw[->] (a4) to node [below right] {$t_{r-1}/t_r$} (a10);
\draw[->] (a5) to node [below] {$t_{r-2}/(t_{r-1}t_r)$} (a11);
\draw[->] (a6) to node [above] {$1/t_1$} (a12);

\end{tikzpicture}
\end{center}
\end{figure}
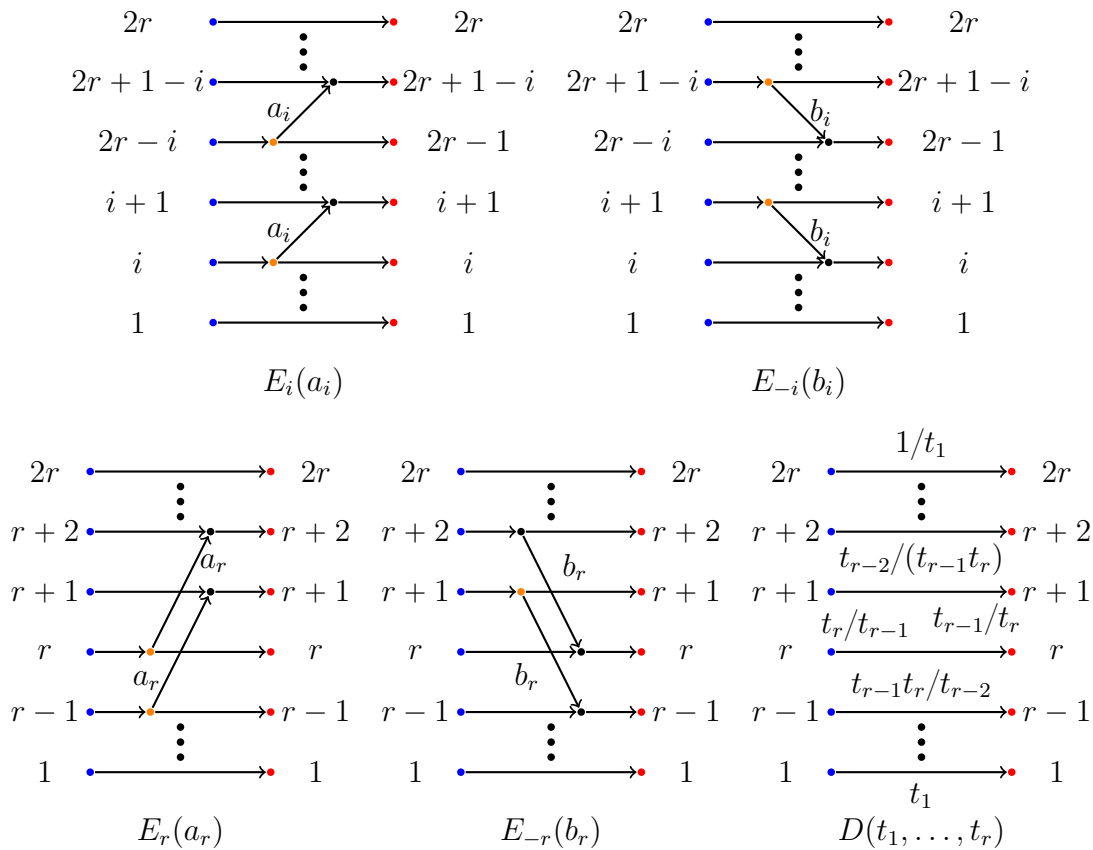
\begin{align}
E_i(a_i)&=I+a_ie_{i,i+1}+a_ie_{2r-i,2r+1-i},\label{eq:2.30}\\
E_{-i}(b_i)&=I+b_ie_{i+1,i}+b_ie_{2r+1-i,2r-i},\label{eq:2.31}\\
E_r(a_r)&=I+a_re_{r-1,r+1}+a_re_{r,r+2},\label{eq:2.32}\\
E_{-r}(b_r)&=I+b_re_{r+1,r-1}+b_re_{r+2,r},\quad\text{and}\label{eq:2.33}\\
D(t_1,\ldots,t_r)&=\diag(t_1,\ldots,t_{r-1}t_r/t_{r-2},t_r/t_{r-1},t_{r-1}/t_r,t_{r-2}/(t_{r-1}t_r),\ldots,1/t_1)\label{eq:2.34}
\end{align}
for all $i\in[1,r]$. The elements $E_i(a_i)$, $E_{-i}(b_i)$ and $D(t_1,\ldots,t_r)$ could be represented in the form of elementary chips as shown in Figure \ref{Figure2.6}.

\subsubsection*{Network representation of elements of the conjugation quotient Coxeter double Bruhat cell}\label{Section2.6.5}

Given a double reduced word $\mathbf{i}=(i_1,\ldots,i_m)$ for $(u,v)$, we may express an element in the double Bruhat cell $G^{u,v}$ that admits the factorization scheme \eqref{eq:2.2} in the form of a network diagram $N_{u,v}(\mathbf{i})$ consisting of $m+1$ elementary chips $C_1,\ldots,C_{m+1}$, as follows: $C_1$ corresponds to the element $D(t_1,\ldots,t_r)$, and $C_{j+1}$ corresponds to the element $E_{i_j}(a_{i_j})$ for all $j\in[1,m]$. The network diagram $N_{u,v}(\mathbf{i})$ is then obtained by gluing the sink of $C_i$ at level $k$ with the source of $C_{i+1}$ at level $k$ for $i\in[1,m]$ and $k\in[1,\dim V(\omega_1)]$, and removing all glued vertices. 

\begin{example}\label{2.11}
An example of the network diagram $N_{u,v}(\mathbf{i})$ is shown in Figure \ref{Figure2.7} in the case where $G=SL_4(\mathbb{C})$, $u=s_1s_2s_3$, $v=s_1s_3s_2$ and $\mathbf{i}=(-1,-2,-3,1,3,2)$.
\end{example}
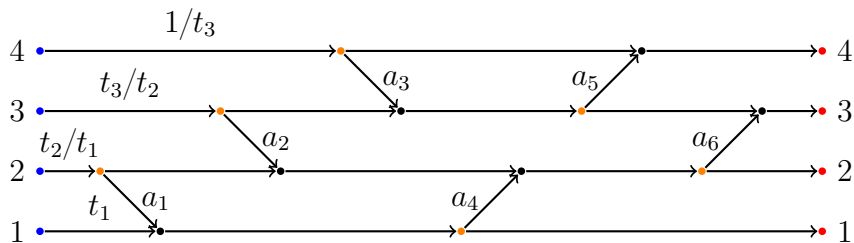
\begin{figure}[t]
\caption{The network diagram $N_{u,v}(\mathbf{i})$ when $G=SL_4(\mathbb{C})$, $u=s_1s_2s_3$, $v=s_1s_3s_2$ and $\mathbf{i}=(-1,-2,-3,1,3,2)$.}
\label{Figure2.7}
\begin{center}
\begin{tikzpicture}[
       thick,
       acteur/.style={
         circle,
         thick,
         inner sep=1pt,
         minimum size=0.1cm
       }
] 
\node (a1) at (0,0) [acteur,fill=blue]{};
\node (a2) at (0,0.8) [acteur,fill=blue]{}; 
\node (a3) at (0,1.6) [acteur,fill=blue]{}; 
\node (a4) at (0,2.4) [acteur,fill=blue]{}; 
\node (a5) at (0.8,0.8) [acteur,fill=orange]{};
\node (a6) at (1.6,0) [acteur,fill=black]{}; 
\node (a7) at (2.4,1.6) [acteur,fill=orange]{}; 
\node (a8) at (3.2,0.8) [acteur,fill=black]{}; 
\node (a9) at (4.0,2.4) [acteur,fill=orange]{};
\node (a10) at (4.8,1.6) [acteur,fill=black]{}; 
\node (a11) at (5.6,0) [acteur,fill=orange]{}; 
\node (a12) at (6.4,0.8) [acteur,fill=black]{}; 
\node (a13) at (7.2,1.6) [acteur,fill=orange]{};
\node (a14) at (8.0,2.4) [acteur,fill=black]{}; 
\node (a15) at (8.8,0.8) [acteur,fill=orange]{}; 
\node (a16) at (9.6,1.6) [acteur,fill=black]{}; 
\node (a17) at (10.4,0) [acteur,fill=red]{};
\node (a18) at (10.4,0.8) [acteur,fill=red]{}; 
\node (a19) at (10.4,1.6) [acteur,fill=red]{}; 
\node (a20) at (10.4,2.4) [acteur,fill=red]{}; 

\node (b1) at (-0.3,0) {1};
\node (b2) at (-0.3,0.8) {2}; 
\node (b3) at (-0.3,1.6) {3}; 
\node (b4) at (-0.3,2.4) {4}; 
\node (b5) at (10.7,0) {1};
\node (b6) at (10.7,0.8) {2}; 
\node (b7) at (10.7,1.6) {3}; 
\node (b8) at (10.7,2.4) {4}; 

\draw[->] (a1) to node [above] {$t_1$} (a6);
\draw[->] (a6) to node {} (a11);
\draw[->] (a11) to node {} (a17);
\draw[->] (a2) to node [above] {$t_2/t_1$} (a5);
\draw[->] (a5) to node {} (a8);
\draw[->] (a8) to node {} (a12);
\draw[->] (a12) to node {} (a15);
\draw[->] (a15) to node {} (a18);
\draw[->] (a3) to node [above] {$t_3/t_2$} (a7);
\draw[->] (a7) to node {} (a10);
\draw[->] (a10) to node {} (a13);
\draw[->] (a13) to node {} (a16);
\draw[->] (a16) to node {} (a19);
\draw[->] (a4) to node [above] {$1/t_3$} (a9);
\draw[->] (a9) to node {} (a14);
\draw[->] (a14) to node {} (a20);
\draw[->] (a5) to node [right] {$a_1$} (a6);
\draw[->] (a7) to node [right] {$a_2$} (a8);
\draw[->] (a9) to node [right] {$a_3$} (a10);
\draw[->] (a11) to node [left] {$a_4$} (a12);
\draw[->] (a13) to node [left] {$a_5$} (a14);
\draw[->] (a15) to node [left] {$a_6$} (a16);

\end{tikzpicture} 
\end{center}
\end{figure}

Similarly, given a pair $(u,v)$ of Coxeter elements in $W$, and an unmixed double reduced word $\mathbf{i}$ for $(u,v)$, we may express an element in the conjugation quotient Coxeter double Bruhat cell $G^{u,v}/H$ that admits the factorization scheme \eqref{eq:2.6} in the form of a network diagram $\overline{N}_{u,v}(\mathbf{i})$, in a similar fashion as in Example \ref{2.11}. 

\begin{example}\label{2.12}
An example of the network diagram $\overline{N}_{u,v}(\mathbf{i})$ is shown in Figure \ref{Figure2.8} in the case where $G=SO_5(\mathbb{C})$, $u=v=s_1s_2$ and $\mathbf{i}=(-1,-2,1,2)$.
\begin{figure}[t]
\caption{The network diagram $\overline{N}_{u,v}(\mathbf{i})$ when $G=SO_5(\mathbb{C})$, $u=v=s_1s_2$ and $\mathbf{i}=(-1,-2,1,2)$.}
\label{Figure2.8}
\begin{center}
\begin{tikzpicture}[
       thick,
       acteur/.style={
         circle,
         fill=black,
         thick,
         inner sep=1pt,
         minimum size=0.1cm
       }
] 
\node (a1) at (0,0) [acteur,fill=blue]{};
\node (a2) at (0,0.8) [acteur,fill=blue]{}; 
\node (a3) at (0,1.6) [acteur,fill=blue]{}; 
\node (a4) at (0,2.4) [acteur,fill=blue]{}; 
\node (a5) at (0,3.2) [acteur,fill=blue]{};
\node (a6) at (0.8,0.8) [acteur,fill=orange]{};
\node (a7) at (1.6,0) [acteur,fill=black]{}; 
\node (a8) at (0.8,3.2) [acteur,fill=orange]{}; 
\node (a9) at (1.6,2.4) [acteur,fill=black]{}; 
\node (a10) at (2.4,1.6) [acteur,fill=orange]{};
\node (a11) at (3.2,0.8) [acteur,fill=black]{}; 
\node (a12) at (2.4,2.4) [acteur,fill=orange]{};
\node (a13) at (3.2,1.6) [acteur,fill=black]{}; 
\node (a14) at (4.8,0) [acteur,fill=orange]{}; 
\node (a15) at (5.6,0.8) [acteur,fill=black]{}; 
\node (a16) at (4.8,2.4) [acteur,fill=orange]{};
\node (a17) at (5.6,3.2) [acteur,fill=black]{}; 
\node (a18) at (6.4,0.8) [acteur,fill=orange]{}; 
\node (a19) at (7.2,1.6) [acteur,fill=black]{}; 
\node (a20) at (6.4,1.6) [acteur,fill=orange]{}; 
\node (a21) at (7.2,2.4) [acteur,fill=black]{}; 
\node (a22) at (8.0,0) [acteur,fill=red]{};
\node (a23) at (8.0,0.8) [acteur,fill=red]{}; 
\node (a24) at (8.0,1.6) [acteur,fill=red]{}; 
\node (a25) at (8.0,2.4) [acteur,fill=red]{}; 
\node (a26) at (8.0,3.2) [acteur,fill=red]{};

\node (b1) at (-0.3,0) {1};
\node (b2) at (-0.3,0.8) {2}; 
\node (b3) at (-0.3,1.6) {3}; 
\node (b4) at (-0.3,2.4) {4}; 
\node (b5) at (-0.3,3.2) {5};
\node (b6) at (8.3,0) {1};
\node (b7) at (8.3,0.8) {2}; 
\node (b8) at (8.3,1.6) {3}; 
\node (b9) at (8.3,2.4) {4}; 
\node (b10) at (8.3,3.2) {5};

\draw[->] (a1) to node {} (a7);
\draw[->] (a7) to node [above left] {$t_1$} (a14);
\draw[->] (a14) to node {} (a22);
\draw[->] (a2) to node {} (a6);
\draw[->] (a6) to node {} (a11);
\draw[->] (a11) to node [above] {$t_2^2/t_1$} (a15);
\draw[->] (a15) to node {} (a18);
\draw[->] (a18) to node {} (a23);
\draw[->] (a3) to node {} (a10);
\draw[->] (a10) to node {} (a13);
\draw[->] (a13) to node [above] {$1$} (a20);
\draw[->] (a20) to node {} (a19);
\draw[->] (a19) to node {} (a24);
\draw[->] (a4) to node {} (a9);
\draw[->] (a9) to node {} (a12);
\draw[->] (a12) to node [above] {$t_1/t_2^2$} (a16);
\draw[->] (a16) to node {} (a21);
\draw[->] (a21) to node {} (a25);
\draw[->] (a5) to node {} (a8);
\draw[->] (a8) to node [above] {$1/t_1$} (a17);
\draw[->] (a17) to node {} (a26);
\draw[->] (a6) to node [right] {$1$} (a7);
\draw[->] (a8) to node [right] {$1$} (a9);
\draw[->] (a10) to node [left=0.2cm] {$\sqrt{2}$} (a11);
\draw[->] (a12) to node [right=0.2cm] {$\sqrt{2}$} (a13);
\draw[->] (a12) to node [above left=0.3cm] {$1$} (a11);
\draw[->] (a14) to node [left] {$c_1$} (a15);
\draw[->] (a16) to node [left] {$c_1$} (a17);
\draw[->] (a18) to node [right=0.3cm] {$\sqrt{2}c_2$} (a19);
\draw[->] (a20) to node [left=0.3cm] {$\sqrt{2}c_2$} (a21);
\draw[->] (a18) to node [above right=0.2cm] {$c_2^2$} (a21);

\end{tikzpicture} 
\end{center}
\end{figure}
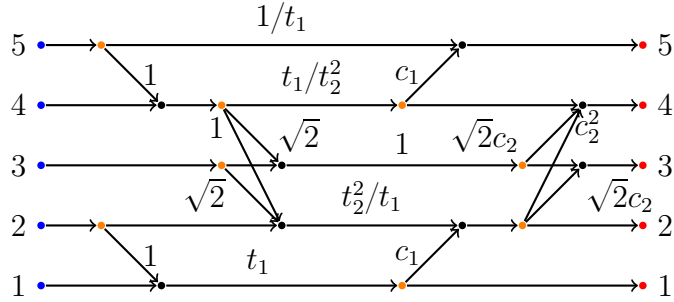
\end{example}
\section{Cluster structures on Coxeter double Bruhat cells}\label{Section3}

\subsection{Cluster structures on double Bruhat cells}\label{Section3.1}

In this subsection, we will recall the construction of cluster structures on double Bruhat cells, as described in \cite{BFZ05, FG06}. More precisely, following \cite{FG06} we will construct a seed $\Sigma_{\mathbf{i}}$ for any double reduced word $\mathbf{i}$ of a pair $(u,v)$ of elements in $W$. As in \cite{BFZ05}, we will construct $\Sigma_{\mathbf{i}}$ in two steps: firstly, we will construct an exchange quiver $\Gamma_{\mathbf{i}}$ that determines the signs of the entries of the exchange matrix $B_{\mathbf{i}}$. We will then determine those absolute values of these entries of $B_{\mathbf{i}}$ using a series of rules. While this construction works for all simple Lie groups, we shall restrict our construction to the case where $G$ is classical, for the ease of presentation.

\begin{remark}
In \cite{FG06}, the notations $\mathbf{J}(D)$ and $\varepsilon(D)$ is used instead of $\Sigma_{\mathbf{i}}$ and $B_{\mathbf{i}}$ respectively, where $D$ denotes a double reduced word for some pair of elements in $W$.
\end{remark}

To begin, we define the indexing set $I$ of $\Sigma_{\mathbf{i}}$ to be $\{-1,\ldots,r\}\cup\{1,\ldots,m\}$. Let us consider the subnetwork $D_\mathbf{i}$ obtained from $N_{u,v}(\mathbf{i})$ by removing all vertices and edges that lie above level $r+1$. It is then easy to see that there are $m+r$ faces in the network $D_\mathbf{i}$ that are bounded between levels $1$ and $r+1$. Moreover, if $G=SL_{r+1}(\mathbb{C}),SO_{2r+1}(\mathbb{C})$ or $Sp_{2r}(\mathbb{C})$, then any such face of $D_\mathbf{i}$ must be bounded between levels $k$ and $k+1$ for some $k\in[1,r]$, and if $G=SO_{2r}(\mathbb{C})$, then any such face of $D_\mathbf{i}$ must be bounded between either levels $k$ and $k+1$ for some $k\in[1,r-1]$, or levels $r-1$ and $r+1$.

Next, we will define the exchange quiver $\Gamma_{\mathbf{i}}$. The vertices of $\Gamma_{\mathbf{i}}$ correspond to each of these $m+r$ faces bounded between levels $1$ and $r+1$ in the network $D_\mathbf{i}$. Next, we label these $m+r$ vertices by the elements of the indexing set $I$ as follows: we first fix $k\in[1,r]$, and let $\{j:|i_j|=k\}=\{j_1<\cdots<j_n\}$. If either $G=SL_{r+1}(\mathbb{C}),SO_{2r+1}(\mathbb{C})$ or $Sp_{2r}(\mathbb{C})$, or $G=SO_{2r}(\mathbb{C})$ and $k\neq r$, we label the $n$ vertices corresponding to the $n$ faces bounded between levels $k$ and $k+1$ by $j_1,\ldots,j_n$ from left to right. If $G=SO_{2r}(\mathbb{C})$ and $k=r$, we label the $n$ vertices corresponding to the $n$ faces bounded between levels $r-1$ and $r+1$ by $j_1,\ldots,j_n$ from left to right. As an example, when $G=Sp_4(\mathbb{C})$, $u=v=s_1s_2$ and $\mathbf{i}=(-1,-2,1,2)$, the subnetwork $D_\mathbf{i}$, along with its labeled faces corresponding to the vertices of the exchange quiver $\Gamma_{\mathbf{i}}$ are shown in Figure \ref{Figure3.1} (for convenience, we have omitted the factorization parameters from $D_\mathbf{i}$).

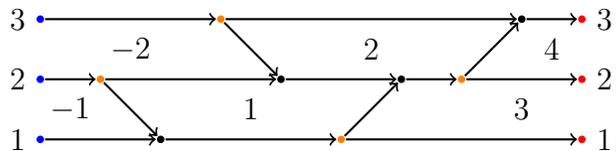
\begin{figure}[t]
\caption{The subnetwork $D_\mathbf{i}$ in the case where $G=Sp_4(\mathbb{C})$, $u=v=s_1s_2$ and $\mathbf{i}=(-1,-2,1,2)$, along with its labeled faces}
\label{Figure3.1}
\begin{center}
\begin{tikzpicture}[
       thick,
       acteur/.style={
         circle,
         thick,
         inner sep=1pt,
         minimum size=0.1cm
       }
] 
\node (a1) at (0,0) [acteur,fill=blue]{};
\node (a2) at (0,0.8) [acteur,fill=blue]{}; 
\node (a3) at (0,1.6) [acteur,fill=blue]{}; 
\node (a4) at (0.8,0.8) [acteur,fill=orange]{};
\node (a5) at (1.6,0) [acteur,fill=black]{}; 
\node (a6) at (2.4,1.6) [acteur,fill=orange]{};
\node (a7) at (3.2,0.8) [acteur,fill=black]{}; 
\node (a8) at (4.0,0) [acteur,fill=orange]{}; 
\node (a9) at (4.8,0.8) [acteur,fill=black]{}; 
\node (a10) at (5.6,0.8) [acteur,fill=orange]{}; 
\node (a11) at (6.4,1.6) [acteur,fill=black]{}; 
\node (a12) at (7.2,0) [acteur,fill=red]{};
\node (a13) at (7.2,0.8) [acteur,fill=red]{}; 
\node (a14) at (7.2,1.6) [acteur,fill=red]{}; 

\node (b1) at (-0.3,0) {1};
\node (b2) at (-0.3,0.8) {2}; 
\node (b3) at (-0.3,1.6) {3}; 
\node (b5) at (7.5,0) {1};
\node (b6) at (7.5,0.8) {2}; 
\node (b7) at (7.5,1.6) {3}; 

\node (c-1) at (0.4,0.4) {$-1$};
\node (c-2) at (1.2,1.2) {$-2$}; 
\node (c1) at (2.8,0.4) {$1$}; 
\node (c2) at (4.4,1.2) {$2$}; 
\node (c3) at (6.4,0.4) {$3$};
\node (c4) at (6.8,1.2) {$4$}; 

\draw[->] (a1) to node {} (a5);
\draw[->] (a5) to node {} (a8);
\draw[->] (a8) to node {} (a12);
\draw[->] (a2) to node {} (a4);
\draw[->] (a4) to node {} (a7);
\draw[->] (a7) to node {} (a9);
\draw[->] (a9) to node {} (a10);
\draw[->] (a10) to node {} (a13);
\draw[->] (a3) to node {} (a6);
\draw[->] (a6) to node {} (a11);
\draw[->] (a11) to node {} (a14);
\draw[->] (a4) to node {} (a5);
\draw[->] (a6) to node {} (a7);
\draw[->] (a8) to node {} (a9);
\draw[->] (a10) to node {} (a11);
\end{tikzpicture} 
\end{center}
\end{figure}

We say that an index $i\in I$ is unfrozen if the face corresponding to the vertex $i$ is bounded, and frozen otherwise. This then yields a partition of $I$ into a disjoint union $I=I_f\sqcup I_u$ of frozen and unfrozen indices. In the example $G=Sp_4(\mathbb{C})$, $u=v=s_1s_2$ and $\mathbf{i}=(-1,-2,1,2)$, we see from Figure \ref{Figure3.1} that the indices $-1,-2,3,4$ are frozen, and the indices $1,2$ are unfrozen.

Next, we will define the directed edges of $\Gamma_{\mathbf{i}}$, following the procedure described by Schrader and Shapiro in \cite[Section 4.1]{SS17} used to recover the quiver associated to a cluster seed from planar directed networks. To this end, we pick any two indices $j,k\in I$. We say that vertex $j$ is connected to vertex $k$ in $\Gamma_{\mathbf{i}}$ if the following conditions are satisfied: 
\begin{enumerate}
\item The faces corresponding to the vertices $j$ and $k$ share (exactly) a common directed edge $e$, 
\item The two faces lie on opposite sides of $e$, and
\item One of the two vertices of $e$ is either orange or black. More precisely, $e$ has to be from one of the following six edges:
\end{enumerate}

\begin{center}
\begin{tikzpicture}[
       thick,
       acteur/.style={
         circle,
         thick,
         inner sep=1pt,
         minimum size=0.1cm
       }
] 

\node (a1) at (0,0) [acteur,fill=orange]{};
\node (a2) at (1.6,0) [acteur,fill=black]{}; 
\node (a3) at (2.4,0) [acteur,fill=black]{};
\node (a4) at (4.0,0) [acteur,fill=orange]{}; 
\node (a5) at (4.8,0) [acteur,fill=blue]{};
\node (a6) at (6.4,0) [acteur,fill=black]{}; 
\node (a7) at (7.2,0) [acteur,fill=blue]{};
\node (a8) at (8.8,0) [acteur,fill=orange]{}; 
\node (a9) at (9.6,0) [acteur,fill=black]{};
\node (a10) at (11.2,0) [acteur,fill=red]{}; 
\node (a11) at (12.0,0) [acteur,fill=orange]{};
\node (a12) at (13.6,0) [acteur,fill=red]{}; 

\draw[->] (a1) to node {} (a2);
\draw[->] (a3) to node {} (a4);
\draw[->] (a5) to node {} (a6);
\draw[->] (a7) to node {} (a8);
\draw[->] (a9) to node {} (a10);
\draw[->] (a11) to node {} (a12);
\end{tikzpicture} 
\end{center}

The vertices corresponding to two faces that share a common edge that is from the above edges are then connected as follows:

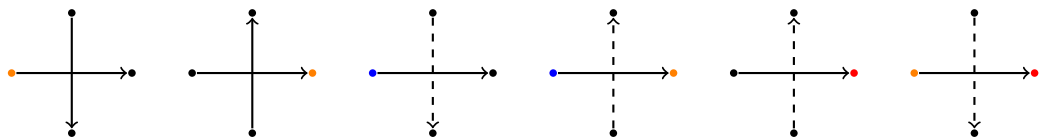
\begin{center}
\begin{tikzpicture}[
       thick,
       acteur/.style={
         circle,
         thick,
         inner sep=1pt,
         minimum size=0.1cm
       }
] 

\node (a1) at (0,0) [acteur,fill=orange]{};
\node (a2) at (1.6,0) [acteur,fill=black]{}; 
\node (a3) at (2.4,0) [acteur,fill=black]{};
\node (a4) at (4.0,0) [acteur,fill=orange]{}; 
\node (a5) at (4.8,0) [acteur,fill=blue]{};
\node (a6) at (6.4,0) [acteur,fill=black]{}; 
\node (a7) at (7.2,0) [acteur,fill=blue]{};
\node (a8) at (8.8,0) [acteur,fill=orange]{}; 
\node (a9) at (9.6,0) [acteur,fill=black]{};
\node (a10) at (11.2,0) [acteur,fill=red]{}; 
\node (a11) at (12.0,0) [acteur,fill=orange]{};
\node (a12) at (13.6,0) [acteur,fill=red]{}; 

\node (b1) at (0.8,-0.8) [acteur,fill=black]{};
\node (b2) at (0.8,0.8) [acteur,fill=black]{}; 
\node (b3) at (3.2,-0.8) [acteur,fill=black]{};
\node (b4) at (3.2,0.8) [acteur,fill=black]{}; 
\node (b5) at (5.6,-0.8) [acteur,fill=black]{};
\node (b6) at (5.6,0.8) [acteur,fill=black]{}; 
\node (b7) at (8.0,-0.8) [acteur,fill=black]{};
\node (b8) at (8.0,0.8) [acteur,fill=black]{}; 
\node (b9) at (10.4,-0.8) [acteur,fill=black]{};
\node (b10) at (10.4,0.8) [acteur,fill=black]{}; 
\node (b11) at (12.8,-0.8) [acteur,fill=black]{};
\node (b12) at (12.8,0.8) [acteur,fill=black]{}; 

\draw[->] (a1) to node {} (a2);
\draw[->] (a3) to node {} (a4);
\draw[->] (a5) to node {} (a6);
\draw[->] (a7) to node {} (a8);
\draw[->] (a9) to node {} (a10);
\draw[->] (a11) to node {} (a12);

\draw[->] (b2) to node {} (b1);
\draw[->] (b3) to node {} (b4);
\draw[dashed,->] (b6) to node {} (b5);
\draw[dashed,->] (b7) to node {} (b8);
\draw[dashed,->] (b9) to node {} (b10);
\draw[dashed,->] (b12) to node {} (b11);
\end{tikzpicture} 
\end{center}

Pictorially, if we denote the directed edge between the two vertices corresponding to two faces by $\overrightarrow{e}$, and we view $\overrightarrow{e}$ from top to bottom, then the black vertex incident to $e$ (if it exists) must lie on the right of $\overrightarrow{e}$, and the orange vertex incident to $e$ (if it exists) must lie on the left of $\overrightarrow{e}$. Here, we note that two vertices are connected to each other by a dashed edge only if the two vertices correspond to frozen indices. One can check using the definition of $\Gamma_{\mathbf{i}}$ that the exchange quiver $\tilde{\Gamma}(\mathbf{i})$ defined in \cite[Definition 2.2]{BFZ05} can be obtained from $\Gamma_{\mathbf{i}}$ by removing the dashed edges.

Finally, we will explain how the entries of the exchange matrix $B_{\mathbf{i}}$ are defined. The row and column indices of $B_{\mathbf{i}}$ are given by $I$. Next, we pick any two indices $j,k\in I$. Then $B_{j,k}$ is determined by the following rules:
\begin{enumerate}
\item $B_{j,k}\neq0$ if and only if vertex $j$ is connected to vertex $k$ in $\Gamma_{\mathbf{i}}$, and $B_{j,k}>0$ (respectively $B_{j,k}<0$) if this edge is directed from $j$ to $k$ (respectively $k$ to $j$).
\item If vertex $j$ is connected to vertex $k$ in $\Gamma_{\mathbf{i}}$, then 
\begin{equation*}
|B_{j,k}|=
\begin{cases}
1 & \text{if }|i_j|=|i_k|,\\
-C_{|i_j|,|i_k|} & \text{if }|i_j|\neq|i_k|\text{ and the edge is solid},\\ 
-\frac{C_{|i_j|,|i_k|}}{2} & \text{if }|i_j|\neq|i_k|\text{ and the edge is dashed}.
\end{cases}
\end{equation*}
\end{enumerate}

It is clear that $B_{\mathbf{i}}$ is skew-symmetrizable with skew-symmetrizers defined by $d_{-i}=d_i'$ for $i\in[1,r]$ and $d_j=d_{|i_j|}'$ for all $j\in[1,m]$. Also, the exchange matrix $\tilde{B}(\mathbf{i})$ defined in \cite[Definition 2.3]{BFZ05} can be obtained from $B_{\mathbf{i}}$ by removing the columns of $B_{\mathbf{i}}$ that correspond to frozen indices.

\begin{example}\label{3.1}
We let $G=Sp_4(\mathbb{C})$, $u=v=s_1s_2$ and $\mathbf{i}=(-1,-2,1,2)$. From the network diagram $D_{\mathbf{i}}$ as depicted in Figure \ref{Figure3.1}, we deduce that $\Gamma_{\mathbf{i}}$ is given by the following quiver:
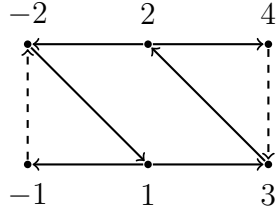
\begin{center}
\begin{tikzpicture}[
       thick,
       acteur/.style={
         circle,
         thick,
         inner sep=1pt,
         minimum size=0.1cm
       }
] 

\node (a-1) at (0,0) [acteur,fill=black] {};
\node (a-2) at (0,1.6) [acteur,fill=black] {};
\node (a1) at (1.6,0) [acteur,fill=black] {}; 
\node (a2) at (1.6,1.6) [acteur,fill=black] {};
\node (a3) at (3.2,0) [acteur,fill=black] {};
\node (a4) at (3.2,1.6) [acteur,fill=black] {};

\node (b-1) at (0,-0.4) {$-1$};
\node (b-2) at (0,2.0) {$-2$};
\node (b1) at (1.6,-0.4) {$1$}; 
\node (b2) at (1.6,2.0) {$2$};
\node (b3) at (3.2,-0.4) {$3$};
\node (b4) at (3.2,2.0) {$4$};

\draw[->] (a1) to node {} (a-1);
\draw[->] (a1) to node {} (a3);
\draw[->] (a2) to node {} (a-2);
\draw[->] (a2) to node {} (a4);
\draw[->] (a-2) to node {} (a1);
\draw[->] (a3) to node {} (a2);
\draw[dashed,->] (a-1) to node {} (a-2);
\draw[dashed,->] (a4) to node {} (a3);

\end{tikzpicture} 
\end{center}

Now, from $\Gamma_{\mathbf{i}}$ and the rules determining the entries of $B_{\mathbf{i}}$, we can check that the exchange matrix $B_{\mathbf{i}}$ is given by
\begin{equation*}
B_{\mathbf{i}}=
\begin{pmatrix}
0 & 1 & -1 & 0 & 0 & 0\\
-\frac{1}{2} & 0 & 1 & -1 & 0 & 0\\
1 & -2 & 0 & 0 & 1 & 0\\
0 & 1 & 0 & 0 & -1 & 1\\
0 & 0 & -1 & 2 & 0 & -1\\
0 & 0 & 0 & -1 & \frac{1}{2} & 0
\end{pmatrix},
\end{equation*}
where we have ordered the row and column indices as $-1,-2,1,2,3,4$.
\end{example}

Next, we will recall the connection between $\Sigma_{\mathbf{i}}$ and the double Bruhat cell $G^{u,v}$. To make this connection more transparent, we will need a few extra definitions. To this end, we let $G_{\Ad}^{u,v}$ denote the double Bruhat cell of the adjoint group $G_{\Ad}$, and $H_{\Ad}$ its Cartan subgroup. Here, $G_{\Ad}$ is the quotient of $G$ by its center $Z(G)$. Likewise, $G_{\Ad}^{u,v}$ and $H_{\Ad}$ are quotients of $G^{u,v}$ and $H$ respectively by $Z(G)$. 

The root lattice $P$ can be identified with the character lattice $\Hom(H_{\Ad},\mathbb{C}^\times)$ of $H_{\Ad}$, with basis given by the set $\{\alpha_1,\ldots,\alpha_r\}\subseteq\mathfrak{h}^*$ of simple roots of $\mathfrak{g}$. The character lattice $\Hom(H_{\Ad},\mathbb{C}^{\times})$ of $H_{\Ad}$ has a dual cocharacter lattice $\Hom(\mathbb{C}^{\times},H_{\Ad})$, and can be identified with the coweight lattice of $\mathfrak{g}$, with a basis given by the set $\{\omega_1^{\vee},\ldots,\omega_r^{\vee}\}\subseteq\mathfrak{h}^*$ of fundamental coweights of $\mathfrak{g}$, defined by $\langle\alpha_i,\omega_j^{\vee}\rangle=\delta_{i,j}$ for all $i,j=1,\ldots,r$. For each $\omega^{\vee}$ in the coweight lattice, $\omega^{\vee}$ defines a cocharacter $t\mapsto t^{\omega^{\vee}}$.

\begin{definition}\cite{FG06}\label{3.2}
Let $(u,v)\in W\times W$ and $\mathbf{i}=(i_1,\ldots,i_m)$ be a double reduced word for $(u,v)$. We let $x_{\mathbf{i}}:\mathcal{X}_{\Sigma_{\mathbf{i}}}\hookrightarrow G_{\Ad}^{u,v}$ be the open immersion defined by
\begin{equation}\label{eq:3.1}
x_{\mathbf{i}}(X_{-r},\ldots,X_{-1},X_1,\ldots,X_m)=X_{-r}^{\omega_r^{\vee}}\cdots X_{-1}^{\omega_1^{\vee}}E_{i_1}X_1^{\omega_{|i_1|}^{\vee}}\cdots E_{i_m}X_m^{\omega_{|i_m|}^{\vee}}, 
\end{equation}
The map $x_{\mathbf{i}}$ is Poisson with respect to the log-canonical Poisson structure on $\mathcal{X}_{\Sigma_{\mathbf{i}}}$ as defined by \eqref{eq:3.10}, and the Poisson structure on $G_{\Ad}^{u,v}$ induced from the standard Poisson-Lie structure on $G$.
\end{definition}

The following Proposition tells us that any pair of coordinates on $G_{\Ad}^{u,v}$ given by cluster $\mathcal{X}$-coordinates arising from two different double reduced words for $(u,v)$ are related by a sequence of cluster transformations (and permutations of coordinates, where necessary).

\begin{proposition}\cite{FG06}\label{3.3}
Let $(u,v)\in W\times W$ and $\mathbf{i}=(i_1,\ldots,i_m)$, $\mathbf{i}'=(i_1',\ldots,i_m')$ be double reduced words for $(u,v)$ that differ by swapping two adjacent indices differing only by a sign, that is, there exists some (unfrozen) index $k$ such that $i_k=-i_{k+1}$, and
\begin{equation*}
i_{\ell}'=
\begin{cases}
-i_{\ell} & \text{if } \ell=k,k+1,\\
i_{\ell} & \text{otherwise}.
\end{cases}
\end{equation*}
Then the coordinates on $G_{\Ad}^{u,v}$ given by cluster $\mathcal{X}$-coordinates differ by the cluster transformation at $k$:
\begin{center}
\begin{tikzpicture}
  \node (X) {$\mathcal{X}_{\Sigma_{\mathbf{i}}}$};
  \node (Y) [right=4.5cm of X] {$\mathcal{X}_{\Sigma_{\mathbf{i}'}}$};
  \node (W) [below right = 2cm and 1.8cm of X] {$G_{\Ad}^{u,v}$};
  \draw[->, dotted] (X) to node [above] {$\mu_k$} (Y);
  \draw[->] (X) to node [above] {$x_{\mathbf{i}}$} (W);
  \draw[->] (Y) to node [above] {$x_{\mathbf{i}'}$} (W);
\end{tikzpicture}
\end{center}
\end{proposition}

\subsection{Cluster structures on conjugation quotients of Coxeter double Bruhat cells}\label{Section3.2}

In this subsection, our goal is to describe the rational coordinates $\overline{a}_k=a_{k_-}a_{k_+}$ and $t_k$ on the quotient $G^{u,v}/H$ of a Coxeter double Bruhat cell $G^{u,v}$, with respect to the factorization scheme \eqref{eq:2.2}, in terms of cluster variables. Our first step is to show that we can obtain a set of rational coordinates on the quotient reduced double Bruhat cell $G_{\Ad}^{u,v}/H_{\Ad}$ given by cluster $\mathcal{X}$-coordinates from amalgamation:

\begin{proposition}\label{3.4}
Let $(u,v)$ be a pair of Coxeter elements, and $\mathbf{i}=(i_1,\ldots,i_{2r})$ be a double reduced word for $(u,v)$. Then there exist a seed $\widetilde{\Sigma}_{\mathbf{i}}$, such that $\widetilde{\Sigma}_{\mathbf{i}}$ is the unique amalgamation of $\Sigma_{\mathbf{i}}$ for which the open immersion $x_{\mathbf{i}}:\mathcal{X}_{\Sigma_{\mathbf{i}}}\hookrightarrow G_{\Ad}^{u,v}$ descends to an open immersion $\widetilde{x}_{\mathbf{i}}:\mathcal{X}_{\widetilde{\Sigma}_{\mathbf{i}}}\hookrightarrow G_{\Ad}^{u,v}/H_{\Ad}$ intertwining the quotient and the amalgamation maps:
\begin{center}
\begin{tikzpicture} [node distance=3cm]
  \node (X) {$\mathcal{X}_{\Sigma_{\mathbf{i}}}$};
  \node (Y) [right of=X] {$G_{\Ad}^{u,v}$};
  \node (W) [below of=X] {$\mathcal{X}_{\widetilde{\Sigma}_{\mathbf{i}}}$};
  \node (Z) [below of=Y] {$G_{\Ad}^{u,v}/H_{\Ad}$};
  \draw[->] (X) to node [above] {$x_{\mathbf{i}}$} (Y);
  \draw[->] (Y) to node [right] {$\pi$} (Z);
  \draw[->] (X) to node [left] {$\pi$} (W);
  \draw[->] (W) to node [below] {$\widetilde{x}_{\mathbf{i}}$} (Z);
\end{tikzpicture}
\end{center}
\end{proposition}

Here, we remark that a special case of Proposition \ref{3.4} was proved in \cite[Theorem 3.4]{Williams15} with $u=v=s_1\cdots s_r$ and $\mathbf{i}=(-1,\ldots,-r,1,\ldots,r)$. As in \cite{BFZ05}, we will construct $\widetilde{\Sigma}_{\mathbf{i}}$ in two steps: firstly, we will construct an exchange quiver $\widetilde{\Gamma}_{\mathbf{i}}$ that determines the signs of the entries of the exchange matrix $\widetilde{B}_{\mathbf{i}}$. We will then determine those absolute values of these entries of $\widetilde{B}_{\mathbf{i}}$ using a similar set of rules that we have defined in Section \ref{Section3.1}.

To begin, we will need some notations. We let $\widetilde{I}=\{\pm1,\ldots,\pm r\}$, and for each $j\in[1,r]$, we let $j_-<j_+$ be the unique elements in $[1,2r]$ that satisfy $|i_{j_-}|=j=|i_{j_+}|$. Then it is easy to see that $I_f=\{-i,i_+:i\in[1,r]\}$, and $I_u=\{i_-:i\in[1,r]\}$. Henceforth, let us define the surjection $\pi:I\twoheadrightarrow\widetilde{I}$ of indexing sets by \footnote{In \cite{Williams15}, where $i_-=i$ and $i_+=r+i$ for all $i\in[1,r]$, the surjection $\pi:I\twoheadrightarrow\widetilde{I}$ is defined by $\pi(-i)=\pi(i+r)=-i$ and $\pi(i)=i$ instead.}

\begin{equation*}
\pi(-i)=\pi(i_+)=i,\quad\pi(i_-)=-i
\end{equation*}
for all $i\in[1,r]$. 

We define the exchange quiver $\widetilde{\Gamma}_{\mathbf{i}}$ as follows: we consider the network $\widetilde{D}_\mathbf{i}$ obtained from $D_\mathbf{i}$ by gluing the source vertex of $D_\mathbf{i}$ with the sink vertex of $D_\mathbf{i}$ on level $i$ for all $i\in[1,r+1]$, and subsequently removing all glued vertices, so that all vertices of $\widetilde{D}_\mathbf{i}$ are colored either black or orange. Put another way, the subnetwork $\widetilde{D}_\mathbf{i}$ obtained by placing the subnetwork $D_\mathbf{i}$ on a cylinder instead of a plane. It is then easy to see that there are $2r$ faces in the network $\widetilde{D}_\mathbf{i}$ that are bounded between levels $1$ and $r+1$. The vertices of $\widetilde{\Gamma}_{\mathbf{i}}$ then correspond to each of these $2r$ faces in the network $\widetilde{D}_\mathbf{i}$. We label these $2r$ vertices of $\widetilde{\Gamma}_{\mathbf{i}}$ by the elements of the indexing set $I$ as follows: we label a vertex corresponding to one such face of $\widetilde{D}_\mathbf{i}$ by $-i$ if it corresponds to the bounded face of $D_\mathbf{i}$ whose vertex is labeled $i_-$, and $i$ if it corresponds to the unbounded faces of $D_\mathbf{i}$ whose vertices are labeled $-i$ and $i_+$, and which are subsequently identified under the gluing process. As an example, when $G=Sp_4(\mathbb{C})$, $u=v=s_1s_2$ and $\mathbf{i}=(-1,-2,1,2)$, the subnetwork $D_\mathbf{i}$, along with its labeled faces corresponding to the vertices of the exchange quiver $\Gamma_{\mathbf{i}}$ are given in Figure \ref{Figure3.2}. For convenience, we have omitted the factorization parameters from $D_\mathbf{i}$, and to avoid confusion, we have labeled the levels in red. 

\begin{figure}[t]
\caption{The network $\widetilde{D}_\mathbf{i}$ in the case where $G=Sp_4(\mathbb{C})$, $u=v=s_1s_2$ and $\mathbf{i}=(-1,-2,1,2)$, along with its labeled faces}
\label{Figure3.2}
\begin{center}
\begin{tikzpicture}[
       thick,
       acteur/.style={
         circle,
         thick,
         inner sep=1pt,
         minimum size=0.1cm
       }
] 

\draw (0,0) ellipse (8cm and 4cm);
\draw (0,0) ellipse (6cm and 3cm);
\draw (0,0) ellipse (4cm and 2cm);

\node (a1) at ($(0,0)+(-125:6 and 3)$) [acteur,fill=orange]{};
\node (a2) at ($(0,0)+(-110:8 and 4)$) [acteur,fill=black]{};
\node (a3) at ($(0,0)+(-115:4 and 2)$) [acteur,fill=orange]{};
\node (a4) at ($(0,0)+(-100:6 and 3)$) [acteur,fill=black]{};
\node (a5) at ($(0,0)+(-90:8 and 4)$) [acteur,fill=orange]{};
\node (a6) at ($(0,0)+(-80:6 and 3)$) [acteur,fill=black]{};
\node (a7) at ($(0,0)+(-70:6 and 3)$) [acteur,fill=orange]{};
\node (a8) at ($(0,0)+(-45:4 and 2)$) [acteur,fill=black]{};

\node (b1) at ($(0,0)+(45:8 and 4)$) [above right] {\textcolor{red}{$1$}};
\node (b2) at ($(0,0)+(45:6 and 3)$) [above right] {\textcolor{red}{$2$}};
\node (b3) at ($(0,0)+(45:4 and 2)$) [above right] {\textcolor{red}{$3$}};

\node (c1) at ($(0,0)+(-100:7 and 3.5)$) {$-1$};
\node (c2) at ($(0,0)+(-85:5 and 2.5)$) {$-2$};
\node (c1) at ($(0,0)+(90:7 and 3.5)$) {$1$};
\node (c2) at ($(0,0)+(90:5 and 2.5)$) {$2$};

\draw[->] (a1) to node {} (a2);
\draw[->] (a3) to node {} (a4);
\draw[->] (a5) to node {} (a6);
\draw[->] (a7) to node {} (a8);
\end{tikzpicture}
\end{center}
\end{figure}

Next, we will define the directed edges of $\Gamma_{\mathbf{i}}$. To this end, we pick any two indices $j,k\in\widetilde{I}$. We say that vertex $j$ is connected to vertex $k$ in $\widetilde{\Gamma}_{\mathbf{i}}$ if the faces corresponding to the vertices $j$ and $k$ share at least an edge $e$ whose two vertices incident to $e$ are of different colors, such that the two faces lie on opposite sides of $e$. For each common edge $e$ between these two faces whose two vertices incident to $e$ are of different colors, we draw an directed edge $\overrightarrow{e}$ between vertices $j$ and $k$, with the rule that the black vertex incident to $e$ must lie on the right of $\overrightarrow{e}$, and the orange vertex incident to $e$ must lie on the left of $\overrightarrow{e}$, similar to our rules for these edges in Section \ref{Section3.1}. 

Finally, we will explain how the entries of the exchange matrix $\widetilde{B}_{\mathbf{i}}$ are defined. The row and column indices of $\widetilde{B}_{\mathbf{i}}$ are given by $\widetilde{I}$. Next, we pick any two indices $j,k\in\widetilde{I}$. Then $\widetilde{B}_{j,k}$ is determined by the following rules:
\begin{enumerate}
\item $\widetilde{B}_{j,k}\neq0$ if and only if vertex $j$ is connected to vertex $k$ in $\widetilde{\Gamma}_{\mathbf{i}}$, and $\widetilde{B}_{j,k}>0$ (respectively $\widetilde{B}_{j,k}<0$) if this edge is directed from $j$ to $k$ (respectively $k$ to $j$).
\item If vertex $j$ is connected to vertex $k$ in $\widetilde{\Gamma}_{\mathbf{i}}$, then 
\begin{equation*}
|\widetilde{B}_{j,k}|=
\begin{cases}
2 & \text{if }|j|=|k|,\\
-C_{|i_j|,|i_k|}\times\#\{\text{arrows between vertices }j\text{ and }k\} & \text{otherwise}.
\end{cases}
\end{equation*}
\end{enumerate}

With the above rules, it is then clear that $\widetilde{B}_{\mathbf{i}}$ is skew-symmetrizable with skew-symmetrizers defined by $\widetilde{d}_{-i}=d_i'=\widetilde{d}_{-i}$ for $i\in[1,r]$. Thus, it follows from the definitions of $\Gamma_{\mathbf{i}}$, $\widetilde{\Gamma}_{\mathbf{i}}$, $B_{\mathbf{i}}$ and $\widetilde{B}_{\mathbf{i}}$ that $\widetilde{\Sigma}_{\mathbf{i}}$ is the amalgamation of $\Sigma_{\mathbf{i}}$ along $\pi$. Moreover, there is only one open immersion $\widetilde{x}_{\mathbf{i}}:\mathcal{X}_{\widetilde{\Sigma}_{\mathbf{i}}}\hookrightarrow G_{\Ad}^{u,v}/H_{\Ad}$ that intertwines the quotient and the amalgamation maps, and that is given by
\begin{equation}\label{eq:3.2}
\widetilde{x}_{\mathbf{i}}(\widetilde{X}_{-1},\ldots,\widetilde{X}_{-r},\widetilde{X}_1,\ldots,\widetilde{X}_r)
=
E_{i_1}Y_1^{\omega_{|i_1|}^{\vee}}\cdots E_{i_{2r}}Y_{2r}^{\omega_{|i_{2r}|}^{\vee}},
\end{equation} 
where
\begin{equation}\label{eq:3.3}
\widetilde{x}_{\mathbf{i}}^*(Y_j)=
\begin{cases}
\widetilde{X}_{-|i_j|} & \text{if }j=|i_j|_-,\\
\widetilde{X}_{|i_j|} & \text{if }j=|i_j|_+.
\end{cases}
\end{equation}
This proves Proposition \ref{3.4}.

\begin{example}\label{3.5}
We let $G=Sp_4(\mathbb{C})$, $u=v=s_1s_2$ and $\mathbf{i}=(-1,-2,1,2)$. From the network diagram $\widetilde{D}_{\mathbf{i}}$ as depicted in Figure \ref{Figure3.2}, we deduce that $\widetilde{\Gamma}_{\mathbf{i}}$ is given by the following quiver:
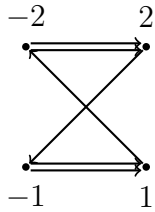
\begin{center}
\begin{tikzpicture}[
       thick,
       acteur/.style={
         circle,
         thick,
         inner sep=1pt,
         minimum size=0.1cm
       }
] 

\node (a-1) at (0,0) [acteur,fill=black] {};
\node (a-2) at (0,1.6) [acteur,fill=black] {};
\node (a1) at (1.6,0) [acteur,fill=black] {}; 
\node (a2) at (1.6,1.6) [acteur,fill=black] {};

\node (b-1) at (0,-0.4) {$-1$};
\node (b-2) at (0,2.0) {$-2$};
\node (b1) at (1.6,-0.4) {$1$}; 
\node (b2) at (1.6,2.0) {$2$};

\draw[transform canvas={yshift=0.25ex},->] (a-1) to node {} (a1);
\draw[transform canvas={yshift=0.25ex},->] (a-2) to node {} (a2);
\draw[transform canvas={yshift=-0.25ex},->] (a-1) to node {} (a1);
\draw[transform canvas={yshift=-0.25ex},->] (a-2) to node {} (a2);
\draw[->] (a2) to node {} (a-1);
\draw[->] (a1) to node {} (a-2);

\end{tikzpicture} 
\end{center}

Now, from $\widetilde{\Gamma}_{\mathbf{i}}$ and the rules determining the entries of $B_{\mathbf{i}}$, we can check that the exchange matrix $\widetilde{B}_{\mathbf{i}}$ is given by
\begin{equation*}
\widetilde{B}_{\mathbf{i}}=
\begin{pmatrix}
0 & 0 & 2 & -2\\
0 & 0 & -1 & 2\\
-2 & 2 & 0 & 0\\
1 & -2 & 0 & 0\\
\end{pmatrix},
\end{equation*}
where we have ordered the row and column indices as $-1,-2,1,2$.
\end{example}

Our next step is to lift these coordinates on $G_{\Ad}^{u,v}/H_{\Ad}$ to coordinates on $G^{u,v}/H$ that are given by cluster variables. Before we do so, let us give explicit formulas for the entries of the exchange matrix $\widetilde{B}_{\mathbf{i}}$. To this end, we define $\varepsilon_j=1$ if $i_{j_-}=1$, and $\varepsilon_j=-1$ otherwise. Then it follows from the definition of $\widetilde{\Gamma}_{\mathbf{i}}$, as well as the rules determining the entries of $\widetilde{B}_{\mathbf{i}}$ that we have
\begin{align}
\widetilde{B}_{-j,-k}
&=\frac{C_{j,k}}{2}[\varepsilon_k(\delta_{j_-<k_-<j_+}+\delta_{j_-<k_+<j_+})-\varepsilon_j(\delta_{k_-<j_-<k_+}+\delta_{k_-<j_+<k_+})],\label{eq:3.4}\\
\widetilde{B}_{-j,k}
&=-\frac{C_{j,k}}{2}[\varepsilon_j(2\delta_{j,k}+\delta_{j_-<k_-}+\delta_{j_+<k_-}+\delta_{k_+<j_-}+\delta_{k_+<j_+})\nonumber\\
&\qquad\qquad+\varepsilon_k(\delta_{j_-<k_-<j_+}+\delta_{j_-<k_+<j_+})],\label{eq:3.5}\\
\widetilde{B}_{j,-k}
&=\frac{C_{j,k}}{2}[\varepsilon_k(2\delta_{j,k}+\delta_{k_-<j_-}+\delta_{k_+<j_-}+\delta_{j_+<k_-}+\delta_{j_+<k_+})\nonumber\\
&\qquad\quad+\varepsilon_j(\delta_{k_-<j_-<k_+}+\delta_{k_-<j_+<k_+})],\label{eq:3.6}\quad\text{and}\\
\widetilde{B}_{j,k}
&=\frac{C_{j,k}}{2}[\varepsilon_j(\delta_{j_-<k_-}+\delta_{j_+<k_-}+\delta_{k_+<j_-}+\delta_{k_+<j_+})\nonumber\\
&\qquad\quad-\varepsilon_k(\delta_{k_-<j_-}+\delta_{k_+<j_-}+\delta_{j_+<k_-}+\delta_{j_+<k_+})].\label{eq:3.7}
\end{align}
for all $j,k\in[1,r]$.

Next, let us record some properties of $\widetilde{B}_{\mathbf{i}}$. 

\begin{proposition}\label{3.6}
Let $(u,v)\in W\times W$ be a pair of Coxeter elements, $j,\ell\in[1,r]$, and $\mathbf{i}=(i_1,\ldots,i_{2r})$, $\mathbf{i}'=(i_1',\ldots,i_{2r}')$ be double reduced words for $(u,v)$ that differ by swapping two adjacent indices $k$ and $k+1$, that is, we have
\begin{equation*}
i_{\ell}'=
\begin{cases}
i_{k+1} & \text{if } \ell=k,\\
i_{k} & \text{if } \ell=k+1,\\
i_{\ell} & \text{otherwise}.
\end{cases}
\end{equation*}
Then:
\begin{enumerate}
\item If $i_k=\pm j$ and $i_{k+1}=\mp j$, then $\widetilde{\Sigma}_{\mathbf{i}'}=\mu_{-j}(\widetilde{\Sigma}_{\mathbf{i}})$.
\item If $i_k=\pm j$ and $i_{k+1}=\pm\ell$ with $C_{j,\ell}=0$, then $\widetilde{\Sigma}_{\mathbf{i}'}=\widetilde{\Sigma}_{\mathbf{i}}$.
\item If $i_k=\pm j$ and $i_{k+1}=\mp\ell$ with $j\neq\ell$, then $\widetilde{\Sigma}_{\mathbf{i}'}=\widetilde{\Sigma}_{\mathbf{i}}$.
\end{enumerate}
\end{proposition}

\begin{proof}
Statement (1) follows from Propositions \ref{3.3} and \ref{3.4}, as well as the definitions of $\widetilde{\Sigma}_{\mathbf{i}'}$ and $\widetilde{\Sigma}_{\mathbf{i}}$, and we may check case-by-case using \eqref{eq:3.4}-\eqref{eq:3.7} that statements (2) and (3) hold. We will omit the details here.
\end{proof}

In light of the above proposition, we may define $\widetilde{\Sigma}_{(u,v)}=\widetilde{\Sigma}_{\mathbf{i}}$, where $\mathbf{i}$ is any unmixed double reduced word for $(u,v)$. In this case, we have $j_-\in[1,r]$, $j_+\in[r+1,2r]$ and $i_{j_+}=-i_{j_-}=j$ for all $j\in[1,r]$, and it follows from equations \eqref{eq:3.4}-\eqref{eq:3.7} that the entries of $\widetilde{B}_{(u,v)}$ are given by
\begin{align}
\widetilde{B}_{-j,-k}
&=\widetilde{B}_{j,k}
=\frac{C_{j,k}}{2}(\delta_{k_-<j_-}-\delta_{j_-<k_-}+\delta_{j_+<k_+}-\delta_{k_+<j_+}),\label{eq:3.8}\\
\widetilde{B}_{-j,k}
&=C_{j,k}(\delta_{j,k}+\delta_{j_-<k_-}+\delta_{k_+<j_+}),\label{eq:3.9}\quad\text{and}\\
\widetilde{B}_{j,-k}
&=-C_{j,k}(\delta_{j,k}+\delta_{k_-<j_-}+\delta_{j_+<k_+}),\label{eq:3.10}
\end{align}
for all $j,k\in[1,r]$.

\begin{remark}
By ordering the row and column indices as $-1,1,-2,2,\ldots,-r,r$, it is easy to check that in the case when $G=SL_{r+1}(\mathbb{C})$, the exchange matrix $\widetilde{B}_{(u,v)}$ is precisely the $2r\times 2r$ submatrix of the $(2r+2)\times(2r+2)$ exchange matrix defined in \cite[Equation 5.16]{GSV11} by removing the last two columns and rows. 
\end{remark}

Our next proposition shows that cyclic shifts in double reduced words amounts to permutations of the indexing set $\widetilde{I}$:

\begin{proposition}\label{3.7}
Let $(u,v),(u',v')\in W\times W$ be pairs of Coxeter elements. Let $\mathbf{i}=(i_1,\ldots,i_{2r})$, $\mathbf{i}'=(i_1',\ldots,i_{2r}')$ be double reduced words for $(u,v)$ and $(u',v')$, such that $i_{2r}'=i_1$ and $i_k'=i_{k+1}$ for all $k\in[1,2r-1]$, and $j=|i_1|$. Then $\widetilde{\Sigma}_{\mathbf{i}'}=\sigma_k(\widetilde{\Sigma}_{\mathbf{i}})$, where $\sigma_j$ is the permutation of $\widetilde{I}$ defined by 
\begin{equation}\label{eq:3.11}
\sigma_j(i)=
\begin{cases}
i & \text{if } |i|\neq j,\\
-i & \text{if } |i|=j.
\end{cases}
\end{equation}
\end{proposition}

\begin{proof}
We need to show that 
\begin{equation}\label{eq:3.12}
\widetilde{B}_{\sigma_j(i),\sigma_j(k)}'=\widetilde{B}_{i,k}
\end{equation}
holds for all $i,k\in\widetilde{I}$. Firstly, it follows from the definition of $\mathbf{i}'$ that we have $\varepsilon_j'=-\varepsilon_j$, $j_-=1$, $j_+'=2r$, $j_-'=j_+-1$, $\varepsilon_k'=\varepsilon_k$, $k_-'=k_--1$ and $k_+'=k_+-1$ for all $k\in[1,r]\setminus\{j\}$. Thus, by \eqref{eq:3.4}-\eqref{eq:3.7}, it is easy to see that \eqref{eq:3.12} holds if either $|i|=|k|=|j|$ or $|i|\neq|j|$ and $|k|\neq|j|$. By the skew-symmetry of $\widetilde{B}_{\mathbf{i}}$ and $\widetilde{B}_{\mathbf{i}'}$, it suffices to show that \eqref{eq:3.12} holds when $|i|=|j|$ and $|k|\neq|j|$. It follows from \eqref{eq:3.4}-\eqref{eq:3.7} that we have
{\allowdisplaybreaks
\begin{align*}
\widetilde{B}_{-j,-k}'
&=\frac{C_{j,k}}{2}[\varepsilon_k'(\delta_{j_-'<k_-'}+\delta_{j_-'<k_+'})-\varepsilon_j'\delta_{k_-'<j_-'<k_+'}]\\
&=\frac{C_{j,k}}{2}[\varepsilon_k(\delta_{j_+<k_-}+\delta_{j_+<k_+})+\varepsilon_j\delta_{k_-<j_+<k_+}]\\
&=\widetilde{B}_{j,-k},\\
\widetilde{B}_{-j,k}'
&=-\frac{C_{j,k}}{2}[\varepsilon_j'(\delta_{j_-'<k_-'}+\delta_{k_+'<j_-'}+1)+\varepsilon_k'(\delta_{j_-'<k_-'}+\delta_{j_-'<k_+'})]\\
&=-\frac{C_{j,k}}{2}[-\varepsilon_j(\delta_{j_+<k_-}+\delta_{k_+<j_+}+1)+\varepsilon_k(\delta_{j_+<k_-}+\delta_{j_+<k_+})]\\
&=\widetilde{B}_{j,k},\\
\widetilde{B}_{j,-k}'
&=\frac{C_{j,k}}{2}[\varepsilon_k'(\delta_{k_-'<j_-'}+\delta_{k_+'<j_-'})+\varepsilon_j'\delta_{k_-'<j_-'<k_+'}]\\
&=\frac{C_{j,k}}{2}[\varepsilon_k(\delta_{k_-<j_+}+\delta_{k_+<j_+})-\varepsilon_j\delta_{k_-<j_+<k_+}]\\
&=\widetilde{B}_{-j,-k},\quad\text{and}\\
\widetilde{B}_{j,k}'
&=\frac{C_{j,k}}{2}[\varepsilon_j'(\delta_{j_-'<k_-'}+\delta_{k_+'<j_-'}+1)-\varepsilon_k'(\delta_{k_-'<j_-'}+\delta_{k_+'<j_-'})]\\
&=\frac{C_{j,k}}{2}[-\varepsilon_j(\delta_{j_+<k_-}+\delta_{k_+<j_+}+1)-\varepsilon_k(\delta_{k_-<j_+}+\delta_{k_+<j_+})]\\
&=\widetilde{B}_{-j,k},
\end{align*}
which shows that \eqref{eq:3.12} holds when $|i|=|j|$ and $|k|\neq|j|$. This completes the proof of Proposition \ref{3.7}.
}
\end{proof}

Our next goal is to relate two seeds $\widetilde{\Sigma}_{(u,v)}$ and $\widetilde{\Sigma}_{(u',v')}$ to each other in the case where $(u,v)$ and $(u',v')$ are two pairs of Coxeter elements in $W$. To this end, we will need some definitions. We let $D(G)$ be the Coxeter diagram of $G$, and we let $v_1,\ldots,v_r$ be the vertices of $D(G)$. We note that there is a bijection between the set $C(W)$ of Coxeter elements of $W$ and the set $\mathcal{O}(D(G))$ of directed graphs with underlying undirected graph $D(G)$. Indeed, given a Coxeter element $u$ of $W$ and an edge $v_iv_j$ of the graph $D(G)$, we give the edge $v_iv_j$ the orientation $v_i\to v_j$ if $s_i$ precedes $s_j$ in any reduced word of $u$, and the orientation $v_j\to v_i$ otherwise. We will denote the resulting directed graph $\mathcal{O}(u)$.

\begin{example}\label{3.8}
Let $G=SL_6(\mathbb{C})$ and $u=s_1s_4s_3s_2s_5$. Then the directed graph $\mathcal{O}(u)$ is given by 
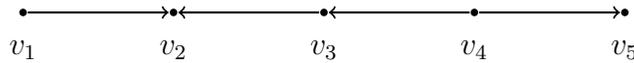
\begin{center}
\begin{tikzpicture}[
       thick,
       acteur/.style={
         circle,
         fill=black,
         thick,
         inner sep=1pt,
         minimum size=0.1cm
       }
] 

\node (a1) at (0,0) [acteur]{};
\node (a2) at (2.0,0) [acteur]{}; 
\node (a3) at (4.0,0) [acteur]{}; 
\node (a4) at (6.0,0) [acteur]{}; 
\node (a5) at (8.0,0) [acteur]{};

\node (b1) at (0,-0.5) {$v_1$};
\node (b2) at (2.0,-0.5) {$v_2$}; 
\node (b3) at (4.0,-0.5) {$v_3$}; 
\node (b4) at (6.0,-0.5) {$v_4$}; 
\node (b5) at (8.0,-0.5) {$v_5$};

\draw[->] (a1) to node {} (a2);
\draw[->] (a3) to node {} (a2);
\draw[->] (a4) to node {} (a3);
\draw[->] (a4) to node {} (a5);

\end{tikzpicture}
\end{center}
\end{example}

Conversely, given a directed graph $\mathcal{O}$ with underlying undirected graph $D(G)$, we can obtain a Coxeter element $u_{\mathcal{O}}$ of $W$ as follows: firstly, we let $\mathcal{O}_1=\mathcal{O}$, and $v_{i_{1,1}},v_{i_{1,2}},\ldots,v_{i_{1,n_1}}$ be the sources of the directed graph $\mathcal{O}_1$. Next, we remove the sources $v_{i_{1,1}},v_{i_{1,2}}$, $\ldots,v_{i_{1,n_1}}$, along with the edges that these vertices are incident to from $\mathcal{O}_1$, and call the new (directed) graph $\mathcal{O}_2$. We would then repeat the same procedure until only the direct graph contains only isolated vertices, say $v_{i_{m,1}},v_{i_{m,2}},\ldots,v_{i_{m,n_m}}$, and the resulting Coxeter element $u_{\mathcal{O}}$ obtained would be the element $s_{i_{1,1}}s_{i_{1,2}}\ldots s_{i_{1,n_1}}s_{i_{2,1}}s_{i_{2,2}}\ldots s_{i_{2,n_2}}\ldots s_{i_{m,n_m}}$. 

\begin{example}\label{3.9}
Let $G=SO_{10}(\mathbb{C})$ and $\mathcal{O}$ be the following orientation on $D(G)$:
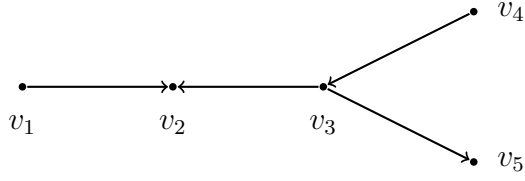
\begin{center}
\begin{tikzpicture}[
       thick,
       acteur/.style={
         circle,
         fill=black,
         thick,
         inner sep=1pt,
         minimum size=0.1cm
       }
] 

\node (a1) at (0,0) [acteur]{};
\node (a2) at (2.0,0) [acteur]{}; 
\node (a3) at (4.0,0) [acteur]{}; 
\node (a4) at (6.0,1.0) [acteur]{}; 
\node (a5) at (6.0,-1.0) [acteur]{};

\node (b1) at (0,-0.5) {$v_1$};
\node (b2) at (2.0,-0.5) {$v_2$}; 
\node (b3) at (4.0,-0.5) {$v_3$}; 
\node (b4) at (6.5,1.0) {$v_4$}; 
\node (b5) at (6.5,-1.0) {$v_5$};

\draw[->] (a1) to node {} (a2);
\draw[->] (a3) to node {} (a2);
\draw[->] (a4) to node {} (a3);
\draw[->] (a3) to node {} (a5);

\end{tikzpicture}
\end{center}
The sources of the directed graph $\mathcal{O}_1=\mathcal{O}$ are $v_1$ and $v_4$, so we will remove these vertices, along with their incident edges, from $\mathcal{O}_1$, in order to form the directed graph $\mathcal{O}_2$. Next, $v_3$ is the only source of $\mathcal{O}_2$, and $\mathcal{O}_3$ consists solely of the isolated vertices $v_2$ and $v_5$. Thus, we have $u_{\mathcal{O}}=s_1s_4s_3s_2s_5$.
\end{example}

It is easy to see that $u_{\mathcal{O}(u)}=u$ for any $u\in C(W)$ and $\mathcal{O}(u_{\mathcal{O}})=\mathcal{O}$ for any $\mathcal{O}\in\mathcal{O}(D(G))$, so this establishes a bijection between $C(W)$ and $\mathcal{O}(D(G))$.

\begin{lemma}\label{3.10}
Let $(u,v),(u',v')$ be pairs of Coxeter elements. Suppose that there exists some $k\in[1,r]$, such that either of these hold:
\begin{enumerate}
\item $v_k$ is a source in $\mathcal{O}(u)$ and a sink in $\mathcal{O}(u')$, $\mathcal{O}(u)\setminus\{v_k\}=\mathcal{O}(u')\setminus\{v_k\}$, and $\mathcal{O}(v)=\mathcal{O}(v')$, or
\item $v_k$ is a sink in $\mathcal{O}(v)$ and a source in $\mathcal{O}(v')$, $\mathcal{O}(v)\setminus\{v_k\}=\mathcal{O}(v')\setminus\{v_k\}$, and $\mathcal{O}(u)=\mathcal{O}(u')$.
\end{enumerate}
Then $\widetilde{\Sigma}_{(u',v')}=\tau_k(\widetilde{\Sigma}_{(u,v)})$, where $\tau_k=\mu_{-k}\circ\sigma_k$, and $\sigma_k$ is the permutation of $\widetilde{I}$ defined by \eqref{eq:3.11}.
\end{lemma}

\begin{proof}
By symmetry, it suffices to consider the first case. As $v_k$ is a source in $\mathcal{O}(u)$ and a sink in $\mathcal{O}(u')$, $\mathcal{O}(u)\setminus\{v_k\}=\mathcal{O}(u')\setminus\{v_k\}$, and $\mathcal{O}(v)=\mathcal{O}(v')$, there exists an unmixed double reduced word $\mathbf{i}=(i_1,\ldots,i_{2r})$ for $(u,v)$, such that $\overline{\mathbf{i}}'=(i_2,i_3,\ldots,i_{2r},i_1)$ is a double reduced word for $(u',v')$. Consequently, by Proposition \ref{3.7}, we have $\widetilde{\Sigma}_{\overline{\mathbf{i}}'}=\sigma_k(\widetilde{\Sigma}_{\mathbf{i}})=\sigma_k(\widetilde{\Sigma}_{(u,v)})$. 

Also, as we have $i_1,\ldots,i_r<0$ and $i_{r+1},\ldots,i_{2r}>0$, it follows that $\mathbf{i}'=(i_2,i_3,\ldots,i_r,i_1,i_{r+1},$ $\ldots,i_{2r})$ is an unmixed double reduced word for $(u',v')$, and by Proposition \ref{3.6}, we have $\widetilde{\Sigma}_{(u',v')}=\widetilde{\Sigma}_{\mathbf{i}'}=\mu_{-k}(\widetilde{\Sigma}_{\overline{\mathbf{i}}'})$. Consequently, we have $\widetilde{\Sigma}_{(u',v')}=\tau_k(\widetilde{\Sigma}_{(u,v)})$ as claimed.
\end{proof}

\begin{corollary}\label{3.11}
Let $(u,v)$ and $(u',v')$ be two pairs of Coxeter elements in $W$. Then:
\begin{enumerate}
\item $\widetilde{\Sigma}_{(u,v)}$ and $\widetilde{\Sigma}_{(u',v')}$ are mutation equivalent to each other, up to a permutation of indexing sets, and
\item $\widetilde{B}_{(u,v)}$ has full rank.
\end{enumerate}
\end{corollary}

\begin{proof}
~
\begin{enumerate}
\item Using Lemma \ref{3.10}, we can prove by induction that if $\mathcal{O}(u)$ and $\mathcal{O}(u')$ differ by orientations of the edges incident to a single vertex in the Coxeter diagram, then $\widetilde{\Sigma}_{(u,v)}$ and $\widetilde{\Sigma}_{(u',v)}$ are mutation equivalent to each other, up to a permutation of indexing sets. As we can get from $\mathcal{O}(u)$ to $\mathcal{O}(u')$ by turning sources into sinks, and vice versa, and similarly from $\mathcal{O}(v)$ to $\mathcal{O}(v')$ in an analogous fashion, statement (1) follows.
\item As cluster mutations preserve the rank of exchange matrices \cite[Lemma 1.2]{GSV03}, it suffices to show, using statement (1), that $\widetilde{B}_{(u,u)}$ has full rank. In this case, we have
\begin{equation*}
\widetilde{B}_{(u,u)}
=\begin{pmatrix}
0 & C\\
-C & 0
\end{pmatrix},
\end{equation*}
where the row and column indices are ordered as $-1,\ldots,-r,1,\ldots,r$. As the Cartan matrix $C$ has full rank, so does $\widetilde{B}_{(u,u)}$, and this completes the proof of statement (2).
\end{enumerate}
\end{proof}

Following Corollary \ref{3.11}, it follows that the torus $\mathcal{A}_{\widetilde{\Sigma}_{\mathbf{i}}}$ inherits a log-canonical Poisson structure from $\mathcal{X}_{\widetilde{\Sigma}_{\mathbf{i}}}$, given by
\begin{equation}\label{eq:3.13}
\{\widetilde{A}_i,\widetilde{A}_j\}=\widetilde{\Lambda}_{i,j}\widetilde{A}_i\widetilde{A}_j
\end{equation}
for all $i,j\in\widetilde{I}$. Here, $\widetilde{\Lambda}_{i,j}$ are the entries of the $2r\times 2r$ skew-symmetric matrix $\widetilde{\Lambda}_{\mathbf{i}}$ defined by the equation $\widetilde{\Lambda}_{\mathbf{i}}\widetilde{B}_{\mathbf{i}}=-\widetilde{D}$, and $\widetilde{D}$ is the $2r\times 2r$ diagonal matrix with row and column indices $-1,\ldots,-r,1,\ldots,r$ whose $(i,i)$-th entry is equal to $d_{|i|}'$.

We are now ready to lift the cluster $\mathcal{X}$-coordinates on $G_{\Ad}^{u,v}/H_{\Ad}$ to coordinates on $G^{u,v}/H$ that are given by cluster variables. 

\begin{proposition}\label{3.12}
Let $(u,v)$ be a pair of Coxeter elements of $W$, and $\mathbf{i}$ be a double reduced word for $(u,v)$. Then the map $\widetilde{x}_{\mathbf{i}}:\mathcal{X}_{\widetilde{\Sigma}_{\mathbf{i}}}\to G_{\Ad}^{u,v}/H_{\Ad}$ defined by \eqref{eq:3.2} lifts to a regular map $\widetilde{a}_{\mathbf{i}}:\mathcal{A}_{\widetilde{\Sigma}_{\mathbf{i}}}\to G^{u,v}/H$, given by
\begin{equation}\label{eq:3.14}
\widetilde{a}_{\mathbf{i}}(\widetilde{A}_{-1},\ldots,\widetilde{A}_{-r},\widetilde{A}_1,\ldots,\widetilde{A}_r)=D(t_1,\ldots,t_r)E_{i_1}(a_{i_1})\cdots E_{i_{2r}}(a_{i_{2r}})
\end{equation} 
where 
\begin{align}
\widetilde{a}_{\mathbf{i}}^*(t_j)&=\widetilde{A}_j^{\varepsilon_j}\widetilde{A}_{-j}^{-\varepsilon_j}\label{eq:3.15}\\
\widetilde{a}_{\mathbf{i}}^*(\overline{a}_j)&=\prod_{k\in\widetilde{I}}\widetilde{A}_k^{-\varepsilon_j\widetilde{B}_{k,-j}},\label{eq:3.16}
\end{align} 
and $\overline{a}_j=a_{j_-}a_{j_+}$ for all $j\in[1,r]$, such that the maps $\widetilde{a}_{\mathbf{i}}$ and $\widetilde{x}_{\mathbf{i}}$ intertwine the maps $p_{\widetilde{\Sigma}_{{\mathbf{i}}}}:\mathcal{A}_{\widetilde{\Sigma}_{\mathbf{i}}}\to\mathcal{X}_{\widetilde{\Sigma}_{\mathbf{i}}}$ and $\pi:G^{u,v}/H\to G_{\Ad}^{u,v}/H_{\Ad}$, that is, we have
\begin{center}
\begin{tikzpicture} [node distance=3cm]
  \node (X) {$\mathcal{A}_{\widetilde{\Sigma}_{\mathbf{i}}}$};
  \node (Y) [right of=X] {$G^{u,v}/H$};
  \node (W) [below of=X] {$\mathcal{X}_{\widetilde{\Sigma}_{\mathbf{i}}}$};
  \node (Z) [below of=Y] {$G_{\Ad}^{u,v}/H_{\Ad}$};
  \draw[->] (X) to node [above] {$\widetilde{a}_{\mathbf{i}}$} (Y);
  \draw[->] (Y) to node [right] {$\pi$} (Z);
  \draw[->] (X) to node [left] {$p_{\widetilde{\Sigma}_{{\mathbf{i}}}}$} (W);
  \draw[->] (W) to node [below] {$\widetilde{x}_{\mathbf{i}}$} (Z);
\end{tikzpicture}
\end{center}
\end{proposition}

\begin{proof}
Firstly, we note from \eqref{eq:3.2} that a generic element in $G_{\Ad}^{u,v}/H_{\Ad}$ can be written as
\begin{equation*}
\prod_{j=1}^r u_j^{\omega_j^{\vee}}\cdot E_{i_1}(b_{i_1})\cdots E_{i_{2r}}(b_{i_{2r}})
\end{equation*}
for some $b_1,\ldots,b_{2r},u_1,\ldots,u_r\in\mathbb{C}^{\times}$. We need to show that 
\begin{align}
(\pi\circ\widetilde{a}_{\mathbf{i}})^*(u_j)&=(\widetilde{x}_{\mathbf{i}}\circ p_{\widetilde{\Sigma}_{{\mathbf{i}}}})^*(u_j),\label{eq:3.17}\quad\text{and}\\
(\pi\circ\widetilde{a}_{\mathbf{i}})^*(\overline{b}_j)&=(\widetilde{x}_{\mathbf{i}}\circ p_{\widetilde{\Sigma}_{{\mathbf{i}}}})^*(\overline{b}_j),\label{eq:3.18}
\end{align}
where $\overline{b}_j=b_{j_-}b_{j_+}$ for all $j\in[1,r]$. To this end, we first note that we have $\alpha_j^{\vee}=\sum_{k=1}^r C_{j,k}\omega_k^{\vee}$ for all $j\in[1,r]$. Thus, it follows that for all $j\in[1,r]$, we have
\begin{equation*}
\pi^*(u_j)=\prod_{k=1}^rt_k^{C_{k,j}},
\end{equation*}
and using \eqref{eq:3.15}, we have
\begin{equation*}
(\pi\circ\widetilde{a}_{\mathbf{i}})^*(u_j)=\widetilde{a}_{\mathbf{i}}^*(\pi^*(u_j))=\prod_{k=1}^r\widetilde{A}_k^{\varepsilon_kC_{k,j}}\widetilde{A}_{-k}^{-\varepsilon_kC_{k,j}}.
\end{equation*}
On the other hand, we have
\begin{equation*}
(\widetilde{x}_{\mathbf{i}}\circ p_{\widetilde{\Sigma}_{{\mathbf{i}}}})^*(u_j)
=p_{\widetilde{\Sigma}_{{\mathbf{i}}}}^*(\widetilde{x}_{\mathbf{i}}^*(u_j))
=p_{\widetilde{\Sigma}_{{\mathbf{i}}}}^*(\widetilde{X}_j\widetilde{X}_{-j})
=\prod_{k=1}^r\widetilde{A}_k^{\widetilde{B}_{k,j}+\widetilde{B}_{k,-j}}\widetilde{A}_{-k}^{\widetilde{B}_{-k,j}+\widetilde{B}_{-k,-j}}
\end{equation*}
for all $j\in[1,r]$. Using equations \eqref{eq:3.4}-\eqref{eq:3.7}, we can check case-by-case to deduce that we have $\widetilde{B}_{k,j}+\widetilde{B}_{k,-j}=\varepsilon_kC_{k,j}$ and $\widetilde{B}_{-k,j}+\widetilde{B}_{-k,-j}=-\varepsilon_kC_{k,j}$ for all $j,k\in[1,r]$. Consequently, it follows that \eqref{eq:3.17} holds.

Next, it follows from equations \eqref{eq:3.15} and \eqref{eq:3.16} that we have
\begin{equation*}
(\pi\circ\widetilde{a}_{\mathbf{i}})^*(\overline{b}_j)
=\widetilde{a}_{\mathbf{i}}^*(\pi^*(\overline{b}_j))
=\widetilde{a}_{\mathbf{i}}^*(\overline{a}_j)
=\prod_{k\in\widetilde{I}}\widetilde{A}_k^{-\varepsilon_j\widetilde{B}_{k,-j}}.
\end{equation*}
On the other hand, using the commutation relation $E_{\pm j}t^{\omega_j^{\vee}}=t^{\omega_j^{\vee}}E_{\pm j}(t^{\mp 1})$, we have
\begin{equation*}
(\widetilde{x}_{\mathbf{i}}\circ p_{\widetilde{\Sigma}_{{\mathbf{i}}}})^*(\overline{b}_j)
=p_{\widetilde{\Sigma}_{{\mathbf{i}}}}^*(\widetilde{x}_{\mathbf{i}}^*(\overline{b}_j))
=p_{\widetilde{\Sigma}_{{\mathbf{i}}}}^*(\widetilde{X}_{-j}^{-\varepsilon_j})
=\prod_{k\in\widetilde{I}}\widetilde{A}_k^{-\varepsilon_j\widetilde{B}_{k,-j}}.
\end{equation*}
Consequently, it follows that \eqref{eq:3.18} holds, and this completes the proof.
\end{proof}

\begin{remark}
Here, we note that even though $\widetilde{x}_{\mathbf{i}}:\mathcal{X}_{\widetilde{\Sigma}_{\mathbf{i}}}\to G_{\Ad}^{u,v}/H_{\Ad}$ is an open immersion, the map $\widetilde{a}_{\mathbf{i}}:\mathcal{A}_{\widetilde{\Sigma}_{\mathbf{i}}}\to G^{u,v}/H$ is not an open immersion.
\end{remark}

\begin{proposition}\label{3.13}
Let $(u,v)$ be a pair of Coxeter elements of $W$, and $\mathbf{i}$ be a double reduced word for $(u,v)$. Then the map $\widetilde{a}_{\mathbf{i}}$ is Poisson, with respect to the log-canonical Poisson structure on $\mathcal{A}_{\widetilde{\Sigma}_{\mathbf{i}}}$ as defined by \eqref{eq:3.13}, and the Poisson structure on $G^{u,v}/H$ induced from the standard Poisson-Lie structure on $G$.
\end{proposition}

In order to prove Proposition \ref{3.13}, we will need a few auxiliary lemmas.

\begin{lemma}\label{3.14}
Let $(u,v)$ be a pair of Coxeter elements, and $\mathbf{i}=(i_1,\ldots,i_{2r})$ be a double reduced word for $(u,v)$. Then with respect to the following factorization
\begin{equation*}
D(t_1,\ldots,t_r)E_{i_1}(a_{i_1})\cdots E_{i_{2r}}(a_{i_{2r}})
\end{equation*}
of a generic element in $G^{u,v}/H$ (up to conjugation), the Poisson brackets between the rational functions $\overline{a}_j=a_{j_-}a_{j_+}$ and $t_k$, $j,k\in[1,r]$, on $G^{u,v}/H$ are given by
\begin{align}
\{\overline{a}_j,\overline{a}_k\}_{u,v}&=d_j'\widetilde{B}_{-j,-k}\varepsilon_j\varepsilon_k\overline{a}_j\overline{a}_k,\label{eq:3.19}\\
\{\overline{a}_j,t_k\}_{u,v}&=d_j'\delta_{j,k}\overline{a}_jt_k,\label{eq:3.20}\quad\text{and}\\
\{t_j,t_k\}_{u,v}&=0\label{eq:3.21}
\end{align}
for all $j,k\in[1,r]$.
\end{lemma}

Lemma \ref{3.14} follows from equations \eqref{eq:2.4}-\eqref{eq:2.6}, \eqref{eq:3.4}, along with a case-by-case analysis. We shall leave the details to the reader.

\begin{lemma}\label{3.15}
Let $(u,v)\in W\times W$ be a pair of Coxeter elements, $j,\ell\in[1,r]$, and $\mathbf{i}=(i_1,\ldots,i_{2r})$, $\mathbf{i}'=(i_1',\ldots,i_{2r}')$ be double reduced words for $(u,v)$ that differ by swapping two adjacent indices $k$ and $k+1$ differing only by a sign, that is, we have $i_k=-i_{k+1}$, and
\begin{equation*}
i_{\ell}'=
\begin{cases}
i_{k+1} & \text{if } \ell=k,\\
i_k & \text{if } \ell=k+1,\\
i_{\ell} & \text{otherwise}.
\end{cases}
\end{equation*}
Suppose that we have 
\begin{equation*}
\widetilde{\Lambda}_{i,j}-\widetilde{\Lambda}_{-i,j}-\widetilde{\Lambda}_{i,-j}+\widetilde{\Lambda}_{-i,-j}=0
\end{equation*}
for all $i,j\in[1,r]$. Then we also have
\begin{equation*}
\widetilde{\Lambda}_{i,j}'-\widetilde{\Lambda}_{-i,j}'-\widetilde{\Lambda}_{i,-j}'+\widetilde{\Lambda}_{-i,-j}'=0
\end{equation*}
for all $i,j\in[1,r]$.
\end{lemma}

\begin{proof}
Let $\ell=|i_k|$. Then we have \cite[(3.2), (3.4)]{BZ05}
\begin{equation}\label{eq:3.22}
\widetilde{\Lambda}'=E^T\widetilde{\Lambda}E,
\end{equation} 
where $E$ is the $2r\times 2r$ matrix with row and column indices $-1,\ldots,-r,1,\ldots,r$ with entries defined by \cite[(3.2),(3.4)]{BZ05}
\begin{equation*}
E_{p,q}=
\begin{cases}
\delta_{p,q} & \text{if } q\neq-\ell,\\
-1 & \text{if } p=q=-\ell,\\
[\varepsilon_{\ell}\widetilde{B}_{p,-\ell}']_+ & \text{if } p\neq q=-\ell
\end{cases}
\end{equation*}
for all $p,q\in\widetilde{I}$. Then it is easy to see from the definition of $E_{p,q}$ that if $i,j\in[1,r]\setminus\{|\ell|\}$, then we have $\widetilde{\Lambda}_{\pm i,\pm j}'=\widetilde{\Lambda}_{\pm i,\pm j}$ and $\widetilde{\Lambda}_{\pm i,\mp j}'=\widetilde{\Lambda}_{\pm i,\mp j}$, and thus Lemma \ref{3.15} holds in the case where $i,j\in[1,r]\setminus\{|\ell|\}$. Consequently, we only need to consider the case where either $i=|\ell|$ or $j=|\ell|$. As $\widetilde{\Lambda}$ and $\widetilde{\Lambda}'$ are skew-symmetric, it suffices to consider the latter case, and we may also assume that $i\neq|\ell|$.

Now, we have $\varepsilon_{\ell}\widetilde{B}_{-p,-\ell}=C_{p,\ell}\delta_{p_-<\ell_-<p_+}$ and $\varepsilon_{\ell}\widetilde{B}_{p,-\ell}=C_{p,\ell}(\delta_{p,\ell}+\delta_{\ell_-<p_+}+\delta_{p_+<\ell_-}$ for all $p\in[1,r]\setminus\{|\ell|\}$. This implies that for all $q\in\widetilde{I}$, we have $\varepsilon_{\ell}\widetilde{B}_{q,-\ell}'>0$ if and only if $q=\ell$. Thus, by \eqref{eq:3.22}, along with the definition of $E_{p,q}$, it follows that we have 
\begin{align}
\widetilde{\Lambda}_{i,\ell}&=\widetilde{\Lambda}_{i,\ell},\quad\widetilde{\Lambda}_{-i,\ell}=\widetilde{\Lambda}_{-i,\ell},\label{eq:3.23}\\
\widetilde{\Lambda}_{i,-\ell}&=-\widetilde{\Lambda}_{i,-\ell}+2\widetilde{\Lambda}_{i,\ell},\quad\widetilde{\Lambda}_{-i,-\ell}=-\widetilde{\Lambda}_{-i,-\ell}+2\widetilde{\Lambda}_{-i,\ell}.\label{eq:3.24}
\end{align}
By \eqref{eq:3.23} and \eqref{eq:3.24}, it follows that Lemma \ref{3.15} holds in the case where where either $i=|\ell|$ or $j=|\ell|$. This completes the proof of Lemma \ref{3.15}.
\end{proof}

\begin{corollary}\label{3.16}
Let $(u,v)$ be a pair of Coxeter elements, and $\mathbf{i}=(i_1,\ldots,i_{2r})$ be a double reduced word for $(u,v)$. Then for all $i,j\in[1,r]$, we have
\begin{equation*}
\widetilde{\Lambda}_{i,j}-\widetilde{\Lambda}_{-i,j}-\widetilde{\Lambda}_{i,-j}+\widetilde{\Lambda}_{-i,-j}=0.
\end{equation*}
\end{corollary}

\begin{proof}
We will proceed in several steps. Firstly, we will show that Corollary \ref{3.16} holds in the case where $u=v$ and $\mathbf{i}$ is unmixed, in which case we have $\widetilde{\Sigma}_{\mathbf{i}}=\widetilde{\Sigma}_{(u,u)}$. In this case, the entries of $\widetilde{\Lambda}_{\mathbf{i}}$ are given by
\begin{equation*}
\widetilde{\Lambda}_{i,j}=\widetilde{\Lambda}_{-i,-j}=0,\quad\widetilde{\Lambda}_{-i,j}=-\widetilde{\Lambda}_{i,-j}=d_i'(C^{-1})_{i,j}
\end{equation*}
for all $i,j\in[1,r]$, from which we see that Corollary \ref{3.16} holds.

Next, we will show that Corollary \ref{3.16} holds in the case where $\mathbf{i}$ is unmixed, in which case we have $\widetilde{\Sigma}_{\mathbf{i}}=\widetilde{\Sigma}_{(u,v)}$. By Lemma \ref{3.10} and Corollary \ref{3.11}, there exist $j_1,\ldots,j_m\in[1,r]$ and $n_1,\ldots,n_m\in\{\pm1\}$, such that 
\begin{equation*}
\widetilde{\Sigma}_{(u,v)}=(\tau_{j_m}^{n_m}\circ\cdots\circ\tau_{j_1}^{n_1})(\widetilde{\Sigma}_{(u,u)}).
\end{equation*}
We shall proceed by induction on $m$ that Corollary \ref{3.16} holds for all $i,j\in[1,r]$, with the base case $m=0$ being trivial. Suppose that Corollary \ref{3.16} holds for the case $m=k-1$, where $k\geq 1$. We would like to show that Corollary \ref{3.16} holds for the case $m=k$ as well. We let $(u',v')$ be a pair of Coxeter elements satisfying
\begin{equation*}
\widetilde{\Sigma}_{(u',v')}=(\tau_{j_{k-1}}^{n_{k-1}}\circ\cdots\circ\tau_{j_1}^{n_1})(\widetilde{\Sigma}_{(u,u)}),
\end{equation*}
so that we have $\widetilde{\Sigma}_{(u,v)}=\tau_{j_k}^{n_k}(\widetilde{\Sigma}_{(u',v')})$. As $\tau_{j_k}$ is the composition of the permutation $\sigma_{j_k}$ of $\widetilde{I}$, followed by a cluster mutation at $-j_k$, it follows from Lemma \ref{3.15}, along with the induction hypothesis, that Corollary \ref{3.16} holds for $(u,v)$. Consequently, it follows that Corollary \ref{3.16} holds in the case where $\mathbf{i}$ is unmixed.

Finally, if $\mathbf{i}$ is any double reduced word for $(u,v)$, then there exists an unmixed double reduced word $\mathbf{i}'$ for $(u,v)$, and $k_1,\ldots,k_n\in\{\pm1,\ldots,\pm r\}$, such that 
\begin{equation*}
\widetilde{\Sigma}_{\mathbf{i}}=(\sigma_{-k_n}\circ\cdots\circ\sigma_{-k_1})(\widetilde{\Sigma}_{\mathbf{i}'}).
\end{equation*}
Now, by a similar argument as above, we can prove by induction on $n$ that Corollary \ref{3.16} holds for all $i,j\in[1,r]$. This completes the proof of Corollary \ref{3.16}.
\end{proof}

\begin{proof}[Proof of Proposition \ref{3.13}]
As $\widetilde{a}_{\mathbf{i}}$ is a dominant map, it suffices to show that
\begin{align}
\widetilde{a}_{\mathbf{i}}^*(\{\overline{a}_j,\overline{a}_k\}_{u,v})&=\{\widetilde{a}_{\mathbf{i}}^*(\overline{a}_j),\widetilde{a}_{\mathbf{i}}^*(\overline{a}_k)\},\label{eq:3.25}\\
\widetilde{a}_{\mathbf{i}}^*(\{\overline{a}_j,t_k\}_{u,v})&=\{\widetilde{a}_{\mathbf{i}}^*(\overline{a}_j),\widetilde{a}_{\mathbf{i}}^*(t_k)\},\label{eq:3.26}\quad\text{and}\\
\widetilde{a}_{\mathbf{i}}^*(\{t_j,t_k\}_{u,v})&=\{\widetilde{a}_{\mathbf{i}}^*(t_j),\widetilde{a}_{\mathbf{i}}^*(t_k)\}\label{eq:3.27}
\end{align}
for all $j,k\in[1,r]$. Firstly, by \eqref{eq:3.16} and \eqref{eq:3.19}, we have
{\allowdisplaybreaks
\begin{align*}
\widetilde{a}_{\mathbf{i}}^*(\{\overline{a}_j,\overline{a}_k\}_{u,v})
&=\widetilde{a}_{\mathbf{i}}^*(d_j'\widetilde{B}_{-j,-k}\varepsilon_j\varepsilon_k\overline{a}_j\overline{a}_k)\\
&=d_j'\widetilde{B}_{-j,-k}\varepsilon_j\varepsilon_k\prod_{i\in\widetilde{I}}\widetilde{A}_i^{-\varepsilon_j\widetilde{B}_{i,-j}-\varepsilon_k\widetilde{B}_{i,-k}},\\
\{\widetilde{a}_{\mathbf{i}}^*(\overline{a}_j),\widetilde{a}_{\mathbf{i}}^*(\overline{a}_k)\}
&=\left\{\prod_{i\in\widetilde{I}}\widetilde{A}_i^{-\varepsilon_j\widetilde{B}_{i,-j}},\prod_{\ell\in\widetilde{I}}^r\widetilde{A}_{\ell}^{-\varepsilon_k\widetilde{B}_{\ell,-k}}\right\}\\
&=\varepsilon_j\varepsilon_k\left(\sum_{i,\ell\in\widetilde{I}}\widetilde{B}_{i,-j}\widetilde{\Lambda}_{i,\ell}\widetilde{B}_{\ell,-k}\right)\prod_{i\in\widetilde{I}}\widetilde{A}_i^{-\varepsilon_j\widetilde{B}_{i,-j}-\varepsilon_k\widetilde{B}_{i,-k}}\\
&=-\varepsilon_j\varepsilon_k\left(\sum_{i\in\widetilde{I}}\widetilde{B}_{i,-j}\widetilde{D}_{i,-k}\right)\prod_{i\in\widetilde{I}}\widetilde{A}_i^{-\varepsilon_j\widetilde{B}_{i,-j}-\varepsilon_k\widetilde{B}_{i,-k}}\\
&=-\varepsilon_j\varepsilon_kd_k'\widetilde{B}_{-k,-j}\prod_{i\in\widetilde{I}}\widetilde{A}_i^{-\varepsilon_j\widetilde{B}_{i,-j}-\varepsilon_k\widetilde{B}_{i,-k}}\\
&=\varepsilon_j\varepsilon_kd_j'\widetilde{B}_{-j,-k}\prod_{i\in\widetilde{I}}\widetilde{A}_i^{-\varepsilon_j\widetilde{B}_{i,-j}-\varepsilon_k\widetilde{B}_{i,-k}},
\end{align*}
and so \eqref{eq:3.25} holds. Next, by \eqref{eq:3.15}, \eqref{eq:3.16} and \eqref{eq:3.20}, we have
}
\begin{align*}
\widetilde{a}_{\mathbf{i}}^*(\{\overline{a}_j,t_k\}_{u,v})
&=\widetilde{a}_{\mathbf{i}}^*(d_j'\delta_{j,k}\overline{a}_jt_k)\\
&=d_j'\delta_{j,k}\widetilde{A}_j^{\varepsilon_j}\widetilde{A}_{-j}^{-\varepsilon_j}\prod_{i\in\widetilde{I}}\widetilde{A}_i^{-\varepsilon_k\widetilde{B}_{i,-k}}\\
\{\widetilde{a}_{\mathbf{i}}^*(\overline{a}_j),\widetilde{a}_{\mathbf{i}}^*(t_k)\}
&=\left\{\prod_{i\in\widetilde{I}}\widetilde{A}_i^{-\varepsilon_j\widetilde{B}_{i,-j}},\widetilde{A}_k^{\varepsilon_k}\widetilde{A}_{-k}^{-\varepsilon_k}\right\}\\
&=\varepsilon_j\varepsilon_k\left(\sum_{i\in\widetilde{I}}(\widetilde{B}_{i,-j}\widetilde{\Lambda}_{i,-k}-\widetilde{B}_{i,-j}\widetilde{\Lambda}_{i,k})\right)\widetilde{A}_j^{\varepsilon_j}\widetilde{A}_{-j}^{-\varepsilon_j}\prod_{i\in\widetilde{I}}\widetilde{A}_i^{-\varepsilon_k\widetilde{B}_{i,-k}}\\
&=\varepsilon_j\varepsilon_k(\widetilde{D}_{-j,-k}-\widetilde{D}_{-j,k})\widetilde{A}_j^{\varepsilon_j}\widetilde{A}_{-j}^{-\varepsilon_j}\prod_{i\in\widetilde{I}}\widetilde{A}_i^{-\varepsilon_k\widetilde{B}_{i,-k}}\\
&=d_j'\delta_{j,k}\widetilde{A}_j^{\varepsilon_j}\widetilde{A}_{-j}^{-\varepsilon_j}\prod_{i\in\widetilde{I}}\widetilde{A}_i^{-\varepsilon_k\widetilde{B}_{i,-k}},
\end{align*}
which implies that \eqref{eq:3.26} holds. Finally, by \eqref{eq:3.15} and \eqref{eq:3.21}, we have
\begin{align*}
\widetilde{a}_{\mathbf{i}}^*(\{t_j,t_k\}_{u,v})
&=0\\
\{\widetilde{a}_{\mathbf{i}}^*(t_j),\widetilde{a}_{\mathbf{i}}^*(t_k)\}
&=\left\{\widetilde{A}_j^{\varepsilon_j}\widetilde{A}_{-j}^{-\varepsilon_j},\widetilde{A}_k^{\varepsilon_k}\widetilde{A}_{-k}^{-\varepsilon_k}\right\}\\
&=\varepsilon_j\varepsilon_k(\widetilde{\Lambda}_{j,k}-\widetilde{\Lambda}_{-j,k}-\widetilde{\Lambda}_{j,-k}+\widetilde{\Lambda}_{-j,-k})\widetilde{A}_j^{\varepsilon_j}\widetilde{A}_{-j}^{-\varepsilon_j}\widetilde{A}_k^{\varepsilon_k}\widetilde{A}_{-k}^{-\varepsilon_k}\\
&=0,
\end{align*}
where the last equality follows from Corollary \ref{3.16}. So \eqref{eq:3.27} holds as well, and thus the map $\widetilde{a}_{\mathbf{i}}$ is Poisson as desired.
\end{proof}
\section{Generalized B\"{a}cklund-Darboux transformations}\label{Section4}

Our goal in this section is to construct generalized B\"{a}cklund-Darboux transformations between two quotients $G^{u,v}/H$ and $G^{u',v'}/H$ of double Bruhat cells for any pairs of Coxeter elements $(u,v)$ and $(u',v')$, and show that these transformations preserve Hamiltonian flows generated by the trace function of any representation of the simple Lie group, which would then show that the family of Coxeter-Toda systems on a simple Lie group forms a single cluster integrable system. These transformations arise naturally as a consequence of Lemma \ref{3.10} and Corollary \ref{3.11}, where we have shown that the two seeds $\widetilde{\Sigma}_{(u,v)}$ and $\widetilde{\Sigma}_{(u',v')}$ are mutation-equivalent to each other, up to a permutation of indexing sets. These transformations were constructed in \cite{GSV11} in the case where $G=SL_{r+1}(\mathbb{C})$. Our goal here is to construct these transformations for the other finite Dynkin types.

To make our notations and constructions more consistent with the notations used in \cite{GSV11}, we will need a few notations. We will write $\widetilde{a}_{(u,v)}=\widetilde{a}_{\mathbf{i}}$ where $\mathbf{i}=(i_1,\ldots,i_{2r})$ is any unmixed double reduced word for a pair $(u,v)$ of Coxeter elements in $W$. By Proposition \ref{3.6}, $\widetilde{a}_{(u,v)}$ is independent of the choice of an unmixed double reduced word for $(u,v)$, and in particular, $\widetilde{a}_{(u,v)}$ is well-defined.

Next, we note that since $\mathbf{i}$ is unmixed, it follows that the map $\widetilde{a}_{(u,v)}:\mathcal{A}_{\widetilde{\Sigma}_{(u,v)}}\to G^{u,v}/H$ can be rewritten as
\begin{equation}\label{eq:4.1}
\widetilde{a}_{\mathbf{i}}(\widetilde{A}_{-1},\ldots,\widetilde{A}_{-r},\widetilde{A}_1,\ldots,\widetilde{A}_r)=E_{i_1}(1)\cdots E_{i_r}(1)D(t_1,\ldots,t_r)E_{i_{r+1}}(c_{i_{r+1}})\cdots E_{i_{2r}}(c_{i_{2r}})
\end{equation} 
where 
\begin{align}
\widetilde{a}_{(u,v)}^*(t_j)
&=\widetilde{A}_{-j}\widetilde{A}_j^{-1}\label{eq:4.2}\\
\widetilde{a}_{(u,v)}^*(c_j)
&=\widetilde{a}_{(u,v)}^*\left(\overline{a}_j\prod_{i=1}^rt_i^{-C_{i,j}}\right)
=\prod_{k\in\widetilde{I}}\widetilde{A}_k^{-\widetilde{B}_{k,j}}\label{eq:4.3}
\end{align} 
for all $j,k\in[1,r]$, and $\widetilde{B}_{i,j}$ ($i,j\in\widetilde{I}$) are the entries of $\widetilde{B}_{(u,v)}$ given by \eqref{eq:3.8}-\eqref{eq:3.10}. 

\subsection{The elementary generalized B\"{a}cklund-Darboux transformations}\label{Section4.1}

Our first step is to construct generalized B\"{a}cklund-Darboux transformations between $G^{u,v}/H$ and $G^{u',v'}/H$ in the case where $(u,v),(u',v')$ satisfy the hypotheses of Lemma \ref{3.10}, so that we have $\widetilde{\Sigma}_{(u',v')}=\tau_k(\widetilde{\Sigma}_{(u,v)})=(\mu_{-k}\circ\sigma_k)(\widetilde{\Sigma}_{(u,v)})$. 
As in \cite[Section 6.1]{GSV11}, we will show that the corresponding transformation from $G^{u,v}/H$ to $G^{u',v'}/H$ can be realized as the composition of an appropriate conjugation map from $G^{u,v}/H$ to $G^{u',v'}/H$, followed by a refactorization relation on $G^{u',v'}/H$.

Before we construct these generalized B\"{a}cklund-Darboux transformations, we will start off with the construction of the conjugation map from $G^{u,v}/H$ to $G^{u',v'}/H$, followed by the refactorization relation on $G^{u',v'}/H$.

\begin{proposition}\label{4.1}
Let $(u,v),(u',v')\in W\times W$ be pairs of Coxeter elements. Let $\mathbf{i}=(i_1,\ldots,i_{2r})$, $\mathbf{i}'=(i_1',\ldots,i_{2r}')$ be double reduced words for $(u,v)$ and $(u',v')$, such that $i_{2r}'=i_1$ and $i_j'=i_{j+1}$ for all $j\in[1,2r-1]$, and $k=|i_1|$. Then the map $\sigma_k:\mathcal{A}_{\widetilde{\Sigma}_{\mathbf{i}}}\to\mathcal{A}_{\widetilde{\Sigma}_{\mathbf{i}'}}$ induces a birational Poisson map $C_{i_1,-}:G^{u,v}/H\to G^{u',v'}/H$ given by
\begin{equation}\label{eq:4.4}
D(t_1,\ldots,t_r)E_{i_1}(a_{i_1})\cdots E_{i_{2r}}(a_{i_{2r}})\mapsto
D(t_1',\ldots,t_r')E_{i_1'}(a_{i_1'}')\cdots E_{i_{2r}'}(a_{i_{2r}'}'),
\end{equation} 
where 
\begin{align}
C_{i_1,-}^*(t_j')&=t_j\label{eq:4.5}\\
C_{i_1,-}^*(\overline{a}_j')&=\overline{a}_j\prod_{i=1}^rt_i^{\delta_{j,k}\varepsilon_kC_{i,k}},\label{eq:4.6}
\end{align} 
and $\overline{a}_j=a_{j_-}a_{j_+}$, $\overline{a}_j'=a_{j_-}'a_{j_+}'$ for all $j\in[1,r]$, with its inverse $C_{i_{2r},+}:G^{u',v'}/H\to G^{u,v}/H$ given by 
\begin{align}
C_{i_{2r},+}^*(t_j)&=t_j'\label{eq:4.7}\\
C_{i_{2r},+}^*(\overline{a}_j)&=\overline{a}_j'\prod_{i=1}^rt_i'^{\delta_{j,k}\varepsilon_k'C_{i,k}},\label{eq:4.8}
\end{align} 
such that the maps $\sigma_k$ and $C_{i_1,-}$ intertwine the maps $\widetilde{a}_{\mathbf{i}}:\mathcal{A}_{\widetilde{\Sigma}_{\mathbf{i}}}\to G^{u,v}/H$ and $\widetilde{a}_{\mathbf{i}'}:\mathcal{A}_{\widetilde{\Sigma}_{\mathbf{i}'}}\to G^{u',v'}/H$, that is, we have
\begin{center}
\begin{tikzpicture} [node distance=3cm]
  \node (X) {$\mathcal{A}_{\widetilde{\Sigma}_{\mathbf{i}}}$};
  \node (Y) [right of=X] {$\mathcal{A}_{\widetilde{\Sigma}_{\mathbf{i}'}}$};
  \node (W) [below of=X] {$G^{u,v}/H$};
  \node (Z) [below of=Y] {$G^{u',v'}/H$};
  \draw[->] (X) to node [above] {$\sigma_k$} (Y);
  \draw[->] (Y) to node [right] {$\widetilde{a}_{\mathbf{i}'}$} (Z);
  \draw[->] (X) to node [left] {$\widetilde{a}_{\mathbf{i}}$} (W);
  \draw[->] (W) to node [below] {$C_{i_1,-}$} (Z);
\end{tikzpicture}
\end{center}
\end{proposition}

\begin{proof}
Firstly, we will show that $\widetilde{a}_{\mathbf{i}'}\circ\sigma_k=C_{i_1,-}\circ\widetilde{a}_{\mathbf{i}}$, by showing that
\begin{align}
(\widetilde{a}_{\mathbf{i}'}\circ\sigma_k)^*(t_j')&=(C_{i_1,-}\circ\widetilde{a}_{\mathbf{i}})^*(t_j'),\label{eq:4.9}\quad\text{and}\\
(\widetilde{a}_{\mathbf{i}'}\circ\sigma_k)^*(\overline{a}_j')&=(C_{i_1,-}\circ\widetilde{a}_{\mathbf{i}})^*(\overline{a}_j'),\label{eq:4.10}
\end{align}
for all $j\in[1,r]$. To this end, we first note that we have $\varepsilon_j'=\varepsilon_j$ and $\sigma_k^*(\widetilde{A}_{\pm j}')=\widetilde{A}_{\pm j}$ for all $j\in[1,r]\setminus\{k\}$, and we have $\varepsilon_k'=-\varepsilon_k$ and $\sigma_k^*(\widetilde{A}_{\pm k}')=\widetilde{A}_{\mp k}$. This implies that for all $j\in[1,r]\setminus\{k\}$, we have
\begin{align*}
(\widetilde{a}_{\mathbf{i}'}\circ\sigma_k)^*(t_j')
&=\sigma_k^*(\widetilde{a}_{\mathbf{i}'}^*(t_j'))
=\sigma_k^*(\widetilde{A}_j'^{\varepsilon_j'}\widetilde{A}_{-j}^{-\varepsilon_j'})
=\widetilde{A}_j^{\varepsilon_j}\widetilde{A}_{-j}^{-\varepsilon_j},\\
(C_{i_1,-}\circ\widetilde{a}_{\mathbf{i}})^*(t_j')
&=\widetilde{a}_{\mathbf{i}}^*(C_{i_1,-}^*(t_j'))
=\widetilde{a}_{\mathbf{i}}^*(t_j)
=\widetilde{A}_j^{\varepsilon_j}\widetilde{A}_{-j}^{-\varepsilon_j},\\
(\widetilde{a}_{\mathbf{i}'}\circ\sigma_k)^*(t_k')
&=\sigma_k^*(\widetilde{a}_{\mathbf{i}'}^*(t_k'))
=\sigma_k^*(\widetilde{A}_k'^{\varepsilon_k'}\widetilde{A}_{-k}^{-\varepsilon_k'})
=\widetilde{A}_{-k}^{-\varepsilon_k}\widetilde{A}_k^{\varepsilon_k},\\
(C_{i_1,-}\circ\widetilde{a}_{\mathbf{i}})^*(t_k')
&=\widetilde{a}_{\mathbf{i}}^*(C_{i_1,-}^*(t_k'))
=\widetilde{a}_{\mathbf{i}}^*(t_k)
=\widetilde{A}_k^{\varepsilon_k}\widetilde{A}_{-k}^{-\varepsilon_k},
\end{align*}
from which we see that \eqref{eq:4.9} holds. Next, we note that we have $\widetilde{B}_{\sigma_k(p),\sigma_k(q)}'=\widetilde{B}_{p,q}$ for all $p,q\in\widetilde{I}$. Thus, it follows that for all $j\in[1,r]\setminus\{k\}$, we have
{\allowdisplaybreaks
\begin{align*}
(\widetilde{a}_{\mathbf{i}'}\circ\sigma_k)^*(\overline{a}_j')
&=\sigma_k^*(\widetilde{a}_{\mathbf{i}'}^*(\overline{a}_j'))\\
&=\sigma_k^*\left(\widetilde{A}_k'^{-\varepsilon_j'\widetilde{B}_{k,-j}'}\widetilde{A}_{-k}'^{-\varepsilon_j'\widetilde{B}_{-k,-j}'}\prod_{i\in[1,r]\setminus\{k\}}\widetilde{A}_i'^{-\varepsilon_j'\widetilde{B}_{i,-j}'}\widetilde{A}_{-i}'^{-\varepsilon_j'\widetilde{B}_{-i,-j}'}\right)\\
&=\widetilde{A}_{-k}^{-\varepsilon_j\widetilde{B}_{-k,-j}}\widetilde{A}_k^{-\varepsilon_j\widetilde{B}_{k,-j}}\prod_{i\in[1,r]\setminus\{k\}}\widetilde{A}_i^{-\varepsilon_j\widetilde{B}_{i,-j}}\widetilde{A}_{-i}^{-\varepsilon_j\widetilde{B}_{-i,-j}}\\
&=\prod_{i\in\widetilde{I}}\widetilde{A}_i^{-\varepsilon_j\widetilde{B}_{i,-j}},\\
(C_{i_1,-}\circ\widetilde{a}_{\mathbf{i}})^*(\overline{a}_j')
&=\widetilde{a}_{\mathbf{i}}^*(C_{i_1,-}^*(\overline{a}_j'))\\
&=\widetilde{a}_{\mathbf{i}}^*(\overline{a}_j)\\
&=\prod_{i\in\widetilde{I}}\widetilde{A}_i^{-\varepsilon_j\widetilde{B}_{i,-j}},
\end{align*}
from which we deduce that \eqref{eq:4.10} holds for all $j\in[1,r]\setminus\{k\}$. 

Finally, we observe that we have 
}
\begin{align*}
-\varepsilon_k'\widetilde{B}_{j,-k}'+\varepsilon_k\widetilde{B}_{j,-k}
&=\varepsilon_k(\widetilde{B}_{j,k}+\widetilde{B}_{j,-k})
=\varepsilon_j\varepsilon_kC_{j,k},\\
-\varepsilon_k'\widetilde{B}_{-j,-k}'+\varepsilon_k\widetilde{B}_{-j,-k}
&=\varepsilon_k(\widetilde{B}_{-j,k}+\widetilde{B}_{-j,k})
=-\varepsilon_j\varepsilon_kC_{j,k}
\end{align*}
for all $j\in[1,r]\setminus\{k\}$. Consequently, we have
\begin{align*}
(\widetilde{a}_{\mathbf{i}'}\circ\sigma_k)^*(\overline{a}_k')
&=\sigma_k^*(\widetilde{a}_{\mathbf{i}'}^*(\overline{a}_k'))\\
&=\sigma_k^*\left(\widetilde{A}_k'^{-2}\prod_{i\in[1,r]\setminus\{k\}}\widetilde{A}_i'^{-\varepsilon_k'\widetilde{B}_{i,-k}'}\widetilde{A}_{-i}'^{-\varepsilon_k'\widetilde{B}_{-i,-k}'}\right)\\
&=\widetilde{A}_{-k}^{-2}\prod_{i\in[1,r]\setminus\{k\}}\widetilde{A}_i^{-\varepsilon_k\widetilde{B}_{i,-k}}\widetilde{A}_{-i}^{-\varepsilon_k\widetilde{B}_{-i,-k}}
\prod_{i\in[1,r]\setminus\{k\}}\widetilde{A}_i^{\varepsilon_i\varepsilon_kC_{i,k}}\widetilde{A}_{-i}^{-\varepsilon_i\varepsilon_kC_{i,k}}\\
&=\prod_{i\in[1,r]}\widetilde{A}_i^{-\varepsilon_k\widetilde{B}_{i,-k}}\widetilde{A}_{-i}^{-\varepsilon_k\widetilde{B}_{-i,-k}}
\prod_{i\in[1,r]}\widetilde{A}_i^{\varepsilon_i\varepsilon_kC_{i,k}}\widetilde{A}_{-i}^{-\varepsilon_i\varepsilon_kC_{i,k}},\\
(C_{i_1,-}\circ\widetilde{a}_{\mathbf{i}})^*(\overline{a}_k')
&=\widetilde{a}_{\mathbf{i}}^*(C_{i_1,-}^*(\overline{a}_k'))\\
&=\widetilde{a}_{\mathbf{i}}^*\left(\overline{a}_k\prod_{i=1}^rt_i^{\varepsilon_kC_{i,k}}\right)\\
&=\prod_{i\in[1,r]}\widetilde{A}_i^{-\varepsilon_k\widetilde{B}_{i,-k}}\widetilde{A}_{-i}^{-\varepsilon_k\widetilde{B}_{-i,-k}}
\prod_{i\in[1,r]}\widetilde{A}_i^{\varepsilon_i\varepsilon_kC_{i,k}}\widetilde{A}_{-i}^{-\varepsilon_i\varepsilon_kC_{i,k}},
\end{align*}
from which we deduce that \eqref{eq:4.10} holds for $j=k$. 

Next, we need to show that $C_{i_1,-}$ is Poisson. As $C_{i_1,-}$ is a dominant map, it suffices to show that
{\allowdisplaybreaks
\begin{align}
C_{i_1,-}^*(\{\overline{a}_i,\overline{a}_j\}_{u,v})&=\{C_{i_1,-}^*(\overline{a}_i),C_{i_1,-}^*(\overline{a}_j)\},\label{eq:4.11}\\
C_{i_1,-}^*(\{\overline{a}_i,t_j\}_{u,v})&=\{C_{i_1,-}^*(\overline{a}_i),C_{i_1,-}^*(t_j)\},\label{eq:4.12}\quad\text{and}\\
C_{i_1,-}^*(\{t_i,t_j\}_{u,v})&=\{C_{i_1,-}^*(t_i),C_{i_1,-}^*(t_j)\}\label{eq:4.13}
\end{align}
for all $j,k\in[1,r]$. It is easy to see using \eqref{eq:3.20}-\eqref{eq:3.21} and \eqref{eq:4.5}-\eqref{eq:4.6} that \eqref{eq:4.12} and \eqref{eq:4.13} holds trivially, while \eqref{eq:4.11} follows from the fact that we have $\widetilde{B}_{-i,-j}'=\widetilde{B}_{i,j}$, as well as $\widetilde{B}_{-j,k}+\widetilde{B}_{-j,-k}=-\varepsilon_jC_{j,k}$ for all $j\in[1,r]$. We will omit the details here.
}
\end{proof}

When $\mathbf{i}$ is unmixed, we see that the map $C_{i_1,-}:G^{u,v}/H\to G^{u',v'}/H$ is given by the map 
\begin{align}\label{eq:4.14}
&\quad E_{i_1}(1)\cdots E_{i_r}(1)D(t_1,\ldots,t_r)E_{i_{r+1}}(c_{i_{r+1}})\cdots E_{i_{2r}}(c_{i_{2r}})\nonumber\\
&\mapsto E_{i_2}(1)\cdots E_{i_r}(1)D(t_1,\ldots,t_r)E_{i_{r+1}}(c_{i_{r+1}})\cdots E_{i_{2r}}(c_{i_{2r}})E_{i_1}(1),
\end{align} 
or equivalently, the birational map $C_{i_1,-}$ is given by conjugation by $E_{i_1}(-1)$. Indeed, this follows from the fact that if
\begin{align*}
&\quad\, E_{i_2}(1)\cdots E_{i_r}(1)D(t_1,\ldots,t_r)E_{i_{r+1}}(c_{i_{r+1}})\cdots E_{i_{2r}}(c_{i_{2r}})E_{i_1}(1)\\
&=D(t_1,\ldots,t_r)E_{i_2}(a_{i_2}')\cdots E_{i_{2r}}(a_{i_{2r}}')E_{i_1}(1),
\end{align*}
then we have $\overline{a}_k'=c_k$ and $\overline{a}_j'=c_j\prod_{i=1}^rt_i^{-C_{i,j}}$ for all $j\neq k$. Together with \eqref{eq:4.3}, we have
\begin{equation*}
C_{i_1,-}^*(\overline{a}_j')=\overline{a}_j\prod_{i=1}^rt_i^{-\delta_{j,k}C_{i,k}}
\end{equation*} 
for all $j\in[1,r]$, which coincides with \eqref{eq:4.5}, owing to the fact that $\mathbf{i}$ is unmixed. Similarly, when $\mathbf{i}'$ is unmixed, we see that the map $C_{i_{2r},+}:G^{u',v'}/H\to G^{u,v}/H$ is given by the map 
\begin{align}\label{eq:4.15}
&\quad E_{i_1'}(1)\cdots E_{i_r'}(1)D(t_1',\ldots,t_r')E_{i_{r+1}'}(c_{i_{r+1}'}')\cdots E_{i_{2r}'}(c_{i_{2r}'}')\nonumber\\
&\mapsto E_{i_{2r}'}(c_{i_{2r}'}')E_{i_1'}(1)\cdots E_{i_r'}(1)D(t_1',\ldots,t_r')E_{i_{r+1}'}(c_{i_{r+1}'}')\cdots E_{i_{2r-1}'}(c_{i_{2r-1}'}'),
\end{align} 
or equivalently, the birational map $C_{i_{2r},+}$ is given by conjugation by $E_{i_{2r}'}(c_{i_{2r}'}')$.

Next, we have the following analogue of Proposition \ref{3.3}, which describes refactorization relations in conjugation quotient Coxeter double Bruhat cells in terms of cluster transformations:

\begin{proposition}\label{4.2}
Let $(u,v)$ be a pair of Coxeter elements of $W$, $\mathbf{i}=(i_1,\ldots,i_{2r}),\mathbf{i}'=(i_1',\ldots,i_{2r}')$ be double reduced words for $(u,v)$ that differ by swapping two adjacent indices $\ell$ and $\ell+1$ differing only by a sign, and $|i_{\ell}|=j$. Then the coordinates on $G^{u,v}/H$ given by cluster variables differ by the cluster transformation at $-j$:
\begin{center}
\begin{tikzpicture}
  \node (X) {$\mathcal{A}_{\widetilde{\Sigma}_{\mathbf{i}}}$};
  \node (Y) [right=4.5cm of X] {$\mathcal{A}_{\widetilde{\Sigma}_{\mathbf{i}'}}$};
  \node (W) [below right = 2cm and 1.8cm of X] {$G^{u,v}/H$};
  \draw[->, dotted] (X) to node [above] {$\mu_{-j}$} (Y);
  \draw[->] (X) to node [above] {$\widetilde{a}_{\mathbf{i}}$} (W);
  \draw[->] (Y) to node [above] {$\widetilde{a}_{\mathbf{i}'}$} (W);
\end{tikzpicture}
\end{center}
\end{proposition}

\begin{proof}
Firstly, we will show that if
\begin{equation}\label{eq:4.16}
D(t_1,\ldots,t_r)E_{i_1}(a_{i_1})\cdots E_{i_{2r}}(a_{i_{2r}})=D(t_1',\ldots,t_r')E_{i_1'}(a_{i_1'}')\cdots E_{i_{2r}'}(a_{i_{2r}'}'),
\end{equation}
then we have
\begin{align}
t_i'&=t_i(1+\overline{a}_j)^{\delta_{i,j}\varepsilon_j},\label{eq:4.17}\\
\overline{a}_i'&=\overline{a}_i(1+\overline{a}_j)^{-\varepsilon_i\varepsilon_jC_{j,i}\delta_{i_-<j_-<i_+}}\label{eq:4.18}
\end{align}
for all $i\in[1,r]$. To this end, we first recall the following refactorization relations \cite{FZ99}:
\begin{equation*}
E_i(t')E_{-i}(t)=(1+tt')^{\alpha_i^{\vee}}E_{-i}\left(t'(1+tt')\right)E_i\left(\frac{t}{1+tt'}\right),\quad i\in[1,r].
\end{equation*}
Using the above refactorization relation, we have the following relation in $G^{u,v}/H$:
\begin{align}\label{eq:4.19}
&\quad D(t_1,\ldots,t_r)E_{i_1}(a_{i_1})\cdots E_{i_{\ell}}(a_{j_-})E_{i_{\ell+1}}(a_{j_+})\cdots E_{i_{2r}}(a_{i_{2r}})\nonumber\\
&=D(t_1,\ldots,t_r)E_{i_1}(a_{i_1})\cdots [(1+\overline{a}_j)^{\varepsilon_j}]^{\alpha_j^{\vee}}E_{i_{\ell+1}}(a_{j_+}(1+\overline{a}_j))E_{i_{\ell}}(a_{j_-}/(1+\overline{a}_j))\cdots E_{i_{2r}}(a_{i_{2r}}).
\end{align}
The equations \eqref{eq:4.17} and \eqref{eq:4.18} now follow by equating the RHS of both \eqref{eq:4.16} and \eqref{eq:4.19}, using the commutation relations $t^{\alpha_i^{\vee}}E_{\pm k}(a)=E_{\pm k}(at^{\pm C_{i,k}})t^{\alpha_i^{\vee}}$ where necessary.

Next, we need to show that 
\begin{align}
(\widetilde{a}_{\mathbf{i}'}\circ\mu_{-j})^*(t_i')&=\widetilde{a}_{\mathbf{i}}^*(t_i'),\label{eq:4.20}\quad\text{and}\\
(\widetilde{a}_{\mathbf{i}'}\circ\mu_{-j})^*(\overline{a}_i')&=\widetilde{a}_{\mathbf{i}}^*(\overline{a}_i'),\label{eq:4.21}
\end{align}
for all $i\in[1,r]$. As we have $\varepsilon_i'=\varepsilon_i$ and $\mu_{-j}^*(\widetilde{A}_{\pm i}')=\widetilde{A}_{\pm i}$ for all $i\neq j$, it follows that \eqref{eq:4.20} holds for all $i\in[1,r]\setminus\{j\}$. Similarly, as we have $\varepsilon_j'\widetilde{B}_{k,-j}'=\varepsilon_j\widetilde{B}_{k,-j}$ and $\varepsilon_j'\widetilde{B}_{-k,-j}'=\varepsilon_j\widetilde{B}_{-k,-j}$ for all $k\in[1,r]$, and $\widetilde{B}_{-k,-k}=0=\widetilde{B}_{-k,-k}'$, it follows that \eqref{eq:4.21} holds for $i=j$.

Next, we observe from equations \eqref{eq:3.4} and \eqref{eq:3.6} that we have
{\allowdisplaybreaks
\begin{align}
\widetilde{B}_{-i,-j}
&=C_{i,j}\varepsilon_j\delta_{i_-<j_-<i_+},\label{eq:4.22}\\
\widetilde{B}_{i,-j}
&=C_{i,j}\varepsilon_j(1-\delta_{i_-<j_-<i_+}),\label{eq:4.23}\\
\widetilde{B}_{-j,-i}
&=-C_{j,i}\varepsilon_j\delta_{i_-<j_-<i_+},\label{eq:4.24}\\
\widetilde{B}_{j,-i}'-\widetilde{B}_{j,-i}
&=2\widetilde{B}_{-j,-i}\label{eq:4.25}
\end{align}
for all $i\neq j$. In particular, it follows from \eqref{eq:4.22} and \eqref{eq:4.23} that $\sgn(\widetilde{B}_{j,-j})\sgn(\widetilde{B}_{\pm i,-j})\leq 0$ for all $i\neq j$, from which we deduce that
}
\begin{equation}\label{eq:4.26}
\mu_{-j}^*(\widetilde{A}_{-j}')=\widetilde{A}_{-j}^{-1}\widetilde{A}_j^2\left(1+\prod_{i\in\widetilde{I}}\widetilde{A}_i^{-\varepsilon_j\widetilde{B}_{i,-j}}\right).
\end{equation}
By \eqref{eq:3.15} and \eqref{eq:4.26}, we have
\begin{align*}
(\widetilde{a}_{\mathbf{i}'}\circ\mu_{-j})^*(t_j')
&=\mu_{-j}^*(\widetilde{a}_{\mathbf{i}'}^*(t_j'))\\
&=\mu_{-j}^*(\widetilde{A}_j'^{\varepsilon_j'}\widetilde{A}_{-j}'^{-\varepsilon_j'})\\
&=\mu_{-j}^*(\widetilde{A}_j'^{-\varepsilon_j}\widetilde{A}_{-j}'^{\varepsilon_j})\\
&=\widetilde{A}_j^{\varepsilon_j}\widetilde{A}_{-j}^{-\varepsilon_j}\left(1+\prod_{i\in\widetilde{I}}\widetilde{A}_i^{-\varepsilon_j\widetilde{B}_{i,-j}}\right)^{\varepsilon_j}.
\end{align*}
On the other hand, it follows from \eqref{eq:3.15} and \eqref{eq:4.17} that we have 
\begin{equation*}
\widetilde{a}_{\mathbf{i}}^*(t_j')
=\widetilde{a}_{\mathbf{i}}^*(t_j(1+\overline{a}_j)^{\varepsilon_j})
=\widetilde{A}_j^{\varepsilon_j}\widetilde{A}_{-j}^{-\varepsilon_j}\left(1+\prod_{i\in\widetilde{I}}\widetilde{A}_i^{-\varepsilon_j\widetilde{B}_{i,-j}}\right)^{\varepsilon_j},
\end{equation*}
from which we deduce that \eqref{eq:4.20} holds for $i=j$.

Finally, for all $i,k\neq j$, we have $\widetilde{B}_{k,-i}'=\widetilde{B}_{k,-i}$ and $\widetilde{B}_{-k,-i}'=\widetilde{B}_{-k,-i}$. Thus, it follows from \eqref{eq:3.16} and \eqref{eq:4.24}-\eqref{eq:4.26} that for all $i\neq j$, we have
\begin{align*}
(\widetilde{a}_{\mathbf{i}'}\circ\mu_{-j})^*(\overline{a}_i')
&=\mu_{-j}^*(\widetilde{a}_{\mathbf{i}'}^*(\overline{a}_i'))\\
&=\mu_{-j}^*\left(\prod_{k\in\widetilde{I}}\widetilde{A}_k'^{-\varepsilon_i'\widetilde{B}_{k,-i}'}\right)\\
&=\left(1+\prod_{m\in\widetilde{I}}\widetilde{A}_m^{-\varepsilon_j\widetilde{B}_{m,-j}}\right)^{-\varepsilon_i\varepsilon_jC_{j,i}\delta_{i_-<j_-<i_+}}\prod_{k\in\widetilde{I}}^r\widetilde{A}_k^{-\varepsilon_i\widetilde{B}_{k,-i}}.
\end{align*}
On the other hand, it follows from \eqref{eq:3.16} and \eqref{eq:4.21} that we have 
\begin{align*}
\widetilde{a}_{\mathbf{i}}^*(\overline{a}_i')
&=\widetilde{a}_{\mathbf{i}}^*(\overline{a}_i(1+\overline{a}_j)^{-\varepsilon_i\varepsilon_jC_{j,i}\delta_{i_-<j_-<i_+}})\\
&=\left(1+\prod_{m\in\widetilde{I}}\widetilde{A}_m^{-\varepsilon_j\widetilde{B}_{m,-j}}\right)^{-\varepsilon_i\varepsilon_jC_{j,i}\delta_{i_-<j_-<i_+}}\prod_{k\in\widetilde{I}}\widetilde{A}_k^{-\varepsilon_i\widetilde{B}_{k,-i}}
\end{align*}
for all $i\in[1,r]\setminus\{j\}$. So this shows that \eqref{eq:4.21} holds for all $i\in[1,r]\setminus\{j\}$, and we are done.
\end{proof}

We are now ready to construct the generalized B\"{a}cklund-Darboux transformations between $G^{u,v}/H$ and $G^{u',v'}/H$ in the case where $(u,v),(u',v')$ satisfy the hypotheses of Lemma \ref{3.10}. To begin, let us first consider the case where $(u,v),(u',v')$ satisfy the first case of the hypotheses of Lemma \ref{3.10}. Then there exists an unmixed double reduced word $\mathbf{i}=(i_1,\ldots,i_{2r})$ for $(u,v)$, such that $i_1=-k$. This implies that $\overline{\mathbf{i}}'=(i_2,\ldots,i_{2r},i_1)$ and $\mathbf{i}'=(i_2,\ldots,i_r,i_1,i_{r+1},\ldots,i_{2r})$ are double reduced words for $(u',v')$, with $\mathbf{i}'$ being unmixed. Thus, we have a map $\rho_{\tau_k}:G^{u,v}/H\to G^{u',v'}/H$, given in terms of the composition of the conjugation map $C_{-k,-}:G^{u,v}/H\to G^{u',v'}/H$ given in \eqref{eq:4.14}, and the refactorization map $G^{u',v'}/H\to G^{u',v'}/H$ given in \eqref{eq:4.19}, as follows:
\begin{align}\label{eq:4.27}
&\quad E_{i_1}(1)\cdots E_{i_r}(1)D(t_1,\ldots,t_r)E_{i_{r+1}}(c_{i_{r+1}})\cdots E_{i_{2r}}(c_{i_{2r}})\nonumber\\
&\mapsto E_{i_2}(1)\cdots E_{i_r}(1)D(t_1,\ldots,t_r)E_{i_{r+1}}(c_{i_{r+1}})\cdots E_{i_{2r}}(c_{i_{2r}})E_{i_1}(1)\nonumber\\
&=E_{i_2}(1)\cdots E_{i_r}(1)E_{i_1}(1)D(t_1',\ldots,t_r')E_{i_{r+1}}(c_{i_{r+1}}')\cdots E_{i_{2r}}(c_{i_{2r}}')
\end{align} 
As the map $C_{-k,-}$ is Poisson, it implies that the map $\rho_{\tau_k}$ is also Poisson as well, and the maps $\tau_k=\mu_{-k}\circ\sigma_k:\mathcal{A}_{\widetilde{\Sigma}_{(u,v)}}\to \mathcal{A}_{\widetilde{\Sigma}_{(u',v')}}$ and $\rho_{\tau_k}:G^{u,v}/H\to G^{u',v'}/H$ intertwine the maps $\widetilde{a}_{(u,v)}:\mathcal{A}_{\widetilde{\Sigma}_{(u,v)}}\to G^{u,v}/H$ and $\widetilde{a}_{(u',v')}:\mathcal{A}_{\widetilde{\Sigma}_{(u',v')}}\to G^{u',v'}/H$, that is, we have
\begin{center}
\begin{tikzpicture} [node distance=3cm]
  \node (X) {$\mathcal{A}_{\widetilde{\Sigma}_{(u,v)}}$};
  \node (Y) [below of=X] {$G^{u,v}/H$};
  \node (W) [right of=X] {$\mathcal{A}_{\widetilde{\Sigma}_{(u',v')}}$};
  \node (Z) [right of=Y] {$G^{u',v'}/H$};
  \draw[->] (X) to node [left] {$\widetilde{a}_{(u,v)}$} (Y);
  \draw[->, dotted] (Y) to node [below] {$\rho_{\tau_k}$} (Z);
  \draw[->, dotted] (X) to node [above] {$\tau_k$} (W);
  \draw[->] (W) to node [right] {$\widetilde{a}_{(u',v')}$} (Z);
\end{tikzpicture}
\end{center}
Moreover, we can check that the factorization parameters $c_1,\ldots,c_r,t_1,\ldots,t_r$ on $G^{u,v}/H$ and the factorization parameters $c_1',\ldots,c_r',t_1',\ldots,t_r'$ on $G^{u',v'}/H$ are related to each other as follows:
{\allowdisplaybreaks
\begin{align}
\rho_{\tau_k}^*(t_i')
&=t_i(1+c_k)^{\delta_{i,k}},\label{eq:4.28}\\
\rho_{\tau_k}^*(c_i')
&=
\begin{cases}
\frac{c_k}{(1+c_k)^2}\prod_{j=1}^rt_j^{-C_{j,k}} & \text{if }i=k,\\
c_i(1+c_k)^{-C_{k,i}\delta_{i_+<k_+}} & \text{otherwise},
\end{cases}
\label{eq:4.29}\\
(\rho_{\tau_k}^{-1})^*(t_i)
&=t_i'\left(1+c_k'\prod_{j=1}^rt_j'^{C_{j,k}}\right)^{-\delta_{i,k}},\label{eq:4.30}\\
(\rho_{\tau_k}^{-1})^*(c_i)
&=
\begin{cases}
c_k'\prod_{j=1}^rt_j'^{C_{j,k}} & \text{if }i=k,\\
c_i'\left(1+c_k'\prod_{j=1}^rt_j'^{C_{j,k}}\right)^{C_{k,i}\delta_{i_+<k_+}} & \text{otherwise},
\end{cases}
\label{eq:4.31}
\end{align}
When $(u,v),(u',v')$ satisfy the second case of the hypotheses of Lemma \ref{3.10}, the birational Poisson map $\rho_{\tau_k}:G^{u,v}/H\to G^{u',v'}/H$ and its inverse $\rho_{\tau_k}^{-1}:G^{u',v'}/H\to G^{u,v}/H$ are defined in a similar manner as well. We will omit the details here. 

Following \cite{GSV11}, we call these maps $\rho_{\tau_k}$ and $\rho_{\tau_k}^{-1}$ elementary generalized B\"{a}cklund-Darboux transformations. The partial rationale behind such a term comes from the fact that the map $C_{-k,-}$ interchanges the factors $E_{i_1}(1),\ldots,E_{i_r}(1),E_{i_{r+1}}(c_{i_{r+1}}),\ldots,E_{i_{2r}}(c_{i_{2r}}),$ $D(t_1,\ldots,t_r)$ via conjugation. However, as in \cite{GSV11}, the rearrangement of factors is followed by a refactorization relation, and so we have a transformation of the factorization parameters $c_1,\ldots,c_r,t_1,\ldots,t_r$. Later, we will show that these transformations preserve Hamiltonian flows generated by Coxeter-Toda Hamiltonians arising from trace functions of representations of $G$.
}

\begin{remark}
Here, we remark that in the case where $G=SL_{r+1}(\mathbb{C})$ and $k\neq 1$, these formulas are essentially the same as that obtained in \cite[Section 6.1]{GSV11}.
\end{remark}

For the general case, we let $(u,v),(u',v')$ be any pairs of Coxeter elements. By Lemma \ref{3.10} and Corollary \ref{3.11}, there exist $j_1,\ldots,j_m\in[1,r]$ and $n_1,\ldots,n_m\in\{\pm1\}$, such that 
\begin{equation*}
\widetilde{\Sigma}_{(u',v')}=(\tau_{j_m}^{n_m}\circ\cdots\circ\tau_{j_1}^{n_1})(\widetilde{\Sigma}_{(u,v)}).
\end{equation*}
We let $\tau=\tau_{j_m}^{n_m}\circ\cdots\circ\tau_{j_1}^{n_1}$ and $\rho_{\tau}=\rho_{\tau_{j_m}}^{n_m}\circ\cdots\circ\rho_{\tau_{j_1}}^{n_1}$. Then it follows that the map $\tau:\mathcal{A}_{\widetilde{\Sigma}_{(u,v)}}\to\mathcal{A}_{\widetilde{\Sigma}_{(u',v')}}$ and the birational Poisson map $\rho_{\tau}:G^{u,v}/H\to G^{u',v'}/H$ intertwine the maps $\widetilde{a}_{(u,v)}:\mathcal{A}_{\widetilde{\Sigma}_{(u,v)}}\to G^{u,v}/H$ and $\widetilde{a}_{(u',v')}:\mathcal{A}_{\widetilde{\Sigma}_{(u',v')}}\to G^{u',v'}/H$, that is, we have:
\begin{center}
\begin{tikzpicture} [node distance=3cm]
  \node (X) {$\mathcal{A}_{\widetilde{\Sigma}_{(u,v)}}$};
  \node (Y) [below of=X] {$G^{u,v}/H$};
  \node (W) [right of=X] {$\mathcal{A}_{\widetilde{\Sigma}_{(u',v')}}$};
  \node (Z) [right of=Y] {$G^{u',v'}/H$};
  \draw[->] (X) to node [left] {$\widetilde{a}_{(u,v)}$} (Y);
  \draw[->, dotted] (Y) to node [below] {$\rho_{\tau}$} (Z);
  \draw[->, dotted] (X) to node [above] {$\tau$} (W);
  \draw[->] (W) to node [right] {$\widetilde{a}_{(u',v')}$} (Z);
\end{tikzpicture}
\end{center}
Following \cite{GSV11}, we call $\rho_{\tau}$ a generalized B\"{a}cklund-Darboux transformation from $G^{u,v}/H$ to $G^{u',v'}/H$.  

Our next goal is to show that these generalized B\"{a}cklund-Darboux transformations preserve Hamiltonian flows generated by Coxeter-Toda Hamiltonians arising from trace functions of representations of $G$. Before we describe our result further, we would like to recall the results proved in \cite{GSV11} for the case $G=SL_{r+1}(\mathbb{C})$.

For each $j=1,\ldots,r$, $X\in G$ and a pair of Coxeter elements $(u,v)$, we let $\tilde{H}_j(X)=\tr(X^j)$, and $\tilde{H}_j^{u,v}$ denote the restriction of the Hamiltonian $\tilde{H}_j$ to $G^{u,v}/H$. Then with respect to any Coxeter-Toda system on $G^{u,v}/H$, it follows that $\{\tilde{H}_1^{u,v},\ldots,\tilde{H}_r^{u,v}\}$ is also a full set of algebraically independent integrals of motion \cite{Kostant79}. Gekhtman \emph{et al}. \cite{GSV11} showed that the generalized B\"{a}cklund-Darboux transformation $\rho_{\tau}:G^{u,v}/H\to G^{u',v'}/H$ sends the Hamiltonian flow generated by $\tilde{H}_j^{u,v}$ on $G^{u,v}/H$ to the Hamiltonian flow generated by $\tilde{H}_j^{u',v'}$ on $G^{u',v'}/H$ for all $j=1,\ldots,r$, that is, if $c_i(t)$ and $t_i(t)$ satisfy
\begin{equation}\label{eq:4.32}
\frac{dc_i}{dt}=\left\{c_i,\tilde{H}_j^{u,v}\right\}_{u,v},\quad\frac{dt_i}{dt}=\left\{t_i,\tilde{H}_j^{u,v}\right\}_{u,v},
\end{equation}
then $c_i'(t)$ and $t_i'(t)$ satisfy
\begin{equation}\label{eq:4.33}
\frac{d(c_i'\circ\rho_{\tau})}{dt}=\left\{c_i',\tilde{H}_j^{u',v'}\right\}_{u',v'}\circ\rho_{\tau},\quad\frac{d(t_i'\circ\rho_{\tau})}{dt}=\left\{t_i',\tilde{H}_j^{u',v'}\right\}_{u',v'}\circ\rho_{\tau}.
\end{equation}
In particular, as $H_1,\ldots,H_r$ and $\tilde{H}_1,\ldots,\tilde{H}_r$ can be obtained from each other via the Faddeev–LeVerrier algorithm, it follows that if $c_i(t)$ and $t_i(t)$ satisfy \begin{equation}\label{eq:4.34}
\frac{dc_i}{dt}=\left\{c_i,H_j^{u,v}\right\}_{u,v},\quad\frac{dt_i}{dt}=\left\{t_i,H_j^{u,v}\right\}_{u,v},
\end{equation}
then $c_i'(t)$ and $t_i'(t)$ satisfy
\begin{equation}\label{eq:4.35}
\frac{d(c_i'\circ\rho_{\tau})}{dt}=\left\{c_i',H_j^{u',v'}\right\}_{u',v'}\circ\rho_{\tau},\quad\frac{d(t_i'\circ\rho_{\tau})}{dt}=\left\{t_i',H_j^{u',v'}\right\}_{u',v'}\circ\rho_{\tau}.
\end{equation}
Their proof of the statement essentially follows from the following ingredients:
\begin{enumerate}
\item Gekhtman \emph{et al}. \cite{GSV11} showed that if an element $X$ in $G^{u,v}/H$ admits the factorization scheme \eqref{eq:2.3}, then $X$ can be recovered from its Weyl function $m(\lambda; X)=(\lambda I-X)^{-1}_{1,1}=\sum_{i=0}^{\infty}h_i(X)\lambda^i$ via Hankel determinants, 
\item If $\rho_{\tau}=\rho_{\tau_{k_n}}\circ\cdots\circ\rho_{\tau_{k_1}}:G^{u,v}/H\to G^{u',v'}/H$ is a generalized B\"{a}cklund-Darboux transformation with $|k_1|,\ldots,|k_n|\neq1$, and $X$ is a generic element in $G^{u,v}/H$ with $X'=\rho_{\tau}(X)$, then $m(\lambda; X)=m(\lambda; X')$, and
\item For any $j=1,\ldots,r$, the Hamiltonian equations obtained from the Hamiltonian flow generated by $\tilde{H}_j^{u,v}$ on $G^{u,v}/H$ is equivalent to the following evolution of the coefficients of the Laurent expansion of the Weyl function of $X$:
\begin{equation}\label{eq:4.36}
\frac{d}{dt}h_i(X)=h_{i+j}(X)-h_i(X)h_j(X).
\end{equation}
\end{enumerate}
As we have $m(\lambda; X)=\frac{q(\lambda)}{p(\lambda)}$, where $p(\lambda)$ is the characteristic polynomial of $X$ and $q(\lambda)$ is the characteristic polynomial of the $r\times r$ submatrix of $X$ formed by deleting the first row and first column of $X$, and the non-trivial coefficients of $p(\lambda)$ are precisely the Hamiltonians $H_1^{u,v}(X),\ldots,H_r^{u,v}(X)$, we have $H_j^{u,v}=H_j^{u',v'}\circ\rho_{\tau}$ for all $j=1,\ldots,r$.

We have the following generalization of the above result obtained in \cite{GSV11} to the other finite Dynkin types:

\begin{theorem}\label{4.3}
Let $(u,v)$ and $(u',v')$ be two pairs of Coxeter elements, and $\rho_{\tau}:G^{u,v}/H\to G^{u',v'}/H$ be a generalized B\"{a}cklund-Darboux transformation. Let $\pi_V:G\to SL(V)$ be a finite-dimensional representation of $G$, and for all $X\in G$, we define $H_V(X)=\tr(\pi_V(X))$. Then the birational Poisson map $\rho_{\tau}:G^{u,v}/H\to G^{u',v'}/H$ sends the Hamiltonian flow generated by the Coxeter-Toda Hamiltonian $H_V^{u,v}$ on $G^{u,v}/H$ to the Hamiltonian flow generated by the Coxeter-Toda Hamiltonian $H_V^{u',v'}$ on $G^{u',v'}/H$. More precisely, if $c_i(t),t_i(t)$ satisfy
\begin{equation*}
\frac{dc_i}{dt}=\left\{c_i,H_V^{u,v}\right\}_{u,v},\quad\frac{dt_i}{dt}=\left\{t_i,H_V^{u,v}\right\}_{u,v},
\end{equation*}
for all $i=1,\ldots,r$, then we have 
\begin{equation*}
\frac{d(c_i'\circ\rho_{\tau})}{dt}=\left\{c_i',H_V^{u',v'}\right\}_{u',v'}\circ\rho_{\tau},\quad\frac{d(t_i'\circ\rho_{\tau})}{dt}=\left\{t_i',H_V^{u',v'}\right\}_{u',v'}\circ\rho_{\tau}.
\end{equation*}
\end{theorem}

In particular, Theorem \ref{4.3} implies that the generalized B\"{a}cklund-Darboux transformation $\rho_{\tau}:G^{u,v}/H\to G^{u',v'}/H$ preserves the Coxeter-Toda systems on both Coxeter double Bruhat cells $G^{u,v}/H$ and $G^{u',v'}/H$. As $\rho_{\tau}$ is induced from the map $\tau:\mathcal{A}_{\widetilde{\Sigma}_{(u,v)}}\to\mathcal{A}_{\widetilde{\Sigma}_{(u',v')}}$, which itself is a composition of cluster transformations and permutations of cluster seeds, this shows that the family of Coxeter-Toda systems on $G$ forms a single cluster integrable system.

Instead of following the approach using Weyl functions employed by Gekhtman \emph{et al}. \cite{GSV11}, we will prove Theorem \ref{4.3} in a different manner using the following theorem:

\begin{theorem}\label{4.4}
Let $(u,v)$ and $(u',v')$ be pairs of Coxeter elements, and $\rho_{\tau}:G^{u,v}/H\to G^{u',v'}/H$ be a generalized B\"{a}cklund-Darboux transformation. Let $\pi_V:G\to SL(V)$ be a finite-dimensional representation of $G$, and $H_V=\tr\circ\pi_V:G\to\mathbb{C}$ be the trace function of $V$. Then we have
\begin{equation*}
H_V^{u,v}=H_V^{u',v'}\circ\rho_{\tau}.
\end{equation*}
\end{theorem} 

\begin{proof}
Let us write $\tau=\tau_{j_m}^{n_m}\circ\cdots\circ\tau_{j_1}^{n_1}$ for some $j_1,\ldots,j_m\in[1,r]$ and $n_1,\ldots,n_m\in\{\pm1\}$. As each $\rho_{\tau_{j_i}}$ is the composition of a conjugation map followed by a refactorization relation, it follows that the trace function $H_V$ is preserved under the map $\rho_{\tau_{j_i}}^{n_i}$ for all $i\in[1,m]$, from which the desired statement holds.
\end{proof}

\begin{proof}[Proof of Theorem \ref{4.3}]
By equations \eqref{eq:4.28}-\eqref{eq:4.31}, it is easy to check using the chain rule on composite functions and the Leibniz rule satisfied by the Poisson bracket $\{\cdot,\cdot\}_{u,v}$ that we have 
\begin{align*}
\frac{d}{dt}(t_i'\circ\rho_{\tau})&=\{t_i'\circ\rho_{\tau},H_V^{u,v}\}_{u,v},\quad\text{and}\\
\frac{d}{dt}(c_i'\circ\rho_{\tau})&=\{c_i'\circ\rho_{\tau},H_V^{u,v}\}_{u,v}
\end{align*}
By Theorem \ref{4.4}, along with the fact that the map $\rho_{\tau}$ is Poisson, we have
\begin{equation*}
\frac{d}{dt}(t_i'\circ\rho_{\tau})
=\{t_i'\circ\rho_{\tau},H_V^{u,v}\}_{u,v}
=\{t_i'\circ\rho_{\tau},H_V^{u',v'}\circ\rho_{\tau}\}_{u,v}
=\{t_i',H_V^{u',v'}\}_{u',v'}\circ\rho_{\tau}.
\end{equation*}
By a similar argument as above, we have 
\begin{equation*}
\frac{d}{dt}(c_i'\circ\rho_{\tau})
=\{c_i',H_V^{u',v'}\}_{u',v'}\circ\rho_{\tau},
\end{equation*}
and this completes the proof.
\end{proof}

\subsection{The factorization mapping in terms of generalized B\"{a}cklund-Darboux transformations}\label{Section4.2}

Throughout this subsection, we shall fix a pair $(u,v)$ of Coxeter elements, and an unmixed double reduced word $\mathbf{i}=(i_1,\ldots,i_{2r})$ for $(u,v)$. It was shown in \cite[Section 6.2]{GSV11} that the following factorization mapping 
\begin{align}\label{eq:4.37}
&\quad E_{i_1}(1)\cdots E_{i_r}(1)D(t_1,\ldots,t_r)E_{i_{r+1}}(c_{i_{r+1}})\cdots E_{i_{2r}}(c_{i_{2r}})\nonumber\\
&\mapsto D(t_1,\ldots,t_r)E_{i_{r+1}}(c_{i_{r+1}})\cdots E_{i_{2r}}(c_{i_{2r}})E_{i_1}(1)\cdots E_{i_r}(1)\nonumber\\
&=E_{i_1}(1)\cdots E_{i_r}(1)D(t_1',\ldots,t_r')E_{i_{r+1}}(c_{i_{r+1}}')\cdots E_{i_{2r}}(c_{i_{2r}}')
\end{align}
can be described in terms of elementary generalized B\"{a}cklund-Darboux transformations in the case where $G=SL_{r+1}(\mathbb{C})$. Our goal in this subsection is to describe the factorization mapping \eqref{eq:4.37} in general:

\begin{theorem}\label{4.5}
The factorization mapping \eqref{eq:4.37} coincides with the generalized B\"{a}cklund-Darboux transformation $\rho_{\tau}:G^{u,v}/H\to G^{u,v}/H$, where $\tau=\tau_{|i_r|}\circ\cdots\circ\tau_{|i_1|}$.
\end{theorem} 

In order to prove Theorem \ref{4.5}, we need a few auxiliary lemmas:

\begin{lemma}\label{4.6}
Let $(w_1,w_2)$ be a pair of Coxeter elements, and $\mathbf{j}=(j_1,\ldots,j_{2r})$ be a double reduced word for $(w_1,w_2)$. Then for any distinct $i,j\in[1,r]$, we have $(\mu_{-j}\circ\sigma_i)(\widetilde{\Sigma}_{\mathbf{j}})=(\sigma_i\circ\mu_{-j})(\widetilde{\Sigma}_{\mathbf{j}})$, that is, the cluster mutation $\mu_{-j}$ commutes with the permutation $\sigma_i$.
\end{lemma}

\begin{lemma}\label{4.7}
Let $\mathbf{i}'=(i_{r+1},\ldots,i_{2r},i_1,\ldots,i_r)$. Then $\widetilde{\Sigma}_{\mathbf{i}'}=\sigma(\widetilde{\Sigma}_{\mathbf{i}})$, where $\sigma$ is the permutation of $\widetilde{I}$ given by $\sigma(i)=-i$ for all $i\in\widetilde{I}$.
\end{lemma}

\begin{proof}
By the definition of $\mathbf{i}'$, it follows that we have $\varepsilon_j'=-\varepsilon_j=1$, $j_-'=j_+-r$ and $j_+'=j_-+r$. Using \eqref{eq:3.4}-\eqref{eq:3.10}, it follows that we have
{\allowdisplaybreaks
\begin{align*}
\widetilde{B}_{-j,-k}'
&=\frac{C_{j,k}}{2}[(\delta_{j_-'<k_-'<j_+'}+\delta_{j_-'<k_+'<j_+'})-(\delta_{k_-'<j_-'<k_+'}+\delta_{k_-'<j_+'<k_+'})]\\
&=\frac{C_{j,k}}{2}(\delta_{j_+<k_+}+\delta_{k_-<j_-}-\delta_{k_+<j_+}-\delta_{j_-<k_-})\\
&=\widetilde{B}_{j,k},\\
\widetilde{B}_{-j,k}'
&=-\frac{C_{j,k}}{2}[2\delta_{j,k}+\delta_{j_-'<k_-'}+\delta_{j_+'<k_-'}+\delta_{k_+'<j_-'}+\delta_{k_+'<j_+'}+\delta_{j_-'<k_-'<j_+'}+\delta_{j_-'<k_+'<j_+'}]\\
&=-C_{j,k}(\delta_{j,k}+\delta_{j_+<k_+}+\delta_{k_-<j_-})\\
&=\widetilde{B}_{j,-k},\\
\widetilde{B}_{j,-k}'
&=\frac{C_{j,k}}{2}[2\delta_{j,k}+\delta_{k_-'<j_-'}+\delta_{k_+'<j_-'}+\delta_{j_+'<k_-'}+\delta_{j_+'<k_+'}+\delta_{k_-'<j_-'<k_+'}+\delta_{k_-'<j_+'<k_+'}]\\
&=C_{j,k}(\delta_{j,k}+\delta_{j_-<k_-}+\delta_{k_+<j_+})\\
&=\widetilde{B}_{-j,k},\\
\widetilde{B}_{j,k}'
&=\frac{C_{j,k}}{2}[(\delta_{j_-'<k_-'}+\delta_{j_+'<k_-'}+\delta_{k_+'<j_-'}+\delta_{k_+'<j_+'})-(\delta_{k_-'<j_-'}+\delta_{k_+'<j_-'}+\delta_{j_+'<k_-'}+\delta_{j_+'<k_+'})]\\
&=\frac{C_{j,k}}{2}(\delta_{j_+<k_+}+\delta_{k_-<j_-}-\delta_{k_+<j_+}-\delta_{j_-<k_-})\\
&=\widetilde{B}_{-j,-k}
\end{align*}
for all $j,k\in[1,r]$. This completes the proof of Lemma \ref{4.7}.
}
\end{proof}

\begin{proposition}\label{4.8}
Let $\mathbf{i}'=(i_{r+1},\ldots,i_{2r},i_1,\ldots,i_r)$, and $\sigma$ be the permutation of $\widetilde{I}$ given by $\sigma(i)=-i$ for all $i\in\widetilde{I}$. Then the map $\sigma:\mathcal{A}_{\widetilde{\Sigma}_{\mathbf{i}}}\to\mathcal{A}_{\widetilde{\Sigma}_{\mathbf{i}'}}$ induces a map $f_{\sigma}:G^{u,v}/H\to G^{u,v}/H$, given by 
\begin{align}\label{eq:4.38}
&\quad E_{i_1}(1)\cdots E_{i_r}(1)D(t_1,\ldots,t_r)E_{i_{r+1}}(c_{i_{r+1}})\cdots E_{i_{2r}}(c_{i_{2r}})\nonumber\\
&\mapsto D(t_1,\ldots,t_r)E_{i_{r+1}}(c_{i_{r+1}})\cdots E_{i_{2r}}(c_{i_{2r}})E_{i_1}(1)\cdots E_{i_r}(1),
\end{align} 
such that the maps $\sigma$ and $f_{\tau}$ intertwine the maps $\widetilde{a}_{\mathbf{i}}:\mathcal{A}_{\widetilde{\Sigma}_{\mathbf{i}}}\to G^{u,v}/H$ and $\widetilde{a}_{\mathbf{i}'}:\mathcal{A}_{\widetilde{\Sigma}_{\mathbf{i}'}}\to G^{u,v}/H$, that is, we have
\begin{center}
\begin{tikzpicture} [node distance=3cm]
  \node (X) {$\mathcal{A}_{\widetilde{\Sigma}_{\mathbf{i}}}$};
  \node (Y) [right of=X] {$\mathcal{A}_{\widetilde{\Sigma}_{\mathbf{i}'}}$};
  \node (W) [below of=X] {$G^{u,v}/H$};
  \node (Z) [below of=Y] {$G^{u,v}/H$};
  \draw[->] (X) to node [above] {$\sigma$} (Y);
  \draw[->] (Y) to node [right] {$\widetilde{a}_{\mathbf{i}'}$} (Z);
  \draw[->] (X) to node [left] {$\widetilde{a}_{\mathbf{i}}$} (W);
  \draw[->] (W) to node [below] {$f_{\sigma}$} (Z);
\end{tikzpicture}
\end{center}
\end{proposition}

\begin{proof}
Firstly, we have $\overline{a}_j'=c_j$ for all $j\in[1,r]$, which implies that we have 
\begin{equation}\label{eq:4.39}
f_{\sigma}^*(\overline{a}_j')=\overline{a}_j\prod_{i=1}^rt_i^{-C_{i,j}},\quad f_{\sigma}^*(t_j')=t_j.
\end{equation}
Let us show that 
\begin{align}
(\widetilde{a}_{\mathbf{i}'}\circ\sigma)^*(t_j')&=(f_{\sigma}\circ\widetilde{a}_{\mathbf{i}})^*(t_j'),\label{eq:4.40}\quad\text{and}\\
(\widetilde{a}_{\mathbf{i}'}\circ\sigma)^*(\overline{a}_j')&=(f_{\sigma}\circ\widetilde{a}_{\mathbf{i}})^*(\overline{a}_j'),\label{eq:4.41}
\end{align}
for all $j\in[1,r]$. Firstly, by Lemma \ref{4.7}, \eqref{eq:3.15} and \eqref{eq:4.39}, we have
\begin{align*}
(\widetilde{a}_{\mathbf{i}'}\circ\sigma)^*(t_j')
&=\sigma^*(\widetilde{a}_{\mathbf{i}'}^*(t_j'))
=\sigma^*(\widetilde{A}_j'\widetilde{A}_{-j}'^{-1})
=\widetilde{A}_{-j}\widetilde{A}_j^{-1},\\
(f_{\sigma}\circ\widetilde{a}_{\mathbf{i}})^*(t_j')
&=\widetilde{a}_{\mathbf{i}}^*(f_{\sigma}^*(t_j'))
=\widetilde{a}_{\mathbf{i}}^*(t_j)
=\widetilde{A}_j^{-1}\widetilde{A}_{-j},
\end{align*}
from which we see that \eqref{eq:4.40} holds. Next, by Lemma \ref{4.7}, \eqref{eq:3.15}, \eqref{eq:3.16} and \eqref{eq:4.39}, we have
\begin{align*}
(\widetilde{a}_{\mathbf{i}'}\circ\sigma)^*(\overline{a}_j')
&=\sigma^*(\widetilde{a}_{\mathbf{i}'}^*(\overline{a}_j'))
=\sigma^*\left(\prod_{k\in\widetilde{I}}\widetilde{A}_k'^{-\widetilde{B}_{k,-j}'}\right)
=\prod_{k\in\widetilde{I}}\widetilde{A}_{-k}^{-\widetilde{B}_{-k,j}}
=\prod_{k\in\widetilde{I}}\widetilde{A}_k^{-\widetilde{B}_{k,j}},\\
(f_{\sigma}\circ\widetilde{a}_{\mathbf{i}})^*(\overline{a}_j')
&=\widetilde{a}_{\mathbf{i}}^*(f_{\sigma}^*(\overline{a}_j'))
=\widetilde{a}_{\mathbf{i}}^*\left(\overline{a}_j\prod_{i=1}^rt_i^{-C_{i,j}}\right)
=\prod_{k\in\widetilde{I}}\widetilde{A}_k^{\widetilde{B}_{k,-j}}\prod_{i=1}^r\widetilde{A}_i^{C_{i,j}}\widetilde{A}_{-i}^{-C_{i,j}}
=\prod_{k\in\widetilde{I}}\widetilde{A}_k^{\widetilde{B}_{k,j}},
\end{align*}
where the last equality follows from the fact that we have $-\widetilde{B}_{k,j}=\widetilde{B}_{k,-j}+C_{k,j}$ and $-\widetilde{B}_{-k,j}=\widetilde{B}_{-k,-j}-C_{k,j}$ for all $j,k\in[1,r]$. So this shows that \eqref{eq:4.41} holds, and we are done.
\end{proof}

\begin{proof}[Proof of Theorem \ref{4.5}]
Firstly, by Lemma \ref{4.6}, we have 
\begin{equation}\label{eq:4.42}
\tau=\tau_{|i_r|}\circ\cdots\circ\tau_{|i_1|}=\mu_{-|i_r|}\circ\sigma_{|i_r|}\circ\cdots\circ\mu_{-|i_1|}\circ\sigma_{|i_1|}=\mu_{-|i_r|}\circ\cdots\circ\mu_{-|i_1|}\circ\sigma,
\end{equation}
where $\sigma$ is the permutation of $\widetilde{I}$ given by $\sigma(i)=-i$ for all $i\in\widetilde{I}$. Next, by Proposition \ref{4.8}, the map $\sigma:\mathcal{A}_{\widetilde{\Sigma}_{\mathbf{i}}}\to\mathcal{A}_{\widetilde{\Sigma}_{\mathbf{i}'}}$ induces a map $f_{\sigma}:G^{u,v}/H\to G^{u,v}/H$, given by \eqref{eq:4.38}. Finally, by Propositions \ref{3.6} and \ref{4.2}, the composition of cluster transformations $\mu=\mu_{-|i_r|}\circ\cdots\circ\mu_{-|i_1|}$ yields a refactorization relation 
\begin{align}\label{eq:4.43}
&\quad D(t_1,\ldots,t_r)E_{i_{r+1}}(c_{i_{r+1}})\cdots E_{i_{2r}}(c_{i_{2r}})E_{i_1}(1)\cdots E_{i_r}(1)\nonumber\\
&=E_{i_1}(1)\cdots E_{i_r}(1)D(t_1',\ldots,t_r')E_{i_{r+1}}(c_{i_{r+1}}')\cdots E_{i_{2r}}(c_{i_{2r}}').
\end{align}
By composing the map $f_{\sigma}:G^{u,v}/H\to G^{u,v}/H$ given by \eqref{eq:4.38} along with the refactorization relation \eqref{eq:4.43}, we deduce from \eqref{eq:4.42} that the factorization mapping given by \eqref{eq:4.37} coincides with the generalized B\"{a}cklund-Darboux transformation $\rho_{\tau}:G^{u,v}/H\to G^{u,v}/H$ as desired.
\end{proof}

Our next step is to make an explicit connection between the factorization mapping on $G^{u,u}/H$ given by \eqref{eq:4.38}, where $u$ is any Coxeter element of $W$, and the normalized $Q$-system dynamics given by \eqref{eq:2.14}. Here, we remark that while such a connection has already been established in \cite{GSV11, Williams15}, we will need an explicit formulation, as we would like to express our Coxeter-Toda Hamiltonians $f_j^{u,u}$ and $H_j^{u,u}$, $j\in[1,r]$, in terms of the initial data $(R_{\alpha,0},R_{\alpha,1})_{\alpha\in [1,r]}$ of normalized $Q$-system variables, in the case where $G$ is of types $ABCD$, which would be the subject of interest in the next Section. While our notations and conventions will differ slightly from the notations used in \ref{Section2.5}, the content will essentially be the same.

To begin, we first recall from \eqref{eq:3.8}-\eqref{eq:3.10} that the exchange matrix corresponding to the seed $\widetilde{\Sigma}_{(u,u)}$ is given by 
\begin{equation}\label{eq:4.44}
\widetilde{B}_{(u,u)}
=\begin{pmatrix}
0 & C\\
-C & 0
\end{pmatrix},
\end{equation}
where the row and column indices are ordered as $-1,\ldots,-r,1,\ldots,r$. Let us define $\widetilde{A}_{\pm i}^{(k)}=(\tau^*)^k(\widetilde{A}_{\pm i})$ for any $k\in\mathbb{Z}$ and $i\in[1,r]$. Then by Theorem \ref{2.10}, we have $A_i^{(m+1)}=A_{-i}^{(m)}$, and 
\begin{equation}\label{eq:4.45}
A_i^{(m+1)}A_i^{(m-1)}=(A_i^{(m)})^2+\prod_{j\sim i}(A_j^{(m)})^{-C_{j,i}}
\end{equation}
for all $m\in\mathbb{Z}$ and $i\in[1,r]$. Hence, by \eqref{eq:2.14} and \eqref{eq:4.45}, there is an identification of the cluster variables $A_i^{(m)}$ with the normalized $Q$-system variables $R_{i,m}$, given by $A_i^{(m)}\mapsto R_{i,m}$, where the relationship between the finite Dynkin type $Y_r$ of the simple complex Lie group $G$ and the corresponding affine Dynkin type $X_m^{(\kappa)}$ of the $Q$-system is given by the table at the start of Section \ref{Section2.5}. With this identification, it follows from \eqref{eq:4.2}-\eqref{eq:4.3} and \eqref{eq:4.44} that we have 
\begin{align}
\widetilde{a}_{(u,v)}^*(t_j)
&=R_{j,1}R_{j,0}^{-1}\label{eq:4.46},\quad\text{and}\\
\widetilde{a}_{(u,v)}^*(c_j)
&=\prod_{k=1}^rR_{k,1}^{-C_{k,j}}\label{eq:4.47}
\end{align} 
for all $j\in[1,r]$. 

Now, as the generalized B\"{a}cklund-Darboux transformation $\rho_{\tau}:G^{u,u}/H\to G^{u,u}/H$ (or equivalently, the factorization mapping \eqref{eq:4.37}) coincides with the normalized $Q$-system evolution $R_{i,k}\mapsto R_{i,k+1}$, and the map $\rho_{\tau}:G^{u,u}/H\to G^{u,u}/H$ preserves the Coxeter-Toda Hamiltonians $H_j^{u,u}$ and $f_j^{u,u}$ for all $j\in[1,r]$ by Theorem \ref{4.4}, it follows that the Coxeter-Toda Hamiltonians $H_j^{u,u}$ and $f_j^{u,u}$ pull back to conserved quantities $C_j=\rho_{\tau}^*(H_j^{u,u})$ and $\widetilde{C}_j=\rho_{\tau}^*(f_j^{c,c})$ for the normalized $Q$-system of type $X_m^{(\kappa)}$, in the sense that we have
\begin{align}
\rho_{\tau}^*(H_j^{u,u})(R_{\alpha,k},R_{\alpha,k+1})_{\alpha\in [1,r]}&=\rho_{\tau}^*(H_j^{u,u})(R_{\alpha,k+1},R_{\alpha,k+2})_{\alpha\in [1,r]},\label{eq:4.48}\\
\rho_{\tau}^*(f_j^{u,u})(R_{\alpha,k},R_{\alpha,k+1})_{\alpha\in [1,r]}&=\rho_{\tau}^*(f_j^{u,u})(R_{\alpha,k+1},R_{\alpha,k+2})_{\alpha\in [1,r]},\label{eq:4.49}
\end{align}
or equivalently,
\begin{align}
C_j(\mathbf{y}_k)
&=\rho_{\tau}^*(H_j^{u,u})(\mathbf{y}_k)
=\rho_{\tau}^*(H_j^{u,u})(\mathbf{y}_{k+1})
=C_j(\mathbf{y}_{k+1}),\label{eq:4.50}\\
\widetilde{C}_j(\mathbf{y}_k)
&=\rho_{\tau}^*(f_j^{u,u})(\mathbf{y}_k)
=\rho_{\tau}^*(f_j^{u,u})(\mathbf{y}_{k+1})
=\widetilde{C}_j(\mathbf{y}_{k+1})\label{eq:4.51}
\end{align}
for all $j\in[1,r]$ and $k\in\mathbb{Z}$, where $\mathbf{y}_k$ is the $k$-th fundamental initial data for the (un-normalized) $X_m^{(\kappa)}$ $Q$-system defined by \eqref{eq:2.13}.
\section{Conserved quantities of \texorpdfstring{$Q$}{\textit{Q}}-systems}\label{Section5}

In this section, we will give combinatorial formulas for the Coxeter-Toda Hamiltonians arising from traces of the fundamental representations or the exterior powers of the defining representation of a simple classical complex Lie group $G$ of rank $r$, or equivalently, combinatorial formulas for the conserved quantities of $Q$-systems of types $A_r^{(1)}$, $D_{r+1}^{(2)}$, $A_{2r-1}^{(2)}$ and $D_r^{(1)}$. Here, we remark that while some of these formulas were previously given in \cite{Reshetikhin00, KM15, Williams16}, they were given in terms of functions of cluster $\mathcal{X}$-coordinates with fractional exponents. In what follows, the combinatorial formulas that we will provide for these Coxeter-Toda Hamiltonians in the subsequent subsections will be given in terms of regular functions of normalized $Q$-system variables via \eqref{eq:4.46} and \eqref{eq:4.47}.

In order to arrive at these combinatorial formulas, we would need a network formulation of our Coxeter-Toda Hamiltonians $f_k^{u,u}$ using the network formulations of elements of the quotient Coxeter double Bruhat cell $G^{u,u}/H$ developed in Section \ref{Section2.6}, where $u$ is a Coxeter element of $W$, and apply the Lindstr\"{o}m-Gessel-Viennot lemma to these networks, which we will explain in detail in Section \ref{Section5.1}. While the combinatorial formulas for these Coxeter-Toda Hamiltonians is independent of the choice of the Coxeter element $u$, it will be more convenient for us to fix $u=c$, where $c=s_1s_2\cdots s_r$, along with the unmixed double reduced word $\mathbf{i}=(-1,\ldots,-r,1,\ldots,r)$ for $(c,c)$ throughout this section. In subsequent subsections, we will refine these network formulations further for each classical Dynkin type in an amenable manner that will allow us to deduce combinatorial formulas for the conserved quantities of $Q$-systems for the aforementioned affine Dynkin types.

\subsection{Network formulations of Coxeter-Toda Hamiltonians}\label{Section5.1}

In this subsection, we will give network formulations of the Coxeter-Toda Hamiltonians $f_k^{c,c}$. Using the fact that the Hamiltonians $H_j$ can be expressed as linear combinations of $f_0,f_1,\ldots,f_j$ as follows: 
\begin{itemize}
\item If $G=SL_{r+1}(\mathbb{C})$, then $H_j=f_j$ for all $j=1,\ldots,r$,
\item If $G=SO_{2r+1}(\mathbb{C})$, then $H_j=f_j$ for all $j=1,\ldots,r-1$, 
\item If $G=Sp_{2r}(\mathbb{C})$, then $H_{2j}=\sum_{i=0}^j f_{2i}$ and $H_{2\ell+1}=\sum_{i=0}^{\ell}f_{2i+1}$ for all $j=1,\ldots,\left\lceil\frac{r-1}{2}\right\rceil$ and $\ell=0,\ldots,\left\lfloor\frac{r-1}{2}\right\rfloor$,
\item If $G=SO_{2r}(\mathbb{C})$, then $H_j=f_j$ for all $j=1,\ldots,r-2$,
\end{itemize}
it follows that we can then yield network formulations of the Coxeter-Toda Hamiltonians $H_k^{c,c}$ in the above cases where the Coxeter-Toda Hamiltonians $H_k^{c,c}$ can be expressed as linear combinations of the Hamiltonians $f_i^{c,c}$. When $G=SO_{2r}(\mathbb{C})$ or $G=SO_{2r+1}(\mathbb{C})$, we can employ similar ideas in \cite{Yang12} to yield network formulations of the Coxeter-Toda Hamiltonians arising from spin representations, though we will consider these Coxeter-Toda Hamiltonians separately at the end of this section.

To begin, we observe that there is only one elementary chip of type $E_i(a_i)$, and only one elementary chip of type $E_{-i}(b_i)$ in $\overline{N}_{c,c}(\mathbf{i})$ for all $i\in[1,r]$. Moreover, the first $r$ elementary chips of $\overline{N}_{c,c}(\mathbf{i})$ are of type $E_{-i}(b_i)$, and the last $r$ elementary chips of $\overline{N}_{c,c}(\mathbf{i})$ are of type $E_{i}(a_i)$. This implies that a path $P$ in $\overline{N}_{c,c}(\mathbf{i})$ whose start point is a source of $\overline{N}_{c,c}(\mathbf{i})$ and whose end point is a sink of $\overline{N}_{c,c}(\mathbf{i})$ is uniquely determined by the level where the start point lies, the level where the end point lies, and its lowest level. Henceforth, given a triple $(m,n,p)$ with $m,p\geq n$, we say that $(m,n,p)$ is an admissible $(c,c)$-triple, if there exists a (unique) path in $\overline{N}_{c,c}(\mathbf{i})$, which we will denote by $\overline{P}_{c,c}(m,n,p)$, whose start point lies on level $m$, whose lowest level is equal to $n$, and whose end point lies on level $p$, and we denote the set of admissible $(c,c)$-triples by $\overline{\mathcal{P}}^{c,c}$. Next, for any $m\geq n$, we say that $(m,n)$ is an admissible $(c,c)$-pair, if $(m,n,m)$ is an admissible $(c,c)$-triple, in which case we will write $P_{c,c}(m,n)$ in lieu of $\overline{P}_{c,c}(m,n,m)$, and we denote the set of admissible $(c,c)$-pairs by $\mathcal{P}^{c,c}$.

Next, we define the weight $\wt(e)$ of an edge $e$ of $\overline{N}_{c,c}(\mathbf{i})$ to be equal to the label of the edge $e$ if $e$ is labeled, and is equal to $1$ otherwise. Also, given a path $P$ in $\overline{N}_{c,c}(\mathbf{i})$ whose start point is a source vertex and whose end point is a sink vertex, we define the weight $\wt(P)$ of $P$ to be equal to the product of the weights of the edges that constitute the path $P$. Then it follows that the $(i,j)$-th entry of $\overline{g}$ is equal to the sum of the weights of the paths in $\overline{N}_{c,c}(\mathbf{i})$ with source vertex labeled $i$ and sink vertex labeled $j$.

\begin{example}\label{5.1}
Let $G=SO_5(\mathbb{C})$. From Figure \ref{Figure5.1}, we see that $(5,1,1)$ is not an admissible $(c,c)$-triple, as there are no paths in $\overline{N}_{c,c}(\mathbf{i})$ whose start point lies on level $5$, and whose valleys and end point lie on level $1$. However, $(3,2,4)$ is an admissible $(c,c)$-triple, and the path $\overline{P}_{c,c}(3,2,4)$ has weight $\wt(\overline{P}_{c,c}(3,2,4))=\sqrt{2}c_2^2t_2^2/t_1$. The pair $(5,4)$ is an admissible $(c,c)$-pair, as $(5,4,5)$ is an admissible $(c,c)$-triple, and the path $P_{c,c}(5,4)$ has weight $\wt(P_{c,c}(5,4))=c_1t_1/t_2^2$.
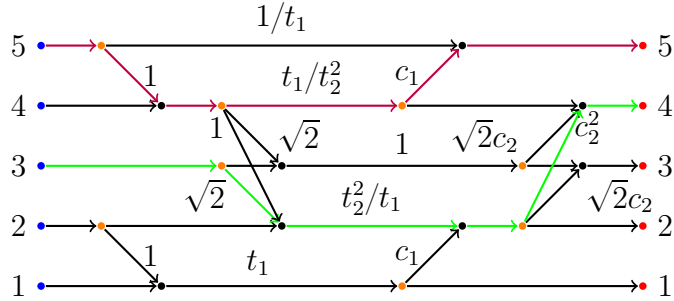
\begin{figure}[t]
\caption{The network diagram $\overline{N}_{c,c}(\mathbf{i})$, along with the highlighted paths $\overline{P}_{c,c}(3,2,4)$ and $P_{c,c}(5,4)$ in green and purple respectively.}
\label{Figure5.1}
\begin{center}
\begin{tikzpicture}[
       thick,
       acteur/.style={
         circle,
         fill=black,
         thick,
         inner sep=1pt,
         minimum size=0.1cm
       }
] 
\node (a1) at (0,0) [acteur,fill=blue]{};
\node (a2) at (0,0.8) [acteur,fill=blue]{}; 
\node (a3) at (0,1.6) [acteur,fill=blue]{}; 
\node (a4) at (0,2.4) [acteur,fill=blue]{}; 
\node (a5) at (0,3.2) [acteur,fill=blue]{};
\node (a6) at (0.8,0.8) [acteur,fill=orange]{};
\node (a7) at (1.6,0) [acteur,fill=black]{}; 
\node (a8) at (0.8,3.2) [acteur,fill=orange]{}; 
\node (a9) at (1.6,2.4) [acteur,fill=black]{}; 
\node (a10) at (2.4,1.6) [acteur,fill=orange]{};
\node (a11) at (3.2,0.8) [acteur,fill=black]{}; 
\node (a12) at (2.4,2.4) [acteur,fill=orange]{};
\node (a13) at (3.2,1.6) [acteur,fill=black]{}; 
\node (a14) at (4.8,0) [acteur,fill=orange]{}; 
\node (a15) at (5.6,0.8) [acteur,fill=black]{}; 
\node (a16) at (4.8,2.4) [acteur,fill=orange]{};
\node (a17) at (5.6,3.2) [acteur,fill=black]{}; 
\node (a18) at (6.4,0.8) [acteur,fill=orange]{}; 
\node (a19) at (7.2,1.6) [acteur,fill=black]{}; 
\node (a20) at (6.4,1.6) [acteur,fill=orange]{}; 
\node (a21) at (7.2,2.4) [acteur,fill=black]{}; 
\node (a22) at (8.0,0) [acteur,fill=red]{};
\node (a23) at (8.0,0.8) [acteur,fill=red]{}; 
\node (a24) at (8.0,1.6) [acteur,fill=red]{}; 
\node (a25) at (8.0,2.4) [acteur,fill=red]{}; 
\node (a26) at (8.0,3.2) [acteur,fill=red]{};

\node (b1) at (-0.3,0) {1};
\node (b2) at (-0.3,0.8) {2}; 
\node (b3) at (-0.3,1.6) {3}; 
\node (b4) at (-0.3,2.4) {4}; 
\node (b5) at (-0.3,3.2) {5};
\node (b6) at (8.3,0) {1};
\node (b7) at (8.3,0.8) {2}; 
\node (b8) at (8.3,1.6) {3}; 
\node (b9) at (8.3,2.4) {4}; 
\node (b10) at (8.3,3.2) {5};

\draw[->] (a1) to node {} (a7);
\draw[->] (a7) to node [above left] {$t_1$} (a14);
\draw[->] (a14) to node {} (a22);
\draw[->] (a2) to node {} (a6);
\draw[->] (a6) to node {} (a11);
\draw[green,->] (a11) to node [above] {\textcolor{black}{$t_2^2/t_1$}} (a15);
\draw[green,->] (a15) to node {} (a18);
\draw[->] (a18) to node {} (a23);
\draw[green,->] (a3) to node {} (a10);
\draw[->] (a10) to node {} (a13);
\draw[->] (a13) to node [above] {$1$} (a20);
\draw[->] (a20) to node {} (a19);
\draw[->] (a19) to node {} (a24);
\draw[->] (a4) to node {} (a9);
\draw[purple,->] (a9) to node {} (a12);
\draw[purple,->] (a12) to node [above] {\textcolor{black}{$t_1/t_2^2$}} (a16);
\draw[->] (a16) to node {} (a21);
\draw[green,->] (a21) to node {} (a25);
\draw[purple,->] (a5) to node {} (a8);
\draw[->] (a8) to node [above] {$1/t_1$} (a17);
\draw[purple,->] (a17) to node {} (a26);
\draw[->] (a6) to node [right] {$1$} (a7);
\draw[purple,->] (a8) to node [right] {\textcolor{black}{$1$}} (a9);
\draw[green,->] (a10) to node [left=0.2cm] {\textcolor{black}{$\sqrt{2}$}} (a11);
\draw[->] (a12) to node [right=0.2cm] {$\sqrt{2}$} (a13);
\draw[->] (a12) to node [above left=0.3cm] {$1$} (a11);
\draw[->] (a14) to node [left] {$c_1$} (a15);
\draw[purple,->] (a16) to node [left] {\textcolor{black}{$c_1$}} (a17);
\draw[->] (a18) to node [right=0.3cm] {$\sqrt{2}c_2$} (a19);
\draw[->] (a20) to node [left=0.3cm] {$\sqrt{2}c_2$} (a21);
\draw[green,->] (a18) to node [above right=0.2cm] {\textcolor{black}{$c_2^2$}} (a21);

\end{tikzpicture} 
\end{center}
\end{figure}
\end{example}

Given a subset $\overline{\mathcal{P}}=\{(m_1,n_1,p_1),\ldots,(m_j,n_j,p_j)\}$ of admissible $(c,c)$-triples with $m_i\neq m_{\ell}$ and $p_i\neq p_{\ell}$ for all $i\neq\ell$, we let $s(\overline{\mathcal{P}})=\{m_1,\ldots,m_j\}$, $v(\overline{\mathcal{P}})=\{n_1,\ldots,n_j\}$, $d(\overline{\mathcal{P}})=\{p_1,\ldots,p_j\}$, and 
\begin{equation*}
\wt_{c,c}(\overline{\mathcal{P}})=\prod_{i=1}^j\wt(\overline{P}_{c,c}(m_i,n_i,p_i)).
\end{equation*}
As an example, when $G=SO_5(\mathbb{C})$, and $\overline{\mathcal{P}}=\{(3,2,4),(5,4,5)\}$, it follows from Example \ref{5.1} that we have $\wt_{c,c}(\overline{\mathcal{P}})=\sqrt{2}c_1c_2^2$.

Likewise, given a subset $\mathcal{P}=\{(m_1,n_1),\ldots,(m_j,n_j)\}$ of admissible $(c,c)$-pairs with $m_i\neq m_{\ell}$ and $n_i\neq n_{\ell}$ for all $i\neq\ell$, we may also define $s(\mathcal{P})$, $v(\mathcal{P})$ and $\wt_{c,c}(\mathcal{P})$ analogously.

For any subset of $\overline{\mathcal{P}}=\{(m_1,n_1,p_1),\ldots,(m_i,n_i,p_i)\}$ of admissible $(c,c)$-triples, we say that $\overline{\mathcal{P}}$ is non-intersecting if the paths $\overline{P}_{c,c}(m_j,n_j,p_j)$ and $\overline{P}_{c,c}(m_k,n_k,p_k)$ do not share any common vertices in $\overline{N}_{c,c}(\mathbf{i})$ for any distinct $j,k\in[1,i]$. Similarly, for any subset of $\mathcal{P}=\{(m_1,n_1),\ldots,(m_i,n_i)\}$ of admissible $(c,c)$-pairs, we say that $\mathcal{P}$ is non-intersecting if the paths $P_{c,c}(m_j,n_j)$ and $P_{c,c}(m_k,n_k)$ are non-intersecting for any distinct $j,k\in[1,i]$. 

\begin{example}\label{5.2}
Let $G=SO_5(\mathbb{C})$. From Figure \ref{Figure5.2}, we see that the set $\overline{\mathcal{P}}=\{(4,2,4),(5,4,5)\}$ is intersecting, as the paths $\overline{P}_{c,c}(4,2,4)$ and $\overline{P}_{c,c}(5,4,5)$ intersect at a vertex of $\overline{N}_{c,c}(\mathbf{i})$ on level $4$ during their descents.  
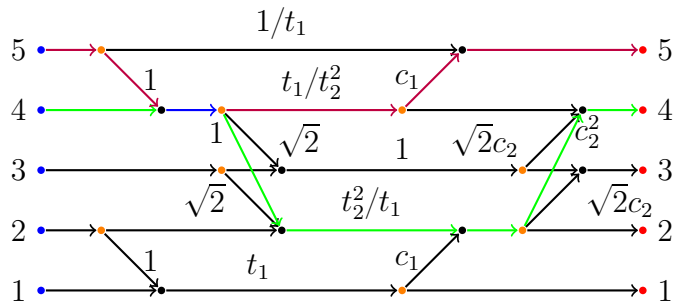
\begin{figure}[t]
\caption{The network diagram $\overline{N}_{c,c}(\mathbf{i})$, along with the intersecting paths $\overline{P}_{c,c}(4,2,4)$ and $\overline{P}_{c,c}(5,4,5)$, as colored in green and purple respectively, with the common edge colored blue.}
\label{Figure5.2}
\begin{center}
\begin{tikzpicture}[
       thick,
       acteur/.style={
         circle,
         fill=black,
         thick,
         inner sep=1pt,
         minimum size=0.1cm
       }
] 
\node (a1) at (0,0) [acteur,fill=blue]{};
\node (a2) at (0,0.8) [acteur,fill=blue]{}; 
\node (a3) at (0,1.6) [acteur,fill=blue]{}; 
\node (a4) at (0,2.4) [acteur,fill=blue]{}; 
\node (a5) at (0,3.2) [acteur,fill=blue]{};
\node (a6) at (0.8,0.8) [acteur,fill=orange]{};
\node (a7) at (1.6,0) [acteur,fill=black]{}; 
\node (a8) at (0.8,3.2) [acteur,fill=orange]{}; 
\node (a9) at (1.6,2.4) [acteur,fill=black]{}; 
\node (a10) at (2.4,1.6) [acteur,fill=orange]{};
\node (a11) at (3.2,0.8) [acteur,fill=black]{}; 
\node (a12) at (2.4,2.4) [acteur,fill=orange]{};
\node (a13) at (3.2,1.6) [acteur,fill=black]{}; 
\node (a14) at (4.8,0) [acteur,fill=orange]{}; 
\node (a15) at (5.6,0.8) [acteur,fill=black]{}; 
\node (a16) at (4.8,2.4) [acteur,fill=orange]{};
\node (a17) at (5.6,3.2) [acteur,fill=black]{}; 
\node (a18) at (6.4,0.8) [acteur,fill=orange]{}; 
\node (a19) at (7.2,1.6) [acteur,fill=black]{}; 
\node (a20) at (6.4,1.6) [acteur,fill=orange]{}; 
\node (a21) at (7.2,2.4) [acteur,fill=black]{}; 
\node (a22) at (8.0,0) [acteur,fill=red]{};
\node (a23) at (8.0,0.8) [acteur,fill=red]{}; 
\node (a24) at (8.0,1.6) [acteur,fill=red]{}; 
\node (a25) at (8.0,2.4) [acteur,fill=red]{}; 
\node (a26) at (8.0,3.2) [acteur,fill=red]{};

\node (b1) at (-0.3,0) {1};
\node (b2) at (-0.3,0.8) {2}; 
\node (b3) at (-0.3,1.6) {3}; 
\node (b4) at (-0.3,2.4) {4}; 
\node (b5) at (-0.3,3.2) {5};
\node (b6) at (8.3,0) {1};
\node (b7) at (8.3,0.8) {2}; 
\node (b8) at (8.3,1.6) {3}; 
\node (b9) at (8.3,2.4) {4}; 
\node (b10) at (8.3,3.2) {5};

\draw[->] (a1) to node {} (a7);
\draw[->] (a7) to node [above left] {$t_1$} (a14);
\draw[->] (a14) to node {} (a22);
\draw[->] (a2) to node {} (a6);
\draw[->] (a6) to node {} (a11);
\draw[green,->] (a11) to node [above] {\textcolor{black}{$t_2^2/t_1$}} (a15);
\draw[green,->] (a15) to node {} (a18);
\draw[->] (a18) to node {} (a23);
\draw[->] (a3) to node {} (a10);
\draw[->] (a10) to node {} (a13);
\draw[->] (a13) to node [above] {$1$} (a20);
\draw[->] (a20) to node {} (a19);
\draw[->] (a19) to node {} (a24);
\draw[green,->] (a4) to node {} (a9);
\draw[blue,->] (a9) to node {} (a12);
\draw[purple,->] (a12) to node [above] {\textcolor{black}{$t_1/t_2^2$}} (a16);
\draw[->] (a16) to node {} (a21);
\draw[green,->] (a21) to node {} (a25);
\draw[purple,->] (a5) to node {} (a8);
\draw[->] (a8) to node [above] {$1/t_1$} (a17);
\draw[purple,->] (a17) to node {} (a26);
\draw[->] (a6) to node [right] {$1$} (a7);
\draw[purple,->] (a8) to node [right] {\textcolor{black}{$1$}} (a9);
\draw[->] (a10) to node [left=0.2cm] {$\sqrt{2}$} (a11);
\draw[->] (a12) to node [right=0.2cm] {$\sqrt{2}$} (a13);
\draw[green,->] (a12) to node [above left=0.3cm] {\textcolor{black}{$1$}} (a11);
\draw[->] (a14) to node [left] {$c_1$} (a15);
\draw[purple,->] (a16) to node [left] {\textcolor{black}{$c_1$}} (a17);
\draw[->] (a18) to node [right=0.3cm] {$\sqrt{2}c_2$} (a19);
\draw[->] (a20) to node [left=0.3cm] {$\sqrt{2}c_2$} (a21);
\draw[green,->] (a18) to node [above right=0.2cm] {\textcolor{black}{$c_2^2$}} (a21);

\end{tikzpicture} 
\end{center}
\end{figure}
\end{example}

\begin{example}\label{5.3}
Let $G=SO_5(\mathbb{C})$. From Figure \ref{Figure5.3}, we see that the set $\overline{\mathcal{P}}=\{(3,2,4),(4,3,3)\}$ is non-intersecting, as the paths $\overline{P}_{c,c}(3,2,4)$ and $\overline{P}_{c,c}(4,3,3)$ do not share any common vertices.
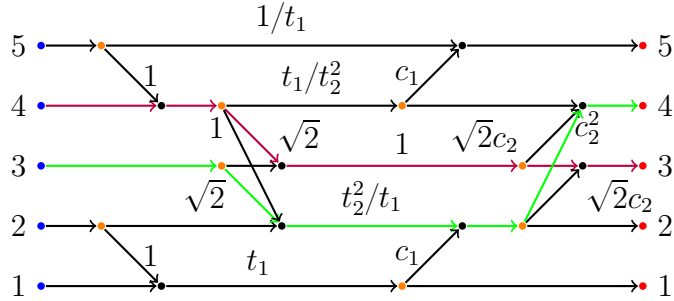
\begin{figure}[t]
\caption{The network diagram $\overline{N}_{c,c}(\mathbf{i})$, along with the non-intersecting paths $\overline{P}_{c,c}(3,2,4)$ and $\overline{P}_{c,c}(4,3,3)$, as colored in green and purple respectively.}
\label{Figure5.3}
\begin{center}
\begin{tikzpicture}[
       thick,
       acteur/.style={
         circle,
         fill=black,
         thick,
         inner sep=1pt,
         minimum size=0.1cm
       }
] 
\node (a1) at (0,0) [acteur,fill=blue]{};
\node (a2) at (0,0.8) [acteur,fill=blue]{}; 
\node (a3) at (0,1.6) [acteur,fill=blue]{}; 
\node (a4) at (0,2.4) [acteur,fill=blue]{}; 
\node (a5) at (0,3.2) [acteur,fill=blue]{};
\node (a6) at (0.8,0.8) [acteur,fill=orange]{};
\node (a7) at (1.6,0) [acteur,fill=black]{}; 
\node (a8) at (0.8,3.2) [acteur,fill=orange]{}; 
\node (a9) at (1.6,2.4) [acteur,fill=black]{}; 
\node (a10) at (2.4,1.6) [acteur,fill=orange]{};
\node (a11) at (3.2,0.8) [acteur,fill=black]{}; 
\node (a12) at (2.4,2.4) [acteur,fill=orange]{};
\node (a13) at (3.2,1.6) [acteur,fill=black]{}; 
\node (a14) at (4.8,0) [acteur,fill=orange]{}; 
\node (a15) at (5.6,0.8) [acteur,fill=black]{}; 
\node (a16) at (4.8,2.4) [acteur,fill=orange]{};
\node (a17) at (5.6,3.2) [acteur,fill=black]{}; 
\node (a18) at (6.4,0.8) [acteur,fill=orange]{}; 
\node (a19) at (7.2,1.6) [acteur,fill=black]{}; 
\node (a20) at (6.4,1.6) [acteur,fill=orange]{}; 
\node (a21) at (7.2,2.4) [acteur,fill=black]{}; 
\node (a22) at (8.0,0) [acteur,fill=red]{};
\node (a23) at (8.0,0.8) [acteur,fill=red]{}; 
\node (a24) at (8.0,1.6) [acteur,fill=red]{}; 
\node (a25) at (8.0,2.4) [acteur,fill=red]{}; 
\node (a26) at (8.0,3.2) [acteur,fill=red]{};

\node (b1) at (-0.3,0) {1};
\node (b2) at (-0.3,0.8) {2}; 
\node (b3) at (-0.3,1.6) {3}; 
\node (b4) at (-0.3,2.4) {4}; 
\node (b5) at (-0.3,3.2) {5};
\node (b6) at (8.3,0) {1};
\node (b7) at (8.3,0.8) {2}; 
\node (b8) at (8.3,1.6) {3}; 
\node (b9) at (8.3,2.4) {4}; 
\node (b10) at (8.3,3.2) {5};

\draw[->] (a1) to node {} (a7);
\draw[->] (a7) to node [above left] {$t_1$} (a14);
\draw[->] (a14) to node {} (a22);
\draw[->] (a2) to node {} (a6);
\draw[->] (a6) to node {} (a11);
\draw[green,->] (a11) to node [above] {\textcolor{black}{$t_2^2/t_1$}} (a15);
\draw[green,->] (a15) to node {} (a18);
\draw[->] (a18) to node {} (a23);
\draw[green,->] (a3) to node {} (a10);
\draw[->] (a10) to node {} (a13);
\draw[purple,->] (a13) to node [above] {\textcolor{black}{$1$}} (a20);
\draw[purple,->] (a20) to node {} (a19);
\draw[purple,->] (a19) to node {} (a24);
\draw[purple,->] (a4) to node {} (a9);
\draw[purple,->] (a9) to node {} (a12);
\draw[->] (a12) to node [above] {$t_1/t_2^2$} (a16);
\draw[->] (a16) to node {} (a21);
\draw[green,->] (a21) to node {} (a25);
\draw[->] (a5) to node {} (a8);
\draw[->] (a8) to node [above] {$1/t_1$} (a17);
\draw[->] (a17) to node {} (a26);
\draw[->] (a6) to node [right] {$1$} (a7);
\draw[->] (a8) to node [right] {$1$} (a9);
\draw[green,->] (a10) to node [left=0.2cm] {\textcolor{black}{$\sqrt{2}$}} (a11);
\draw[purple,->] (a12) to node [right=0.2cm] {\textcolor{black}{$\sqrt{2}$}} (a13);
\draw[->] (a12) to node [above left=0.3cm] {$1$} (a11);
\draw[->] (a14) to node [left] {$c_1$} (a15);
\draw[->] (a16) to node [left] {$c_1$} (a17);
\draw[->] (a18) to node [right=0.3cm] {$\sqrt{2}c_2$} (a19);
\draw[->] (a20) to node [left=0.3cm] {$\sqrt{2}c_2$} (a21);
\draw[green,->] (a18) to node [above right=0.2cm] {\textcolor{black}{$c_2^2$}} (a21);

\end{tikzpicture} 
\end{center}
\end{figure}
\end{example}

With this definition, we will define the set $\overline{\mathcal{P}}_{\nonint}^{c,c}(I)$ for all $I\subseteq[1,\dim V(\omega_1)]$ with $|I|=j$ and $j\in[1,r]$ as follows:
\begin{equation}\label{eq:5.1}
\overline{\mathcal{P}}_{\nonint}^{c,c}(I)=\{\overline{\mathcal{P}}\subseteq\overline{\mathcal{P}}^{c,c}:\overline{\mathcal{P}}\text{ is non-intersecting and }s(\overline{\mathcal{P}})=I=d(\overline{\mathcal{P}})\}.
\end{equation}
Likewise, for all $I\subseteq[1,\dim V(\omega_1)]$ with $|I|=j$ and $j\in[1,r]$, we may define $\mathcal{P}_{\nonint}^{c,c}(I)$ analogously.

Given any $\overline{\mathcal{P}}=\{(m_1,n_1,p_1),\ldots,(m_j,n_j,p_j)\}\in\overline{\mathcal{P}}_{\nonint}^{c,c}(I)$, we say that $\overline{\mathcal{P}}$ is unmixed, if $m_i=p_i$ for all $i=1,\ldots,j$, partially mixed if there exist distinct indices $i,\ell\in[1,j]$, such that $m_i=p_i$ and $m_{\ell}\neq p_{\ell}$, and fully mixed otherwise. 

\begin{example}\label{5.4}
We let $G=SO_5(\mathbb{C})$. Then from Example \ref{5.3}, the non-intersecting set $\overline{\mathcal{P}}=\{(3,2,4),(4,3,3)\}$ is fully mixed.
\end{example}

Next, for any $\overline{\mathcal{P}}\in\overline{\mathcal{P}}_{\nonint}^{c,c}(I)$, we define $\sgn(\overline{\mathcal{P}})$ by
\begin{equation*}
\sgn(\overline{\mathcal{P}})=\sgn(\sigma),
\end{equation*}
where $\sigma$ is the unique permutation in $S_{|I|}$ that satisfy $p_{\sigma(i)}=m_i$ for all $i=1,\ldots,|I|$. It is clear that $\sgn(\overline{\mathcal{P}})$ is well-defined, as $\sgn(\overline{\mathcal{P}})$ is independent of the choice of the ordering of elements of $\overline{\mathcal{P}}$. Then it follows from the Lindstr\"{o}m-Gessel-Viennot lemma that we have
\begin{equation}\label{eq:5.2}
f_j^{c,c}=\sum_{\substack{I\subseteq[1,\dim V(\omega_1)],\\|I|=j}}\sum_{\overline{\mathcal{P}}\in \overline{\mathcal{P}}_{\nonint}^{c,c}(I)}\sgn(\overline{\mathcal{P}})\wt_{c,c}(\overline{\mathcal{P}}).
\end{equation}
for all $j=1,\ldots,r$. In particular, when $j=1$, we have
\begin{equation}\label{eq:5.3}
H_1^{c,c}=f_1^{c,c}=\sum_{(m,n)\in\mathcal{P}^{c,c}}\wt(P_{c,c}(m,n)).
\end{equation}
While \eqref{eq:5.2} does give us a network formulation of the Coxeter-Toda Hamiltonian $f_j^{c,c}$, this formulation is still too coarse for us to deduce explicit combinatorial formulas for $f_j^{c,c}$. As such, our goal in subsequent Subsections \ref{Section5.2}-\ref{Section5.5} is to refine \eqref{eq:5.2}, by further analyzing the structure of non-intersecting subsets $\overline{\mathcal{P}}$ of admissible $(c,c)$-triples. This refinement will lead to a generalization of \eqref{eq:5.3} for the other Coxeter-Toda Hamiltonians $f_j^{c,c}$, $j>1$, in that the Coxeter-Toda Hamiltonians $f_j^{c,c}$, $j>1$ can be expressed as sums of weights of non-intersecting paths, each of which start and end on the same level.

\subsection{The Coxeter-Toda Hamiltonians of type \textit{A}}\label{Section5.2}

In this subsection, we will derive combinatorial formulas for the Coxeter-Toda Hamiltonians $H_j^{c,c}=f_j^{c,c}$, in the case where $G=SL_{r+1}(\mathbb{C})$. While these formulas were previously discovered by Di Francesco and Kedem \cite{DFK09-2}, the techniques that we employ to refine \eqref{eq:5.2} in this case will be instructive in refining \eqref{eq:5.2} for the other classical finite Dynkin types.

To begin, we have the following characterization of the non-intersecting subsets $\overline{\mathcal{P}}$ of admissible $(c,c)$-triples in type $A$:

\begin{lemma}\label{5.5}
Let $\overline{\mathcal{P}}=\{(m_1,n_1,p_1),\ldots,(m_j,n_j,p_j)\}$ be non-intersecting. Then $\overline{\mathcal{P}}$ is unmixed, that is, we have $m_i=p_i$ for all $i\in[1,j]$. In particular, $(m_i,n_i)$ is an admissible $(c,c)$-pair for all $i=1,\ldots,j$, and $\sgn(\overline{\mathcal{P}})=1$.
\end{lemma}

\begin{proof}
Suppose on the contrary that $\overline{\mathcal{P}}$ is not unmixed. By removing paths $\overline{P}_{c,c}(m_i,n_i,p_i)$ in $\overline{\mathcal{P}}$ that satisfy $m_i=p_i$, we may assume that $\overline{\mathcal{P}}$ is a fully mixed non-intersecting set. Without loss of generality, let us assume that $m_1<m_2<\cdots<m_j$, so that we have $m_1<p_1$. As we have $s(\overline{\mathcal{P}})=d(\overline{\mathcal{P}})$ by assumption, there must exist some $\ell>1$, such that $m_{\ell}>p_{\ell}=m_1$. 

As the elementary chip corresponding to $E_{-(m-1)}$ appears before the elementary chip corresponding to $E_{-m}$ in $\overline{N}_{c,c}(\mathbf{i})$ for any $m\in[1,r-1]$, $m_i\geq n_i$, and $n_i$ is the level on which the valleys of the path $\overline{P}_{c,c}(m_i,n_i,p_i)$ lie for all $i\in[1,j]$, it follows that we have either $n_i=m_i$ or $n_i=m_i-1$ for all $i\in[1,j]$, as shown in Figure \ref{Figure5.4}. In particular, we have $n_{\ell}\geq m_{\ell}-1$, which implies that $p_{\ell}\geq n_{\ell}\geq m_{\ell}-1\geq p_{\ell}$. Thus, equality must hold throughout, and we deduce that $m_1=p_{\ell}=n_{\ell}=m_{\ell}-1$, and so we must have $\ell=2$. Now, as we have $p_1\geq m_1+1=m_2$, this necessarily implies that the paths $\overline{P}_{c,c}(m_1,n_1,p_1)$ and $\overline{P}_{c,c}(m_2,n_2,p_2)$ must intersect at level $m_1$, as shown in Figure \ref{Figure5.5}, which contradicts the fact that $\overline{\mathcal{P}}$ is non-intersecting. So $\overline{\mathcal{P}}$ is unmixed as desired.
\end{proof}

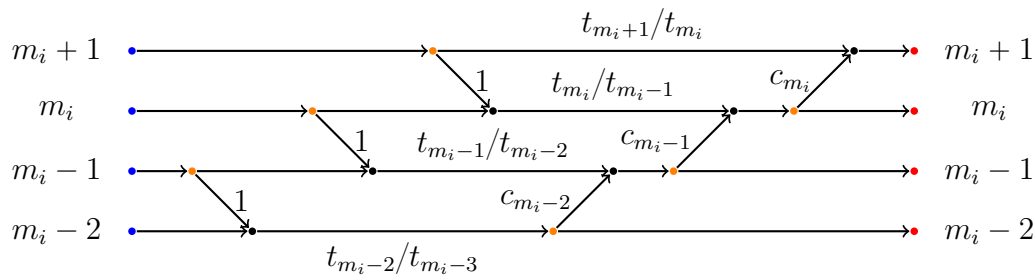
\begin{figure}[t]
\caption{The network diagram $\overline{N}_{c,c}(\mathbf{i})$ involving levels $m_i-2$, $m_i-1$, $m_i$ and $m_i+1$.}
\label{Figure5.4}
\begin{center}
\begin{tikzpicture}[
       thick,
       acteur/.style={
         circle,
         thick,
         inner sep=1pt,
         minimum size=0.1cm
       }
] 
\node (a1) at (0,0) [acteur,fill=blue]{};
\node (a2) at (0,0.8) [acteur,fill=blue]{}; 
\node (a3) at (0,1.6) [acteur,fill=blue]{}; 
\node (a4) at (0,2.4) [acteur,fill=blue]{}; 
\node (a5) at (0.8,0.8) [acteur,fill=orange]{};
\node (a6) at (1.6,0) [acteur,fill=black]{}; 
\node (a7) at (2.4,1.6) [acteur,fill=orange]{}; 
\node (a8) at (3.2,0.8) [acteur,fill=black]{}; 
\node (a9) at (4.0,2.4) [acteur,fill=orange]{};
\node (a10) at (4.8,1.6) [acteur,fill=black]{}; 
\node (a11) at (5.6,0) [acteur,fill=orange]{}; 
\node (a12) at (6.4,0.8) [acteur,fill=black]{}; 
\node (a13) at (7.2,0.8) [acteur,fill=orange]{};
\node (a14) at (8.0,1.6) [acteur,fill=black]{}; 
\node (a15) at (8.8,1.6) [acteur,fill=orange]{}; 
\node (a16) at (9.6,2.4) [acteur,fill=black]{}; 
\node (a17) at (10.4,0) [acteur,fill=red]{};
\node (a18) at (10.4,0.8) [acteur,fill=red]{}; 
\node (a19) at (10.4,1.6) [acteur,fill=red]{}; 
\node (a20) at (10.4,2.4) [acteur,fill=red]{}; 

\node (b1) at (-1,0) {$m_i-2$};
\node (b2) at (-1,0.8) {$m_i-1$}; 
\node (b3) at (-1,1.6) {$m_i$}; 
\node (b4) at (-1,2.4) {$m_i+1$}; 
\node (b5) at (11.4,0) {$m_i-2$};
\node (b6) at (11.4,0.8) {$m_i-1$}; 
\node (b7) at (11.4,1.6) {$m_i$}; 
\node (b8) at (11.4,2.4) {$m_i+1$}; 

\draw[->] (a1) to node {} (a6);
\draw[->] (a6) to node [below] {$t_{m_i-2}/t_{m_i-3}$} (a11);
\draw[->] (a11) to node {} (a17);
\draw[->] (a2) to node {} (a5);
\draw[->] (a5) to node {} (a8);
\draw[->] (a8) to node [above] {$t_{m_i-1}/t_{m_i-2}$} (a12);
\draw[->] (a12) to node {} (a13);
\draw[->] (a13) to node {} (a18);
\draw[->] (a3) to node {} (a7);
\draw[->] (a7) to node {} (a10);
\draw[->] (a10) to node [above] {$t_{m_i}/t_{m_i-1}$} (a14);
\draw[->] (a14) to node {} (a15);
\draw[->] (a15) to node {} (a19);
\draw[->] (a4) to node {} (a9);
\draw[->] (a9) to node [above] {$t_{m_i+1}/t_{m_i}$} (a16);
\draw[->] (a16) to node {} (a20);
\draw[->] (a5) to node [right] {$1$} (a6);
\draw[->] (a7) to node [right] {$1$} (a8);
\draw[->] (a9) to node [right] {$1$} (a10);
\draw[->] (a11) to node [left] {$c_{m_i-2}$} (a12);
\draw[->] (a13) to node [left] {$c_{m_i-1}$} (a14);
\draw[->] (a15) to node [left] {$c_{m_i}$} (a16);
\end{tikzpicture} 
\end{center}
\end{figure}

As a consequence of Lemma \ref{5.5}, it follows that \eqref{eq:5.2} reduces to:
\begin{equation}\label{eq:5.4}
H_j^{c,c}=f_j^{c,c}=\sum_{\substack{I\subseteq[1,r+1],\\|I|=j}}\sum_{\mathcal{P}\in\mathcal{P}_{\nonint}^{c,c}(I)}\wt_{c,c}(\mathcal{P}).
\end{equation}

\begin{figure}[t]
\caption{The paths $\overline{P}_{c,c}(m_1,n_1,p_1)$ and $\overline{P}_{c,c}(m_2,n_2,p_2)$ in $\overline{N}_{c,c}(\mathbf{i})$ involving levels $m_1-1$, $m_1$, $m_1+1=m_2$ and $p_1$, as colored in purple and green respectively, with the common edges colored blue, in the cases where $n_1=m_1$ and $n_1=m_1-1$ respectively.}
\label{Figure5.5}
\begin{center}
\begin{tikzpicture}[
       thick,
       acteur/.style={
         circle,
         thick,
         inner sep=1pt,
         minimum size=0.1cm
       }
] 
\node (a1) at (0,0) [acteur,fill=blue]{};
\node (a2) at (0,0.8) [acteur,fill=blue]{}; 
\node (a3) at (0,1.6) [acteur,fill=blue]{}; 
\node (a4) at (0.8,0.8) [acteur,fill=orange]{};
\node (a5) at (1.6,0) [acteur,fill=black]{}; 
\node (a6) at (2.4,1.6) [acteur,fill=orange]{}; 
\node (a7) at (3.2,0.8) [acteur,fill=black]{}; 
\node (a8) at (5.6,0) [acteur,fill=orange]{}; 
\node (a9) at (6.4,0.8) [acteur,fill=black]{}; 
\node (a10) at (7.2,0.8) [acteur,fill=orange]{};
\node (a11) at (8.0,1.6) [acteur,fill=black]{}; 
\node (a12) at (8.8,1.6) [acteur,fill=black]{}; 
\node (a13) at (10.4,0) [acteur,fill=red]{};
\node (a14) at (10.4,0.8) [acteur,fill=red]{}; 
\node (a15) at (10.4,1.6) [acteur,fill=red]{}; 
\node (a16) at (10.4,3.2) [acteur,fill=red]{}; 

\node (b1) at (-1.3,0) {$m_1-1$};
\node (b2) at (-1.3,0.8) {$m_1=n_2$}; 
\node (b3) at (-1.3,1.6) {$m_1+1=m_2$}; 
\node (b4) at (-1.3,3.2) {$p_1$}; 
\node (b5) at (11.7,0) {$m_1-1$};
\node (b6) at (11.7,0.8) {$m_1=n_2$}; 
\node (b7) at (11.7,1.6) {$m_1+1=m_2$}; 
\node (b8) at (11.7,3.2) {$p_1$}; 

\draw[->] (a1) to node {} (a5);
\draw[->] (a5) to node [below] {$t_{m_1-1}/t_{m_1-2}$} (a8);
\draw[->] (a8) to node {} (a13);
\draw[purple,->] (a2) to node {} (a4);
\draw[purple,->] (a4) to node {} (a7);
\draw[blue,->] (a7) to node [above] {\textcolor{black}{$t_{m_1}/t_{m_1-1}$}} (a9);
\draw[blue,->] (a9) to node {} (a10);
\draw[green,->] (a10) to node {} (a14);
\draw[green,->] (a3) to node {} (a6);
\draw[->] (a6) to node [above] {$t_{m_1+1}/t_{m_1}$} (a11);
\draw[purple,->] (a11) to node {} (a12);
\draw[->] (a12) to node {} (a15);
\draw[->] (a4) to node [right] {$1$} (a5);
\draw[green,->] (a6) to node [right] {\textcolor{black}{$1$}} (a7);
\draw[->] (a8) to node [left] {$c_{m_1-1}$} (a9);
\draw[purple,->] (a10) to node [left] {\textcolor{black}{$c_{m_1}$}} (a11);
\draw[purple,dashed,->] (a12) to node {} (a16);
\end{tikzpicture}

\begin{tikzpicture}[
       thick,
       acteur/.style={
         circle,
         thick,
         inner sep=1pt,
         minimum size=0.1cm
       }
] 
\node (a1) at (0,0) [acteur,fill=blue]{};
\node (a2) at (0,0.8) [acteur,fill=blue]{}; 
\node (a3) at (0,1.6) [acteur,fill=blue]{}; 
\node (a4) at (0.8,0.8) [acteur,fill=orange]{};
\node (a5) at (1.6,0) [acteur,fill=black]{}; 
\node (a6) at (2.4,1.6) [acteur,fill=orange]{}; 
\node (a7) at (3.2,0.8) [acteur,fill=black]{}; 
\node (a8) at (5.6,0) [acteur,fill=orange]{}; 
\node (a9) at (6.4,0.8) [acteur,fill=black]{}; 
\node (a10) at (7.2,0.8) [acteur,fill=orange]{};
\node (a11) at (8.0,1.6) [acteur,fill=black]{}; 
\node (a12) at (8.8,1.6) [acteur,fill=black]{}; 
\node (a13) at (10.4,0) [acteur,fill=red]{};
\node (a14) at (10.4,0.8) [acteur,fill=red]{}; 
\node (a15) at (10.4,1.6) [acteur,fill=red]{}; 
\node (a16) at (10.4,3.2) [acteur,fill=red]{}; 

\node (b1) at (-1.3,0) {$m_1-1$};
\node (b2) at (-1.3,0.8) {$m_1=n_2$}; 
\node (b3) at (-1.3,1.6) {$m_1+1=m_2$}; 
\node (b4) at (-1.3,3.2) {$p_1$}; 
\node (b5) at (11.7,0) {$m_1-1$};
\node (b6) at (11.7,0.8) {$m_1=n_2$}; 
\node (b7) at (11.7,1.6) {$m_1+1=m_2$}; 
\node (b8) at (11.7,3.2) {$p_1$}; 

\draw[->] (a1) to node {} (a5);
\draw[purple,->] (a5) to node [below] {\textcolor{black}{$t_{m_1-1}/t_{m_1-2}$}} (a8);
\draw[->] (a8) to node {} (a13);
\draw[purple,->] (a2) to node {} (a4);
\draw[->] (a4) to node {} (a7);
\draw[green,->] (a7) to node [above] {\textcolor{black}{$t_{m_1}/t_{m_1-1}$}} (a9);
\draw[blue,->] (a9) to node {} (a10);
\draw[green,->] (a10) to node {} (a14);
\draw[green,->] (a3) to node {} (a6);
\draw[->] (a6) to node [above] {$t_{m_1+1}/t_{m_1}$} (a11);
\draw[purple,->] (a11) to node {} (a12);
\draw[->] (a12) to node {} (a15);
\draw[purple,->] (a4) to node [right] {\textcolor{black}{$1$}} (a5);
\draw[green,->] (a6) to node [right] {\textcolor{black}{$1$}} (a7);
\draw[purple,->] (a8) to node [left] {\textcolor{black}{$c_{m_1-1}$}} (a9);
\draw[purple,->] (a10) to node [left] {\textcolor{black}{$c_{m_1}$}} (a11);
\draw[purple,dashed,->] (a12) to node {} (a16);
\end{tikzpicture} 
\end{center}
\end{figure}
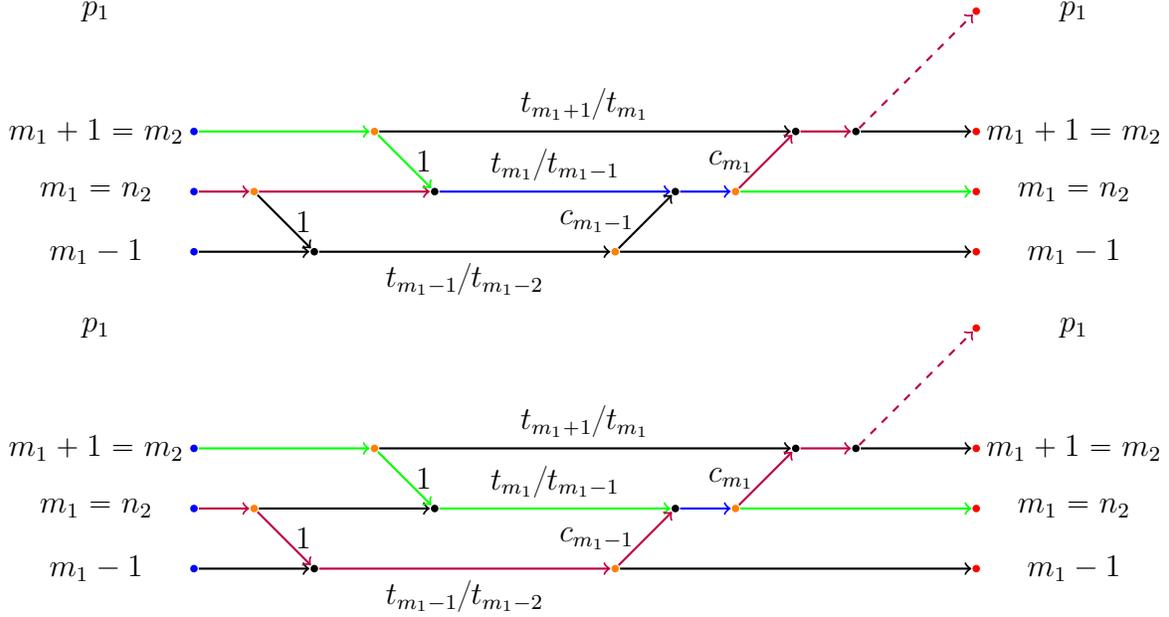

Thus, our next step is to characterize the admissible $(c,c)$-pairs $(m,n)$, as well as when two paths $P_{c,c}(m,n)$ and $P_{c,c}(m',n')$ corresponding to two different admissible $(c,c)$-pairs $(m,n)$ and $(m',n')$ intersect.

\begin{proposition}\label{5.6}
~
\begin{enumerate}
\item The admissible $(c,c)$-pairs are of the form $(i,i)$ and $(j+1,j)$, where $i\in[1,r+1]$ and $j\in[1,r]$. 
\item For all $i\in[1,r+1]$ and $j\in[1,r]$, the weights of the paths $P_{2i-1}:=P_{c,c}(i,i)$ and $P_{2j}:=P_{c,c}(j+1,j)$ are given by
\begin{equation}\label{eq:5.5}
\wt(P_{2i-1})=\frac{t_i}{t_{i-1}},\quad\wt(P_{2j})=\frac{c_jt_j}{t_{j-1}},
\end{equation}
where $t_0=t_{r+1}=1$. In particular, these path weights $y_i:=\rho_{\tau}^*(\wt(P_i))$, written in terms of normalized $A_r^{(1)}$ $Q$-system variables using \eqref{eq:4.46} and \eqref{eq:4.47}, are given by
\begin{equation}\label{eq:5.6}
y_{2i-1}=\frac{R_{i,1}R_{i-1,0}}{R_{i,0}R_{i-1,1}},\quad
y_{2i}=\frac{R_{j-1,0}R_{j+1,1}}{R_{j,0}R_{j,1}},
\end{equation}
where $R_{0,k}=R_{r+1,k}=1$ for all $k\in\mathbb{Z}$.
\item The paths $P_{2i-1}=P_{c,c}(i,i)$ and $P_{2i'-1}=P_{c,c}(i',i')$ do not intersect for any distinct $i,i'\in[1,r+1]$, and for any $j\in[1,r]$ and admissible $(c,c)$-pair $(m,n)\neq(j+1,j)$, the path $P_{2j}=P_{c,c}(j+1,j)$ intersects $P_k=P_{c,c}(m,n)$ if and only if $(m,n)=(j,j)$, $(j+1,j+1),(j,j-1)$ or $(j+2,j+1)$, or equivalently, $k=2j\pm1,2j\pm2$.
\end{enumerate}
\end{proposition} 

\begin{proof}
By the proof of Lemma \ref{5.5}, we have that if $(m,n)$ is an admissible $(c,c)$-pair, then either $m=n$ or $m=n+1$, which proves statement (1). Statement (2) follows from the definition of path weights, and statement (3) follow from observation using Figure \ref{Figure5.6}.

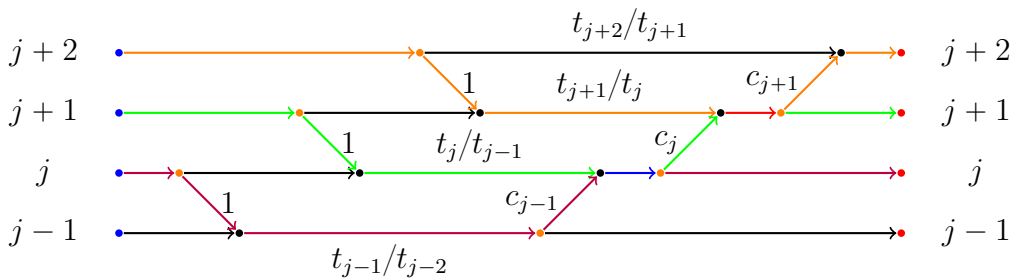
\begin{figure}[t]
\caption{The paths $P_{c,c}(j+1,j)$, $P_{c,c}(j,j-1)$ and $P_{c,c}(j+2,j+1)$ in the network diagram $\overline{N}_{c,c}(\mathbf{i})$ as colored in green, purple and orange respectively. The common edge in the paths $P_{c,c}(j+1,j)$ and $P_{c,c}(j,j-1)$ is colored blue, and the common edge in the paths $P_{c,c}(j+1,j)$ and $P_{c,c}(j+2,j+1)$ is colored red.}
\label{Figure5.6}
\begin{center}
\begin{tikzpicture}[
       thick,
       acteur/.style={
         circle,
         thick,
         inner sep=1pt,
         minimum size=0.1cm
       }
] 
\node (a1) at (0,0) [acteur,fill=blue]{};
\node (a2) at (0,0.8) [acteur,fill=blue]{}; 
\node (a3) at (0,1.6) [acteur,fill=blue]{}; 
\node (a4) at (0,2.4) [acteur,fill=blue]{}; 
\node (a5) at (0.8,0.8) [acteur,fill=orange]{};
\node (a6) at (1.6,0) [acteur,fill=black]{}; 
\node (a7) at (2.4,1.6) [acteur,fill=orange]{}; 
\node (a8) at (3.2,0.8) [acteur,fill=black]{}; 
\node (a9) at (4.0,2.4) [acteur,fill=orange]{};
\node (a10) at (4.8,1.6) [acteur,fill=black]{}; 
\node (a11) at (5.6,0) [acteur,fill=orange]{}; 
\node (a12) at (6.4,0.8) [acteur,fill=black]{}; 
\node (a13) at (7.2,0.8) [acteur,fill=orange]{};
\node (a14) at (8.0,1.6) [acteur,fill=black]{}; 
\node (a15) at (8.8,1.6) [acteur,fill=orange]{}; 
\node (a16) at (9.6,2.4) [acteur,fill=black]{}; 
\node (a17) at (10.4,0) [acteur,fill=red]{};
\node (a18) at (10.4,0.8) [acteur,fill=red]{}; 
\node (a19) at (10.4,1.6) [acteur,fill=red]{}; 
\node (a20) at (10.4,2.4) [acteur,fill=red]{}; 

\node (b1) at (-1,0) {$j-1$};
\node (b2) at (-1,0.8) {$j$}; 
\node (b3) at (-1,1.6) {$j+1$}; 
\node (b4) at (-1,2.4) {$j+2$}; 
\node (b5) at (11.4,0) {$j-1$};
\node (b6) at (11.4,0.8) {$j$}; 
\node (b7) at (11.4,1.6) {$j+1$}; 
\node (b8) at (11.4,2.4) {$j+2$}; 

\draw[->] (a1) to node {} (a6);
\draw[purple,->] (a6) to node [below] {\textcolor{black}{$t_{j-1}/t_{j-2}$}} (a11);
\draw[->] (a11) to node {} (a17);
\draw[purple,->] (a2) to node {} (a5);
\draw[->] (a5) to node {} (a8);
\draw[green,->] (a8) to node [above] {\textcolor{black}{$t_j/t_{j-1}$}} (a12);
\draw[blue,->] (a12) to node {} (a13);
\draw[purple,->] (a13) to node {} (a18);
\draw[green,->] (a3) to node {} (a7);
\draw[->] (a7) to node {} (a10);
\draw[orange,->] (a10) to node [above] {\textcolor{black}{$t_{j+1}/t_j$}} (a14);
\draw[red,->] (a14) to node {} (a15);
\draw[green,->] (a15) to node {} (a19);
\draw[orange,->] (a4) to node {} (a9);
\draw[->] (a9) to node [above] {$t_{j+2}/t_{j+1}$} (a16);
\draw[orange,->] (a16) to node {} (a20);
\draw[purple,->] (a5) to node [right] {\textcolor{black}{$1$}} (a6);
\draw[green,->] (a7) to node [right] {\textcolor{black}{$1$}} (a8);
\draw[orange,->] (a9) to node [right] {\textcolor{black}{$1$}} (a10);
\draw[purple,->] (a11) to node [left] {\textcolor{black}{$c_{j-1}$}} (a12);
\draw[green,->] (a13) to node [left] {\textcolor{black}{$c_j$}} (a14);
\draw[orange,->] (a15) to node [left] {\textcolor{black}{$c_{j+1}$}} (a16);
\end{tikzpicture} 
\end{center}
\end{figure}
\end{proof}

\begin{remark}
In particular, the formulas for the weights $y_{2i-1}$ and $y_{2j}$ given in \eqref{eq:5.6} coincide with \cite[(3.9)]{DFK09-2}.
\end{remark}

The non-intersecting condition given in statement (3) of Proposition \ref{5.6} can be described in terms of adjacency of vertices on a graph $\mathcal{G}_A$, where the graph $\mathcal{G}_A$ is given as follows \cite[Section 3.3]{DFK09-2}:
\begin{center}
\begin{tikzpicture}[
       thick,
       acteur/.style={
         circle,
         thick,
         inner sep=1pt,
         minimum size=0.1cm
       }
] 

\node (a1) at (0,0) [acteur,fill=red]{};
\node (a2) at (2,0) [acteur,fill=blue]{};
\node (a3) at (3,1) [acteur,fill=red]{};
\node (a4) at (4,0) [acteur,fill=blue]{};
\node (a5) at (5,1) [acteur,fill=red]{};
\node (a6) at (6,0) [acteur,fill=blue]{};
\node (d1) at (6.5,0) [acteur,fill=black]{};
\node (d2) at (7,0) [acteur,fill=black]{};
\node (d3) at (7.5,0) [acteur,fill=black]{};
\node (a2r-2) at (8,0) [acteur,fill=blue]{};
\node (a2r-1) at (9,1) [acteur,fill=red]{};
\node (a2r) at (10,0) [acteur,fill=blue]{};
\node (a2r+1) at (12,0) [acteur,fill=red]{};

\node (b1) at (0,-0.2) [below] {$1$};
\node (b2) [above] at (2,-0.2) [below] {$2$};
\node (b3) at (3,1.2) [above] {$3$};
\node (b4) at (4,-0.2) [below] {$4$};
\node (b5) at (5,1.2) [above] {$5$};
\node (b6) at (6,-0.2) [below] {$6$};
\node (b2r-2) [above] at (8,-0.2) [below] {$2r-2$};
\node (b2r-1) [below] at (9,1.2) [above] {$2r-1$};
\node (b2r) at (10,-0.2) [below] {$2r$};
\node (b2r+1) at (12,-0.2) [below] {$2r+1$};

\draw[-] (a1) to node {} (a2);
\draw[-] (a2) to node {} (a3);
\draw[-] (a2) to node {} (a4);
\draw[-] (a3) to node {} (a4);
\draw[-] (a4) to node {} (a5);
\draw[-] (a4) to node {} (a6);
\draw[-] (a5) to node {} (a6);
\draw[-] (a2r-2) to node {} (a2r-1);
\draw[-] (a2r-2) to node {} (a2r);
\draw[-] (a2r-1) to node {} (a2r);
\draw[-] (a2r) to node {} (a2r+1);

\end{tikzpicture} 
\end{center}

With the undirected graph $\mathcal{G}_A$ given as above, it follows that for all distinct $j,k\in[1,2r+1]$, the paths $P_j$ and $P_k$ are non-intersecting if and only if vertices $j$ and $k$ are not adjacent to each other in $\mathcal{G}_A$. 

\begin{definition}\label{5.7}
A hard particle configuration $C$ on $\mathcal{G}_A$ is a subset of $[1,2r+1]$, such that if $j,k\in C$, then the vertices $j$ and $k$ are not adjacent to each other in $\mathcal{G}_A$. We denote the weight $\wt(C)$ of the hard particle configuration $C$ is defined by $\wt(C)=\prod_{i\in C}y_i$, and we denote the set of hard particle configurations on $\mathcal{G}_A$ by $\HPC(\mathcal{G}_A)$.
\end{definition}

By \eqref{eq:5.4}, Lemma \ref{5.5}, Proposition \ref{5.6} and Definition \ref{5.7}, we recover the following formula for the $j$-th conserved quantity $C_j=\rho_{\tau}^*(H_j^{c,c})$ of the normalized $A_r^{(1)}$ $Q$-system, or equivalently, the $j$-th Coxeter-Toda Hamiltonian, written in terms of the normalized $A_r^{(1)}$ $Q$-system variables \cite[Theorem 3.10]{DFK09-2}:

\begin{theorem}\label{5.8}
Let $j\in[1,r]$. Then the $j$-th conserved quantity $C_j=\rho_{\tau}^*(H_j^{c,c})$ of the normalized $A_r^{(1)}$ $Q$-system is given by
\begin{equation}\label{eq:5.7}
C_j=\sum_{\substack{C\in\HPC(\mathcal{G}_A)\\ |C|=j}}\wt(C).
\end{equation}
More precisely, we have
\begin{equation}\label{eq:5.8}
C_j=\sum_{\substack{1\leq i_1<i_2<\cdots<i_j\leq 2r+1\\i_k\leq i_{k+1}-2\text{ if }k\text{ is }odd,\\i_k\leq i_{k+1}-3\text{ if }k\text{ is }even}}y_{i_1}\cdots y_{i_j}.
\end{equation}
\end{theorem}

\subsection{The Coxeter-Toda Hamiltonians of type \textit{C}}\label{Section5.3}

In this subsection, we will derive combinatorial formulas for the Coxeter-Toda Hamiltonians $f_j^{c,c}$, in the case where $G=Sp_{2r}(\mathbb{C})$. Using the fact that $H_{2j}=\sum_{i=0}^j f_{2i}$ and $H_{2\ell+1}=\sum_{i=0}^{\ell}f_{2i+1}$ for all $j=1,\ldots,\left\lceil\frac{r-1}{2}\right\rceil$ and $\ell=0,\ldots,\left\lfloor\frac{r-1}{2}\right\rfloor$, we will then have combinatorial formulas for the Coxeter-Toda Hamiltonians $H_j^{c,c}$ in this case, which we will omit here.

To begin, we have the following characterization of the non-intersecting subsets $\overline{\mathcal{P}}$ of admissible $(c,c)$-triples in type $C$:

\begin{lemma}\label{5.9}
Let $\overline{\mathcal{P}}=\{(m_1,n_1,p_1),\ldots,(m_j,n_j,p_j)\}$ be non-intersecting. Then $\overline{\mathcal{P}}$ is unmixed. In particular, $(m_i,n_i)$ is an admissible $(c,c)$-pair for all $i=1,\ldots,j$, and $\sgn(\overline{\mathcal{P}})=1$.
\end{lemma}

\begin{proof}
Suppose on the contrary that $\overline{\mathcal{P}}$ is not unmixed. By removing paths $\overline{P}_{c,c}(m_i,n_i,p_i)$ in $\overline{\mathcal{P}}$ that satisfy $m_i=p_i$, we may assume that $\overline{\mathcal{P}}$ is a fully mixed non-intersecting set. Without loss of generality, let us assume that $m_1<m_2<\cdots<m_j$, so that we have $m_1<p_1$. As we have $s(\overline{\mathcal{P}})=d(\overline{\mathcal{P}})$ by assumption, there must exist some $\ell>1$, such that $m_{\ell}>p_{\ell}=m_1$. We shall arrive at a contradiction in steps.

Firstly, we assert that $s(\overline{\mathcal{P}})=d(\overline{\mathcal{P}})\subseteq[r+1,2r]$. Suppose on the contrary that this is not the case. Then by the minimality of $m_1$, we must have $n_1\leq m_1\leq r$. As the elementary chip corresponding to $E_{-(m-1)}$ appears before the elementary chip corresponding to $E_{-m}$ in $\overline{N}_{c,c}(\mathbf{i})$ for any $m\in[1,r-1]$, $m_i\geq n_i$, and $n_i$ is the level on which the valleys of the path $\overline{P}_{c,c}(m_i,n_i,p_i)$ lie for all $i\in[1,j]$, it follows from Figure \ref{Figure5.7} that if $m_{\ell}\geq r+1$, then we must have $n_{\ell}\geq r$, and if $m_{\ell}\leq r$, then we must have either $n_{\ell}=m_{\ell}$ or $n_{\ell}=m_{\ell}-1$.

\begin{figure}[t]
\caption{The network diagram $\overline{N}_{c,c}(\mathbf{i})$ involving levels $1$, $2$, $3$, $r$, $r+1$, $2r-2$, $2r-1$ and $2r$.}
\label{Figure5.7}
\begin{center}
\begin{tikzpicture}[
       thick,
       acteur/.style={
         circle,
         thick,
         inner sep=1pt,
         minimum size=0.1cm
       }
] 
\node (a1) at (0,0) [acteur,fill=blue]{};
\node (a2) at (0,0.8) [acteur,fill=blue]{}; 
\node (a3) at (0,1.6) [acteur,fill=blue]{}; 
\node (a4) at (0,2.4) [acteur,fill=blue]{}; 
\node (a5) at (0,3.2) [acteur,fill=blue]{};
\node (a6) at (0,4.0) [acteur,fill=blue]{}; 
\node (a7) at (0,4.8) [acteur,fill=blue]{};
\node (a8) at (0,5.6) [acteur,fill=blue]{}; 
\node (a9) at (0.8,5.6) [acteur,fill=orange]{}; 
\node (a10) at (1.6,4.8) [acteur,fill=black]{}; 
\node (a11) at (0.8,0.8) [acteur,fill=orange]{}; 
\node (a12) at (1.6,0) [acteur,fill=black]{}; 
\node (a13) at (2.4,4.8) [acteur,fill=orange]{}; 
\node (a14) at (3.2,4.0) [acteur,fill=black]{}; 
\node (a15) at (2.4,1.6) [acteur,fill=orange]{}; 
\node (a16) at (3.2,0.8) [acteur,fill=black]{}; 
\node (a17) at (4.0,3.2) [acteur,fill=orange]{}; 
\node (a18) at (4.8,2.4) [acteur,fill=black]{}; 
\node (a19) at (5.6,4.8) [acteur,fill=orange]{}; 
\node (a20) at (6.4,5.6) [acteur,fill=black]{}; 
\node (a21) at (5.6,0) [acteur,fill=orange]{}; 
\node (a22) at (6.4,0.8) [acteur,fill=black]{}; 
\node (a23) at (7.2,4.0) [acteur,fill=orange]{}; 
\node (a24) at (8.0,4.8) [acteur,fill=black]{}; 
\node (a25) at (7.2,0.8) [acteur,fill=orange]{}; 
\node (a26) at (8.0,1.6) [acteur,fill=black]{}; 
\node (a27) at (8.8,2.4) [acteur,fill=orange]{}; 
\node (a28) at (9.6,3.2) [acteur,fill=black]{}; 
\node (a29) at (10.4,0) [acteur,fill=red]{};
\node (a30) at (10.4,0.8) [acteur,fill=red]{}; 
\node (a31) at (10.4,1.6) [acteur,fill=red]{}; 
\node (a32) at (10.4,2.4) [acteur,fill=red]{}; 
\node (a33) at (10.4,3.2) [acteur,fill=red]{};
\node (a34) at (10.4,4.0) [acteur,fill=red]{}; 
\node (a35) at (10.4,4.8) [acteur,fill=red]{};
\node (a36) at (10.4,5.6) [acteur,fill=red]{};

\node (b1) at (-0.6,0) {$1$};
\node (b2) at (-0.6,0.8) {$2$}; 
\node (b3) at (-0.6,1.6) {$3$}; 
\node (b4) at (-0.6,2.4) {$r$}; 
\node (b5) at (-0.6,3.2) {$r+1$};
\node (b6) at (-0.6,4.0) {$2r-2$}; 
\node (b7) at (-0.6,4.8) {$2r-1$}; 
\node (b8) at (-0.6,5.6) {$2r$}; 
\node (b9) at (11,0) {$1$};
\node (b10) at (11,0.8) {$2$}; 
\node (b11) at (11,1.6) {$3$}; 
\node (b12) at (11,2.4) {$r$}; 
\node (b13) at (11,3.2) {$r+1$};
\node (b14) at (11,4.0) {$2r-2$}; 
\node (b15) at (11,4.8) {$2r-1$}; 
\node (b16) at (11,5.6) {$2r$}; 

\node (c1) at (3.6,1.8) [acteur,fill=black]{};
\node (c2) at (3.6,2.0) [acteur,fill=black]{};
\node (c3) at (3.6,2.2) [acteur,fill=black]{};
\node (c4) at (3.6,3.4) [acteur,fill=black]{};
\node (c5) at (3.6,3.6) [acteur,fill=black]{};
\node (c6) at (3.6,3.8) [acteur,fill=black]{};
\node (c7) at (8.4,1.8) [acteur,fill=black]{};
\node (c8) at (8.4,2.0) [acteur,fill=black]{};
\node (c9) at (8.4,2.2) [acteur,fill=black]{};
\node (c10) at (8.4,3.4) [acteur,fill=black]{};
\node (c11) at (8.4,3.6) [acteur,fill=black]{};
\node (c12) at (8.4,3.8) [acteur,fill=black]{};

\draw[->] (a1) to node {} (a12);
\draw[->] (a12) to node [below] {$t_1$} (a21);
\draw[->] (a21) to node {} (a29);
\draw[->] (a2) to node {} (a11);
\draw[->] (a11) to node {} (a16);
\draw[->] (a16) to node [below] {$t_2/t_1$} (a22);
\draw[->] (a22) to node {} (a25);
\draw[->] (a25) to node {} (a30);
\draw[->] (a3) to node {} (a15);
\draw[->] (a15) to node [below] {$t_3/t_2$} (a26);
\draw[->] (a26) to node {} (a31);
\draw[->] (a4) to node {} (a18);
\draw[->] (a18) to node [below] {$t_r/t_{r-1}$} (a27);
\draw[->] (a27) to node {} (a32);
\draw[->] (a5) to node {} (a17);
\draw[->] (a17) to node [above] {$t_{r-1}/t_r$} (a28);
\draw[->] (a28) to node {} (a33);
\draw[->] (a6) to node {} (a14);
\draw[->] (a14) to node [above] {$t_2/t_3$} (a23);
\draw[->] (a23) to node {} (a34);
\draw[->] (a7) to node {} (a10);
\draw[->] (a10) to node {} (a13);
\draw[->] (a13) to node [above] {$t_1/t_2$} (a19);
\draw[->] (a19) to node {} (a24);
\draw[->] (a24) to node {} (a35);
\draw[->] (a8) to node {} (a9);
\draw[->] (a9) to node [above] {$1/t_1$} (a20);
\draw[->] (a20) to node {} (a36);
\draw[->] (a9) to node [left] {$1$} (a10);
\draw[->] (a11) to node [left] {$1$} (a12);
\draw[->] (a13) to node [left] {$1$} (a14);
\draw[->] (a15) to node [left] {$1$} (a16);
\draw[->] (a17) to node [left] {$1$} (a18);
\draw[->] (a19) to node [right] {$c_1$} (a20);
\draw[->] (a21) to node [right] {$c_1$} (a22);
\draw[->] (a23) to node [right] {$c_2$} (a24);
\draw[->] (a25) to node [right] {$c_2$} (a26);
\draw[->] (a27) to node [right] {$c_r$} (a28);

\end{tikzpicture} 
\end{center}
\end{figure}
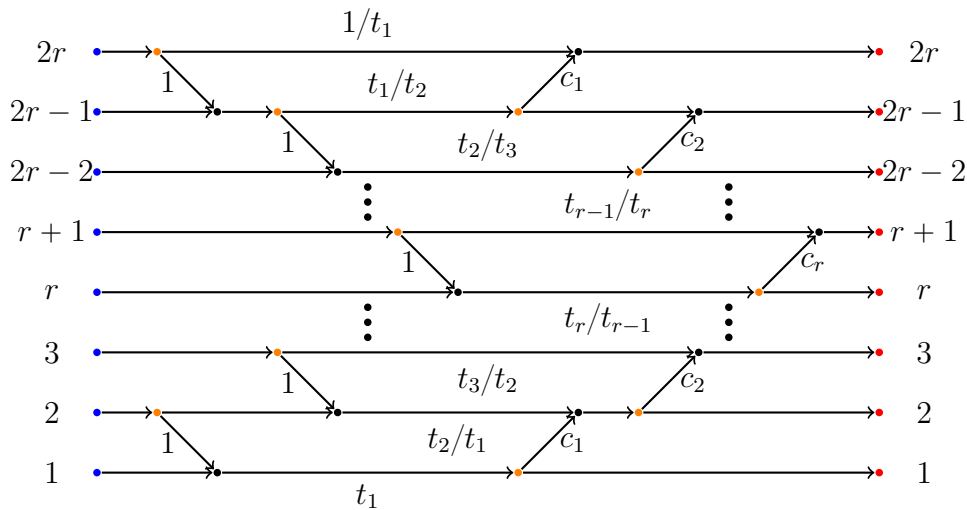

Now, if $m_{\ell}\geq r+1$ and $n_{\ell}\neq r$, then this would necessarily imply that we have $p_{\ell}\geq n_{\ell}\geq r+1>r\geq m_1$, which contradicts the fact that $p_{\ell}=m_1$. Therefore, we must have $n_{\ell}=r$. As we have $p_{\ell}\geq n_{\ell}=r\geq m_1$, we must have equality to hold throughout, that is, we have $m_1=p_{\ell}=n_{\ell}=r$. As $m_1\leq r$, we must have either $n_1=m_1$ or $n_1=m_1-1$, and by a similar argument as in the proof of Lemma \ref{5.5}, it follows that the paths $\overline{P}_{c,c}(m_1,n_1,p_1)$ and $\overline{P}_{c,c}(m_{\ell},n_{\ell},p_{\ell})$ must intersect at level $m_1=r$, which is a contradiction.

On the other hand, if $m_{\ell}\leq r$, then we must have $n_{\ell}\geq m_{\ell}-1$, which implies that $p_{\ell}\geq n_{\ell}\geq m_{\ell}-1\geq p_{\ell}$. Thus, equality must hold throughout, and we deduce that $m_1=p_{\ell}=n_{\ell}=m_{\ell}-1$, and so we must have $\ell=2$. Now, as we have $p_1\geq m_1+1=m_2$, it follows from a similar argument as in the proof of Lemma \ref{5.5} that the paths $\overline{P}_{c,c}(m_1,n_1,p_1)$ and $\overline{P}_{c,c}(m_2,n_2,p_2)$ must intersect at level $m_1$, which is a contradiction. Therefore, we must have $s(\overline{\mathcal{P}})=d(\overline{\mathcal{P}})\subseteq[r+1,2r]$ as claimed.

Next, we let $\sigma$ be the unique permutation in $S_j$ that satisfy $p_{\sigma(1)}<p_{\sigma(2)}<\cdots<p_{\sigma(j)}$, so that $p_{\sigma(1)}<m_{\sigma(1)}$, and we let $k>1$ be the unique index satisfying $p_{\sigma(k)}>m_{\sigma(k)}=p_{\sigma(1)}$. As the elementary chip corresponding to $E_{m-1}(c_{m-1})$ appears before the elementary chip corresponding to $E_m(c_m)$ in $\overline{N}_{c,c}(\mathbf{i})$ for any $m\in[1,r-1]$, and we have $p_i\geq n_i$ for all $i\in[1,j]$, it follows from Figure \ref{Figure5.8} that we must have either $n_i=p_i$ or $n_i=p_i-1$ for all $i\in[1,j]$. In particular, we have $n_{\sigma(k)}\geq p_{\sigma(k)}-1$, which implies that $m_{\sigma(k)}\geq n_{\sigma(k)}\geq p_{\sigma(k)}-1\geq m_{\sigma(k)}$. Thus, equality must hold throughout, and we deduce that $p_{\sigma(1)}=m_{\sigma(k)}=n_{\sigma(k)}=p_{\sigma(k)}-1$, and so this forces $k=2$. 

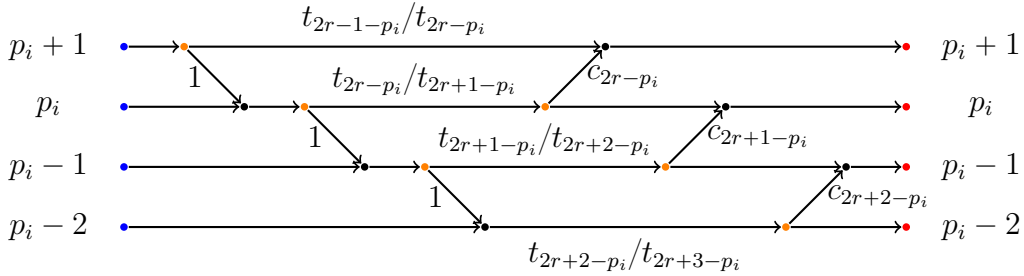
\begin{figure}[t]
\caption{The network diagram $\overline{N}_{c,c}(\mathbf{i})$ involving levels $p_i-2$, $p_i-1$, $p_i$ and $p_i+1$.}
\label{Figure5.8}
\begin{center}
\begin{tikzpicture}[
       thick,
       acteur/.style={
         circle,
         thick,
         inner sep=1pt,
         minimum size=0.1cm
       }
] 
\node (a1) at (0,0) [acteur,fill=blue]{};
\node (a2) at (0,0.8) [acteur,fill=blue]{}; 
\node (a3) at (0,1.6) [acteur,fill=blue]{}; 
\node (a4) at (0,2.4) [acteur,fill=blue]{}; 
\node (a5) at (0.8,2.4) [acteur,fill=orange]{};
\node (a6) at (1.6,1.6) [acteur,fill=black]{}; 
\node (a7) at (2.4,1.6) [acteur,fill=orange]{}; 
\node (a8) at (3.2,0.8) [acteur,fill=black]{}; 
\node (a9) at (4.0,0.8) [acteur,fill=orange]{};
\node (a10) at (4.8,0) [acteur,fill=black]{}; 
\node (a11) at (5.6,1.6) [acteur,fill=orange]{}; 
\node (a12) at (6.4,2.4) [acteur,fill=black]{}; 
\node (a13) at (7.2,0.8) [acteur,fill=orange]{};
\node (a14) at (8.0,1.6) [acteur,fill=black]{}; 
\node (a15) at (8.8,0) [acteur,fill=orange]{}; 
\node (a16) at (9.6,0.8) [acteur,fill=black]{}; 
\node (a17) at (10.4,0) [acteur,fill=red]{};
\node (a18) at (10.4,0.8) [acteur,fill=red]{}; 
\node (a19) at (10.4,1.6) [acteur,fill=red]{}; 
\node (a20) at (10.4,2.4) [acteur,fill=red]{}; 

\node (b1) at (-1,0) {$p_i-2$};
\node (b2) at (-1,0.8) {$p_i-1$}; 
\node (b3) at (-1,1.6) {$p_i$}; 
\node (b4) at (-1,2.4) {$p_i+1$}; 
\node (b5) at (11.4,0) {$p_i-2$};
\node (b6) at (11.4,0.8) {$p_i-1$}; 
\node (b7) at (11.4,1.6) {$p_i$}; 
\node (b8) at (11.4,2.4) {$p_i+1$}; 

\draw[->] (a1) to node {} (a10);
\draw[->] (a10) to node [below] {$t_{2r+2-p_i}/t_{2r+3-p_i}$} (a15);
\draw[->] (a15) to node {} (a17);
\draw[->] (a2) to node {} (a8);
\draw[->] (a8) to node {} (a9);
\draw[->] (a9) to node [above] {$t_{2r+1-p_i}/t_{2r+2-p_i}$} (a13);
\draw[->] (a13) to node {} (a16);
\draw[->] (a16) to node {} (a18);
\draw[->] (a3) to node {} (a6);
\draw[->] (a6) to node {} (a7);
\draw[->] (a7) to node [above] {$t_{2r-p_i}/t_{2r+1-p_i}$} (a11);
\draw[->] (a11) to node {} (a14);
\draw[->] (a14) to node {} (a19);
\draw[->] (a4) to node {} (a5);
\draw[->] (a5) to node [above] {$t_{2r-1-p_i}/t_{2r-p_i}$} (a12);
\draw[->] (a12) to node {} (a20);
\draw[->] (a5) to node [left] {$1$} (a6);
\draw[->] (a7) to node [left] {$1$} (a8);
\draw[->] (a9) to node [left] {$1$} (a10);
\draw[->] (a11) to node [right] {$c_{2r-p_i}$} (a12);
\draw[->] (a13) to node [right] {$c_{2r+1-p_i}$} (a14);
\draw[->] (a15) to node [right] {$c_{2r+2-p_i}$} (a16);
\end{tikzpicture} 
\end{center}
\end{figure}

Now, as we have $m_{\sigma(1)}\leq p_{\sigma(1)}-1=p_{\sigma(2)}$, this necessarily implies that the paths $\overline{P}_{c,c}(m_{\sigma(1)},n_{\sigma(1)},p_{\sigma(1)})$ and $\overline{P}_{c,c}(m_{\sigma(2)},n_{\sigma(2)},p_{\sigma(2)})$ must intersect at level $p_{\sigma(2)}$, as shown in Figure \ref{Figure5.9}, which contradicts the fact that $\overline{\mathcal{P}}$ is non-intersecting. So $\overline{\mathcal{P}}$ is unmixed as desired.
\end{proof}

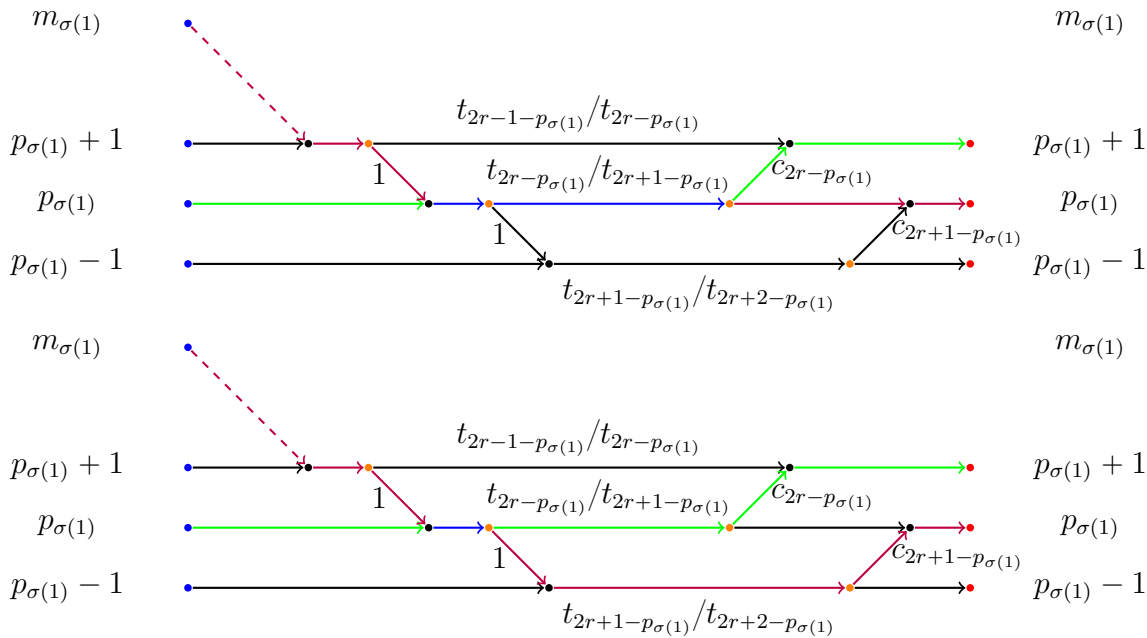
\begin{figure}[t]
\caption{The paths $\overline{P}_{c,c}(m_{\sigma(1)},n_{\sigma(1)},p_{\sigma(1)})$ and $\overline{P}_{c,c}(m_{\sigma(2)},n_{\sigma(2)},p_{\sigma(2)})$ in $\overline{N}_{c,c}(\mathbf{i})$ involving levels $p_{\sigma(1)}-1$, $p_{\sigma(1)}$, $p_{\sigma(1)}+1=p_{\sigma(2)}$ and $m_{\sigma(1)}$, as colored in purple and green respectively, with the common edges colored blue, in the cases where $n_{\sigma(1)}=p_{\sigma(1)}$ and $n_{\sigma(1)}=p_{\sigma(1)}-1$ respectively.}
\label{Figure5.9}
\begin{center}
\begin{tikzpicture}[
       thick,
       acteur/.style={
         circle,
         thick,
         inner sep=1pt,
         minimum size=0.1cm
       }
] 
\node (a1) at (0,0) [acteur,fill=blue]{};
\node (a2) at (0,0.8) [acteur,fill=blue]{}; 
\node (a3) at (0,1.6) [acteur,fill=blue]{}; 
\node (a4) at (0,3.2) [acteur,fill=blue]{};
\node (a5) at (1.6,1.6) [acteur,fill=black]{};
\node (a6) at (2.4,1.6) [acteur,fill=orange]{};
\node (a7) at (3.2,0.8) [acteur,fill=black]{}; 
\node (a8) at (4.0,0.8) [acteur,fill=orange]{}; 
\node (a9) at (4.8,0) [acteur,fill=black]{}; 
\node (a10) at (7.2,0.8) [acteur,fill=orange]{}; 
\node (a11) at (8.0,1.6) [acteur,fill=black]{}; 
\node (a12) at (8.8,0) [acteur,fill=orange]{};
\node (a13) at (9.6,0.8) [acteur,fill=black]{}; 
\node (a14) at (10.4,0) [acteur,fill=red]{}; 
\node (a15) at (10.4,0.8) [acteur,fill=red]{}; 
\node (a16) at (10.4,1.6) [acteur,fill=red]{}; 

\node (b1) at (-1.6,0) {$p_{\sigma(1)}-1$};
\node (b2) at (-1.6,0.8) {$p_{\sigma(1)}$}; 
\node (b3) at (-1.6,1.6) {$p_{\sigma(1)}+1$}; 
\node (b4) at (-1.6,3.2) {$m_{\sigma(1)}$}; 
\node (b5) at (12,0) {$p_{\sigma(1)}-1$};
\node (b6) at (12,0.8) {$p_{\sigma(1)}$}; 
\node (b7) at (12,1.6) {$p_{\sigma(1)}+1$}; 
\node (b8) at (12,3.2) {$m_{\sigma(1)}$}; 

\draw[->] (a1) to node {} (a9);
\draw[->] (a9) to node [below] {\textcolor{black}{$t_{2r+1-p_{\sigma(1)}}/t_{2r+2-p_{\sigma(1)}}$}} (a12);
\draw[->] (a12) to node {} (a14);
\draw[green,->] (a2) to node {} (a7);
\draw[blue,->] (a7) to node {} (a8);
\draw[blue,->] (a8) to node [above] {\textcolor{black}{$t_{2r-p_{\sigma(1)}}/t_{2r+1-p_{\sigma(1)}}$}} (a10);
\draw[purple,->] (a10) to node {} (a13);
\draw[purple,->] (a13) to node {} (a15);
\draw[->] (a3) to node {} (a5);
\draw[purple,->] (a5) to node {} (a6);
\draw[->] (a6) to node [above] {\textcolor{black}{$t_{2r-1-p_{\sigma(1)}}/t_{2r-p_{\sigma(1)}}$}} (a11);
\draw[green,->] (a11) to node {} (a16);
\draw[purple,dashed,->] (a4) to node {} (a5);
\draw[purple,->] (a6) to node [left] {\textcolor{black}{$1$}} (a7);
\draw[->] (a8) to node [left] {\textcolor{black}{$1$}} (a9);
\draw[green,->] (a10) to node [right] {\textcolor{black}{$c_{2r-p_{\sigma(1)}}$}} (a11);
\draw[->] (a12) to node [right] {\textcolor{black}{$c_{2r+1-p_{\sigma(1)}}$}} (a13);

\end{tikzpicture}
\begin{tikzpicture}[
       thick,
       acteur/.style={
         circle,
         thick,
         inner sep=1pt,
         minimum size=0.1cm
       }
] 
\node (a1) at (0,0) [acteur,fill=blue]{};
\node (a2) at (0,0.8) [acteur,fill=blue]{}; 
\node (a3) at (0,1.6) [acteur,fill=blue]{}; 
\node (a4) at (0,3.2) [acteur,fill=blue]{};
\node (a5) at (1.6,1.6) [acteur,fill=black]{};
\node (a6) at (2.4,1.6) [acteur,fill=orange]{};
\node (a7) at (3.2,0.8) [acteur,fill=black]{}; 
\node (a8) at (4.0,0.8) [acteur,fill=orange]{}; 
\node (a9) at (4.8,0) [acteur,fill=black]{}; 
\node (a10) at (7.2,0.8) [acteur,fill=orange]{}; 
\node (a11) at (8.0,1.6) [acteur,fill=black]{}; 
\node (a12) at (8.8,0) [acteur,fill=orange]{};
\node (a13) at (9.6,0.8) [acteur,fill=black]{}; 
\node (a14) at (10.4,0) [acteur,fill=red]{}; 
\node (a15) at (10.4,0.8) [acteur,fill=red]{}; 
\node (a16) at (10.4,1.6) [acteur,fill=red]{}; 

\node (b1) at (-1.6,0) {$p_{\sigma(1)}-1$};
\node (b2) at (-1.6,0.8) {$p_{\sigma(1)}$}; 
\node (b3) at (-1.6,1.6) {$p_{\sigma(1)}+1$}; 
\node (b4) at (-1.6,3.2) {$m_{\sigma(1)}$}; 
\node (b5) at (12,0) {$p_{\sigma(1)}-1$};
\node (b6) at (12,0.8) {$p_{\sigma(1)}$}; 
\node (b7) at (12,1.6) {$p_{\sigma(1)}+1$}; 
\node (b8) at (12,3.2) {$m_{\sigma(1)}$}; 

\draw[->] (a1) to node {} (a9);
\draw[purple,->] (a9) to node [below] {\textcolor{black}{$t_{2r+1-p_{\sigma(1)}}/t_{2r+2-p_{\sigma(1)}}$}} (a12);
\draw[->] (a12) to node {} (a14);
\draw[green,->] (a2) to node {} (a7);
\draw[blue,->] (a7) to node {} (a8);
\draw[green,->] (a8) to node [above] {\textcolor{black}{$t_{2r-p_{\sigma(1)}}/t_{2r+1-p_{\sigma(1)}}$}} (a10);
\draw[->] (a10) to node {} (a13);
\draw[purple,->] (a13) to node {} (a15);
\draw[->] (a3) to node {} (a5);
\draw[purple,->] (a5) to node {} (a6);
\draw[->] (a6) to node [above] {\textcolor{black}{$t_{2r-1-p_{\sigma(1)}}/t_{2r-p_{\sigma(1)}}$}} (a11);
\draw[green,->] (a11) to node {} (a16);
\draw[purple,dashed,->] (a4) to node {} (a5);
\draw[purple,->] (a6) to node [left] {\textcolor{black}{$1$}} (a7);
\draw[purple,->] (a8) to node [left] {\textcolor{black}{$1$}} (a9);
\draw[green,->] (a10) to node [right] {\textcolor{black}{$c_{2r-p_{\sigma(1)}}$}} (a11);
\draw[purple,->] (a12) to node [right] {\textcolor{black}{$c_{2r+1-p_{\sigma(1)}}$}} (a13);

\end{tikzpicture}
\end{center}
\end{figure}

As a consequence of Lemma \ref{5.9}, it follows that \eqref{eq:5.2} reduces to:
\begin{equation}\label{eq:5.9}
f_j^{c,c}=\sum_{\substack{I\subseteq[1,2r],\\|I|=j}}\sum_{\mathcal{P}\in\mathcal{P}_{\nonint}^{c,c}(I)}\wt_{c,c}(\mathcal{P}).
\end{equation}
Thus, our next step is to characterize the admissible $(c,c)$-pairs $(m,n)$, as well as when two paths $P_{c,c}(m,n)$ and $P_{c,c}(m',n')$ corresponding to two different admissible $(c,c)$-pairs $(m,n)$ and $(m',n')$ intersect.

\begin{proposition}\label{5.10}
~
\begin{enumerate}
\item The admissible $(c,c)$-pairs are of the form $(i,i)$ and $(j+1,j)$, where $i\in[1,2r]$ and $j\in[1,2r-1]$. 
\item For all $i\in[1,2r]$ and $j\in[1,2r-1]$, the weights of the paths $P_{2i-1}:=P_{c,c}(i,i)$ and $P_{2j}:=P_{c,c}(j+1,j)$ are given by
\begin{align}
\wt(P_{2k-1})&=\frac{t_k}{t_{k-1}},\quad \wt(P_{4r+1-2k})=\wt(P_{2k-1})^{-1}=\frac{t_{k-1}}{t_k},\quad k\in[1,r],\label{eq:5.10}\\
\wt(P_{2\ell})&=\frac{c_{\ell}t_{\ell}}{t_{\ell-1}},\quad\ell\in[1,r],\label{eq:5.11}\\
\wt(P_{4r-2\ell})&=\frac{c_{\ell}t_{\ell}}{t_{\ell+1}},\quad\ell\in[1,r-1],\label{eq:5.12}
\end{align}
where $t_0=1$. In particular, the path weights $y_i=\rho_{\tau}^*(\wt(P_i))$, written in terms of normalized $A_{2r-1}^{(2)}$ $Q$-system variables using \eqref{eq:4.46} and \eqref{eq:4.47}, are given by
{\allowdisplaybreaks
\begin{align}
y_{2k-1}&=\frac{R_{k,1}R_{k-1,0}}{R_{k,0}R_{k-1,1}},\quad y_{4r+1-2k}=\frac{R_{k,0}R_{k-1,1}}{R_{k,1}R_{k-1,0}},\quad k\in[1,r],\label{eq:5.13}\\
y_{2\ell}&=\frac{R_{\ell-1,0}R_{\ell+1,1}}{R_{\ell,0}R_{\ell,1}},\quad y_{4r-2\ell}=\frac{R_{\ell-1,1}R_{\ell+1,0}}{R_{\ell,0}R_{\ell,1}},\quad\ell\in[1,r-1],\label{eq:5.14}\\
y_{2r}&=\frac{R_{r-1,0}R_{r-1,1}}{R_{r,0}R_{r,1}},\label{eq:5.15}
\end{align}
where $R_{0,k}=1$ for all $k\in\mathbb{Z}$.
}
\item The paths $P_{2i-1}=P_{c,c}(i,i), P_{2i'-1}=P_{c,c}(i',i')$ do not intersect for any distinct $i,i'\in[1,2r]$, and for any $j\in[1,2r-1]$ and admissible $(c,c)$-pair $(m,n)\neq(j+1,j)$, the path $P_{2j}=P_{c,c}(j+1,j)$ intersects $P_k=P_{c,c}(m,n)$ if and only if $(m,n)=(j,j)$, $(j+1,j+1),(j,j-1)$ or $(j+2,j+1)$, or equivalently, $k=2j\pm1,2j\pm2$.
\end{enumerate}
\end{proposition} 

\begin{proof}
By the proof of Lemma \ref{5.9}, we have that if $(m,n)$ is an admissible $(c,c)$-pair, then either $m=n$ or $m=n+1$, which proves statement (1). Statement (2) follows from the definition of path weights, and statement (3) follow from observation using Figures \ref{Figure5.6}, \ref{Figure5.10} and \ref{Figure5.11}.

\begin{figure}[t]
\caption{The paths $P_{c,c}(2r+1-j,2r-j)$, $P_{c,c}(2r-j,2r-1-j)$ and $P_{c,c}(2r+2-j,2r+1-j)$ in the network diagram $\overline{N}_{c,c}(\mathbf{i})$ as colored in green, purple and orange respectively, where $j\in[1,r-1]$. The common edge in the paths $P_{c,c}(2r+1-j,2r-j)$ and $P_{c,c}(2r-j,2r-1-j)$ is colored blue, and the common edge in the paths $P_{c,c}(2r+1-j,2r-j)$ and $P_{c,c}(2r+2-j,2r+1-j)$ is colored red.}
\label{Figure5.10}
\begin{center}
\begin{tikzpicture}[
       thick,
       acteur/.style={
         circle,
         thick,
         inner sep=1pt,
         minimum size=0.1cm
       }
] 
\node (a1) at (0,0) [acteur,fill=blue]{};
\node (a2) at (0,0.8) [acteur,fill=blue]{}; 
\node (a3) at (0,1.6) [acteur,fill=blue]{}; 
\node (a4) at (0,2.4) [acteur,fill=blue]{}; 
\node (a5) at (0.8,2.4) [acteur,fill=orange]{};
\node (a6) at (1.6,1.6) [acteur,fill=black]{}; 
\node (a7) at (2.4,1.6) [acteur,fill=orange]{}; 
\node (a8) at (3.2,0.8) [acteur,fill=black]{}; 
\node (a9) at (4.0,0.8) [acteur,fill=orange]{};
\node (a10) at (4.8,0) [acteur,fill=black]{}; 
\node (a11) at (5.6,1.6) [acteur,fill=orange]{}; 
\node (a12) at (6.4,2.4) [acteur,fill=black]{}; 
\node (a13) at (7.2,0.8) [acteur,fill=orange]{};
\node (a14) at (8.0,1.6) [acteur,fill=black]{}; 
\node (a15) at (8.8,0) [acteur,fill=orange]{}; 
\node (a16) at (9.6,0.8) [acteur,fill=black]{}; 
\node (a17) at (10.4,0) [acteur,fill=red]{};
\node (a18) at (10.4,0.8) [acteur,fill=red]{}; 
\node (a19) at (10.4,1.6) [acteur,fill=red]{}; 
\node (a20) at (10.4,2.4) [acteur,fill=red]{}; 

\node (b1) at (-1,0) {$2r-1-j$};
\node (b2) at (-1,0.8) {$2r-j$}; 
\node (b3) at (-1,1.6) {$2r+1-j$}; 
\node (b4) at (-1,2.4) {$2r+2-j$}; 
\node (b5) at (11.4,0) {$2r-1-j$};
\node (b6) at (11.4,0.8) {$2r-j$}; 
\node (b7) at (11.4,1.6) {$2r+1-j$}; 
\node (b8) at (11.4,2.4) {$2r+2-j$}; 

\draw[->] (a1) to node {} (a10);
\draw[purple,->] (a10) to node [below] {\textcolor{black}{$t_{j+1}/t_{j+2}$}} (a15);
\draw[->] (a15) to node {} (a17);
\draw[purple,->] (a2) to node {} (a8);
\draw[blue,->] (a8) to node {} (a9);
\draw[green,->] (a9) to node [above] {\textcolor{black}{$t_j/t_{j+1}$}} (a13);
\draw[->] (a13) to node {} (a16);
\draw[purple,->] (a16) to node {} (a18);
\draw[green,->] (a3) to node {} (a6);
\draw[red,->] (a6) to node {} (a7);
\draw[orange,->] (a7) to node [above] {\textcolor{black}{$t_{j-1}/t_j$}} (a11);
\draw[->] (a11) to node {} (a14);
\draw[green,->] (a14) to node {} (a19);
\draw[orange,->] (a4) to node {} (a5);
\draw[->] (a5) to node [above] {$t_{j-2}/t_{j-1}$} (a12);
\draw[orange,->] (a12) to node {} (a20);
\draw[orange,->] (a5) to node [left] {\textcolor{black}{$1$}} (a6);
\draw[green,->] (a7) to node [left] {\textcolor{black}{$1$}} (a8);
\draw[purple,->] (a9) to node [left] {\textcolor{black}{$1$}} (a10);
\draw[orange,->] (a11) to node [right] {\textcolor{black}{$c_{j-1}$}} (a12);
\draw[green,->] (a13) to node [right] {\textcolor{black}{$c_j$}} (a14);
\draw[purple,->] (a15) to node [right] {\textcolor{black}{$c_{j+1}$}} (a16);
\end{tikzpicture} 
\end{center}
\end{figure}
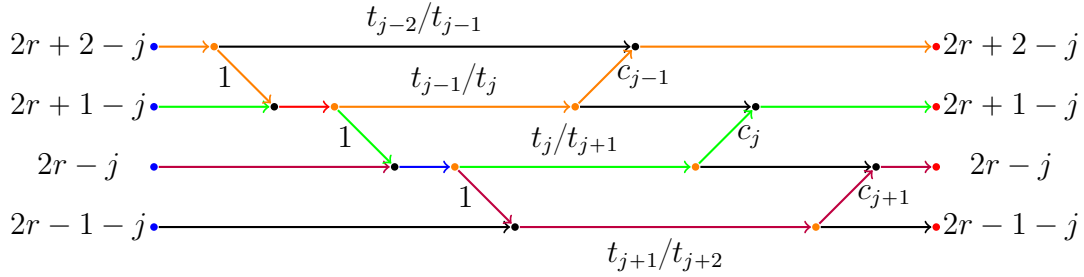
\end{proof}

The non-intersecting condition given in statement (3) of Proposition \ref{5.10} can be described in terms of adjacency of vertices on a graph $\mathcal{G}_C$, where the graph $\mathcal{G}_C$ is given as follows:
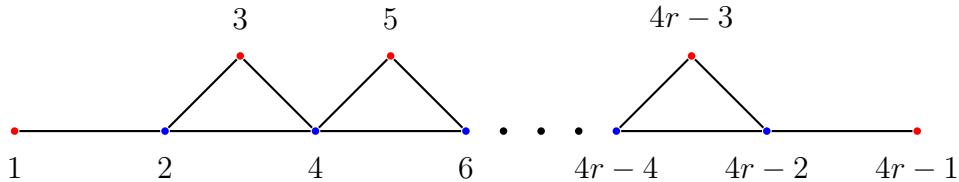
\begin{center}
\begin{tikzpicture}[
       thick,
       acteur/.style={
         circle,
         thick,
         inner sep=1pt,
         minimum size=0.1cm
       }
] 

\node (a1) at (0,0) [acteur,fill=red]{};
\node (a2) at (2,0) [acteur,fill=blue]{};
\node (a3) at (3,1) [acteur,fill=red]{};
\node (a4) at (4,0) [acteur,fill=blue]{};
\node (a5) at (5,1) [acteur,fill=red]{};
\node (a6) at (6,0) [acteur,fill=blue]{};
\node (d1) at (6.5,0) [acteur,fill=black]{};
\node (d2) at (7,0) [acteur,fill=black]{};
\node (d3) at (7.5,0) [acteur,fill=black]{};
\node (a4r-4) at (8,0) [acteur,fill=blue]{};
\node (a4r-3) at (9,1) [acteur,fill=red]{};
\node (a4r-2) at (10,0) [acteur,fill=blue]{};
\node (a4r-1) at (12,0) [acteur,fill=red]{};

\node (b1) at (0,-0.2) [below] {$1$};
\node (b2) [above] at (2,-0.2) [below] {$2$};
\node (b3) at (3,1.2) [above] {$3$};
\node (b4) at (4,-0.2) [below] {$4$};
\node (b5) at (5,1.2) [above] {$5$};
\node (b6) at (6,-0.2) [below] {$6$};
\node (b4r-4) [above] at (8,-0.2) [below] {$4r-4$};
\node (b4r-3) [below] at (9,1.2) [above] {$4r-3$};
\node (b4r-2) at (10,-0.2) [below] {$4r-2$};
\node (b4r-1) at (12,-0.2) [below] {$4r-1$};

\draw[-] (a1) to node {} (a2);
\draw[-] (a2) to node {} (a3);
\draw[-] (a2) to node {} (a4);
\draw[-] (a3) to node {} (a4);
\draw[-] (a4) to node {} (a5);
\draw[-] (a4) to node {} (a6);
\draw[-] (a5) to node {} (a6);
\draw[-] (a2r-2) to node {} (a2r-1);
\draw[-] (a2r-2) to node {} (a2r);
\draw[-] (a2r-1) to node {} (a2r);
\draw[-] (a2r) to node {} (a2r+1);

\end{tikzpicture} 
\end{center}

With the undirected graph $\mathcal{G}_C$ given as above, it follows that for all distinct $j,k\in[1,4r-1]$, the paths $P_j$ and $P_k$ are non-intersecting if and only if vertices $j$ and $k$ are not adjacent to each other in $\mathcal{G}_C$. Similarly to type $A$, we may define a hard particle configuration on $\mathcal{G}_C$, as well as the weight $\wt(C)$ of a hard particle configuration $C$ on $\mathcal{G}_C$, and the set of hard particle configurations $\HPC(\mathcal{G}_C)$ on $\mathcal{G}_C$ analogously.

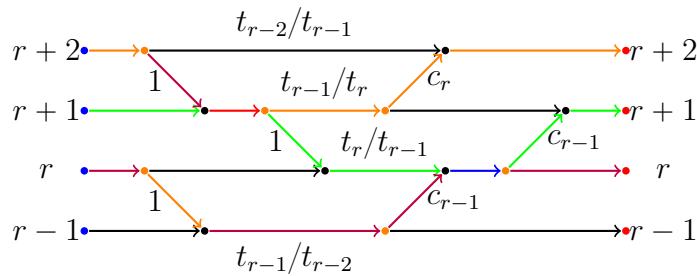
\begin{figure}[t]
\caption{The paths $P_{c,c}(r+1,r)$, $P_{c,c}(r,r-1)$ and $P_{c,c}(r+2,r+1)$ in the network diagram $\overline{N}_{c,c}(\mathbf{i})$ as colored in green, purple and orange respectively, where $j\in[1,r-1]$. The common edge in the paths $P_{c,c}(r+1,r)$ and $P_{c,c}(r,r-1)$ is colored blue, and the common edge in the paths $P_{c,c}(r+1,r)$ and $P_{c,c}(r,r-1)$ is colored red.}
\label{Figure5.11}
\begin{center}
\begin{tikzpicture}[
       thick,
       acteur/.style={
         circle,
         thick,
         inner sep=1pt,
         minimum size=0.1cm
       }
] 
\node (a1) at (0,0) [acteur,fill=blue]{};
\node (a2) at (0,0.8) [acteur,fill=blue]{}; 
\node (a3) at (0,1.6) [acteur,fill=blue]{}; 
\node (a4) at (0,2.4) [acteur,fill=blue]{}; 
\node (a5) at (0.8,2.4) [acteur,fill=orange]{};
\node (a6) at (1.6,1.6) [acteur,fill=black]{}; 
\node (a7) at (0.8,0.8) [acteur,fill=orange]{}; 
\node (a8) at (1.6,0) [acteur,fill=black]{}; 
\node (a9) at (2.4,1.6) [acteur,fill=orange]{};
\node (a10) at (3.2,0.8) [acteur,fill=black]{}; 
\node (a11) at (4.0,0) [acteur,fill=orange]{}; 
\node (a12) at (4.8,0.8) [acteur,fill=black]{}; 
\node (a13) at (4.0,1.6) [acteur,fill=orange]{};
\node (a14) at (4.8,2.4) [acteur,fill=black]{}; 
\node (a15) at (5.6,0.8) [acteur,fill=orange]{}; 
\node (a16) at (6.4,1.6) [acteur,fill=black]{}; 
\node (a17) at (7.2,0) [acteur,fill=red]{};
\node (a18) at (7.2,0.8) [acteur,fill=red]{}; 
\node (a19) at (7.2,1.6) [acteur,fill=red]{}; 
\node (a20) at (7.2,2.4) [acteur,fill=red]{}; 

\node (b1) at (-0.5,0) {$r-1$};
\node (b2) at (-0.5,0.8) {$r$}; 
\node (b3) at (-0.5,1.6) {$r+1$}; 
\node (b4) at (-0.5,2.4) {$r+2$}; 
\node (b5) at (7.7,0) {$r-1$};
\node (b6) at (7.7,0.8) {$r$}; 
\node (b7) at (7.7,1.6) {$r+1$}; 
\node (b8) at (7.7,2.4) {$r+2$}; 

\draw[->] (a1) to node {} (a8);
\draw[purple,->] (a8) to node [below] {\textcolor{black}{$t_{r-1}/t_{r-2}$}} (a11);
\draw[->] (a11) to node {} (a17);
\draw[purple,->] (a2) to node {} (a7);
\draw[->] (a7) to node {} (a10);
\draw[green,->] (a10) to node [above] {\textcolor{black}{$t_r/t_{r-1}$}} (a12);
\draw[blue,->] (a12) to node {} (a15);
\draw[purple,->] (a15) to node {} (a18);
\draw[green,->] (a3) to node {} (a6);
\draw[red,->] (a6) to node {} (a9);
\draw[orange,->] (a9) to node [above] {\textcolor{black}{$t_{r-1}/t_r$}} (a13);
\draw[->] (a13) to node {} (a16);
\draw[green,->] (a16) to node {} (a19);
\draw[orange,->] (a4) to node {} (a5);
\draw[->] (a5) to node [above] {$t_{r-2}/t_{r-1}$} (a14);
\draw[orange,->] (a14) to node {} (a20);
\draw[purple,->] (a5) to node [left] {\textcolor{black}{$1$}} (a6);
\draw[orange,->] (a7) to node [left] {\textcolor{black}{$1$}} (a8);
\draw[green,->] (a9) to node [left] {\textcolor{black}{$1$}} (a10);
\draw[purple,->] (a11) to node [right] {\textcolor{black}{$c_{r-1}$}} (a12);
\draw[orange,->] (a13) to node [right] {\textcolor{black}{$c_r$}} (a14);
\draw[green,->] (a15) to node [right] {\textcolor{black}{$c_{r-1}$}} (a16);
\end{tikzpicture} 
\end{center}
\end{figure}

By \eqref{eq:5.9}, Definition \ref{5.7}, Lemma \ref{5.9} and Proposition \ref{5.10}, we recover the following formula for the conserved quantity $\widetilde{C}_j=\rho_{\tau}^*(f_j^{c,c})$ of the normalized $A_{2r-1}^{(2)}$ $Q$-system, or equivalently, the Coxeter-Toda Hamiltonian $f_j^{c,c}$, written in terms of the normalized $A_{2r-1}^{(2)}$ $Q$-system variables:

\begin{theorem}\label{5.11}
Let $j\in[1,r]$. Then the conserved quantity $\widetilde{C}_j=\rho_{\tau}^*(f_j^{c,c})$ of the normalized $A_{2r-1}^{(2)}$ $Q$-system is given by
\begin{equation}\label{eq:5.16}
\widetilde{C}_j=\sum_{\substack{C\in\HPC(\mathcal{G}_C)\\ |C|=j}}\wt(C).
\end{equation}
More precisely, we have
\begin{equation}\label{eq:5.17}
\widetilde{C}_j=\sum_{\substack{1\leq i_1<i_2<\cdots<i_j\leq 4r-1\\i_k\leq i_{k+1}-2\text{ if }k\text{ is }odd,\\i_k\leq i_{k+1}-3\text{ if }k\text{ is }even}}y_{i_1}\cdots y_{i_j}.
\end{equation}
\end{theorem}

\subsection{The Coxeter-Toda Hamiltonians of type \textit{B} arising from exterior powers of the defining representation}\label{Section5.4}

In this subsection, we will derive combinatorial formulas for the Coxeter-Toda Hamiltonians $H_j^{c,c}=f_j^{c,c}$, $j\in[1,r-1]$, in the case where $G=SO_{2r+1}(\mathbb{C})$. Unlike in types $A$ and $C$, there exist fully mixed, non-intersecting subsets $\overline{\mathcal{P}}$ of $\overline{\mathcal{P}}^{c,c}$, as Example \ref{5.3} shows. Thus, our first goal in this subsection is to characterize the fully mixed non-intersecting subsets of $\overline{\mathcal{P}}^{c,c}$. To do so, we will proceed in several steps.

\begin{lemma}\label{5.12}
Let $\overline{\mathcal{P}}=\{(m_1,n_1,p_1),\ldots,(m_j,n_j,p_j)\}$ be a fully mixed non-intersecting set, and $I=\{m_1,\ldots,m_j\}=\{p_1,\ldots,p_j\}$. Then $I\subseteq[r+1,2r+1]$, and $I\not\subseteq[r+3,2r+1]$.
\end{lemma}

\begin{proof}
Firstly, we may argue in a similar way as in the first part of the proof of Lemma \ref{5.9}, using Figure \ref{Figure5.12} where necessary, to deduce that $I\subseteq[r+1,2r+1]$. Subsequently, we may argue in a similar way as in the second part of the proof of Lemma \ref{5.9}, using Figure \ref{Figure5.12} where necessary again, to deduce that $I\not\subseteq[r+3,2r+1]$. We will omit the details here.
\begin{figure}[t]
\caption{The network diagram $\overline{N}_{c,c}(\mathbf{i})$ involving levels $1$, $2$, $3$, $r$, $r+1$, $r+2$, $2r-1$, $2r$ and $2r+1$.}
\label{Figure5.12}
\begin{center}
\begin{tikzpicture}[
       thick,
       acteur/.style={
         circle,
         thick,
         inner sep=1pt,
         minimum size=0.1cm
       }
] 
\node (a1) at (0,0) [acteur,fill=blue]{};
\node (a2) at (0,0.8) [acteur,fill=blue]{}; 
\node (a3) at (0,1.6) [acteur,fill=blue]{}; 
\node (a4) at (0,2.4) [acteur,fill=blue]{}; 
\node (a5) at (0,3.2) [acteur,fill=blue]{};
\node (a6) at (0,4.0) [acteur,fill=blue]{};
\node (a7) at (0,4.8) [acteur,fill=blue]{}; 
\node (a8) at (0,5.6) [acteur,fill=blue]{};
\node (a9) at (0,6.4) [acteur,fill=blue]{}; 
\node (a10) at (0.8,6.4) [acteur,fill=orange]{}; 
\node (a11) at (1.6,5.6) [acteur,fill=black]{}; 
\node (a12) at (0.8,0.8) [acteur,fill=orange]{}; 
\node (a13) at (1.6,0) [acteur,fill=black]{}; 
\node (a14) at (2.4,5.6) [acteur,fill=orange]{}; 
\node (a15) at (3.2,4.8) [acteur,fill=black]{}; 
\node (a16) at (2.4,1.6) [acteur,fill=orange]{}; 
\node (a17) at (3.2,0.8) [acteur,fill=black]{}; 
\node (a18) at (4.0,3.2) [acteur,fill=orange]{}; 
\node (a19) at (4.8,2.4) [acteur,fill=black]{}; 
\node (a20) at (4.0,4.0) [acteur,fill=orange]{}; 
\node (a21) at (4.8,3.2) [acteur,fill=black]{}; 
\node (a22) at (5.6,5.6) [acteur,fill=orange]{}; 
\node (a23) at (6.4,6.4) [acteur,fill=black]{}; 
\node (a24) at (5.6,0) [acteur,fill=orange]{}; 
\node (a25) at (6.4,0.8) [acteur,fill=black]{}; 
\node (a26) at (7.2,4.8) [acteur,fill=orange]{}; 
\node (a27) at (8.0,5.6) [acteur,fill=black]{}; 
\node (a28) at (7.2,0.8) [acteur,fill=orange]{}; 
\node (a29) at (8.0,1.6) [acteur,fill=black]{}; 
\node (a30) at (8.8,2.4) [acteur,fill=orange]{}; 
\node (a31) at (9.6,3.2) [acteur,fill=black]{}; 
\node (a32) at (8.8,3.2) [acteur,fill=orange]{}; 
\node (a33) at (9.6,4.0) [acteur,fill=black]{}; 
\node (a34) at (10.4,0) [acteur,fill=red]{};
\node (a35) at (10.4,0.8) [acteur,fill=red]{}; 
\node (a36) at (10.4,1.6) [acteur,fill=red]{}; 
\node (a37) at (10.4,2.4) [acteur,fill=red]{}; 
\node (a38) at (10.4,3.2) [acteur,fill=red]{};
\node (a39) at (10.4,4.0) [acteur,fill=red]{}; 
\node (a40) at (10.4,4.8) [acteur,fill=red]{}; 
\node (a41) at (10.4,5.6) [acteur,fill=red]{};
\node (a42) at (10.4,6.4) [acteur,fill=red]{};

\node (b1) at (-0.6,0) {$1$};
\node (b2) at (-0.6,0.8) {$2$}; 
\node (b3) at (-0.6,1.6) {$3$}; 
\node (b4) at (-0.6,2.4) {$r$}; 
\node (b5) at (-0.6,3.2) {$r+1$};
\node (b6) at (-0.6,4.0) {$r+2$};
\node (b7) at (-0.6,4.8) {$2r-1$}; 
\node (b8) at (-0.6,5.6) {$2r$}; 
\node (b9) at (-0.6,6.4) {$2r+1$}; 
\node (b10) at (11,0) {$1$};
\node (b11) at (11,0.8) {$2$}; 
\node (b12) at (11,1.6) {$3$}; 
\node (b13) at (11,2.4) {$r$}; 
\node (b14) at (11,3.2) {$r+1$};
\node (b15) at (11,4.0) {$r+2$};
\node (b16) at (11,4.8) {$2r-1$}; 
\node (b17) at (11,5.6) {$2r$}; 
\node (b18) at (11,6.4) {$2r+1$}; 

\node (c1) at (3.6,1.8) [acteur,fill=black]{};
\node (c2) at (3.6,2.0) [acteur,fill=black]{};
\node (c3) at (3.6,2.2) [acteur,fill=black]{};
\node (c4) at (3.6,4.2) [acteur,fill=black]{};
\node (c5) at (3.6,4.4) [acteur,fill=black]{};
\node (c6) at (3.6,4.6) [acteur,fill=black]{};
\node (c7) at (8.4,1.8) [acteur,fill=black]{};
\node (c8) at (8.4,2.0) [acteur,fill=black]{};
\node (c9) at (8.4,2.2) [acteur,fill=black]{};
\node (c10) at (8.4,4.2) [acteur,fill=black]{};
\node (c11) at (8.4,4.4) [acteur,fill=black]{};
\node (c12) at (8.4,4.6) [acteur,fill=black]{};

\draw[->] (a1) to node {} (a13);
\draw[->] (a13) to node [below] {$t_1$} (a24);
\draw[->] (a24) to node {} (a34);
\draw[->] (a2) to node {} (a12);
\draw[->] (a12) to node {} (a17);
\draw[->] (a17) to node [below] {$t_2/t_1$} (a25);
\draw[->] (a25) to node {} (a28);
\draw[->] (a28) to node {} (a35);
\draw[->] (a3) to node {} (a16);
\draw[->] (a16) to node [below] {$t_3/t_2$} (a29);
\draw[->] (a29) to node {} (a36);
\draw[->] (a4) to node {} (a19);
\draw[->] (a19) to node [below] {$t_r^2/t_{r-1}$} (a30);
\draw[->] (a30) to node {} (a37);
\draw[->] (a5) to node {} (a18);
\draw[->] (a18) to node {} (a21);
\draw[->] (a21) to node [above] {$1$} (a32);
\draw[->] (a32) to node {} (a31);
\draw[->] (a31) to node {} (a38);
\draw[->] (a6) to node {} (a20);
\draw[->] (a20) to node [above] {$t_{r-1}/t_r^2$} (a33);
\draw[->] (a33) to node {} (a39);
\draw[->] (a7) to node {} (a15);
\draw[->] (a15) to node [above] {$t_2/t_3$} (a26);
\draw[->] (a26) to node {} (a40);
\draw[->] (a8) to node {} (a11);
\draw[->] (a11) to node {} (a14);
\draw[->] (a14) to node [above] {$t_1/t_2$} (a22);
\draw[->] (a22) to node {} (a27);
\draw[->] (a27) to node {} (a41);
\draw[->] (a9) to node {} (a10);
\draw[->] (a10) to node [above] {$1/t_1$} (a23);
\draw[->] (a23) to node {} (a42);
\draw[->] (a10) to node [left] {$1$} (a11);
\draw[->] (a12) to node [left] {$1$} (a13);
\draw[->] (a14) to node [left] {$1$} (a15);
\draw[->] (a16) to node [left] {$1$} (a17);
\draw[->] (a22) to node [right] {$c_1$} (a23);
\draw[->] (a24) to node [right] {$c_1$} (a25);
\draw[->] (a26) to node [right] {$c_2$} (a27);
\draw[->] (a28) to node [right] {$c_2$} (a29);
\draw[->] (a18) to node [left=0.2cm] {$\sqrt{2}$} (a19);
\draw[->] (a20) to node [right=0.2cm] {$\sqrt{2}$} (a21);
\draw[->] (a20) to node [above left=0.3cm] {$1$} (a19);
\draw[->] (a30) to node [right=0.3cm] {$\sqrt{2}c_r$} (a31);
\draw[->] (a32) to node [left=0.3cm] {$\sqrt{2}c_r$} (a33);
\draw[->] (a30) to node [above right=0.2cm] {$c_r^2$} (a33);

\end{tikzpicture} 
\end{center}
\end{figure}
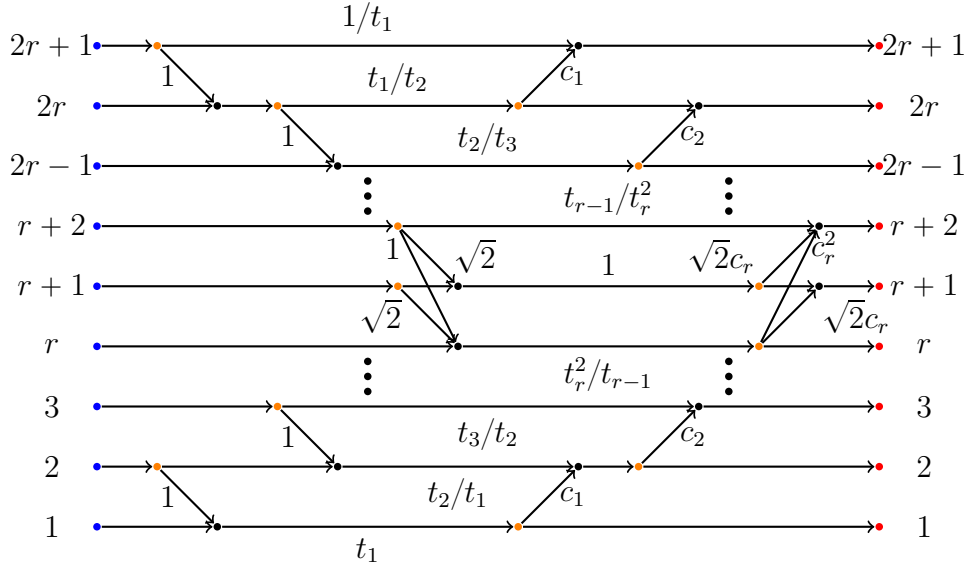
\end{proof}

\begin{lemma}\label{5.13}
Let $\overline{\mathcal{P}}=\{(m_1,n_1,p_1),\ldots,(m_j,n_j,p_j)\}$ be a fully mixed non-intersecting set, and $I=\{m_1,\ldots,m_j\}=\{p_1,\ldots,p_j\}$. Then $r+1\in I$.
\end{lemma}

\begin{proof}
By Lemma \ref{5.12}, we must have either $r+1\in I$, or $r+2\in I$. Suppose on the contrary that we have $r+1\notin I$. As before, we may assume without loss of generality that we have $r+2=m_1<m_2<\cdots<m_j$. Then there exists some $\ell\in[2,j]$, such that $m_{\ell}>p_{\ell}=m_1=r+2$. As the elementary chip corresponding to $E_{r-2}(c_{r-2})$ appears before the elementary chip corresponding to $E_{r-1}(c_{r-1})$ in $\overline{N}_{c,c}(\mathbf{i})$, the elementary chip corresponding to $E_{r-1}(c_{r-1})$ appears before the elementary chip corresponding to $E_r(c_r)$ in $\overline{N}_{c,c}(\mathbf{i})$, and we have $m_1,p_1\geq n_1$, we must have $n_1=m_1=r+2$, and $p_1=m_1+1=r+3$, as shown in Figure \ref{Figure5.13}. As $m_{\ell}\geq r+3=p_1=m_1+1$, we must have $\ell=2$. Now, it is easy to observe using Figure \ref{Figure5.13} that the paths $\overline{P}_{c,c}(m_1,n_1,p_1)=\overline{P}_{c,c}(r+2,r+2,r+3)$ and $\overline{P}_{c,c}(m_2,n_2,p_2)=\overline{P}_{c,c}(r+3,n_2,r+2)$ must intersect at level $r+2$ during their descents, which contradicts the fact that $\overline{\mathcal{P}}$ is non-intersecting. Therefore, we must have $r+1\in I$.

\begin{figure}[t]
\caption{The paths $\overline{P}_{c,c}(r+2,r+2,r+3)$ and $\overline{P}_{c,c}(r+3,n_2,r+2)$ in $\overline{N}_{c,c}(\mathbf{i})$ as colored in green and purple respectively, with the common edges colored blue, in the cases where $n_2=r+2$, $n_2=r+1$ and $n_2=r$ respectively.}
\label{Figure5.13}
\begin{center}
\begin{tikzpicture}[
       thick,
       acteur/.style={
         circle,
         thick,
         inner sep=1pt,
         minimum size=0.1cm
       }
] 
\node (a1) at (0,0) [acteur,fill=blue]{}; 
\node (a2) at (0,0.8) [acteur,fill=blue]{};
\node (a3) at (0,1.6) [acteur,fill=blue]{};
\node (a4) at (0,2.4) [acteur,fill=blue]{};
\node (a5) at (0,3.2) [acteur,fill=blue]{}; 
\node (a6) at (0.8,3.2) [acteur,fill=orange]{}; 
\node (a7) at (1.6,2.4) [acteur,fill=black]{}; 
\node (a8) at (2.4,2.4) [acteur,fill=orange]{}; 
\node (a9) at (3.2,1.6) [acteur,fill=black]{}; 
\node (a10) at (4.0,0.8) [acteur,fill=orange]{}; 
\node (a11) at (4.8,0) [acteur,fill=black]{}; 
\node (a12) at (4.0,1.6) [acteur,fill=orange]{}; 
\node (a13) at (4.8,0.8) [acteur,fill=black]{}; 
\node (a14) at (5.6,2.4) [acteur,fill=orange]{}; 
\node (a15) at (6.4,3.2) [acteur,fill=black]{}; 
\node (a16) at (7.2,1.6) [acteur,fill=orange]{}; 
\node (a17) at (8.0,2.4) [acteur,fill=black]{}; 
\node (a18) at (8.8,0) [acteur,fill=orange]{}; 
\node (a19) at (9.6,0.8) [acteur,fill=black]{}; 
\node (a20) at (8.8,0.8) [acteur,fill=orange]{}; 
\node (a21) at (9.6,1.6) [acteur,fill=black]{}; 
\node (a22) at (10.4,0) [acteur,fill=red]{}; 
\node (a23) at (10.4,0.8) [acteur,fill=red]{};
\node (a24) at (10.4,1.6) [acteur,fill=red]{}; 
\node (a25) at (10.4,2.4) [acteur,fill=red]{};
\node (a26) at (10.4,3.2) [acteur,fill=red]{};

\node (b1) at (-0.6,0) {$r$}; 
\node (b2) at (-0.6,0.8) {$r+1$};
\node (b3) at (-0.6,1.6) {$r+2$};
\node (b4) at (-0.6,2.4) {$r+3$}; 
\node (b5) at (-0.6,3.2) {$r+4$}; 
\node (b6) at (11,0) {$r$}; 
\node (b7) at (11,0.8) {$r+1$};
\node (b8) at (11,1.6) {$r+2$};
\node (b9) at (11,2.4) {$r+3$}; 
\node (b10) at (11,3.2) {$r+4$}; 

\draw[->] (a1) to node {} (a11);
\draw[->] (a11) to node [below] {$t_r^2/t_{r-1}$} (a18);
\draw[->] (a18) to node {} (a22);
\draw[->] (a2) to node {} (a10);
\draw[->] (a10) to node {} (a13);
\draw[->] (a13) to node [above] {$1$} (a20);
\draw[->] (a20) to node {} (a19);
\draw[->] (a19) to node {} (a23);
\draw[green,->] (a3) to node {} (a9);
\draw[blue,->] (a9) to node {} (a12);
\draw[blue,->] (a12) to node [above] {\textcolor{black}{$t_{r-1}/t_r^2$}} (a16);
\draw[purple,->] (a16) to node {} (a21);
\draw[purple,->] (a21) to node {} (a24);
\draw[purple,->] (a4) to node {} (a7);
\draw[purple,->] (a7) to node {} (a8);
\draw[->] (a8) to node [above] {$t_{r-2}/t_{r-1}$} (a14);
\draw[->] (a14) to node {} (a17);
\draw[green,->] (a17) to node {} (a25);
\draw[->] (a5) to node {} (a6);
\draw[->] (a6) to node [above] {$t_{r-3}/t_{r-2}$} (a15);
\draw[->] (a15) to node {} (a26);
\draw[->] (a6) to node [left] {$1$} (a7);
\draw[purple,->] (a8) to node [left] {\textcolor{black}{$1$}} (a9);
\draw[->] (a14) to node [right] {$c_{r-2}$} (a15);
\draw[green,->] (a16) to node [right] {\textcolor{black}{$c_{r-1}$}} (a17);
\draw[->] (a10) to node [left=0.2cm] {$\sqrt{2}$} (a11);
\draw[->] (a12) to node [right=0.2cm] {$\sqrt{2}$} (a13);
\draw[->] (a12) to node [above left=0.3cm] {$1$} (a11);
\draw[->] (a18) to node [right=0.3cm] {$\sqrt{2}c_r$} (a19);
\draw[->] (a20) to node [left=0.3cm] {$\sqrt{2}c_r$} (a21);
\draw[->] (a18) to node [above right=0.2cm] {$c_r^2$} (a21);

\end{tikzpicture} 
\begin{tikzpicture}[
       thick,
       acteur/.style={
         circle,
         thick,
         inner sep=1pt,
         minimum size=0.1cm
       }
] 
\node (a1) at (0,0) [acteur,fill=blue]{}; 
\node (a2) at (0,0.8) [acteur,fill=blue]{};
\node (a3) at (0,1.6) [acteur,fill=blue]{};
\node (a4) at (0,2.4) [acteur,fill=blue]{};
\node (a5) at (0,3.2) [acteur,fill=blue]{}; 
\node (a6) at (0.8,3.2) [acteur,fill=orange]{}; 
\node (a7) at (1.6,2.4) [acteur,fill=black]{}; 
\node (a8) at (2.4,2.4) [acteur,fill=orange]{}; 
\node (a9) at (3.2,1.6) [acteur,fill=black]{}; 
\node (a10) at (4.0,0.8) [acteur,fill=orange]{}; 
\node (a11) at (4.8,0) [acteur,fill=black]{}; 
\node (a12) at (4.0,1.6) [acteur,fill=orange]{}; 
\node (a13) at (4.8,0.8) [acteur,fill=black]{}; 
\node (a14) at (5.6,2.4) [acteur,fill=orange]{}; 
\node (a15) at (6.4,3.2) [acteur,fill=black]{}; 
\node (a16) at (7.2,1.6) [acteur,fill=orange]{}; 
\node (a17) at (8.0,2.4) [acteur,fill=black]{}; 
\node (a18) at (8.8,0) [acteur,fill=orange]{}; 
\node (a19) at (9.6,0.8) [acteur,fill=black]{}; 
\node (a20) at (8.8,0.8) [acteur,fill=orange]{}; 
\node (a21) at (9.6,1.6) [acteur,fill=black]{}; 
\node (a22) at (10.4,0) [acteur,fill=red]{}; 
\node (a23) at (10.4,0.8) [acteur,fill=red]{};
\node (a24) at (10.4,1.6) [acteur,fill=red]{}; 
\node (a25) at (10.4,2.4) [acteur,fill=red]{};
\node (a26) at (10.4,3.2) [acteur,fill=red]{};

\node (b1) at (-0.6,0) {$r$}; 
\node (b2) at (-0.6,0.8) {$r+1$};
\node (b3) at (-0.6,1.6) {$r+2$};
\node (b4) at (-0.6,2.4) {$r+3$}; 
\node (b5) at (-0.6,3.2) {$r+4$}; 
\node (b6) at (11,0) {$r$}; 
\node (b7) at (11,0.8) {$r+1$};
\node (b8) at (11,1.6) {$r+2$};
\node (b9) at (11,2.4) {$r+3$}; 
\node (b10) at (11,3.2) {$r+4$}; 

\draw[->] (a1) to node {} (a11);
\draw[->] (a11) to node [below] {$t_r^2/t_{r-1}$} (a18);
\draw[->] (a18) to node {} (a22);
\draw[->] (a2) to node {} (a10);
\draw[->] (a10) to node {} (a13);
\draw[purple,->] (a13) to node [above] {\textcolor{black}{$1$}} (a20);
\draw[->] (a20) to node {} (a19);
\draw[->] (a19) to node {} (a23);
\draw[green,->] (a3) to node {} (a9);
\draw[blue,->] (a9) to node {} (a12);
\draw[green,->] (a12) to node [above] {\textcolor{black}{$t_{r-1}/t_r^2$}} (a16);
\draw[->] (a16) to node {} (a21);
\draw[purple,->] (a21) to node {} (a24);
\draw[purple,->] (a4) to node {} (a7);
\draw[purple,->] (a7) to node {} (a8);
\draw[->] (a8) to node [above] {$t_{r-2}/t_{r-1}$} (a14);
\draw[->] (a14) to node {} (a17);
\draw[green,->] (a17) to node {} (a25);
\draw[->] (a5) to node {} (a6);
\draw[->] (a6) to node [above] {$t_{r-3}/t_{r-2}$} (a15);
\draw[->] (a15) to node {} (a26);
\draw[->] (a6) to node [left] {$1$} (a7);
\draw[purple,->] (a8) to node [left] {\textcolor{black}{$1$}} (a9);
\draw[->] (a14) to node [right] {$c_{r-2}$} (a15);
\draw[green,->] (a16) to node [right] {\textcolor{black}{$c_{r-1}$}} (a17);
\draw[->] (a10) to node [left=0.2cm] {$\sqrt{2}$} (a11);
\draw[purple,->] (a12) to node [right=0.2cm] {\textcolor{black}{$\sqrt{2}$}} (a13);
\draw[->] (a12) to node [above left=0.3cm] {$1$} (a11);
\draw[->] (a18) to node [right=0.3cm] {$\sqrt{2}c_r$} (a19);
\draw[purple,->] (a20) to node [left=0.3cm] {\textcolor{black}{$\sqrt{2}c_r$}} (a21);
\draw[->] (a18) to node [above right=0.2cm] {$c_r^2$} (a21);

\end{tikzpicture} 
\begin{tikzpicture}[
       thick,
       acteur/.style={
         circle,
         thick,
         inner sep=1pt,
         minimum size=0.1cm
       }
] 
\node (a1) at (0,0) [acteur,fill=blue]{}; 
\node (a2) at (0,0.8) [acteur,fill=blue]{};
\node (a3) at (0,1.6) [acteur,fill=blue]{};
\node (a4) at (0,2.4) [acteur,fill=blue]{};
\node (a5) at (0,3.2) [acteur,fill=blue]{}; 
\node (a6) at (0.8,3.2) [acteur,fill=orange]{}; 
\node (a7) at (1.6,2.4) [acteur,fill=black]{}; 
\node (a8) at (2.4,2.4) [acteur,fill=orange]{}; 
\node (a9) at (3.2,1.6) [acteur,fill=black]{}; 
\node (a10) at (4.0,0.8) [acteur,fill=orange]{}; 
\node (a11) at (4.8,0) [acteur,fill=black]{}; 
\node (a12) at (4.0,1.6) [acteur,fill=orange]{}; 
\node (a13) at (4.8,0.8) [acteur,fill=black]{}; 
\node (a14) at (5.6,2.4) [acteur,fill=orange]{}; 
\node (a15) at (6.4,3.2) [acteur,fill=black]{}; 
\node (a16) at (7.2,1.6) [acteur,fill=orange]{}; 
\node (a17) at (8.0,2.4) [acteur,fill=black]{}; 
\node (a18) at (8.8,0) [acteur,fill=orange]{}; 
\node (a19) at (9.6,0.8) [acteur,fill=black]{}; 
\node (a20) at (8.8,0.8) [acteur,fill=orange]{}; 
\node (a21) at (9.6,1.6) [acteur,fill=black]{}; 
\node (a22) at (10.4,0) [acteur,fill=red]{}; 
\node (a23) at (10.4,0.8) [acteur,fill=red]{};
\node (a24) at (10.4,1.6) [acteur,fill=red]{}; 
\node (a25) at (10.4,2.4) [acteur,fill=red]{};
\node (a26) at (10.4,3.2) [acteur,fill=red]{};

\node (b1) at (-0.6,0) {$r$}; 
\node (b2) at (-0.6,0.8) {$r+1$};
\node (b3) at (-0.6,1.6) {$r+2$};
\node (b4) at (-0.6,2.4) {$r+3$}; 
\node (b5) at (-0.6,3.2) {$r+4$}; 
\node (b6) at (11,0) {$r$}; 
\node (b7) at (11,0.8) {$r+1$};
\node (b8) at (11,1.6) {$r+2$};
\node (b9) at (11,2.4) {$r+3$}; 
\node (b10) at (11,3.2) {$r+4$}; 

\draw[->] (a1) to node {} (a11);
\draw[purple,->] (a11) to node [below] {\textcolor{black}{$t_r^2/t_{r-1}$}} (a18);
\draw[->] (a18) to node {} (a22);
\draw[->] (a2) to node {} (a10);
\draw[->] (a10) to node {} (a13);
\draw[->] (a13) to node [above] {$1$} (a20);
\draw[->] (a20) to node {} (a19);
\draw[purple,->] (a19) to node {} (a23);
\draw[green,->] (a3) to node {} (a9);
\draw[blue,->] (a9) to node {} (a12);
\draw[green,->] (a12) to node [above] {\textcolor{black}{$t_{r-1}/t_r^2$}} (a16);
\draw[->] (a16) to node {} (a21);
\draw[->] (a21) to node {} (a24);
\draw[purple,->] (a4) to node {} (a7);
\draw[purple,->] (a7) to node {} (a8);
\draw[->] (a8) to node [above] {$t_{r-2}/t_{r-1}$} (a14);
\draw[->] (a14) to node {} (a17);
\draw[green,->] (a17) to node {} (a25);
\draw[->] (a5) to node {} (a6);
\draw[->] (a6) to node [above] {$t_{r-3}/t_{r-2}$} (a15);
\draw[->] (a15) to node {} (a26);
\draw[->] (a6) to node [left] {$1$} (a7);
\draw[purple,->] (a8) to node [left] {\textcolor{black}{$1$}} (a9);
\draw[->] (a14) to node [right] {$c_{r-2}$} (a15);
\draw[green,->] (a16) to node [right] {\textcolor{black}{$c_{r-1}$}} (a17);
\draw[->] (a10) to node [left=0.2cm] {$\sqrt{2}$} (a11);
\draw[->] (a12) to node [right=0.2cm] {$\sqrt{2}$} (a13);
\draw[purple,->] (a12) to node [above left=0.3cm] {\textcolor{black}{$1$}} (a11);
\draw[purple,->] (a18) to node [right=0.3cm] {\textcolor{black}{$\sqrt{2}c_r$}} (a19);
\draw[->] (a20) to node [left=0.3cm] {$\sqrt{2}c_r$} (a21);
\draw[->] (a18) to node [above right=0.2cm] {$c_r^2$} (a21);

\end{tikzpicture} 
\end{center}
\end{figure}
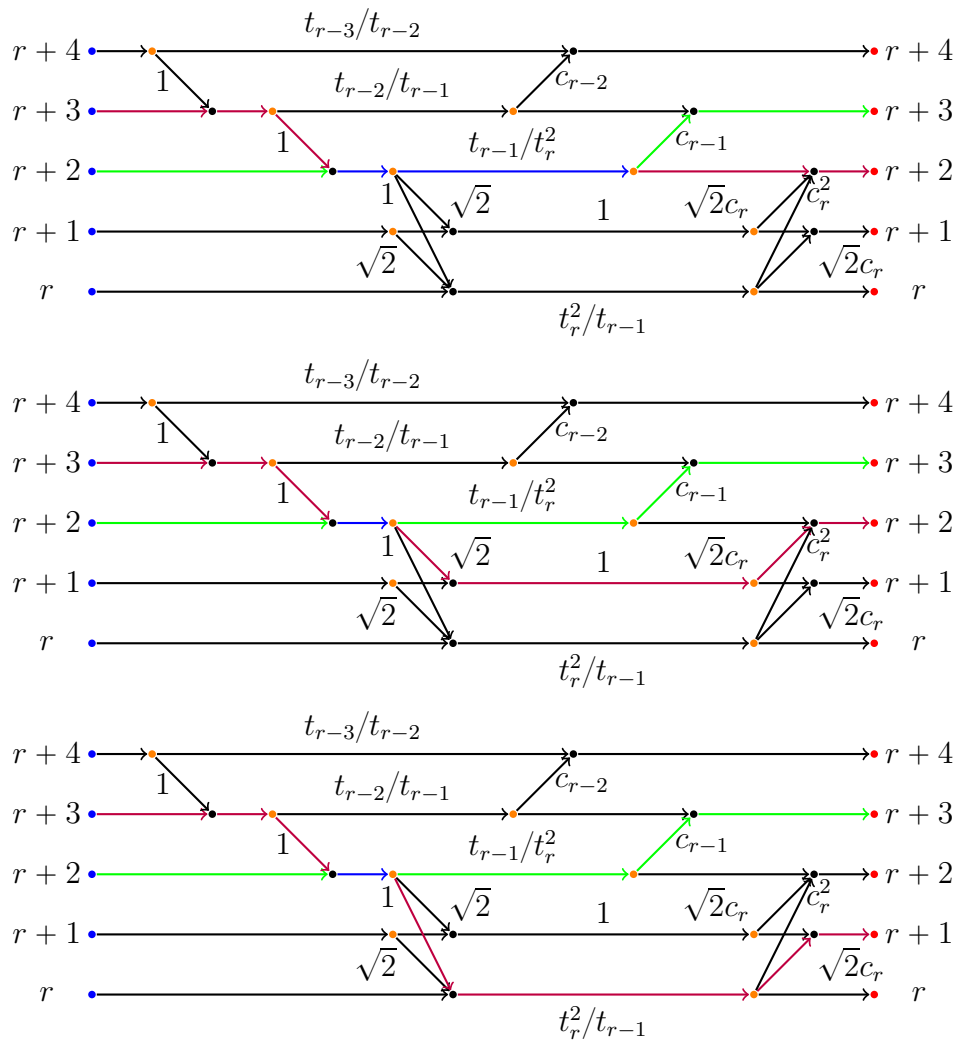
\end{proof}

\begin{figure}[t]
\caption{The paths $\overline{P}_{c,c}(m_2,n_2,p_2)$ and $\overline{P}_{c,c}(m_{\ell},n_{\ell},p_{\ell})$ in $\overline{N}_{c,c}(\mathbf{i})$ as colored in green and purple respectively, with the common edges colored blue, in the case where $(n_{\ell},p_{\ell})=(r+1,r+1)$. The case $(n_{\ell},p_{\ell})=(r,r+1)$ is similar.}
\label{Figure5.14}
\begin{center}
\begin{tikzpicture}[
       thick,
       acteur/.style={
         circle,
         thick,
         inner sep=1pt,
         minimum size=0.1cm
       }
] 

\node (a0) at (1.6,4.0) {}; 
\node (a1) at (0,0) [acteur,fill=blue]{}; 
\node (a2) at (0,0.8) [acteur,fill=blue]{}; 
\node (a3) at (0,1.6) [acteur,fill=blue]{};
\node (a4) at (0,2.4) [acteur,fill=blue]{};
\node (a5) at (0,3.2) [acteur,fill=blue]{};
\node (a6) at (2.4,3.2) [acteur,fill=orange]{}; 
\node (a7) at (3.2,2.4) [acteur,fill=black]{}; 
\node (a8) at (2.4,0.8) [acteur,fill=orange]{}; 
\node (a9) at (3.2,0) [acteur,fill=black]{}; 
\node (a10) at (4.0,1.6) [acteur,fill=orange]{}; 
\node (a11) at (4.8,0.8) [acteur,fill=black]{}; 
\node (a12) at (4.0,2.4) [acteur,fill=orange]{}; 
\node (a13) at (4.8,1.6) [acteur,fill=black]{}; 
\node (a14) at (7.2,2.4) [acteur,fill=orange]{}; 
\node (a15) at (8.0,3.2) [acteur,fill=black]{}; 
\node (a16) at (7.2,0) [acteur,fill=orange]{}; 
\node (a17) at (8.0,0.8) [acteur,fill=black]{}; 
\node (a18) at (8.8,0.8) [acteur,fill=orange]{}; 
\node (a19) at (9.6,1.6) [acteur,fill=black]{}; 
\node (a20) at (8.8,1.6) [acteur,fill=orange]{}; 
\node (a21) at (9.6,2.4) [acteur,fill=black]{}; 
\node (a22) at (10.4,0) [acteur,fill=red]{}; 
\node (a23) at (10.4,0.8) [acteur,fill=red]{}; 
\node (a24) at (10.4,1.6) [acteur,fill=red]{};
\node (a25) at (10.4,2.4) [acteur,fill=red]{}; 
\node (a26) at (10.4,3.2) [acteur,fill=red]{};
\node (a27) at (8.8,4.0) {};

\node (b1) at (-0.6,0) {$r-1$}; 
\node (b2) at (-0.6,0.8) {$r$}; 
\node (b3) at (-0.6,1.6) {$r+1$};
\node (b4) at (-0.6,2.4) {$r+2$};
\node (b5) at (-0.6,3.2) {$r+3$}; 
\node (b6) at (11,0) {$r+1$}; 
\node (b7) at (11,0.8) {$r$}; 
\node (b8) at (11,1.6) {$r+1$};
\node (b9) at (11,2.4) {$r+2$};
\node (b10) at (11,3.2) {$r+3$}; 

\draw[->] (a1) to node {} (a9);
\draw[->] (a9) to node [below] {$t_{r-1}/t_{r-2}$} (a16);
\draw[->] (a16) to node {} (a22);
\draw[->] (a2) to node {} (a8);
\draw[->] (a8) to node {} (a11);
\draw[->] (a11) to node [below] {$t_r^2/t_{r-1}$} (a17);
\draw[->] (a17) to node {} (a18);
\draw[->] (a18) to node {} (a23);
\draw[->] (a3) to node {} (a10);
\draw[->] (a10) to node {} (a13);
\draw[purple,->] (a13) to node [above] {\textcolor{black}{$1$}} (a20);
\draw[purple,->] (a20) to node {} (a19);
\draw[purple,->] (a19) to node {} (a24);
\draw[green,->] (a4) to node {} (a7);
\draw[blue,->] (a7) to node {} (a12);
\draw[green,->] (a12) to node [above] {\textcolor{black}{$t_{r-1}/t_r^2$}} (a14);
\draw[->] (a14) to node {} (a21);
\draw[->] (a21) to node {} (a25);
\draw[->] (a5) to node {} (a6);
\draw[->] (a6) to node [above] {$t_{r-2}/t_{r-1}$} (a15);
\draw[->] (a15) to node {} (a26);
\draw[purple,->] (a6) to node [left] {\textcolor{black}{$1$}} (a7);
\draw[->] (a8) to node [left] {$1$} (a9);
\draw[green,->] (a14) to node [right] {\textcolor{black}{$c_{r-1}$}} (a15);
\draw[->] (a16) to node [right] {$c_{r-1}$} (a17);
\draw[->] (a10) to node [left=0.2cm] {$\sqrt{2}$} (a11);
\draw[purple,->] (a12) to node [right=0.2cm] {\textcolor{black}{$\sqrt{2}$}} (a13);
\draw[->] (a12) to node [above left=0.3cm] {$1$} (a11);
\draw[->] (a18) to node [right=0.3cm] {$\sqrt{2}c_r$} (a19);
\draw[->] (a20) to node [left=0.3cm] {$\sqrt{2}c_r$} (a21);
\draw[->] (a18) to node [above right=0.2cm] {$c_r^2$} (a21);

\draw[purple,dashed,->] (a0) to node {} (a6);
\draw[green,dashed,->] (a15) to node {} (a27);
\end{tikzpicture} 
\end{center}
\end{figure}
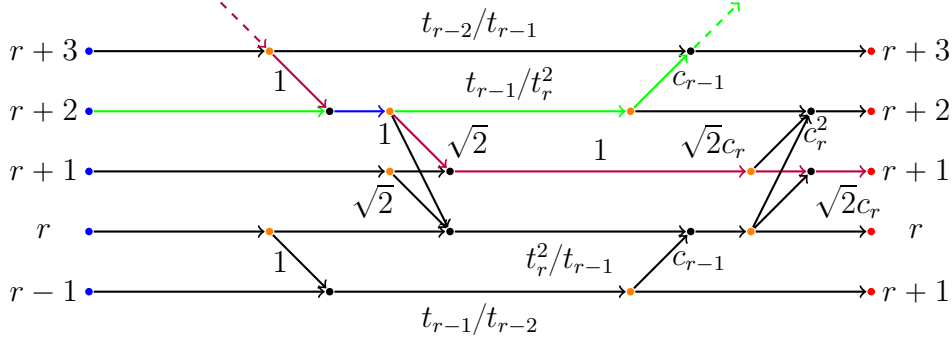

\begin{corollary}\label{5.14}
Let $\overline{\mathcal{P}}=\{(m_1,n_1,p_1),\ldots,(m_j,n_j,p_j)\}$ be a fully mixed non-intersecting set. Then we have either $\overline{\mathcal{P}}=\{(r+1,r,r+2),(r+2,r+1,r+1)\}$ or $\overline{\mathcal{P}}=\{(r+1,r+1,r+2),(r+2,r,r+1)\}$. In particular, we have $\sgn(\overline{\mathcal{P}})=-1$.
\end{corollary}

\begin{proof}
By Lemma \ref{5.13}, we may assume without loss of generality that $r+1=m_1<m_2<\cdots<m_j$. As the elementary chip corresponding to $E_{-(r-1)}$ appears before the elementary chip corresponding to $E_{-r}$ in $\overline{N}_{c,c}(\mathbf{i})$, it follows that we must have either $n_1=r+1$ or $n_1=r$. Next, as the elementary chip corresponding to $E_{r-1}(c_{r-1})$ appears before the elementary chip corresponding to $E_r(c_r)$ in $\overline{N}_{c,c}(\mathbf{i})$, it follows that in both cases, we must have $p_1=r+2$, which forces $m_2=r+2$. 

Next, we will show that $p_2=r+1$. Suppose on the contrary that $p_2\neq r+1$. Then we must have $j>2$, so there must exist some $\ell\in[3,j]$, such that $p_{\ell}=r+1$. As the elementary chip corresponding to $E_{r-1}(c_{r-1})$ appears before the elementary chips corresponding to $E_r(c_r)$ in $\overline{N}_{c,c}(\mathbf{i})$, we necessarily have $n_2=m_2=r+2$, but this would then imply that the paths $\overline{P}_{c,c}(m_2,n_2,p_2)$ and $\overline{P}_{c,c}(m_{\ell},n_{\ell},p_{\ell})$ intersect at level $r+2$ during their descents, as shown in Figure \ref{Figure5.14}, which contradicts the fact that $\overline{\mathcal{P}}$ is non-intersecting. So $p_2=r+1$. Moreover, it follows that if $n_1=r+1$, then $n_2=r$, and if $n_1=r$, then $n_2=r+1$, as shown in Figure \ref{Figure5.15}.

Finally, suppose on the contrary that we have $j>2$. Then $\{(m_3,n_3,p_3),\ldots,(m_j,n_j,p_j)\}$ is also a fully mixed non-intersecting set, but $\{m_3,\ldots,m_j\}\subseteq[r+3,2r+1]$, which contradicts the second part of Lemma \ref{5.12}. Therefore, we have $\overline{\mathcal{P}}=\{(r+1,r,r+2),(r+2,r+1,r+1)\}$ or $\overline{\mathcal{P}}=\{(r+1,r+1,r+2),(r+2,r,r+1)\}$ as desired.
\end{proof}

\begin{figure}[t]
\caption{The paths $\overline{P}_{c,c}(r+1,n_1,r+2)$ and $\overline{P}_{c,c}(r+2,n_2,r+1)$ in $\overline{N}_{c,c}(\mathbf{i})$ as colored in green and purple respectively, in the cases where $(n_1,n_2)=(r+1,r)$ and $(n_1,n_2)=(r,r+1)$ respectively.}
\label{Figure5.15}
\begin{center}
\begin{tikzpicture}[
       thick,
       acteur/.style={
         circle,
         thick,
         inner sep=1pt,
         minimum size=0.1cm
       }
] 
\node (a1) at (0,0) [acteur,fill=blue]{}; 
\node (a2) at (0,0.8) [acteur,fill=blue]{}; 
\node (a3) at (0,1.6) [acteur,fill=blue]{};
\node (a4) at (0,2.4) [acteur,fill=blue]{};
\node (a5) at (0,3.2) [acteur,fill=blue]{};
\node (a6) at (2.4,3.2) [acteur,fill=orange]{}; 
\node (a7) at (3.2,2.4) [acteur,fill=black]{}; 
\node (a8) at (2.4,0.8) [acteur,fill=orange]{}; 
\node (a9) at (3.2,0) [acteur,fill=black]{}; 
\node (a10) at (4.0,1.6) [acteur,fill=orange]{}; 
\node (a11) at (4.8,0.8) [acteur,fill=black]{}; 
\node (a12) at (4.0,2.4) [acteur,fill=orange]{}; 
\node (a13) at (4.8,1.6) [acteur,fill=black]{}; 
\node (a14) at (7.2,2.4) [acteur,fill=orange]{}; 
\node (a15) at (8.0,3.2) [acteur,fill=black]{}; 
\node (a16) at (7.2,0) [acteur,fill=orange]{}; 
\node (a17) at (8.0,0.8) [acteur,fill=black]{}; 
\node (a18) at (8.8,0.8) [acteur,fill=orange]{}; 
\node (a19) at (9.6,1.6) [acteur,fill=black]{}; 
\node (a20) at (8.8,1.6) [acteur,fill=orange]{}; 
\node (a21) at (9.6,2.4) [acteur,fill=black]{}; 
\node (a22) at (10.4,0) [acteur,fill=red]{}; 
\node (a23) at (10.4,0.8) [acteur,fill=red]{}; 
\node (a24) at (10.4,1.6) [acteur,fill=red]{};
\node (a25) at (10.4,2.4) [acteur,fill=red]{}; 
\node (a26) at (10.4,3.2) [acteur,fill=red]{};

\node (b1) at (-0.6,0) {$r-1$}; 
\node (b2) at (-0.6,0.8) {$r$}; 
\node (b3) at (-0.6,1.6) {$r+1$};
\node (b4) at (-0.6,2.4) {$r+2$};
\node (b5) at (-0.6,3.2) {$r+3$}; 
\node (b6) at (11,0) {$r+1$}; 
\node (b7) at (11,0.8) {$r$}; 
\node (b8) at (11,1.6) {$r+1$};
\node (b9) at (11,2.4) {$r+2$};
\node (b10) at (11,3.2) {$r+3$}; 

\draw[->] (a1) to node {} (a9);
\draw[->] (a9) to node [below] {$t_{r-1}/t_{r-2}$} (a16);
\draw[->] (a16) to node {} (a22);
\draw[->] (a2) to node {} (a8);
\draw[->] (a8) to node {} (a11);
\draw[purple,->] (a11) to node [below] {\textcolor{black}{$t_r^2/t_{r-1}$}} (a17);
\draw[purple,->] (a17) to node {} (a18);
\draw[->] (a18) to node {} (a23);
\draw[green,->] (a3) to node {} (a10);
\draw[green,->] (a10) to node {} (a13);
\draw[green,->] (a13) to node [above] {\textcolor{black}{$1$}} (a20);
\draw[->] (a20) to node {} (a19);
\draw[purple,->] (a19) to node {} (a24);
\draw[purple,->] (a4) to node {} (a7);
\draw[purple,->] (a7) to node {} (a12);
\draw[->] (a12) to node [above] {$t_{r-1}/t_r^2$} (a14);
\draw[->] (a14) to node {} (a21);
\draw[green,->] (a21) to node {} (a25);
\draw[->] (a5) to node {} (a6);
\draw[->] (a6) to node [above] {$t_{r-2}/t_{r-1}$} (a15);
\draw[->] (a15) to node {} (a26);
\draw[->] (a6) to node [left] {$1$} (a7);
\draw[->] (a8) to node [left] {$1$} (a9);
\draw[->] (a14) to node [right] {$c_{r-1}$} (a15);
\draw[->] (a16) to node [right] {$c_{r-1}$} (a17);
\draw[->] (a10) to node [left=0.2cm] {$\sqrt{2}$} (a11);
\draw[->] (a12) to node [right=0.2cm] {$\sqrt{2}$} (a13);
\draw[purple,->] (a12) to node [above left=0.3cm] {\textcolor{black}{$1$}} (a11);
\draw[purple,->] (a18) to node [right=0.3cm] {\textcolor{black}{$\sqrt{2}c_r$}} (a19);
\draw[green,->] (a20) to node [left=0.3cm] {\textcolor{black}{$\sqrt{2}c_r$}} (a21);
\draw[->] (a18) to node [above right=0.2cm] {$c_r^2$} (a21);
\end{tikzpicture} 
\begin{tikzpicture}[
       thick,
       acteur/.style={
         circle,
         thick,
         inner sep=1pt,
         minimum size=0.1cm
       }
] 
\node (a1) at (0,0) [acteur,fill=blue]{}; 
\node (a2) at (0,0.8) [acteur,fill=blue]{}; 
\node (a3) at (0,1.6) [acteur,fill=blue]{};
\node (a4) at (0,2.4) [acteur,fill=blue]{};
\node (a5) at (0,3.2) [acteur,fill=blue]{};
\node (a6) at (2.4,3.2) [acteur,fill=orange]{}; 
\node (a7) at (3.2,2.4) [acteur,fill=black]{}; 
\node (a8) at (2.4,0.8) [acteur,fill=orange]{}; 
\node (a9) at (3.2,0) [acteur,fill=black]{}; 
\node (a10) at (4.0,1.6) [acteur,fill=orange]{}; 
\node (a11) at (4.8,0.8) [acteur,fill=black]{}; 
\node (a12) at (4.0,2.4) [acteur,fill=orange]{}; 
\node (a13) at (4.8,1.6) [acteur,fill=black]{}; 
\node (a14) at (7.2,2.4) [acteur,fill=orange]{}; 
\node (a15) at (8.0,3.2) [acteur,fill=black]{}; 
\node (a16) at (7.2,0) [acteur,fill=orange]{}; 
\node (a17) at (8.0,0.8) [acteur,fill=black]{}; 
\node (a18) at (8.8,0.8) [acteur,fill=orange]{}; 
\node (a19) at (9.6,1.6) [acteur,fill=black]{}; 
\node (a20) at (8.8,1.6) [acteur,fill=orange]{}; 
\node (a21) at (9.6,2.4) [acteur,fill=black]{}; 
\node (a22) at (10.4,0) [acteur,fill=red]{}; 
\node (a23) at (10.4,0.8) [acteur,fill=red]{}; 
\node (a24) at (10.4,1.6) [acteur,fill=red]{};
\node (a25) at (10.4,2.4) [acteur,fill=red]{}; 
\node (a26) at (10.4,3.2) [acteur,fill=red]{};

\node (b1) at (-0.6,0) {$r-1$}; 
\node (b2) at (-0.6,0.8) {$r$}; 
\node (b3) at (-0.6,1.6) {$r+1$};
\node (b4) at (-0.6,2.4) {$r+2$};
\node (b5) at (-0.6,3.2) {$r+3$}; 
\node (b6) at (11,0) {$r+1$}; 
\node (b7) at (11,0.8) {$r$}; 
\node (b8) at (11,1.6) {$r+1$};
\node (b9) at (11,2.4) {$r+2$};
\node (b10) at (11,3.2) {$r+3$}; 

\draw[->] (a1) to node {} (a9);
\draw[->] (a9) to node [below] {$t_{r-1}/t_{r-2}$} (a16);
\draw[->] (a16) to node {} (a22);
\draw[->] (a2) to node {} (a8);
\draw[->] (a8) to node {} (a11);
\draw[green,->] (a11) to node [below] {\textcolor{black}{$t_r^2/t_{r-1}$}} (a17);
\draw[green,->] (a17) to node {} (a18);
\draw[->] (a18) to node {} (a23);
\draw[green,->] (a3) to node {} (a10);
\draw[->] (a10) to node {} (a13);
\draw[purple,->] (a13) to node [above] {\textcolor{black}{$1$}} (a20);
\draw[purple,->] (a20) to node {} (a19);
\draw[purple,->] (a19) to node {} (a24);
\draw[purple,->] (a4) to node {} (a7);
\draw[purple,->] (a7) to node {} (a12);
\draw[->] (a12) to node [above] {$t_{r-1}/t_r^2$} (a14);
\draw[->] (a14) to node {} (a21);
\draw[green,->] (a21) to node {} (a25);
\draw[->] (a5) to node {} (a6);
\draw[->] (a6) to node [above] {$t_{r-2}/t_{r-1}$} (a15);
\draw[->] (a15) to node {} (a26);
\draw[->] (a6) to node [left] {$1$} (a7);
\draw[->] (a8) to node [left] {$1$} (a9);
\draw[->] (a14) to node [right] {$c_{r-1}$} (a15);
\draw[->] (a16) to node [right] {$c_{r-1}$} (a17);
\draw[green,->] (a10) to node [left=0.2cm] {\textcolor{black}{$\sqrt{2}$}} (a11);
\draw[purple,->] (a12) to node [right=0.2cm] {\textcolor{black}{$\sqrt{2}$}} (a13);
\draw[->] (a12) to node [above left=0.3cm] {$1$} (a11);
\draw[->] (a18) to node [right=0.3cm] {$\sqrt{2}c_r$} (a19);
\draw[->] (a20) to node [left=0.3cm] {$\sqrt{2}c_r$} (a21);
\draw[green,->] (a18) to node [above right=0.2cm] {\textcolor{black}{$c_r^2$}} (a21);
\end{tikzpicture} 
\end{center}
\end{figure}
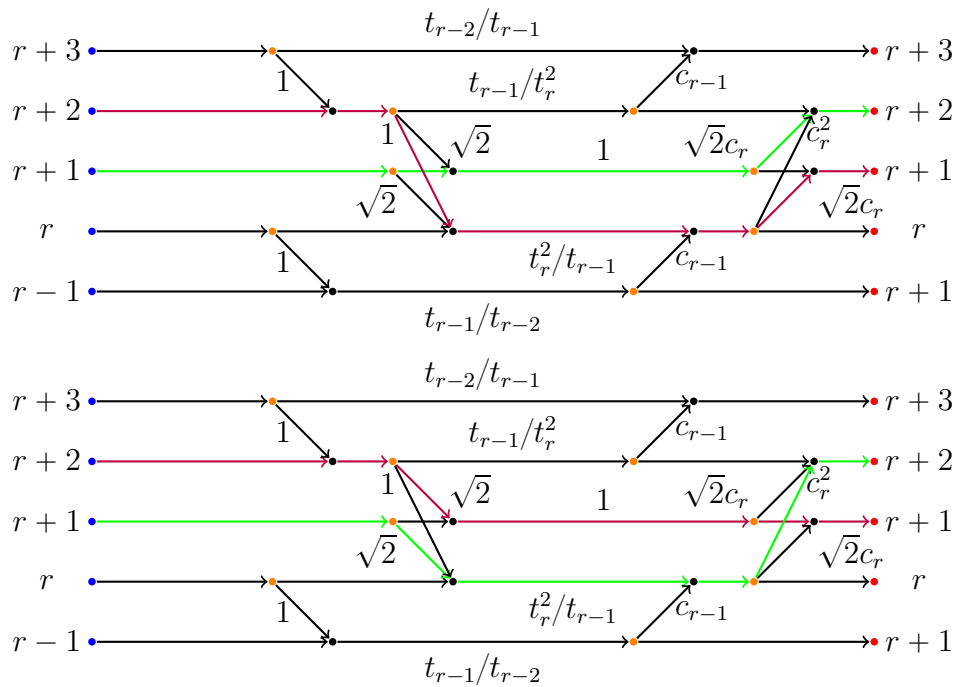

\begin{lemma}\label{5.15}
~
\begin{enumerate}
\item The triples $(r+1,r,r+1)$ and $(r+2,r+1,r+2)$ are both admissible $(c,c)$-triples, and the paths $\overline{P}_{c,c}(r+1,r,r+1)$ and $\overline{P}_{c,c}(r+2,r+1,r+2)$ are non-intersecting, as shown in Figure \ref{Figure5.16}.
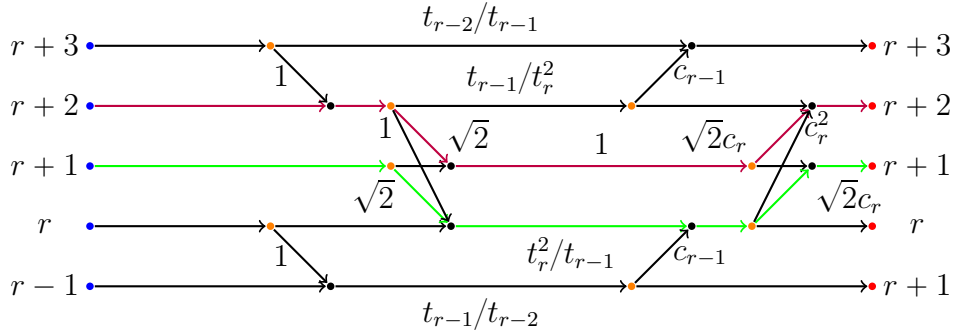
\begin{figure}[t]
\caption{The paths $\overline{P}_{c,c}(r+1,r,r+1)$ and $\overline{P}_{c,c}(r+2,r+1,r+2)$ in $\overline{N}_{c,c}(\mathbf{i})$ as colored in green and purple respectively.}
\label{Figure5.16}
\begin{center}
\begin{tikzpicture}[
       thick,
       acteur/.style={
         circle,
         thick,
         inner sep=1pt,
         minimum size=0.1cm
       }
] 
\node (a1) at (0,0) [acteur,fill=blue]{}; 
\node (a2) at (0,0.8) [acteur,fill=blue]{}; 
\node (a3) at (0,1.6) [acteur,fill=blue]{};
\node (a4) at (0,2.4) [acteur,fill=blue]{};
\node (a5) at (0,3.2) [acteur,fill=blue]{};
\node (a6) at (2.4,3.2) [acteur,fill=orange]{}; 
\node (a7) at (3.2,2.4) [acteur,fill=black]{}; 
\node (a8) at (2.4,0.8) [acteur,fill=orange]{}; 
\node (a9) at (3.2,0) [acteur,fill=black]{}; 
\node (a10) at (4.0,1.6) [acteur,fill=orange]{}; 
\node (a11) at (4.8,0.8) [acteur,fill=black]{}; 
\node (a12) at (4.0,2.4) [acteur,fill=orange]{}; 
\node (a13) at (4.8,1.6) [acteur,fill=black]{}; 
\node (a14) at (7.2,2.4) [acteur,fill=orange]{}; 
\node (a15) at (8.0,3.2) [acteur,fill=black]{}; 
\node (a16) at (7.2,0) [acteur,fill=orange]{}; 
\node (a17) at (8.0,0.8) [acteur,fill=black]{}; 
\node (a18) at (8.8,0.8) [acteur,fill=orange]{}; 
\node (a19) at (9.6,1.6) [acteur,fill=black]{}; 
\node (a20) at (8.8,1.6) [acteur,fill=orange]{}; 
\node (a21) at (9.6,2.4) [acteur,fill=black]{}; 
\node (a22) at (10.4,0) [acteur,fill=red]{}; 
\node (a23) at (10.4,0.8) [acteur,fill=red]{}; 
\node (a24) at (10.4,1.6) [acteur,fill=red]{};
\node (a25) at (10.4,2.4) [acteur,fill=red]{}; 
\node (a26) at (10.4,3.2) [acteur,fill=red]{};

\node (b1) at (-0.6,0) {$r-1$}; 
\node (b2) at (-0.6,0.8) {$r$}; 
\node (b3) at (-0.6,1.6) {$r+1$};
\node (b4) at (-0.6,2.4) {$r+2$};
\node (b5) at (-0.6,3.2) {$r+3$}; 
\node (b6) at (11,0) {$r+1$}; 
\node (b7) at (11,0.8) {$r$}; 
\node (b8) at (11,1.6) {$r+1$};
\node (b9) at (11,2.4) {$r+2$};
\node (b10) at (11,3.2) {$r+3$}; 

\draw[->] (a1) to node {} (a9);
\draw[->] (a9) to node [below] {$t_{r-1}/t_{r-2}$} (a16);
\draw[->] (a16) to node {} (a22);
\draw[->] (a2) to node {} (a8);
\draw[->] (a8) to node {} (a11);
\draw[green,->] (a11) to node [below] {\textcolor{black}{$t_r^2/t_{r-1}$}} (a17);
\draw[green,->] (a17) to node {} (a18);
\draw[->] (a18) to node {} (a23);
\draw[green,->] (a3) to node {} (a10);
\draw[->] (a10) to node {} (a13);
\draw[purple,->] (a13) to node [above] {\textcolor{black}{$1$}} (a20);
\draw[->] (a20) to node {} (a19);
\draw[green,->] (a19) to node {} (a24);
\draw[purple,->] (a4) to node {} (a7);
\draw[purple,->] (a7) to node {} (a12);
\draw[->] (a12) to node [above] {$t_{r-1}/t_r^2$} (a14);
\draw[->] (a14) to node {} (a21);
\draw[purple,->] (a21) to node {} (a25);
\draw[->] (a5) to node {} (a6);
\draw[->] (a6) to node [above] {$t_{r-2}/t_{r-1}$} (a15);
\draw[->] (a15) to node {} (a26);
\draw[->] (a6) to node [left] {$1$} (a7);
\draw[->] (a8) to node [left] {$1$} (a9);
\draw[->] (a14) to node [right] {$c_{r-1}$} (a15);
\draw[->] (a16) to node [right] {$c_{r-1}$} (a17);
\draw[green,->] (a10) to node [left=0.2cm] {\textcolor{black}{$\sqrt{2}$}} (a11);
\draw[purple,->] (a12) to node [right=0.2cm] {\textcolor{black}{$\sqrt{2}$}} (a13);
\draw[->] (a12) to node [above left=0.3cm] {$1$} (a11);
\draw[green,->] (a18) to node [right=0.3cm] {\textcolor{black}{$\sqrt{2}c_r$}} (a19);
\draw[purple,->] (a20) to node [left=0.3cm] {\textcolor{black}{$\sqrt{2}c_r$}} (a21);
\draw[->] (a18) to node [above right=0.2cm] {$c_r^2$} (a21);
\end{tikzpicture} 
\end{center}
\end{figure}
\item Let
{\allowdisplaybreaks
\begin{align}
\overline{\mathcal{P}}_1&=\{(r+1,r,r+2),(r+2,r+1,r+1)\},\label{eq:5.18}\\
\overline{\mathcal{P}}_2&=\{(r+1,r+1,r+2),(r+2,r,r+1)\},\label{eq:5.19}\\
\overline{\mathcal{P}}_3&=\{(r+1,r,r+1),(r+2,r+1,r+2)\},\label{eq:5.20}.
\end{align}
Then we have
}
\begin{equation}\label{eq:5.21}
\wt_{c,c}(\overline{\mathcal{P}}_1)=2c_r^2t_r^2/t_{r-1},\quad
\wt_{c,c}(\overline{\mathcal{P}}_2)=2c_r^2t_r^2/t_{r-1},\quad
\wt_{c,c}(\overline{\mathcal{P}}_3)=4c_r^2t_r^2/t_{r-1}
\end{equation}
In particular, we have
\begin{equation}\label{eq:5.22}
\sum_{i=1}^3\sgn(\overline{\mathcal{P}}_i)\wt_{c,c}(\overline{\mathcal{P}}_i)=0.
\end{equation}
\item For any admissible $(c,c)$-triple $(m,n,p)\notin\overline{\mathcal{P}}_1\cup\overline{\mathcal{P}}_2\cup\overline{\mathcal{P}}_3$, the following statements are equivalent:
\begin{enumerate}
\item $\overline{\mathcal{P}}_1\cup\{(m,n,p)\}$ is non-intersecting.
\item $\overline{\mathcal{P}}_2\cup\{(m,n,p)\}$ is non-intersecting.
\item $\overline{\mathcal{P}}_3\cup\{(m,n,p)\}$ is non-intersecting.
\end{enumerate}
\end{enumerate}
\end{lemma}

The proof of Lemma \ref{5.15} follows from observing Figures \ref{Figure5.15} and \ref{Figure5.16}, and the proof shall be omitted.

As a consequence of Corollary \ref{5.14} and Lemma \ref{5.15}, it follows that \eqref{eq:5.2} reduces to:
\begin{equation}\label{eq:5.23}
H_j^{c,c}=f_j^{c,c}=\sum_{\substack{I\subseteq[1,2r+1],\\|I|=j}}\sum_{\substack{\mathcal{P}\in\mathcal{P}_{\nonint}^{c,c}(I)\\\{(r+1,r),(r+2,r+1)\}\not\subseteq\mathcal{P}}}\wt_{c,c}(\mathcal{P}).
\end{equation}

In view of Corollary \ref{5.14} and Lemma \ref{5.15}, our next step is to characterize the admissible $(c,c)$-pairs $(m,n)$, as well as when two paths $P_{c,c}(m,n)$ and $P_{c,c}(m',n')$ corresponding to two different admissible $(c,c)$-pairs $(m,n)$ and $(m',n')$ intersect.

\begin{proposition}\label{5.16}
~
\begin{enumerate}
\item The admissible $(c,c)$-pairs $(m,n)$ are of the form $(i,i)$, $(j+1,j)$ and $(r+2,r)$, where $i\in[1,2r+1]$ and $j\in[1,2r]$. 
\item For all $i\in[1,r]\cup[r+2,2r+1]$ and $j\in[1,2r]$, the weights of the paths $P_{2i-1}:=P_{c,c}(i,i)$, $P_{2j}:=P_{c,c}(j+1,j)$, $P_{(2r+1,1)}:=P_{c,c}(r+1,r+1)$ and $P_{(2r+1,2)}:=P_{c,c}(r+2,r)$ are given by
{\allowdisplaybreaks
\begin{align}
\wt(P_{2k-1})&=\frac{t_k}{t_{k-1}},\quad\wt(P_{4r+3-2k})=\wt(P_{2k-1})^{-1}=\frac{t_{k-1}}{t_k},\quad k\in[1,r-1],\label{eq:5.24}\\
\wt(P_{2r-1})&=\frac{t_r^2}{t_{r-1}},\quad\wt(P_{2r+3})=\wt(P_{2k-1})^{-1}=\frac{t_{r-1}}{t_r^2},\label{eq:5.25}\\
\wt(P_{2\ell})&=\frac{c_{\ell}t_{\ell}}{t_{\ell-1}},\quad\ell\in[1,r-1],\label{eq:5.26}\\
\wt(P_{4r+2-2\ell})&=\frac{c_{\ell}t_{\ell}}{t_{\ell+1}},\quad\ell\in[1,r-2],\label{eq:5.27}\\
\wt(P_{2r})&=\frac{2c_rt_r^2}{t_{r-1}},\quad\wt(P_{2r+2})=2c_r,\quad\wt(P_{2r+4})=\frac{c_{r-1}t_{r-1}}{t_r^2},\label{eq:5.28}\\
\wt(P_{(2r+1,1)})&=1,\quad\wt(P_{(2r+1,2)})=\frac{c_r^2t_r^2}{t_{r-1}},\label{eq:5.29}
\end{align}
where $t_0=1$. In particular, the weights $y_i:=\rho_{\tau}^*(\wt(P_i))$, $y_{(2r+1,j)}:=\rho_{\tau}^*(\wt(P_{(2r+1,j)}))$, written in terms of normalized $D_{r+1}^{(2)}$ $Q$-system variables using \eqref{eq:4.46} and \eqref{eq:4.47}, are given by
}
\begin{align}
y_{2k-1}&=\frac{R_{k,1}R_{k-1,0}}{R_{k,0}R_{k-1,1}},\quad y_{4r+3-2k}=\frac{R_{k,0}R_{k-1,1}}{R_{k,1}R_{k-1,0}},\quad k\in[1,r-1],\label{eq:5.30}\\
y_{2r-1}&=\frac{R_{r,1}^2R_{r-1,0}}{R_{r,0}^2R_{r-1,1}},\quad y_{2r+3}=\frac{R_{r,0}^2R_{r-1,1}}{R_{r,1}^2R_{r-1,0}},\label{eq:5.31}\\
y_{2\ell}&=\frac{R_{\ell-1,0}R_{\ell+1,1}}{R_{\ell,0}R_{\ell,1}},\quad y_{4r+2-2\ell}=\frac{R_{\ell-1,1}R_{\ell+1,0}}{R_{\ell,0}R_{\ell,1}},\quad\ell\in[1,r-2],\label{eq:5.32}\\
y_{2r-2}&=\frac{R_{r-2,0}R_{r,1}^2}{R_{r-1,0}R_{r-1,1}},\quad y_{2r}=\frac{2R_{r-1,0}}{R_{r,0}^2},\quad y_{2r+2}=\frac{2R_{r-1,1}}{R_{r,1}^2},\quad y_{2r+4}=\frac{R_{r-2,1}R_{r,0}^2}{R_{r-1,0}R_{r-1,1}},\label{eq:5.33}\\
y_{(2r+1,1)}&=1,\quad y_{(2r+1,2)}=\frac{R_{r-1,0}R_{r-1,1}}{R_{r,0}^2R_{r,1}^2},\label{eq:5.34}
\end{align}
where $R_{0,k}=1$ for all $k\in\mathbb{Z}$.
\item The paths $P_{2i-1}$ and $P_{2i'-1}$ do not intersect for any distinct $i,i'\in[1,r]\cup[r+2,2r+1]$, and the paths $P_{2i-1}$ and $P_{(2r+1,1)}$ do not intersect for any $i\in[1,r]\cup[r+2,2r+1]$.
\item For any $j\in[1,r-2]$, the path $P_{2j}$ intersects $P_k$ if and only if $k=2j\pm1,2j\pm2$, and the path $P_{4r+2-2j}$ intersects $P_k$ if and only if $k=4r-2j,4r+1-2j,4r+3-2j,4r+4-2j$.
\item The path $P_{2r-2}$ intersects $P_k$ if and only if $k=2r-4,2r-3,2r-1,2r,(2r+1,2)$.
\item The path $P_{2r+4}=P_{c,c}(r+3,r+2)$ intersects $P_k$ if and only if $k=2r+2,2r+3,2r+5,2r+6,(2r+1,2)$.
\item The path $P_{2r}$ intersects $P_k$ if and only if $k=2r-2,2r-1,(2r+1,1),(2r+1,2)$.
\item The path $P_{2r+2}$ intersects $P_k$ if and only if $k=2r+3,2r+4,(2r+1,1),(2r+1,2)$.
\item The path $P_{(2r+1,2)}$ intersects $P_k$ if and only if $k=2r-2,2r-1,2r,2r+2,2r+3,2r+4$.
\end{enumerate}
\end{proposition} 

\begin{figure}[t]
\caption{The network diagram $\overline{N}_{c,c}(\mathbf{i})$ involving levels $r-2$, $r-1$, $r$, $r+1$, $r+2$, $r+3$ and $r+4$.}
\label{Figure5.17}
\begin{center}
\begin{tikzpicture}[
       thick,
       acteur/.style={
         circle,
         thick,
         inner sep=1pt,
         minimum size=0.1cm
       }
] 
\node (a1) at (0,0) [acteur,fill=blue]{};
\node (a2) at (0,0.8) [acteur,fill=blue]{}; 
\node (a3) at (0,1.6) [acteur,fill=blue]{}; 
\node (a4) at (0,2.4) [acteur,fill=blue]{};
\node (a5) at (0,3.2) [acteur,fill=blue]{};
\node (a6) at (0,4.0) [acteur,fill=blue]{};
\node (a7) at (0,4.8) [acteur,fill=blue]{}; 
\node (a8) at (0.8,4.8) [acteur,fill=orange]{}; 
\node (a9) at (1.6,4.0) [acteur,fill=black]{}; 
\node (a10) at (0.8,0.8) [acteur,fill=orange]{}; 
\node (a11) at (1.6,0) [acteur,fill=black]{}; 
\node (a12) at (2.4,4.0) [acteur,fill=orange]{}; 
\node (a13) at (3.2,3.2) [acteur,fill=black]{}; 
\node (a14) at (2.4,1.6) [acteur,fill=orange]{}; 
\node (a15) at (3.2,0.8) [acteur,fill=black]{}; 
\node (a16) at (4.0,2.4) [acteur,fill=orange]{}; 
\node (a17) at (4.8,1.6) [acteur,fill=black]{}; 
\node (a18) at (4.0,3.2) [acteur,fill=orange]{}; 
\node (a19) at (4.8,2.4) [acteur,fill=black]{}; 
\node (a20) at (5.6,4.0) [acteur,fill=orange]{}; 
\node (a21) at (6.4,4.8) [acteur,fill=black]{}; 
\node (a22) at (5.6,0) [acteur,fill=orange]{}; 
\node (a23) at (6.4,0.8) [acteur,fill=black]{}; 
\node (a24) at (7.2,3.2) [acteur,fill=orange]{}; 
\node (a25) at (8.0,4.0) [acteur,fill=black]{}; 
\node (a26) at (7.2,0.8) [acteur,fill=orange]{}; 
\node (a27) at (8.0,1.6) [acteur,fill=black]{}; 
\node (a28) at (8.8,1.6) [acteur,fill=orange]{}; 
\node (a29) at (9.6,2.4) [acteur,fill=black]{}; 
\node (a30) at (8.8,2.4) [acteur,fill=orange]{}; 
\node (a31) at (9.6,3.2) [acteur,fill=black]{}; 
\node (a32) at (10.4,0) [acteur,fill=red]{};
\node (a33) at (10.4,0.8) [acteur,fill=red]{}; 
\node (a34) at (10.4,1.6) [acteur,fill=red]{}; 
\node (a35) at (10.4,2.4) [acteur,fill=red]{};
\node (a36) at (10.4,3.2) [acteur,fill=red]{}; 
\node (a37) at (10.4,4.0) [acteur,fill=red]{};
\node (a38) at (10.4,4.8) [acteur,fill=red]{};

\node (b1) at (-0.6,0) {$r-2$};
\node (b2) at (-0.6,0.8) {$r-1$}; 
\node (b3) at (-0.6,1.6) {$r$}; 
\node (b4) at (-0.6,2.4) {$r+1$};
\node (b5) at (-0.6,3.2) {$r+2$};
\node (b6) at (-0.6,4.0) {$r+3$}; 
\node (b7) at (-0.6,4.8) {$r+4$}; 
\node (b8) at (11,0) {$r-2$};
\node (b9) at (11,0.8) {$r-1$}; 
\node (b10) at (11,1.6) {$r$}; 
\node (b11) at (11,2.4) {$r+1$};
\node (b12) at (11,3.2) {$r+2$};
\node (b13) at (11,4.0) {$r+3$}; 
\node (b14) at (11,4.8) {$r+4$}; 

\draw[->] (a1) to node {} (a11);
\draw[->] (a11) to node [below] {$t_{r-2}/t_{r-3}$} (a22);
\draw[->] (a22) to node {} (a32);
\draw[->] (a2) to node {} (a10);
\draw[->] (a10) to node {} (a15);
\draw[->] (a15) to node [below] {$t_{r-1}/t_{r-2}$} (a23);
\draw[->] (a23) to node {} (a26);
\draw[->] (a26) to node {} (a33);
\draw[->] (a3) to node {} (a14);
\draw[->] (a14) to node {} (a17);
\draw[->] (a17) to node [below] {$t_r^2/t_{r-1}$} (a27);
\draw[->] (a27) to node {} (a28);
\draw[->] (a28) to node {} (a34);
\draw[->] (a4) to node {} (a16);
\draw[->] (a16) to node {} (a19);
\draw[->] (a19) to node [above] {$1$} (a30);
\draw[->] (a30) to node {} (a29);
\draw[->] (a29) to node {} (a35);
\draw[->] (a5) to node {} (a13);
\draw[->] (a13) to node {} (a18);
\draw[->] (a18) to node [above] {$t_{r-1}/t_r^2$} (a24);
\draw[->] (a24) to node {} (a31);
\draw[->] (a31) to node {} (a36);
\draw[->] (a6) to node {} (a9);
\draw[->] (a9) to node {} (a12);
\draw[->] (a12) to node [above] {$t_{r-2}/t_{r-1}$} (a20);
\draw[->] (a20) to node {} (a25);
\draw[->] (a25) to node {} (a37);
\draw[->] (a7) to node {} (a8);
\draw[->] (a8) to node [above] {$t_{r-3}/t_{r-2}$} (a21);
\draw[->] (a21) to node {} (a38);
\draw[->] (a8) to node [left] {$1$} (a9);
\draw[->] (a10) to node [left] {$1$} (a11);
\draw[->] (a12) to node [left] {$1$} (a13);
\draw[->] (a14) to node [left] {$1$} (a15);
\draw[->] (a20) to node [right] {$c_1$} (a21);
\draw[->] (a22) to node [right] {$c_1$} (a23);
\draw[->] (a24) to node [right] {$c_2$} (a25);
\draw[->] (a26) to node [right] {$c_2$} (a27);
\draw[->] (a16) to node [left=0.2cm] {$\sqrt{2}$} (a17);
\draw[->] (a18) to node [right=0.2cm] {$\sqrt{2}$} (a19);
\draw[->] (a18) to node [above left=0.3cm] {$1$} (a17);
\draw[->] (a28) to node [right=0.3cm] {$\sqrt{2}c_r$} (a29);
\draw[->] (a30) to node [left=0.3cm] {$\sqrt{2}c_r$} (a31);
\draw[->] (a28) to node [above right=0.2cm] {$c_r^2$} (a31);

\end{tikzpicture} 
\end{center}
\end{figure}
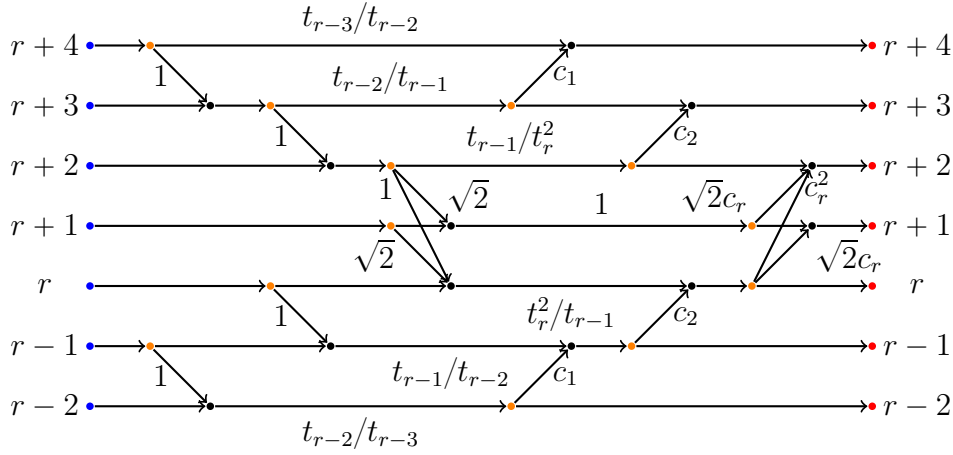

\begin{proof}
By observing Figure \ref{Figure5.12}, we have that if $(m,n)$ is an admissible $(c,c)$-pair, and either $m\leq r$ or $m\geq r+3$, then we have either $m=n$ or $m=n+1$. Thus, it remains to consider the case where either $m=r+1$ or $m=r+2$. If $m=r+1$, then since the elementary chip corresponding to $E_{-(r-1)}$ appears before the elementary chip corresponding to $E_{-r}$ in $\overline{N}_{c,c}(\mathbf{i})$, it follows that we must have either $n=r+1$ or $n=r$. By a similar reasoning, it follows that if $m=r+2$, then we must have either $n=r+2$, $n=r+1$ or $n=r$, and this proves statement (1). Statement (2) follows from the definition of path weights, while statement (3) follows from the fact that the paths involved do not involve any ascents and descents. Statement (4) follow from observing Figures \ref{Figure5.6} and \ref{Figure5.11}, while modifying the labels that appear in Figure \ref{Figure5.6} to adapt to type $B$ where necessary. Statements (5)--(9) follow from observing Figure \ref{Figure5.17}. We will leave the details to the reader.

\end{proof}

Motivated by \eqref{eq:5.23} and statements (3)--(9) of Proposition \ref{5.16}, let us define a graph $\mathcal{G}_B$, where the graph $\mathcal{G}_B$ is given as follows:

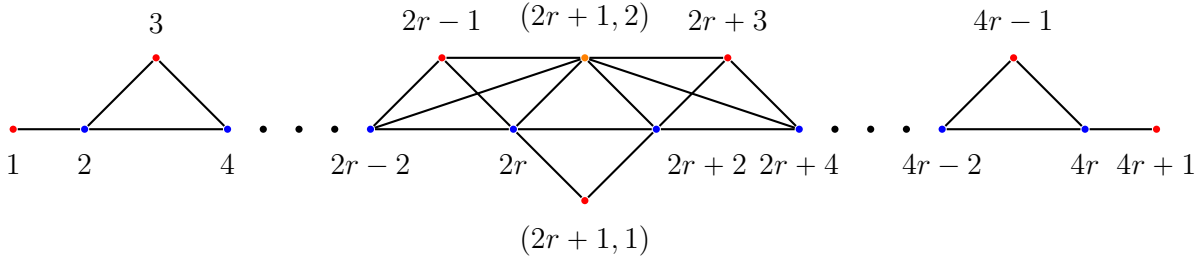
\begin{center}
\begin{tikzpicture}[
scale=0.95, transform shape,
       thick,
       acteur/.style={
         circle,
         thick,
         inner sep=1pt,
         minimum size=0.1cm
       }
] 

\node (a1) at (0,0) [acteur,fill=red]{};
\node (a2) at (1,0) [acteur,fill=blue]{};
\node (a3) at (2,1) [acteur,fill=red]{};
\node (a4) at (3,0) [acteur,fill=blue]{};
\node (d1) at (3.5,0) [acteur,fill=black]{};
\node (d2) at (4,0) [acteur,fill=black]{};
\node (d3) at (4.5,0) [acteur,fill=black]{};
\node (a2rm2) at (5,0) [acteur,fill=blue]{};
\node (a2rm1) at (6,1) [acteur,fill=red]{};
\node (a2r) at (7,0) [acteur,fill=blue]{};
\node (a2rp1c1) at (8,-1) [acteur,fill=red]{};
\node (a2rp1c2) at (8,1) [acteur,fill=orange]{};
\node (a2rp2) at (9,0) [acteur,fill=blue]{};
\node (a2rp3) at (10,1) [acteur,fill=red]{};
\node (a2rp4) at (11,0) [acteur,fill=blue]{};
\node (d4) at (11.5,0) [acteur,fill=black]{};
\node (d5) at (12,0) [acteur,fill=black]{};
\node (d6) at (12.5,0) [acteur,fill=black]{};
\node (a4rm2) at (13,0) [acteur,fill=blue]{};
\node (a4rm1) at (14,1) [acteur,fill=red]{};
\node (a4r) at (15,0) [acteur,fill=blue]{};
\node (a4rp1) at (16,0) [acteur,fill=red]{};

\node (b1) at (0,-0.2) [below] {$1$};
\node (b2) at (1,-0.2) [below] {$2$};
\node (b3) at (2,1.2) [above] {$3$};
\node (b4) at (3,-0.2) [below] {$4$};
\node (b2rm2) at (5,-0.2) [below] {$2r-2$};
\node (b2rm1) at (6,1.2) [above] {$2r-1$};
\node (b2r) at (7,-0.2) [below] {$2r$};
\node (b2rp1c1) at (8,-1.2) [below] {$(2r+1,1)$};
\node (b2rp1c2) at (8,1.2) [above] {$(2r+1,2)$};
\node (b2rp2) at (9,-0.2) [below right] {$2r+2$};
\node (b2rp3) at (10,1.2) [above] {$2r+3$};
\node (b2rp4) at (11,-0.2) [below] {$2r+4$};
\node (b4rm2) at (13,-0.2) [below] {$4r-2$};
\node (b4rm1) at (14,1.2) [above] {$4r-1$};
\node (b4r) at (15,-0.2) [below] {$4r$};
\node (b4rp1) at (16,-0.2) [below] {$4r+1$};

\draw[-] (a1) to node {} (a2);
\draw[-] (a2) to node {} (a3);
\draw[-] (a2) to node {} (a4);
\draw[-] (a3) to node {} (a4);
\draw[-] (a2rm2) to node {} (a2rm1);
\draw[-] (a2rm2) to node {} (a2r);
\draw[-] (a2rm2) to node {} (a2rp1c2);
\draw[-] (a2rm1) to node {} (a2r);
\draw[-] (a2rm1) to node {} (a2rp1c2);
\draw[-] (a2r) to node {} (a2rp1c1);
\draw[-] (a2r) to node {} (a2rp1c2);
\draw[-] (a2r) to node {} (a2rp2);
\draw[-] (a2rp1c1) to node {} (a2rp2);
\draw[-] (a2rp1c2) to node {} (a2rp2);
\draw[-] (a2rp1c2) to node {} (a2rp3);
\draw[-] (a2rp1c2) to node {} (a2rp4);
\draw[-] (a2rp2) to node {} (a2rp3);
\draw[-] (a2rp2) to node {} (a2rp4);
\draw[-] (a2rp3) to node {} (a2rp4);
\draw[-] (a4rm2) to node {} (a4rm1);
\draw[-] (a4rm2) to node {} (a4r);
\draw[-] (a4rm1) to node {} (a4r);
\draw[-] (a4r) to node {} (a4rp1);

\end{tikzpicture} 
\end{center}

With the undirected graph $\mathcal{G}_B$ given as above, it follows that for all distinct $j,k\in[1,2r]\cup[2r+2,4r+1]\cup\{(2r+1,1),(2r+1,2)\}$ that if vertices $j$ and $k$ are not adjacent to each other in $\mathcal{G}_B$, then the paths $P_j$ and $P_k$ are non-intersecting. The converse is true as well, except when $\{j,k\}=\{2r,2r+2\}$. However, in light of Lemma \ref{5.15} and \eqref{eq:5.23}, where there are cancellation of terms involved, we may effectively regard $P_{2r}$ and $P_{2r+2}$ as intersecting paths.

Similarly to types $A$ and $C$, we may define a hard particle configuration on $\mathcal{G}_B$, as well as the weight $\wt(C)$ of a hard particle configuration $C$ on $\mathcal{G}_B$, and the set of hard particle configurations $\HPC(\mathcal{G}_B)$ on $\mathcal{G}_B$ analogously.

By \eqref{eq:5.23}, Definition \ref{5.7}, and Proposition \ref{5.16}, we recover the following formula for the conserved quantity $C_j=\rho_{\tau}^*(H_j^{c,c})$ of the normalized $D_{r+1}^{(2)}$ $Q$-system, or equivalently, the Coxeter-Toda Hamiltonian $H_j^{c,c}$, written in terms of the normalized $D_{r+1}^{(2)}$ $Q$-system variables:

\begin{theorem}\label{5.17}
Let $j\in[1,r-1]$. Then the $j$-th conserved quantity $C_j=\rho_{\tau}^*(H_j^{c,c})$ of the normalized $D_{r+1}^{(2)}$ $Q$-system is given by
\begin{equation}\label{eq:5.35}
C_j=\sum_{\substack{C\in\HPC(\mathcal{G}_B)\\ |C|=j}}\wt(C).
\end{equation}
\end{theorem}

\subsection{The Coxeter-Toda Hamiltonians of type \textit{D} arising from exterior powers of the defining representation}\label{Section5.5}

In this subsection, we will derive combinatorial formulas for the Hamiltonians $H_j^{c,c}=f_j^{c,c}$, $j\in[1,r-2]$, in the case where $G=SO_{2r}(\mathbb{C})$. Similar to type $B$, there exist fully mixed, non-intersecting subsets $\overline{\mathcal{P}}$ of $\overline{\mathcal{P}}^{c,c}$. Thus, our first goal in this subsection is to characterize the fully mixed non-intersecting subsets of $\overline{\mathcal{P}}^{c,c}$. To do so, we will proceed in several steps.

\begin{lemma}\label{5.18}
Let $\overline{\mathcal{P}}=\{(m_1,n_1,p_1),\ldots,(m_j,n_j,p_j)\}$ be a fully mixed non-intersecting set, and $I=\{m_1,\ldots,m_j\}=\{p_1,\ldots,p_j\}$. Then $I\subseteq[r,2r]$, and $I\not\subseteq[r+3,2r]$.
\end{lemma}

\begin{proof}
Firstly, we may argue in a similar way as in the first part of the proof of Lemma \ref{5.9} to deduce that $I\subseteq[r,2r]$. Subsequently, we may argue in a similar way as in the second part of the proof of Lemma \ref{5.9} to deduce that $I\not\subseteq[r+3,2r]$. We will omit the details here.
\end{proof}

\begin{lemma}\label{5.19}
Let $\overline{\mathcal{P}}=\{(m_1,n_1,p_1),\ldots,(m_j,n_j,p_j)\}$ be a fully mixed non-intersecting set, and $I=\{m_1,\ldots,m_j\}=\{p_1,\ldots,p_j\}$. Then $r\in I$.
\end{lemma}

\begin{proof}
By Lemma \ref{5.18}, we must have either $r\in I$, $r+1\in I$, or $r+2\in I$. Let us assume without loss of generality that we have $m_1<m_2<\cdots<m_j$. We let $k\in[2,j]$ be the unique index satisfying $p_k=m_1$. Suppose on the contrary that we have $r\notin I$. Then we have either $r+1\in I$ or $r+2\in I$. Let us first consider the case where $r+1\in I$, so that we have $p_1>m_1=r+1$. As the elementary chip corresponding to $E_{-(r-2)}$ appears before the elementary chip corresponding to $E_{-r}$ in $\overline{N}_{c,c}(\mathbf{i})$, and we have $p_1>m_1\geq n_1$, we must have $(m_1,n_1,p_1)=(r+1,r-1,r+2)$ or $(r+1,r+1,r+2)$, as shown in Figure \ref{Figure5.18}. In both cases, the paths $\overline{P}_{c,c}(m_1,n_1,p_1)$ and $\overline{P}_{c,c}(m_k,n_k,p_k)=\overline{P}_{c,c}(m_k,n_k,r+1)$ must intersect at level $m_1=r+1$ during their descents, which contradicts the fact that $\overline{\mathcal{P}}$ is non-intersecting. Therefore, we must have $r+1\notin I$, and hence we must have $r+2\in I$, so that we have $m_1=r+2$.

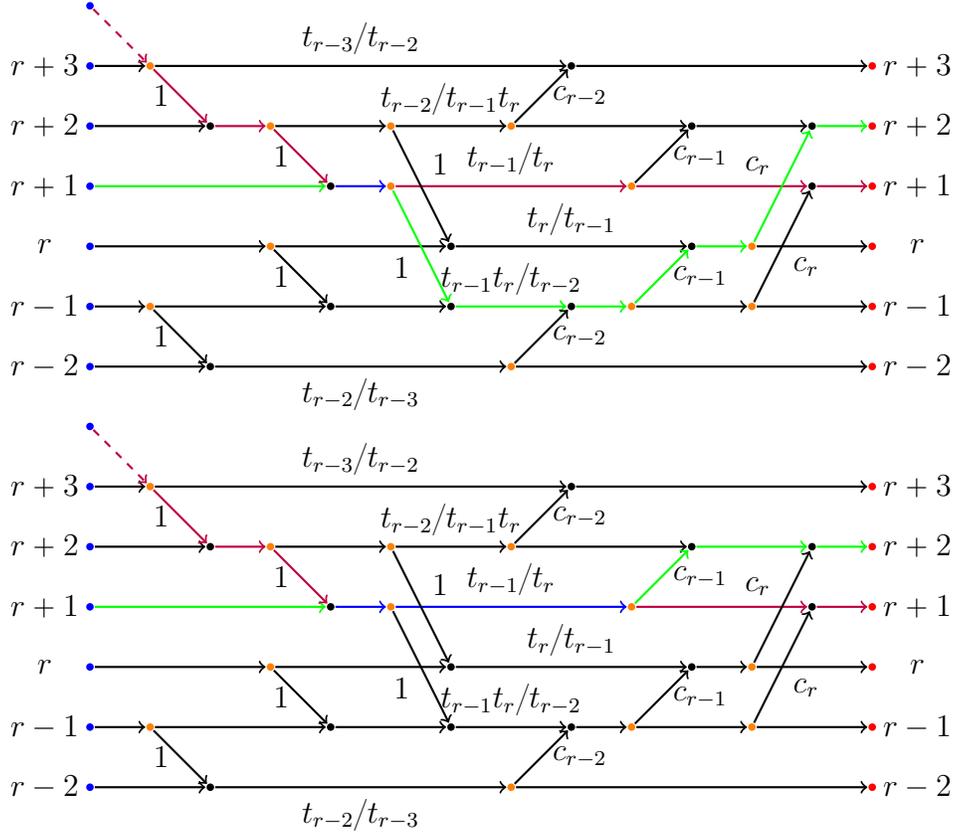
\begin{figure}[t]
\caption{The paths $\overline{P}_{c,c}(m_1,n_1,p_1)$ and $\overline{P}_{c,c}(m_k,n_k,p_k)$ in $\overline{N}_{c,c}(\mathbf{i})$ as colored in green and purple respectively, with the common edges colored blue, in the cases where $(m_1,n_1,p_1)=(r+1,r-1,r+2)$, $(r+1,r+1,r+2)$ respectively, and $(n_k,p_k)=(r+1,r+1)$. The case $(n_k,p_k)=(r-1,r+1)$ is similar.}
\label{Figure5.18}
\begin{center}
\begin{tikzpicture}[
       thick,
       acteur/.style={
         circle,
         thick,
         inner sep=1pt,
         minimum size=0.1cm
       }
] 
\node (a0) at (0,4.8) [acteur,fill=blue]{};
\node (a1) at (0,0) [acteur,fill=blue]{};
\node (a2) at (0,0.8) [acteur,fill=blue]{}; 
\node (a3) at (0,1.6) [acteur,fill=blue]{}; 
\node (a4) at (0,2.4) [acteur,fill=blue]{}; 
\node (a5) at (0,3.2) [acteur,fill=blue]{};
\node (a6) at (0,4.0) [acteur,fill=blue]{};
\node (a7) at (0.8,0.8) [acteur,fill=orange]{};
\node (a8) at (1.6,0) [acteur,fill=black]{}; 
\node (a9) at (0.8,4.0) [acteur,fill=orange]{};
\node (a10) at (1.6,3.2) [acteur,fill=black]{}; 
\node (a11) at (2.4,1.6) [acteur,fill=orange]{};
\node (a12) at (3.2,0.8) [acteur,fill=black]{}; 
\node (a13) at (2.4,3.2) [acteur,fill=orange]{};
\node (a14) at (3.2,2.4) [acteur,fill=black]{}; 
\node (a15) at (4.0,2.4) [acteur,fill=orange]{};
\node (a16) at (4.8,0.8) [acteur,fill=black]{}; 
\node (a17) at (4.0,3.2) [acteur,fill=orange]{};
\node (a18) at (4.8,1.6) [acteur,fill=black]{}; 

\node (a19) at (5.6,0) [acteur,fill=orange]{};
\node (a20) at (6.4,0.8) [acteur,fill=black]{}; 
\node (a21) at (5.6,3.2) [acteur,fill=orange]{};
\node (a22) at (6.4,4.0) [acteur,fill=black]{}; 
\node (a23) at (7.2,0.8) [acteur,fill=orange]{};
\node (a24) at (8.0,1.6) [acteur,fill=black]{}; 
\node (a25) at (7.2,2.4) [acteur,fill=orange]{};
\node (a26) at (8.0,3.2) [acteur,fill=black]{}; 
\node (a27) at (8.8,0.8) [acteur,fill=orange]{};
\node (a28) at (9.6,2.4) [acteur,fill=black]{}; 
\node (a29) at (8.8,1.6) [acteur,fill=orange]{};
\node (a30) at (9.6,3.2) [acteur,fill=black]{}; 
\node (a31) at (10.4,0) [acteur,fill=red]{};
\node (a32) at (10.4,0.8) [acteur,fill=red]{}; 
\node (a33) at (10.4,1.6) [acteur,fill=red]{}; 
\node (a34) at (10.4,2.4) [acteur,fill=red]{}; 
\node (a35) at (10.4,3.2) [acteur,fill=red]{};
\node (a36) at (10.4,4.0) [acteur,fill=red]{};

\node (b1) at (-0.6,0) {$r-2$};
\node (b2) at (-0.6,0.8) {$r-1$}; 
\node (b3) at (-0.6,1.6) {$r$}; 
\node (b4) at (-0.6,2.4) {$r+1$};
\node (b5) at (-0.6,3.2) {$r+2$};
\node (b6) at (-0.6,4.0) {$r+3$}; 
\node (b7) at (11,0) {$r-2$};
\node (b8) at (11,0.8) {$r-1$}; 
\node (b9) at (11,1.6) {$r$}; 
\node (b10) at (11,2.4) {$r+1$};
\node (b11) at (11,3.2) {$r+2$};
\node (b12) at (11,4.0) {$r+3$}; 

\draw[->] (a1) to node {} (a8);
\draw[->] (a8) to node [below] {$t_{r-2}/t_{r-3}$} (a19);
\draw[->] (a19) to node {} (a31);
\draw[->] (a2) to node {} (a7);
\draw[->] (a7) to node {} (a12);
\draw[->] (a12) to node {} (a16);
\draw[green,->] (a16) to node [above] {\textcolor{black}{$t_{r-1}t_r/t_{r-2}$}} (a20);
\draw[green,->] (a20) to node {} (a23);
\draw[->] (a23) to node {} (a27);
\draw[->] (a27) to node {} (a32);
\draw[->] (a3) to node {} (a11);
\draw[->] (a11) to node {} (a18);
\draw[->] (a18) to node [above] {$t_r/t_{r-1}$} (a24);
\draw[green,->] (a24) to node {} (a29);
\draw[->] (a29) to node {} (a33);
\draw[green,->] (a4) to node {} (a14);
\draw[blue,->] (a14) to node {} (a15);
\draw[purple,->] (a15) to node [above] {\textcolor{black}{$t_{r-1}/t_r$}} (a25);
\draw[purple,->] (a25) to node {} (a28);
\draw[purple,->] (a28) to node {} (a34);
\draw[->] (a5) to node {} (a10);
\draw[purple,->] (a10) to node {} (a13);
\draw[->] (a13) to node {} (a17);
\draw[->] (a17) to node [above] {$t_{r-2}/t_{r-1}t_r$} (a21);
\draw[->] (a21) to node {} (a26);
\draw[->] (a26) to node {} (a30);
\draw[green,->] (a30) to node {} (a35);
\draw[->] (a6) to node {} (a9);
\draw[->] (a9) to node [above] {$t_{r-3}/t_{r-2}$} (a22);
\draw[->] (a22) to node {} (a36);

\draw[->] (a7) to node [left] {$1$} (a8);
\draw[purple,->] (a9) to node [left] {\textcolor{black}{$1$}} (a10);
\draw[->] (a11) to node [left] {$1$} (a12);
\draw[purple,->] (a13) to node [left] {\textcolor{black}{$1$}} (a14);
\draw[green,->] (a15) to node [below left] {\textcolor{black}{$1$}} (a16);
\draw[->] (a17) to node [above right] {$1$} (a18);
\draw[->] (a19) to node [right] {$c_{r-2}$} (a20);
\draw[->] (a21) to node [right] {$c_{r-2}$} (a22);
\draw[green,->] (a23) to node [right] {\textcolor{black}{$c_{r-1}$}} (a24);
\draw[->] (a25) to node [right] {$c_{r-1}$} (a26);
\draw[->] (a27) to node [below right] {$c_r$} (a28);
\draw[green,->] (a29) to node [above left] {\textcolor{black}{$c_r$}} (a30);
\draw[purple,dashed,->] (a0) to node {} (a9);

\end{tikzpicture} 
\begin{tikzpicture}[
       thick,
       acteur/.style={
         circle,
         thick,
         inner sep=1pt,
         minimum size=0.1cm
       }
] 
\node (a0) at (0,4.8) [acteur,fill=blue]{};
\node (a1) at (0,0) [acteur,fill=blue]{};
\node (a2) at (0,0.8) [acteur,fill=blue]{}; 
\node (a3) at (0,1.6) [acteur,fill=blue]{}; 
\node (a4) at (0,2.4) [acteur,fill=blue]{}; 
\node (a5) at (0,3.2) [acteur,fill=blue]{};
\node (a6) at (0,4.0) [acteur,fill=blue]{};
\node (a7) at (0.8,0.8) [acteur,fill=orange]{};
\node (a8) at (1.6,0) [acteur,fill=black]{}; 
\node (a9) at (0.8,4.0) [acteur,fill=orange]{};
\node (a10) at (1.6,3.2) [acteur,fill=black]{}; 
\node (a11) at (2.4,1.6) [acteur,fill=orange]{};
\node (a12) at (3.2,0.8) [acteur,fill=black]{}; 
\node (a13) at (2.4,3.2) [acteur,fill=orange]{};
\node (a14) at (3.2,2.4) [acteur,fill=black]{}; 
\node (a15) at (4.0,2.4) [acteur,fill=orange]{};
\node (a16) at (4.8,0.8) [acteur,fill=black]{}; 
\node (a17) at (4.0,3.2) [acteur,fill=orange]{};
\node (a18) at (4.8,1.6) [acteur,fill=black]{}; 

\node (a19) at (5.6,0) [acteur,fill=orange]{};
\node (a20) at (6.4,0.8) [acteur,fill=black]{}; 
\node (a21) at (5.6,3.2) [acteur,fill=orange]{};
\node (a22) at (6.4,4.0) [acteur,fill=black]{}; 
\node (a23) at (7.2,0.8) [acteur,fill=orange]{};
\node (a24) at (8.0,1.6) [acteur,fill=black]{}; 
\node (a25) at (7.2,2.4) [acteur,fill=orange]{};
\node (a26) at (8.0,3.2) [acteur,fill=black]{}; 
\node (a27) at (8.8,0.8) [acteur,fill=orange]{};
\node (a28) at (9.6,2.4) [acteur,fill=black]{}; 
\node (a29) at (8.8,1.6) [acteur,fill=orange]{};
\node (a30) at (9.6,3.2) [acteur,fill=black]{}; 
\node (a31) at (10.4,0) [acteur,fill=red]{};
\node (a32) at (10.4,0.8) [acteur,fill=red]{}; 
\node (a33) at (10.4,1.6) [acteur,fill=red]{}; 
\node (a34) at (10.4,2.4) [acteur,fill=red]{}; 
\node (a35) at (10.4,3.2) [acteur,fill=red]{};
\node (a36) at (10.4,4.0) [acteur,fill=red]{};

\node (b1) at (-0.6,0) {$r-2$};
\node (b2) at (-0.6,0.8) {$r-1$}; 
\node (b3) at (-0.6,1.6) {$r$}; 
\node (b4) at (-0.6,2.4) {$r+1$};
\node (b5) at (-0.6,3.2) {$r+2$};
\node (b6) at (-0.6,4.0) {$r+3$}; 
\node (b7) at (11,0) {$r-2$};
\node (b8) at (11,0.8) {$r-1$}; 
\node (b9) at (11,1.6) {$r$}; 
\node (b10) at (11,2.4) {$r+1$};
\node (b11) at (11,3.2) {$r+2$};
\node (b12) at (11,4.0) {$r+3$}; 

\draw[->] (a1) to node {} (a8);
\draw[->] (a8) to node [below] {$t_{r-2}/t_{r-3}$} (a19);
\draw[->] (a19) to node {} (a31);
\draw[->] (a2) to node {} (a7);
\draw[->] (a7) to node {} (a12);
\draw[->] (a12) to node {} (a16);
\draw[->] (a16) to node [above] {$t_{r-1}t_r/t_{r-2}$} (a20);
\draw[->] (a20) to node {} (a23);
\draw[->] (a23) to node {} (a27);
\draw[->] (a27) to node {} (a32);
\draw[->] (a3) to node {} (a11);
\draw[->] (a11) to node {} (a18);
\draw[->] (a18) to node [above] {$t_r/t_{r-1}$} (a24);
\draw[->] (a24) to node {} (a29);
\draw[->] (a29) to node {} (a33);
\draw[green,->] (a4) to node {} (a14);
\draw[blue,->] (a14) to node {} (a15);
\draw[blue,->] (a15) to node [above] {\textcolor{black}{$t_{r-1}/t_r$}} (a25);
\draw[purple,->] (a25) to node {} (a28);
\draw[purple,->] (a28) to node {} (a34);
\draw[->] (a5) to node {} (a10);
\draw[purple,->] (a10) to node {} (a13);
\draw[->] (a13) to node {} (a17);
\draw[->] (a17) to node [above] {$t_{r-2}/t_{r-1}t_r$} (a21);
\draw[->] (a21) to node {} (a26);
\draw[green,->] (a26) to node {} (a30);
\draw[green,->] (a30) to node {} (a35);
\draw[->] (a6) to node {} (a9);
\draw[->] (a9) to node [above] {$t_{r-3}/t_{r-2}$} (a22);
\draw[->] (a22) to node {} (a36);

\draw[->] (a7) to node [left] {$1$} (a8);
\draw[purple,->] (a9) to node [left] {\textcolor{black}{$1$}} (a10);
\draw[->] (a11) to node [left] {$1$} (a12);
\draw[purple,->] (a13) to node [left] {\textcolor{black}{$1$}} (a14);
\draw[->] (a15) to node [below left] {$1$} (a16);
\draw[->] (a17) to node [above right] {$1$} (a18);
\draw[->] (a19) to node [right] {$c_{r-2}$} (a20);
\draw[->] (a21) to node [right] {$c_{r-2}$} (a22);
\draw[->] (a23) to node [right] {$c_{r-1}$} (a24);
\draw[green,->] (a25) to node [right] {\textcolor{black}{$c_{r-1}$}} (a26);
\draw[->] (a27) to node [below right] {$c_r$} (a28);
\draw[->] (a29) to node [above left] {$c_r$} (a30);
\draw[purple,dashed,->] (a0) to node {} (a9);

\end{tikzpicture} 
\end{center}
\end{figure}

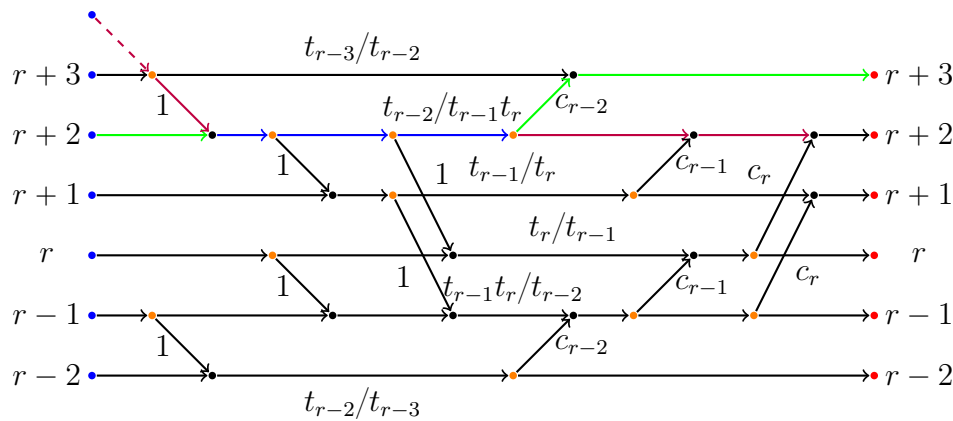
\begin{figure}[t]
\caption{The paths $\overline{P}_{c,c}(m_1,n_1,p_1)$ and $\overline{P}_{c,c}(m_k,n_k,p_k)$ in $\overline{N}_{c,c}(\mathbf{i})$ as colored in green and purple respectively, with the common edges colored blue, in the case where $(n_k,p_k)=(r+2,r+2)$. The other cases $(n_k,p_k)=(r+1,r+2)$, $(r,r+2)$ and $(r-1,r+2)$ are similar.}
\label{Figure5.19}
\begin{center}
\begin{tikzpicture}[
       thick,
       acteur/.style={
         circle,
         thick,
         inner sep=1pt,
         minimum size=0.1cm
       }
] 
\node (a0) at (0,4.8) [acteur,fill=blue]{};
\node (a1) at (0,0) [acteur,fill=blue]{};
\node (a2) at (0,0.8) [acteur,fill=blue]{}; 
\node (a3) at (0,1.6) [acteur,fill=blue]{}; 
\node (a4) at (0,2.4) [acteur,fill=blue]{}; 
\node (a5) at (0,3.2) [acteur,fill=blue]{};
\node (a6) at (0,4.0) [acteur,fill=blue]{};
\node (a7) at (0.8,0.8) [acteur,fill=orange]{};
\node (a8) at (1.6,0) [acteur,fill=black]{}; 
\node (a9) at (0.8,4.0) [acteur,fill=orange]{};
\node (a10) at (1.6,3.2) [acteur,fill=black]{}; 
\node (a11) at (2.4,1.6) [acteur,fill=orange]{};
\node (a12) at (3.2,0.8) [acteur,fill=black]{}; 
\node (a13) at (2.4,3.2) [acteur,fill=orange]{};
\node (a14) at (3.2,2.4) [acteur,fill=black]{}; 
\node (a15) at (4.0,2.4) [acteur,fill=orange]{};
\node (a16) at (4.8,0.8) [acteur,fill=black]{}; 
\node (a17) at (4.0,3.2) [acteur,fill=orange]{};
\node (a18) at (4.8,1.6) [acteur,fill=black]{}; 

\node (a19) at (5.6,0) [acteur,fill=orange]{};
\node (a20) at (6.4,0.8) [acteur,fill=black]{}; 
\node (a21) at (5.6,3.2) [acteur,fill=orange]{};
\node (a22) at (6.4,4.0) [acteur,fill=black]{}; 
\node (a23) at (7.2,0.8) [acteur,fill=orange]{};
\node (a24) at (8.0,1.6) [acteur,fill=black]{}; 
\node (a25) at (7.2,2.4) [acteur,fill=orange]{};
\node (a26) at (8.0,3.2) [acteur,fill=black]{}; 
\node (a27) at (8.8,0.8) [acteur,fill=orange]{};
\node (a28) at (9.6,2.4) [acteur,fill=black]{}; 
\node (a29) at (8.8,1.6) [acteur,fill=orange]{};
\node (a30) at (9.6,3.2) [acteur,fill=black]{}; 
\node (a31) at (10.4,0) [acteur,fill=red]{};
\node (a32) at (10.4,0.8) [acteur,fill=red]{}; 
\node (a33) at (10.4,1.6) [acteur,fill=red]{}; 
\node (a34) at (10.4,2.4) [acteur,fill=red]{}; 
\node (a35) at (10.4,3.2) [acteur,fill=red]{};
\node (a36) at (10.4,4.0) [acteur,fill=red]{};

\node (b1) at (-0.6,0) {$r-2$};
\node (b2) at (-0.6,0.8) {$r-1$}; 
\node (b3) at (-0.6,1.6) {$r$}; 
\node (b4) at (-0.6,2.4) {$r+1$};
\node (b5) at (-0.6,3.2) {$r+2$};
\node (b6) at (-0.6,4.0) {$r+3$}; 
\node (b7) at (11,0) {$r-2$};
\node (b8) at (11,0.8) {$r-1$}; 
\node (b9) at (11,1.6) {$r$}; 
\node (b10) at (11,2.4) {$r+1$};
\node (b11) at (11,3.2) {$r+2$};
\node (b12) at (11,4.0) {$r+3$}; 

\draw[->] (a1) to node {} (a8);
\draw[->] (a8) to node [below] {$t_{r-2}/t_{r-3}$} (a19);
\draw[->] (a19) to node {} (a31);
\draw[->] (a2) to node {} (a7);
\draw[->] (a7) to node {} (a12);
\draw[->] (a12) to node {} (a16);
\draw[->] (a16) to node [above] {$t_{r-1}t_r/t_{r-2}$} (a20);
\draw[->] (a20) to node {} (a23);
\draw[->] (a23) to node {} (a27);
\draw[->] (a27) to node {} (a32);
\draw[->] (a3) to node {} (a11);
\draw[->] (a11) to node {} (a18);
\draw[->] (a18) to node [above] {$t_r/t_{r-1}$} (a24);
\draw[->] (a24) to node {} (a29);
\draw[->] (a29) to node {} (a33);
\draw[->] (a4) to node {} (a14);
\draw[->] (a14) to node {} (a15);
\draw[->] (a15) to node [above] {$t_{r-1}/t_r$} (a25);
\draw[->] (a25) to node {} (a28);
\draw[->] (a28) to node {} (a34);
\draw[green,->] (a5) to node {} (a10);
\draw[blue,->] (a10) to node {} (a13);
\draw[blue,->] (a13) to node {} (a17);
\draw[blue,->] (a17) to node [above] {\textcolor{black}{$t_{r-2}/t_{r-1}t_r$}} (a21);
\draw[purple,->] (a21) to node {} (a26);
\draw[purple,->] (a26) to node {} (a30);
\draw[->] (a30) to node {} (a35);
\draw[->] (a6) to node {} (a9);
\draw[->] (a9) to node [above] {$t_{r-3}/t_{r-2}$} (a22);
\draw[green,->] (a22) to node {} (a36);

\draw[->] (a7) to node [left] {$1$} (a8);
\draw[purple,->] (a9) to node [left] {\textcolor{black}{$1$}} (a10);
\draw[->] (a11) to node [left] {$1$} (a12);
\draw[->] (a13) to node [left] {$1$} (a14);
\draw[->] (a15) to node [below left] {$1$} (a16);
\draw[->] (a17) to node [above right] {$1$} (a18);
\draw[->] (a19) to node [right] {$c_{r-2}$} (a20);
\draw[green,->] (a21) to node [right] {\textcolor{black}{$c_{r-2}$}} (a22);
\draw[->] (a23) to node [right] {$c_{r-1}$} (a24);
\draw[->] (a25) to node [right] {$c_{r-1}$} (a26);
\draw[->] (a27) to node [below right] {$c_r$} (a28);
\draw[->] (a29) to node [above left] {$c_r$} (a30);
\draw[purple,dashed,->] (a0) to node {} (a9);

\end{tikzpicture}  
\end{center}
\end{figure}
\begin{figure}[t]
\caption{The paths $\overline{P}_{c,c}(m_1,n_1,p_1)$ and $\overline{P}_{c,c}(m_k,n_k,p_k)$ in $\overline{N}_{c,c}(\mathbf{i})$ as colored in green and purple respectively, with the common edges colored blue, in the cases where $(m_1,n_1,p_1)=(r,r-1,r+2)$, $(r,r,r+2)$ respectively, and $(n_k,p_k)=(r,r)$. The case $(n_k,p_k)=(r-1,r)$ is similar.}
\label{Figure5.20}
\begin{center}
\begin{tikzpicture}[
       thick,
       acteur/.style={
         circle,
         thick,
         inner sep=1pt,
         minimum size=0.1cm
       }
] 
\node (a0) at (0,4.8) [acteur,fill=blue]{};
\node (a1) at (0,0) [acteur,fill=blue]{};
\node (a2) at (0,0.8) [acteur,fill=blue]{}; 
\node (a3) at (0,1.6) [acteur,fill=blue]{}; 
\node (a4) at (0,2.4) [acteur,fill=blue]{}; 
\node (a5) at (0,3.2) [acteur,fill=blue]{};
\node (a6) at (0,4.0) [acteur,fill=blue]{};
\node (a7) at (0.8,0.8) [acteur,fill=orange]{};
\node (a8) at (1.6,0) [acteur,fill=black]{}; 
\node (a9) at (0.8,4.0) [acteur,fill=orange]{};
\node (a10) at (1.6,3.2) [acteur,fill=black]{}; 
\node (a11) at (2.4,1.6) [acteur,fill=orange]{};
\node (a12) at (3.2,0.8) [acteur,fill=black]{}; 
\node (a13) at (2.4,3.2) [acteur,fill=orange]{};
\node (a14) at (3.2,2.4) [acteur,fill=black]{}; 
\node (a15) at (4.0,2.4) [acteur,fill=orange]{};
\node (a16) at (4.8,0.8) [acteur,fill=black]{}; 
\node (a17) at (4.0,3.2) [acteur,fill=orange]{};
\node (a18) at (4.8,1.6) [acteur,fill=black]{}; 

\node (a19) at (5.6,0) [acteur,fill=orange]{};
\node (a20) at (6.4,0.8) [acteur,fill=black]{}; 
\node (a21) at (5.6,3.2) [acteur,fill=orange]{};
\node (a22) at (6.4,4.0) [acteur,fill=black]{}; 
\node (a23) at (7.2,0.8) [acteur,fill=orange]{};
\node (a24) at (8.0,1.6) [acteur,fill=black]{}; 
\node (a25) at (7.2,2.4) [acteur,fill=orange]{};
\node (a26) at (8.0,3.2) [acteur,fill=black]{}; 
\node (a27) at (8.8,0.8) [acteur,fill=orange]{};
\node (a28) at (9.6,2.4) [acteur,fill=black]{}; 
\node (a29) at (8.8,1.6) [acteur,fill=orange]{};
\node (a30) at (9.6,3.2) [acteur,fill=black]{}; 
\node (a31) at (10.4,0) [acteur,fill=red]{};
\node (a32) at (10.4,0.8) [acteur,fill=red]{}; 
\node (a33) at (10.4,1.6) [acteur,fill=red]{}; 
\node (a34) at (10.4,2.4) [acteur,fill=red]{}; 
\node (a35) at (10.4,3.2) [acteur,fill=red]{};
\node (a36) at (10.4,4.0) [acteur,fill=red]{};

\node (b1) at (-0.6,0) {$r-2$};
\node (b2) at (-0.6,0.8) {$r-1$}; 
\node (b3) at (-0.6,1.6) {$r$}; 
\node (b4) at (-0.6,2.4) {$r+1$};
\node (b5) at (-0.6,3.2) {$r+2$};
\node (b6) at (-0.6,4.0) {$r+3$}; 
\node (b7) at (11,0) {$r-2$};
\node (b8) at (11,0.8) {$r-1$}; 
\node (b9) at (11,1.6) {$r$}; 
\node (b10) at (11,2.4) {$r+1$};
\node (b11) at (11,3.2) {$r+2$};
\node (b12) at (11,4.0) {$r+3$}; 

\draw[->] (a1) to node {} (a8);
\draw[->] (a8) to node [below] {$t_{r-2}/t_{r-3}$} (a19);
\draw[->] (a19) to node {} (a31);
\draw[->] (a2) to node {} (a7);
\draw[->] (a7) to node {} (a12);
\draw[green,->] (a12) to node {} (a16);
\draw[green,->] (a16) to node [above] {\textcolor{black}{$t_{r-1}t_r/t_{r-2}$}} (a20);
\draw[green,->] (a20) to node {} (a23);
\draw[->] (a23) to node {} (a27);
\draw[->] (a27) to node {} (a32);
\draw[green,->] (a3) to node {} (a11);
\draw[->] (a11) to node {} (a18);
\draw[purple,->] (a18) to node [above] {\textcolor{black}{$t_r/t_{r-1}$}} (a24);
\draw[blue,->] (a24) to node {} (a29);
\draw[purple,->] (a29) to node {} (a33);
\draw[->] (a4) to node {} (a14);
\draw[->] (a14) to node {} (a15);
\draw[->] (a15) to node [above] {$t_{r-1}/t_r$} (a25);
\draw[->] (a25) to node {} (a28);
\draw[->] (a28) to node {} (a34);
\draw[->] (a5) to node {} (a10);
\draw[purple,->] (a10) to node {} (a13);
\draw[purple,->] (a13) to node {} (a17);
\draw[->] (a17) to node [above] {$t_{r-2}/t_{r-1}t_r$} (a21);
\draw[->] (a21) to node {} (a26);
\draw[->] (a26) to node {} (a30);
\draw[green,->] (a30) to node {} (a35);
\draw[->] (a6) to node {} (a9);
\draw[->] (a9) to node [above] {$t_{r-3}/t_{r-2}$} (a22);
\draw[->] (a22) to node {} (a36);

\draw[->] (a7) to node [left] {$1$} (a8);
\draw[purple,->] (a9) to node [left] {\textcolor{black}{$1$}} (a10);
\draw[green,->] (a11) to node [left] {\textcolor{black}{$1$}} (a12);
\draw[->] (a13) to node [left] {$1$} (a14);
\draw[->] (a15) to node [below left] {$1$} (a16);
\draw[purple,->] (a17) to node [above right] {\textcolor{black}{$1$}} (a18);
\draw[->] (a19) to node [right] {$c_{r-2}$} (a20);
\draw[->] (a21) to node [right] {$c_{r-2}$} (a22);
\draw[green,->] (a23) to node [right] {\textcolor{black}{$c_{r-1}$}} (a24);
\draw[->] (a25) to node [right] {$c_{r-1}$} (a26);
\draw[->] (a27) to node [below right] {$c_r$} (a28);
\draw[green,->] (a29) to node [above left] {\textcolor{black}{$c_r$}} (a30);
\draw[purple,dashed,->] (a0) to node {} (a9);
\end{tikzpicture} 
\begin{tikzpicture}[
       thick,
       acteur/.style={
         circle,
         thick,
         inner sep=1pt,
         minimum size=0.1cm
       }
] 
\node (a0) at (0,4.8) [acteur,fill=blue]{};
\node (a1) at (0,0) [acteur,fill=blue]{};
\node (a2) at (0,0.8) [acteur,fill=blue]{}; 
\node (a3) at (0,1.6) [acteur,fill=blue]{}; 
\node (a4) at (0,2.4) [acteur,fill=blue]{}; 
\node (a5) at (0,3.2) [acteur,fill=blue]{};
\node (a6) at (0,4.0) [acteur,fill=blue]{};
\node (a7) at (0.8,0.8) [acteur,fill=orange]{};
\node (a8) at (1.6,0) [acteur,fill=black]{}; 
\node (a9) at (0.8,4.0) [acteur,fill=orange]{};
\node (a10) at (1.6,3.2) [acteur,fill=black]{}; 
\node (a11) at (2.4,1.6) [acteur,fill=orange]{};
\node (a12) at (3.2,0.8) [acteur,fill=black]{}; 
\node (a13) at (2.4,3.2) [acteur,fill=orange]{};
\node (a14) at (3.2,2.4) [acteur,fill=black]{}; 
\node (a15) at (4.0,2.4) [acteur,fill=orange]{};
\node (a16) at (4.8,0.8) [acteur,fill=black]{}; 
\node (a17) at (4.0,3.2) [acteur,fill=orange]{};
\node (a18) at (4.8,1.6) [acteur,fill=black]{}; 

\node (a19) at (5.6,0) [acteur,fill=orange]{};
\node (a20) at (6.4,0.8) [acteur,fill=black]{}; 
\node (a21) at (5.6,3.2) [acteur,fill=orange]{};
\node (a22) at (6.4,4.0) [acteur,fill=black]{}; 
\node (a23) at (7.2,0.8) [acteur,fill=orange]{};
\node (a24) at (8.0,1.6) [acteur,fill=black]{}; 
\node (a25) at (7.2,2.4) [acteur,fill=orange]{};
\node (a26) at (8.0,3.2) [acteur,fill=black]{}; 
\node (a27) at (8.8,0.8) [acteur,fill=orange]{};
\node (a28) at (9.6,2.4) [acteur,fill=black]{}; 
\node (a29) at (8.8,1.6) [acteur,fill=orange]{};
\node (a30) at (9.6,3.2) [acteur,fill=black]{}; 
\node (a31) at (10.4,0) [acteur,fill=red]{};
\node (a32) at (10.4,0.8) [acteur,fill=red]{}; 
\node (a33) at (10.4,1.6) [acteur,fill=red]{}; 
\node (a34) at (10.4,2.4) [acteur,fill=red]{}; 
\node (a35) at (10.4,3.2) [acteur,fill=red]{};
\node (a36) at (10.4,4.0) [acteur,fill=red]{};

\node (b1) at (-0.6,0) {$r-2$};
\node (b2) at (-0.6,0.8) {$r-1$}; 
\node (b3) at (-0.6,1.6) {$r$}; 
\node (b4) at (-0.6,2.4) {$r+1$};
\node (b5) at (-0.6,3.2) {$r+2$};
\node (b6) at (-0.6,4.0) {$r+3$}; 
\node (b7) at (11,0) {$r-2$};
\node (b8) at (11,0.8) {$r-1$}; 
\node (b9) at (11,1.6) {$r$}; 
\node (b10) at (11,2.4) {$r+1$};
\node (b11) at (11,3.2) {$r+2$};
\node (b12) at (11,4.0) {$r+3$}; 

\draw[->] (a1) to node {} (a8);
\draw[->] (a8) to node [below] {$t_{r-2}/t_{r-3}$} (a19);
\draw[->] (a19) to node {} (a31);
\draw[->] (a2) to node {} (a7);
\draw[->] (a7) to node {} (a12);
\draw[->] (a12) to node {} (a16);
\draw[->] (a16) to node [above] {$t_{r-1}t_r/t_{r-2}$} (a20);
\draw[->] (a20) to node {} (a23);
\draw[->] (a23) to node {} (a27);
\draw[->] (a27) to node {} (a32);
\draw[green,->] (a3) to node {} (a11);
\draw[green,->] (a11) to node {} (a18);
\draw[blue,->] (a18) to node [above] {\textcolor{black}{$t_r/t_{r-1}$}} (a24);
\draw[blue,->] (a24) to node {} (a29);
\draw[purple,->] (a29) to node {} (a33);
\draw[->] (a4) to node {} (a14);
\draw[->] (a14) to node {} (a15);
\draw[->] (a15) to node [above] {$t_{r-1}/t_r$} (a25);
\draw[->] (a25) to node {} (a28);
\draw[->] (a28) to node {} (a34);
\draw[->] (a5) to node {} (a10);
\draw[purple,->] (a10) to node {} (a13);
\draw[purple,->] (a13) to node {} (a17);
\draw[->] (a17) to node [above] {$t_{r-2}/t_{r-1}t_r$} (a21);
\draw[->] (a21) to node {} (a26);
\draw[->] (a26) to node {} (a30);
\draw[green,->] (a30) to node {} (a35);
\draw[->] (a6) to node {} (a9);
\draw[->] (a9) to node [above] {$t_{r-3}/t_{r-2}$} (a22);
\draw[->] (a22) to node {} (a36);

\draw[->] (a7) to node [left] {$1$} (a8);
\draw[purple,->] (a9) to node [left] {\textcolor{black}{$1$}} (a10);
\draw[->] (a11) to node [left] {$1$} (a12);
\draw[->] (a13) to node [left] {$1$} (a14);
\draw[->] (a15) to node [below left] {$1$} (a16);
\draw[purple,->] (a17) to node [above right] {\textcolor{black}{$1$}} (a18);
\draw[->] (a19) to node [right] {$c_{r-2}$} (a20);
\draw[->] (a21) to node [right] {$c_{r-2}$} (a22);
\draw[->] (a23) to node [right] {$c_{r-1}$} (a24);
\draw[->] (a25) to node [right] {$c_{r-1}$} (a26);
\draw[->] (a27) to node [below right] {$c_r$} (a28);
\draw[green,->] (a29) to node [above left] {\textcolor{black}{$c_r$}} (a30);
\draw[purple,dashed,->] (a0) to node {} (a9);
\end{tikzpicture} 
\end{center}
\end{figure}
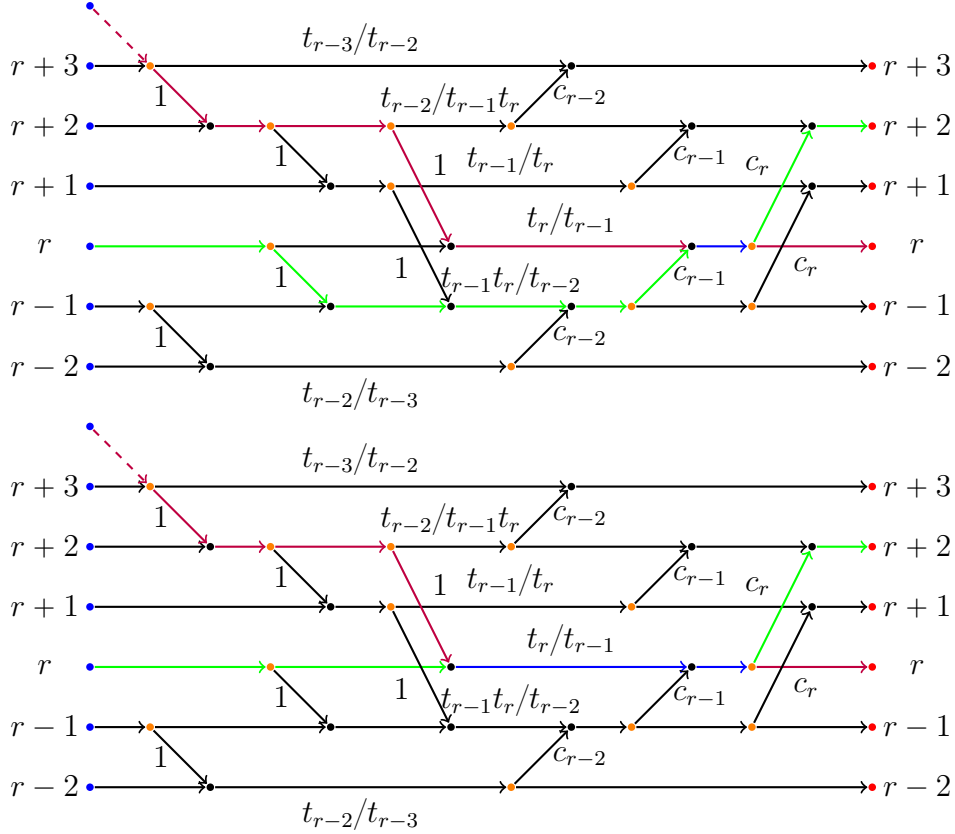

As the elementary chip corresponding to $E_{-(r-2)}$ appears before the elementary chips corresponding to $E_{-(r-1)}$ and $E_{-r}$ in $\overline{N}_{c,c}(\mathbf{i})$, and we have $p_1>m_1=r+2$, we must have $n_1=r+2$, as shown in Figure \ref{Figure5.19}, but this would again imply that the paths $\overline{P}_{c,c}(m_1,n_1,p_1)=\overline{P}_{c,c}(r+2,r+2,p_1)$ and $\overline{P}_{c,c}(m_k,n_k,p_k)=\overline{P}_{c,c}(m_k,n_k,r+2)$ must intersect at level $m_1=r+2$ during their descents, which contradicts the fact that $\overline{\mathcal{P}}$ is non-intersecting. Therefore, we must have $r\in I$ as desired.
\end{proof}

\begin{corollary}\label{5.20}
Let $\overline{\mathcal{P}}=\{(m_1,n_1,p_1),\ldots,(m_j,n_j,p_j)\}$ be a fully mixed non-intersecting set. Then we have $\overline{\mathcal{P}}=\{(r,r-1,r+1),(r+1,r+1,r+2),(r+2,r,r)\}$. In particular, we have $\sgn(\overline{\mathcal{P}})=1$.
\end{corollary}

\begin{figure}[t]
\caption{The non-intersecting paths $\overline{P}_{c,c}(m_1,n_1,p_1)=\overline{P}_{c,c}(r,r-1,r+1)$ and $\overline{P}_{c,c}(m_2,n_2,p_2)=\overline{P}_{c,c}(r+1,r+1,r+2)$ in $\overline{N}_{c,c}(\mathbf{i})$ as colored in green and purple respectively.}
\label{Figure5.21}
\begin{center}
\begin{tikzpicture}[
       thick,
       acteur/.style={
         circle,
         thick,
         inner sep=1pt,
         minimum size=0.1cm
       }
] 
\node (a1) at (0,0) [acteur,fill=blue]{};
\node (a2) at (0,0.8) [acteur,fill=blue]{}; 
\node (a3) at (0,1.6) [acteur,fill=blue]{}; 
\node (a4) at (0,2.4) [acteur,fill=blue]{}; 
\node (a5) at (0,3.2) [acteur,fill=blue]{};
\node (a6) at (0,4.0) [acteur,fill=blue]{};
\node (a7) at (0.8,0.8) [acteur,fill=orange]{};
\node (a8) at (1.6,0) [acteur,fill=black]{}; 
\node (a9) at (0.8,4.0) [acteur,fill=orange]{};
\node (a10) at (1.6,3.2) [acteur,fill=black]{}; 
\node (a11) at (2.4,1.6) [acteur,fill=orange]{};
\node (a12) at (3.2,0.8) [acteur,fill=black]{}; 
\node (a13) at (2.4,3.2) [acteur,fill=orange]{};
\node (a14) at (3.2,2.4) [acteur,fill=black]{}; 
\node (a15) at (4.0,2.4) [acteur,fill=orange]{};
\node (a16) at (4.8,0.8) [acteur,fill=black]{}; 
\node (a17) at (4.0,3.2) [acteur,fill=orange]{};
\node (a18) at (4.8,1.6) [acteur,fill=black]{}; 

\node (a19) at (5.6,0) [acteur,fill=orange]{};
\node (a20) at (6.4,0.8) [acteur,fill=black]{}; 
\node (a21) at (5.6,3.2) [acteur,fill=orange]{};
\node (a22) at (6.4,4.0) [acteur,fill=black]{}; 
\node (a23) at (7.2,0.8) [acteur,fill=orange]{};
\node (a24) at (8.0,1.6) [acteur,fill=black]{}; 
\node (a25) at (7.2,2.4) [acteur,fill=orange]{};
\node (a26) at (8.0,3.2) [acteur,fill=black]{}; 
\node (a27) at (8.8,0.8) [acteur,fill=orange]{};
\node (a28) at (9.6,2.4) [acteur,fill=black]{}; 
\node (a29) at (8.8,1.6) [acteur,fill=orange]{};
\node (a30) at (9.6,3.2) [acteur,fill=black]{}; 
\node (a31) at (10.4,0) [acteur,fill=red]{};
\node (a32) at (10.4,0.8) [acteur,fill=red]{}; 
\node (a33) at (10.4,1.6) [acteur,fill=red]{}; 
\node (a34) at (10.4,2.4) [acteur,fill=red]{}; 
\node (a35) at (10.4,3.2) [acteur,fill=red]{};
\node (a36) at (10.4,4.0) [acteur,fill=red]{};

\node (b1) at (-0.6,0) {$r-2$};
\node (b2) at (-0.6,0.8) {$r-1$}; 
\node (b3) at (-0.6,1.6) {$r$}; 
\node (b4) at (-0.6,2.4) {$r+1$};
\node (b5) at (-0.6,3.2) {$r+2$};
\node (b6) at (-0.6,4.0) {$r+3$}; 
\node (b7) at (11,0) {$r-2$};
\node (b8) at (11,0.8) {$r-1$}; 
\node (b9) at (11,1.6) {$r$}; 
\node (b10) at (11,2.4) {$r+1$};
\node (b11) at (11,3.2) {$r+2$};
\node (b12) at (11,4.0) {$r+3$}; 

\draw[->] (a1) to node {} (a8);
\draw[->] (a8) to node [below] {$t_{r-2}/t_{r-3}$} (a19);
\draw[->] (a19) to node {} (a31);
\draw[->] (a2) to node {} (a7);
\draw[->] (a7) to node {} (a12);
\draw[green,->] (a12) to node {} (a16);
\draw[green,->] (a16) to node [above] {\textcolor{black}{$t_{r-1}t_r/t_{r-2}$}} (a20);
\draw[green,->] (a20) to node {} (a23);
\draw[green,->] (a23) to node {} (a27);
\draw[->] (a27) to node {} (a32);
\draw[green,->] (a3) to node {} (a11);
\draw[->] (a11) to node {} (a18);
\draw[->] (a18) to node [above] {$t_r/t_{r-1}$} (a24);
\draw[->] (a24) to node {} (a29);
\draw[->] (a29) to node {} (a33);
\draw[purple,->] (a4) to node {} (a14);
\draw[purple,->] (a14) to node {} (a15);
\draw[purple,->] (a15) to node [above] {\textcolor{black}{$t_{r-1}/t_r$}} (a25);
\draw[->] (a25) to node {} (a28);
\draw[green,->] (a28) to node {} (a34);
\draw[->] (a5) to node {} (a10);
\draw[->] (a10) to node {} (a13);
\draw[->] (a13) to node {} (a17);
\draw[->] (a17) to node [above] {$t_{r-2}/t_{r-1}t_r$} (a21);
\draw[->] (a21) to node {} (a26);
\draw[purple,->] (a26) to node {} (a30);
\draw[purple,->] (a30) to node {} (a35);
\draw[->] (a6) to node {} (a9);
\draw[->] (a9) to node [above] {$t_{r-3}/t_{r-2}$} (a22);
\draw[->] (a22) to node {} (a36);

\draw[->] (a7) to node [left] {$1$} (a8);
\draw[->] (a9) to node [left] {$1$} (a10);
\draw[green,->] (a11) to node [left] {\textcolor{black}{$1$}} (a12);
\draw[->] (a13) to node [left] {$1$} (a14);
\draw[->] (a15) to node [below left] {$1$} (a16);
\draw[->] (a17) to node [above right] {$1$} (a18);
\draw[->] (a19) to node [right] {$c_{r-2}$} (a20);
\draw[->] (a21) to node [right] {$c_{r-2}$} (a22);
\draw[->] (a23) to node [right] {$c_{r-1}$} (a24);
\draw[purple,->] (a25) to node [right] {\textcolor{black}{$c_{r-1}$}} (a26);
\draw[green,->] (a27) to node [below right] {\textcolor{black}{$c_r$}} (a28);
\draw[->] (a29) to node [above left] {$c_r$} (a30);

\end{tikzpicture} 
\end{center}
\end{figure}
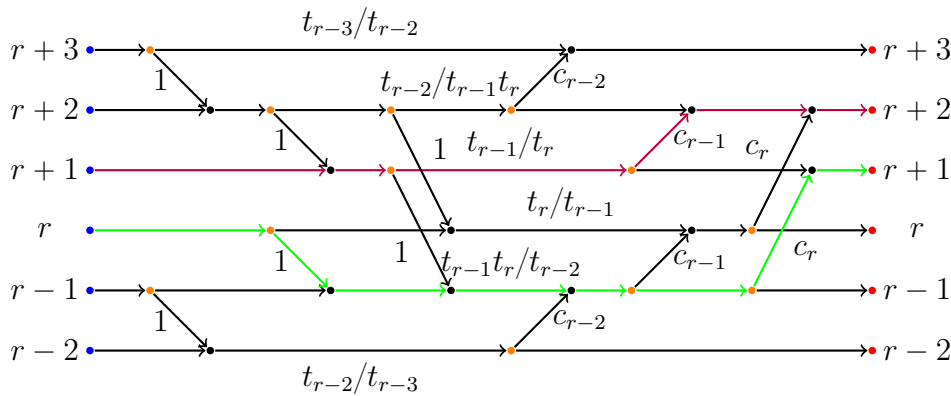

\begin{proof}
By Lemma \ref{5.19}, we may assume without loss of generality that $r=m_1<m_2<\cdots<m_j$. We let $k\in[2,j]$ be the unique index satisfying $p_k=m_1=r$. Let us first show that $r+1\in I$. Suppose on the contrary that $r+1\notin I$. As the elementary chip corresponding to $E_{-(r-2)}$ appears before the elementary chip corresponding to $E_{-r}$ in $\overline{N}_{c,c}(\mathbf{i})$, and we have $p_1>m_1\geq n_1$, we must have $(m_1,n_1,p_1)=(r,r-1,r+2)$ or $(r,r,r+2)$, as shown in Figure \ref{Figure5.20}. In both cases, the paths $\overline{P}_{c,c}(m_1,n_1,p_1)$ and $\overline{P}_{c,c}(m_k,n_k,p_k)=\overline{P}_{c,c}(m_k,n_k,r+1)$ must intersect at level $m_1=r$ during their ascents, which contradicts the fact that $\overline{\mathcal{P}}$ is non-intersecting. Therefore, we must have $r+1\in I$, which forces $m_2=r+1$. In fact, our above argument would also show that $(m_1,n_1,p_1)=(r,r-1,r+1)$. 

Now, as the paths $\overline{P}_{c,c}(m_1,n_1,p_1)=\overline{P}_{c,c}(r,r-1,r+1)$ and $\overline{P}_{c,c}(m_2,n_2,p_2)=\overline{P}_{c,c}(r+1,n_2,p_2)$ do not intersect, we must have $n_2=m_2=r+1$, and hence $p_2=r+2$, as shown in Figure \ref{Figure5.21}. In particular, we have $r+2\in I$, and this forces $m_3=r+2$.

Next, we will show that $p_3=r$. Suppose on the contrary that $p_3\neq r$. Then we must have $j>3$, so there must exist some $\ell\in[4,j]$, such that $p_{\ell}=r$. As the elementary chip corresponding to $E_{r-2}(c_{r-2})$ appears before the elementary chips corresponding to $E_{r-1}(c_{r-1})$ and $E_r(c_r)$ in $\overline{N}_{c,c}(\mathbf{i})$, we necessarily have $n_3=m_3=r+2$, but this would then imply that the paths $\overline{P}_{c,c}(m_3,n_3,p_3)$ and $\overline{P}_{c,c}(m_{\ell},n_{\ell},p_{\ell})$ intersect at level $r+2$ during their descents, as shown in Figure \ref{Figure5.22}, which contradicts the fact that $\overline{\mathcal{P}}$ is non-intersecting. So $p_3=r$, and this forces $n_3=r$. Consequently, we have $\{(r,r-1,r+1),(r+1,r+1,r+2),(r+2,r,r)\}\subseteq\overline{\mathcal{P}}$.

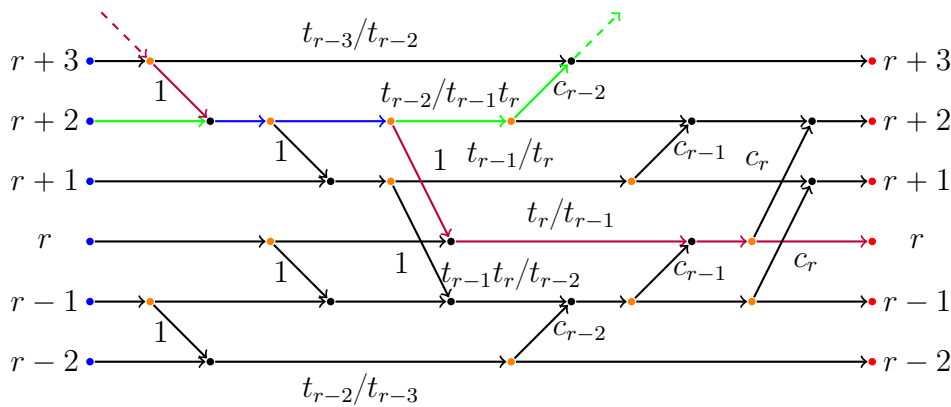
\begin{figure}[t]
\caption{The paths $\overline{P}_{c,c}(m_3,n_3,p_3)$ and $\overline{P}_{c,c}(m_{\ell},n_{\ell},p_{\ell})$ in $\overline{N}_{c,c}(\mathbf{i})$ as colored in green and purple respectively, with the common edges colored blue, in the case where $(n_{\ell},p_{\ell})=(r,r)$. The case $(n_{\ell},p_{\ell})=(r-1,r)$ is similar.}
\label{Figure5.22}
\begin{center}
\begin{tikzpicture}[
       thick,
       acteur/.style={
         circle,
         thick,
         inner sep=1pt,
         minimum size=0.1cm
       }
] 
\node (a0) at (0,4.8) {};
\node (a1) at (0,0) [acteur,fill=blue]{};
\node (a2) at (0,0.8) [acteur,fill=blue]{}; 
\node (a3) at (0,1.6) [acteur,fill=blue]{}; 
\node (a4) at (0,2.4) [acteur,fill=blue]{}; 
\node (a5) at (0,3.2) [acteur,fill=blue]{};
\node (a6) at (0,4.0) [acteur,fill=blue]{};
\node (a7) at (0.8,0.8) [acteur,fill=orange]{};
\node (a8) at (1.6,0) [acteur,fill=black]{}; 
\node (a9) at (0.8,4.0) [acteur,fill=orange]{};
\node (a10) at (1.6,3.2) [acteur,fill=black]{}; 
\node (a11) at (2.4,1.6) [acteur,fill=orange]{};
\node (a12) at (3.2,0.8) [acteur,fill=black]{}; 
\node (a13) at (2.4,3.2) [acteur,fill=orange]{};
\node (a14) at (3.2,2.4) [acteur,fill=black]{}; 
\node (a15) at (4.0,2.4) [acteur,fill=orange]{};
\node (a16) at (4.8,0.8) [acteur,fill=black]{}; 
\node (a17) at (4.0,3.2) [acteur,fill=orange]{};
\node (a18) at (4.8,1.6) [acteur,fill=black]{}; 

\node (a19) at (5.6,0) [acteur,fill=orange]{};
\node (a20) at (6.4,0.8) [acteur,fill=black]{}; 
\node (a21) at (5.6,3.2) [acteur,fill=orange]{};
\node (a22) at (6.4,4.0) [acteur,fill=black]{}; 
\node (a23) at (7.2,0.8) [acteur,fill=orange]{};
\node (a24) at (8.0,1.6) [acteur,fill=black]{}; 
\node (a25) at (7.2,2.4) [acteur,fill=orange]{};
\node (a26) at (8.0,3.2) [acteur,fill=black]{}; 
\node (a27) at (8.8,0.8) [acteur,fill=orange]{};
\node (a28) at (9.6,2.4) [acteur,fill=black]{}; 
\node (a29) at (8.8,1.6) [acteur,fill=orange]{};
\node (a30) at (9.6,3.2) [acteur,fill=black]{}; 
\node (a31) at (10.4,0) [acteur,fill=red]{};
\node (a32) at (10.4,0.8) [acteur,fill=red]{}; 
\node (a33) at (10.4,1.6) [acteur,fill=red]{}; 
\node (a34) at (10.4,2.4) [acteur,fill=red]{}; 
\node (a35) at (10.4,3.2) [acteur,fill=red]{};
\node (a36) at (10.4,4.0) [acteur,fill=red]{};

\node (a37) at (7.2,4.8) {};

\node (b1) at (-0.6,0) {$r-2$};
\node (b2) at (-0.6,0.8) {$r-1$}; 
\node (b3) at (-0.6,1.6) {$r$}; 
\node (b4) at (-0.6,2.4) {$r+1$};
\node (b5) at (-0.6,3.2) {$r+2$};
\node (b6) at (-0.6,4.0) {$r+3$}; 
\node (b7) at (11,0) {$r-2$};
\node (b8) at (11,0.8) {$r-1$}; 
\node (b9) at (11,1.6) {$r$}; 
\node (b10) at (11,2.4) {$r+1$};
\node (b11) at (11,3.2) {$r+2$};
\node (b12) at (11,4.0) {$r+3$}; 

\draw[->] (a1) to node {} (a8);
\draw[->] (a8) to node [below] {$t_{r-2}/t_{r-3}$} (a19);
\draw[->] (a19) to node {} (a31);
\draw[->] (a2) to node {} (a7);
\draw[->] (a7) to node {} (a12);
\draw[->] (a12) to node {} (a16);
\draw[->] (a16) to node [above] {$t_{r-1}t_r/t_{r-2}$} (a20);
\draw[->] (a20) to node {} (a23);
\draw[->] (a23) to node {} (a27);
\draw[->] (a27) to node {} (a32);
\draw[->] (a3) to node {} (a11);
\draw[->] (a11) to node {} (a18);
\draw[purple,->] (a18) to node [above] {\textcolor{black}{$t_r/t_{r-1}$}} (a24);
\draw[purple,->] (a24) to node {} (a29);
\draw[purple,->] (a29) to node {} (a33);
\draw[->] (a4) to node {} (a14);
\draw[->] (a14) to node {} (a15);
\draw[->] (a15) to node [above] {$t_{r-1}/t_r$} (a25);
\draw[->] (a25) to node {} (a28);
\draw[->] (a28) to node {} (a34);
\draw[green,->] (a5) to node {} (a10);
\draw[blue,->] (a10) to node {} (a13);
\draw[blue,->] (a13) to node {} (a17);
\draw[green,->] (a17) to node [above] {\textcolor{black}{$t_{r-2}/t_{r-1}t_r$}} (a21);
\draw[->] (a21) to node {} (a26);
\draw[->] (a26) to node {} (a30);
\draw[->] (a30) to node {} (a35);
\draw[->] (a6) to node {} (a9);
\draw[->] (a9) to node [above] {$t_{r-3}/t_{r-2}$} (a22);
\draw[->] (a22) to node {} (a36);

\draw[->] (a7) to node [left] {$1$} (a8);
\draw[purple,->] (a9) to node [left] {\textcolor{black}{$1$}} (a10);
\draw[->] (a11) to node [left] {$1$} (a12);
\draw[->] (a13) to node [left] {$1$} (a14);
\draw[->] (a15) to node [below left] {$1$} (a16);
\draw[purple,->] (a17) to node [above right] {\textcolor{black}{$1$}} (a18);
\draw[->] (a19) to node [right] {$c_{r-2}$} (a20);
\draw[green,->] (a21) to node [right] {\textcolor{black}{$c_{r-2}$}} (a22);
\draw[->] (a23) to node [right] {$c_{r-1}$} (a24);
\draw[->] (a25) to node [right] {$c_{r-1}$} (a26);
\draw[->] (a27) to node [below right] {$c_r$} (a28);
\draw[->] (a29) to node [above left] {$c_r$} (a30);
\draw[purple,dashed,->] (a0) to node {} (a9);
\draw[green,dashed,->] (a22) to node {} (a37);

\end{tikzpicture} 
\end{center}
\end{figure}

Finally, suppose on the contrary that we have $j>3$. Then $\{(m_4,n_4,p_4),\ldots,(m_j,n_j,p_j)\}$ is also a fully mixed non-intersecting set, but $\{m_4,\ldots,m_j\}\subseteq[r+3,2r]$, which contradicts the second part of Lemma \ref{5.18}. Therefore, we have $\overline{\mathcal{P}}=\{(r,r-1,r+1),(r+1,r+1,r+2),(r+2,r,r)\}$ as desired.
\end{proof}

\begin{figure}[t]
\caption{The non-intersecting paths $\overline{P}_{c,c}(r,r-1,r+1)$, $\overline{P}_{c,c}(r+1,r+1,r+2)$ and $\overline{P}_{c,c}(r+2,r,r)$ in $\overline{N}_{c,c}(\mathbf{i})$ as colored in green, purple and blue respectively.}
\label{Figure5.23}
\begin{center}
\begin{tikzpicture}[
       thick,
       acteur/.style={
         circle,
         thick,
         inner sep=1pt,
         minimum size=0.1cm
       }
] 
\node (a1) at (0,0) [acteur,fill=blue]{};
\node (a2) at (0,0.8) [acteur,fill=blue]{}; 
\node (a3) at (0,1.6) [acteur,fill=blue]{}; 
\node (a4) at (0,2.4) [acteur,fill=blue]{}; 
\node (a5) at (0,3.2) [acteur,fill=blue]{};
\node (a6) at (0,4.0) [acteur,fill=blue]{};
\node (a7) at (0.8,0.8) [acteur,fill=orange]{};
\node (a8) at (1.6,0) [acteur,fill=black]{}; 
\node (a9) at (0.8,4.0) [acteur,fill=orange]{};
\node (a10) at (1.6,3.2) [acteur,fill=black]{}; 
\node (a11) at (2.4,1.6) [acteur,fill=orange]{};
\node (a12) at (3.2,0.8) [acteur,fill=black]{}; 
\node (a13) at (2.4,3.2) [acteur,fill=orange]{};
\node (a14) at (3.2,2.4) [acteur,fill=black]{}; 
\node (a15) at (4.0,2.4) [acteur,fill=orange]{};
\node (a16) at (4.8,0.8) [acteur,fill=black]{}; 
\node (a17) at (4.0,3.2) [acteur,fill=orange]{};
\node (a18) at (4.8,1.6) [acteur,fill=black]{}; 

\node (a19) at (5.6,0) [acteur,fill=orange]{};
\node (a20) at (6.4,0.8) [acteur,fill=black]{}; 
\node (a21) at (5.6,3.2) [acteur,fill=orange]{};
\node (a22) at (6.4,4.0) [acteur,fill=black]{}; 
\node (a23) at (7.2,0.8) [acteur,fill=orange]{};
\node (a24) at (8.0,1.6) [acteur,fill=black]{}; 
\node (a25) at (7.2,2.4) [acteur,fill=orange]{};
\node (a26) at (8.0,3.2) [acteur,fill=black]{}; 
\node (a27) at (8.8,0.8) [acteur,fill=orange]{};
\node (a28) at (9.6,2.4) [acteur,fill=black]{}; 
\node (a29) at (8.8,1.6) [acteur,fill=orange]{};
\node (a30) at (9.6,3.2) [acteur,fill=black]{}; 
\node (a31) at (10.4,0) [acteur,fill=red]{};
\node (a32) at (10.4,0.8) [acteur,fill=red]{}; 
\node (a33) at (10.4,1.6) [acteur,fill=red]{}; 
\node (a34) at (10.4,2.4) [acteur,fill=red]{}; 
\node (a35) at (10.4,3.2) [acteur,fill=red]{};
\node (a36) at (10.4,4.0) [acteur,fill=red]{};

\node (b1) at (-0.6,0) {$r-2$};
\node (b2) at (-0.6,0.8) {$r-1$}; 
\node (b3) at (-0.6,1.6) {$r$}; 
\node (b4) at (-0.6,2.4) {$r+1$};
\node (b5) at (-0.6,3.2) {$r+2$};
\node (b6) at (-0.6,4.0) {$r+3$}; 
\node (b7) at (11,0) {$r-2$};
\node (b8) at (11,0.8) {$r-1$}; 
\node (b9) at (11,1.6) {$r$}; 
\node (b10) at (11,2.4) {$r+1$};
\node (b11) at (11,3.2) {$r+2$};
\node (b12) at (11,4.0) {$r+3$}; 

\draw[->] (a1) to node {} (a8);
\draw[->] (a8) to node [below] {$t_{r-2}/t_{r-3}$} (a19);
\draw[->] (a19) to node {} (a31);
\draw[->] (a2) to node {} (a7);
\draw[->] (a7) to node {} (a12);
\draw[green,->] (a12) to node {} (a16);
\draw[green,->] (a16) to node [above] {\textcolor{black}{$t_{r-1}t_r/t_{r-2}$}} (a20);
\draw[green,->] (a20) to node {} (a23);
\draw[green,->] (a23) to node {} (a27);
\draw[->] (a27) to node {} (a32);
\draw[green,->] (a3) to node {} (a11);
\draw[->] (a11) to node {} (a18);
\draw[blue,->] (a18) to node [above] {\textcolor{black}{$t_r/t_{r-1}$}} (a24);
\draw[blue,->] (a24) to node {} (a29);
\draw[blue,->] (a29) to node {} (a33);
\draw[purple,->] (a4) to node {} (a14);
\draw[purple,->] (a14) to node {} (a15);
\draw[purple,->] (a15) to node [above] {\textcolor{black}{$t_{r-1}/t_r$}} (a25);
\draw[->] (a25) to node {} (a28);
\draw[green,->] (a28) to node {} (a34);
\draw[blue,->] (a5) to node {} (a10);
\draw[blue,->] (a10) to node {} (a13);
\draw[blue,->] (a13) to node {} (a17);
\draw[->] (a17) to node [above] {$t_{r-2}/t_{r-1}t_r$} (a21);
\draw[->] (a21) to node {} (a26);
\draw[purple,->] (a26) to node {} (a30);
\draw[purple,->] (a30) to node {} (a35);
\draw[->] (a6) to node {} (a9);
\draw[->] (a9) to node [above] {$t_{r-3}/t_{r-2}$} (a22);
\draw[->] (a22) to node {} (a36);

\draw[->] (a7) to node [left] {$1$} (a8);
\draw[->] (a9) to node [left] {$1$} (a10);
\draw[green,->] (a11) to node [left] {\textcolor{black}{$1$}} (a12);
\draw[->] (a13) to node [left] {$1$} (a14);
\draw[->] (a15) to node [below left] {$1$} (a16);
\draw[blue,->] (a17) to node [above right] {\textcolor{black}{$1$}} (a18);
\draw[->] (a19) to node [right] {$c_{r-2}$} (a20);
\draw[->] (a21) to node [right] {$c_{r-2}$} (a22);
\draw[->] (a23) to node [right] {$c_{r-1}$} (a24);
\draw[purple,->] (a25) to node [right] {\textcolor{black}{$c_{r-1}$}} (a26);
\draw[green,->] (a27) to node [below right] {\textcolor{black}{$c_r$}} (a28);
\draw[->] (a29) to node [above left] {$c_r$} (a30);

\end{tikzpicture} 
\end{center}
\end{figure}
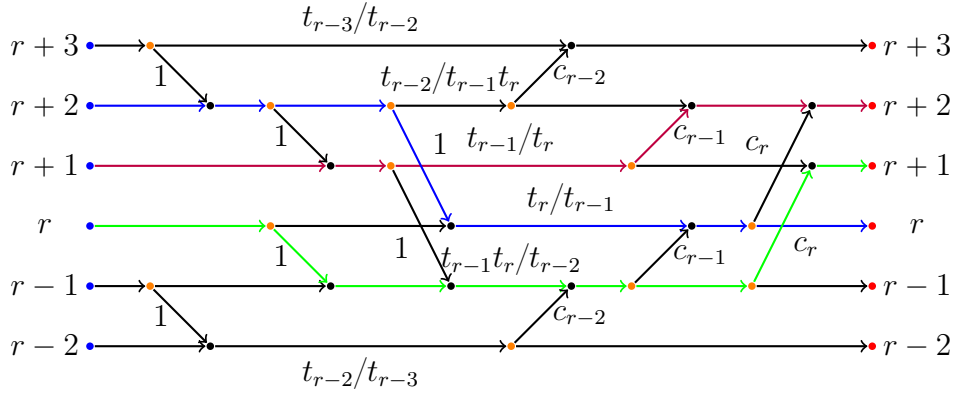

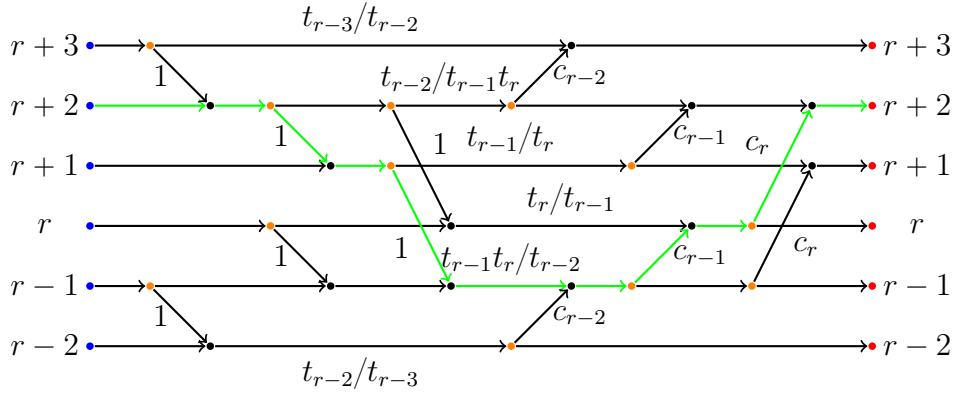
\begin{figure}[t]
\caption{The path $\overline{P}_{c,c}(r+2,r-1,r+2)$ in $\overline{N}_{c,c}(\mathbf{i})$.}
\label{Figure5.24}
\begin{center}
\begin{tikzpicture}[
       thick,
       acteur/.style={
         circle,
         thick,
         inner sep=1pt,
         minimum size=0.1cm
       }
] 
\node (a1) at (0,0) [acteur,fill=blue]{};
\node (a2) at (0,0.8) [acteur,fill=blue]{}; 
\node (a3) at (0,1.6) [acteur,fill=blue]{}; 
\node (a4) at (0,2.4) [acteur,fill=blue]{}; 
\node (a5) at (0,3.2) [acteur,fill=blue]{};
\node (a6) at (0,4.0) [acteur,fill=blue]{};
\node (a7) at (0.8,0.8) [acteur,fill=orange]{};
\node (a8) at (1.6,0) [acteur,fill=black]{}; 
\node (a9) at (0.8,4.0) [acteur,fill=orange]{};
\node (a10) at (1.6,3.2) [acteur,fill=black]{}; 
\node (a11) at (2.4,1.6) [acteur,fill=orange]{};
\node (a12) at (3.2,0.8) [acteur,fill=black]{}; 
\node (a13) at (2.4,3.2) [acteur,fill=orange]{};
\node (a14) at (3.2,2.4) [acteur,fill=black]{}; 
\node (a15) at (4.0,2.4) [acteur,fill=orange]{};
\node (a16) at (4.8,0.8) [acteur,fill=black]{}; 
\node (a17) at (4.0,3.2) [acteur,fill=orange]{};
\node (a18) at (4.8,1.6) [acteur,fill=black]{}; 

\node (a19) at (5.6,0) [acteur,fill=orange]{};
\node (a20) at (6.4,0.8) [acteur,fill=black]{}; 
\node (a21) at (5.6,3.2) [acteur,fill=orange]{};
\node (a22) at (6.4,4.0) [acteur,fill=black]{}; 
\node (a23) at (7.2,0.8) [acteur,fill=orange]{};
\node (a24) at (8.0,1.6) [acteur,fill=black]{}; 
\node (a25) at (7.2,2.4) [acteur,fill=orange]{};
\node (a26) at (8.0,3.2) [acteur,fill=black]{}; 
\node (a27) at (8.8,0.8) [acteur,fill=orange]{};
\node (a28) at (9.6,2.4) [acteur,fill=black]{}; 
\node (a29) at (8.8,1.6) [acteur,fill=orange]{};
\node (a30) at (9.6,3.2) [acteur,fill=black]{}; 
\node (a31) at (10.4,0) [acteur,fill=red]{};
\node (a32) at (10.4,0.8) [acteur,fill=red]{}; 
\node (a33) at (10.4,1.6) [acteur,fill=red]{}; 
\node (a34) at (10.4,2.4) [acteur,fill=red]{}; 
\node (a35) at (10.4,3.2) [acteur,fill=red]{};
\node (a36) at (10.4,4.0) [acteur,fill=red]{};

\node (b1) at (-0.6,0) {$r-2$};
\node (b2) at (-0.6,0.8) {$r-1$}; 
\node (b3) at (-0.6,1.6) {$r$}; 
\node (b4) at (-0.6,2.4) {$r+1$};
\node (b5) at (-0.6,3.2) {$r+2$};
\node (b6) at (-0.6,4.0) {$r+3$}; 
\node (b7) at (11,0) {$r-2$};
\node (b8) at (11,0.8) {$r-1$}; 
\node (b9) at (11,1.6) {$r$}; 
\node (b10) at (11,2.4) {$r+1$};
\node (b11) at (11,3.2) {$r+2$};
\node (b12) at (11,4.0) {$r+3$}; 

\draw[->] (a1) to node {} (a8);
\draw[->] (a8) to node [below] {$t_{r-2}/t_{r-3}$} (a19);
\draw[->] (a19) to node {} (a31);
\draw[->] (a2) to node {} (a7);
\draw[->] (a7) to node {} (a12);
\draw[->] (a12) to node {} (a16);
\draw[green,->] (a16) to node [above] {\textcolor{black}{$t_{r-1}t_r/t_{r-2}$}} (a20);
\draw[green,->] (a20) to node {} (a23);
\draw[->] (a23) to node {} (a27);
\draw[->] (a27) to node {} (a32);
\draw[->] (a3) to node {} (a11);
\draw[->] (a11) to node {} (a18);
\draw[->] (a18) to node [above] {$t_r/t_{r-1}$} (a24);
\draw[green,->] (a24) to node {} (a29);
\draw[->] (a29) to node {} (a33);
\draw[->] (a4) to node {} (a14);
\draw[green,->] (a14) to node {} (a15);
\draw[->] (a15) to node [above] {$t_{r-1}/t_r$} (a25);
\draw[->] (a25) to node {} (a28);
\draw[->] (a28) to node {} (a34);
\draw[green,->] (a5) to node {} (a10);
\draw[green,->] (a10) to node {} (a13);
\draw[->] (a13) to node {} (a17);
\draw[->] (a17) to node [above] {$t_{r-2}/t_{r-1}t_r$} (a21);
\draw[->] (a21) to node {} (a26);
\draw[->] (a26) to node {} (a30);
\draw[green,->] (a30) to node {} (a35);
\draw[->] (a6) to node {} (a9);
\draw[->] (a9) to node [above] {$t_{r-3}/t_{r-2}$} (a22);
\draw[->] (a22) to node {} (a36);

\draw[->] (a7) to node [left] {$1$} (a8);
\draw[->] (a9) to node [left] {$1$} (a10);
\draw[->] (a11) to node [left] {$1$} (a12);
\draw[green,->] (a13) to node [left] {\textcolor{black}{$1$}} (a14);
\draw[green,->] (a15) to node [below left] {\textcolor{black}{$1$}} (a16);
\draw[->] (a17) to node [above right] {$1$} (a18);
\draw[->] (a19) to node [right] {$c_{r-2}$} (a20);
\draw[->] (a21) to node [right] {$c_{r-2}$} (a22);
\draw[green,->] (a23) to node [right] {\textcolor{black}{$c_{r-1}$}} (a24);
\draw[->] (a25) to node [right] {$c_{r-1}$} (a26);
\draw[->] (a27) to node [below right] {$c_r$} (a28);
\draw[green,->] (a29) to node [above left] {\textcolor{black}{$c_r$}} (a30);

\end{tikzpicture} 
\end{center}
\end{figure}

\begin{lemma}\label{5.21}
~
\begin{enumerate}
\item The triple $(r+2,r-1,r+2)$ is an admissible $(c,c)$-triple, and the path $\overline{P}_{c,c}(r+2,r-1,r+2)$ is as displayed in Figure \ref{Figure5.24}.
\item Let $\mathcal{P}=\{(r,r-1,r+1),(r+1,r+1,r+2),(r+2,r,r)\}$. Then
\begin{equation}\label{eq:5.36}
\wt_{c,c}(\overline{\mathcal{P}})=\frac{c_{r-1}c_rt_{r-1}t_r}{t_{r-2}}=\wt(\overline{P}_{c,c}(r+2,r-1,r+2))
\end{equation}
\item For any admissible $(c,c)$-triple $(m,n,p)\notin\overline{\mathcal{P}}\cup\{(r+2,r-1,r+2)\}$, we have $\overline{\mathcal{P}}\cup\{(m,n,p)\}$ is non-intersecting if and only if $\{(m,n,p),(r+2,r-1,r+2)\}$ is non-intersecting.
\end{enumerate}
\end{lemma}

The proof of Lemma \ref{5.21} follows from observing Figures \ref{Figure5.23} and \ref{Figure5.24}, and the proof shall be omitted.

As a consequence of Corollary \ref{5.20}, it follows that if $\overline{\mathcal{P}}$ is a non-intersecting set with $|\overline{\mathcal{P}}|=j\geq 3$, then it is either fully mixed, in which case each of the paths corresponding to the admissible $(c,c)$-triples in $\overline{\mathcal{P}}$ start and end on the same level, or it is not fully mixed, in which case we can write $\overline{\mathcal{P}}=\overline{\mathcal{P}}'\sqcup\{(r,r-1,r+1),(r+1,r+1,r+2),(r+2,r,r)\}$, where $\overline{\mathcal{P}}'$ is a fully mixed non-intersecting set. By Lemma \ref{5.21}, it follows that $\overline{\mathcal{P}}'\sqcup\{(r+2,r-1,r+2)\}$ is also a fully mixed non-intersecting set with $j-2$ elements, and $\wt_{c,c}(\overline{\mathcal{P}})=\wt_{c,c}(\overline{\mathcal{P}}'\sqcup\{(r+2,r-1,r+2)\})$. Consequently, it follows that \eqref{eq:5.2} reduces to:
\begin{equation}\label{eq:5.37}
H_j^{c,c}
=f_j^{c,c}
=\sum_{\substack{I\subseteq[1,2r],\\|I|=j}}\sum_{\mathcal{P}\in\mathcal{P}_{\nonint}^{c,c}(I)}\wt_{c,c}(\mathcal{P})
+\sum_{\substack{I\subseteq[1,2r],\\|I|=j-2}}\sum_{\substack{\mathcal{Q}\in\mathcal{P}_{\nonint}^{c,c}(I)\\(r+2,r-1,r+2)\in\mathcal{Q}}}\wt_{c,c}(\mathcal{Q}).
\end{equation}

In view of Corollary \ref{5.20} and Lemma \ref{5.21}, our next step is to characterize the admissible $(c,c)$-pairs $(m,n)$, as well as when two paths $P_{c,c}(m,n)$ and $P_{c,c}(m',n')$ corresponding to two different admissible $(c,c)$-pairs $(m,n)$ and $(m',n')$ intersect.

\begin{proposition}\label{5.22}
~
\begin{enumerate}
\item The admissible $(c,c)$-pairs $(m,n)$ are of the form $(i,i)$, $(j+1,j)$, $(r+1,r-1)$, $(r+2,r)$ and $(r+2,r-1)$, where $i\in[1,2r]$, and $j\in[1,r-1]\cup[r+1,2r-1]$. 
\item For all $i\in[1,2r]$ and $j\in[1,r-1]\cup[r+1,2r-1]$, the weights of the paths $P_{2i-1}:=P_{c,c}(i,i)$, $P_{2j}:=P_{c,c}(j+1,j)$, $P_{(2r,1)}:=P_{c,c}(r+1,r-1)$, $P_{(2r,2)}:=P_{c,c}(r+2,r-1)$ and $P_{(2r,3)}:=P_{c,c}(r+2,r)$ are given by
\begin{align}
\wt(P_{2k-1})&=\frac{t_k}{t_{k-1}},\quad\wt(P_{4r+1-2k})=\wt(P_{2k-1})^{-1}=\frac{t_{k-1}}{t_k},\quad k\in[1,r]\setminus\{r-1\},\label{eq:5.38}\\
\wt(P_{2r-3})&=\frac{t_{r-1}t_r}{t_{r-2}},\quad\wt(P_{2r+3})=\wt(P_{2r-3})^{-1}=\frac{t_{r-2}}{t_{r-1}t_r},\label{eq:5.39}\\
\wt(P_{2\ell})&=\frac{c_{\ell}t_{\ell}}{t_{\ell-1}},\quad\wt(P_{4r-2\ell})=\frac{c_{\ell}t_{\ell}}{t_{\ell+1}}\quad\ell\in[1,r-3],\label{eq:5.40}\\
\wt(P_{2r-4})&=\frac{c_{r-2}t_{r-2}}{t_{r-3}},\quad\wt(P_{2r+4})=\frac{c_{r-2}t_{r-2}}{t_{r-1}t_r},\label{eq:5.41}\\
\wt(P_{2r-2})&=\frac{c_{r-1}t_{r-1}t_r}{t_{r-2}},\quad\wt(P_{2r+2})=\frac{c_{r-1}t_{r-1}}{t_r},\label{eq:5.42}\\
\wt(P_{(2r,1)})&=\frac{c_rt_{r-1}t_r}{t_{r-2}},\quad\wt(P_{(2r,3)})=\frac{c_rt_r}{t_{r-1}},\label{eq:5.43}\\
\wt(P_{(2r,2)})&=\frac{c_{r-1}c_rt_{r-1}t_r}{t_{r-2}},\label{eq:5.44}
\end{align}
where $t_0=1$. In particular, the weights $y_i:=\rho_{\tau}^*(\wt(P_i))$, $y_{(2r+1,j)}:=\rho_{\tau}^*(\wt(P_{(2r+1,j)}))$, written in terms of normalized $D_r^{(1)}$ $Q$-system variables using \eqref{eq:4.46} and \eqref{eq:4.47}, are given by
{\allowdisplaybreaks
\begin{align}
y_{2k-1}&=\frac{R_{k,1}R_{k-1,0}}{R_{k,0}R_{k-1,1}},\quad y_{4r+1-2k}=\frac{R_{k,0}R_{k-1,1}}{R_{k,1}R_{k-1,0}},\quad k\in[1,r]\setminus\{r-1\},\label{eq:5.45}\\
y_{2r-3}&=\frac{R_{r-2,0}R_{r-1,1}R_{r,1}}{R_{r-2,1}R_{r-1,0}R_{r,0}},\quad y_{2r+3}=\frac{R_{r-2,1}R_{r-1,0}R_{r,0}}{R_{r-2,0}R_{r-1,1}R_{r,1}},\label{eq:5.46}\\
y_{2\ell}&=\frac{R_{\ell-1,0}R_{\ell+1,1}}{R_{\ell,0}R_{\ell,1}},\quad y_{4r-2\ell}=\frac{R_{\ell-1,1}R_{\ell+1,0}}{R_{\ell,0}R_{\ell,1}}\quad\ell\in[1,r-3],\label{eq:5.47}\\
y_{2r-4}&=\frac{R_{r-3,0}R_{r-1,1}R_{r,1}}{R_{r-2,0}R_{r-2,1}},\quad y_{2r+4}=\frac{R_{r-3,1}R_{r-1,0}R_{r,0}}{R_{r-2,0}R_{r-2,1}},\label{eq:5.48}\\
y_{2r-2}&=\frac{R_{r-2,0}R_{r,1}}{R_{r-1,0}R_{r-1,1}R_{r,0}},\quad y_{2r+2}=\frac{R_{r-2,1}R_{r,0}}{R_{r-1,0}R_{r-1,1}R_{r,1}},\label{eq:5.49}\\
y_{(2r,1)}&=\frac{R_{r-2,0}R_{r-1,1}}{R_{r-1,0}R_{r,0}R_{r,1}},\quad y_{(2r,3)}=\frac{R_{r-2,1}R_{r-1,0}}{R_{r-1,1}R_{r,0}R_{r,1}},\label{eq:5.50}\\
y_{(2r,2)}&=\frac{R_{r-2,0}R_{r-2,1}}{R_{r-1,0}R_{r-1,1}R_{r,0}R_{r,1}},\label{eq:5.51}
\end{align}
where $R_{0,k}=1$ for all $k\in\mathbb{Z}$.
}
\item The paths $P_{2i-1}$ and $P_{2i'-1}$ do not intersect for any distinct $i,i'\in[1,2r]$.
\item For any $j\in[1,r-3]$, the path $P_{2j}$ intersects $P_k$ if and only if $k=2j\pm1,2j\pm2$, and the path $P_{4r-2j}=P_{c,c}(2r+1-j,2r-j)$ intersects $P_k$ if and only if $k=4r-2-2j,4r-1-2j,4r+1-2j,4r+2-2j$.
\item The path $P_{2r-4}$ intersects $P_k$ if and only if $k=2r-6,2r-5,2r-3,2r-2,(2r,1),(2r,2)$.
\item The path $P_{2r+4}$ intersects $P_k$ if and only if $k=(2r,2),(2r,3),2r+2,2r+3,2r+5,2r+6$.
\item The path $P_{2r-2}$ intersects $P_k$ if and only if $k=2r-4,2r-3,2r-1,(2r,1),(2r,2),(2r,3)$.
\item The path $P_{2r+2}$ intersects $P_k$ if and only if $k=(2r,1),(2r,2),(2r,3),2r+1,2r+3,2r+4$.
\item The path $P_{(2r,1)}$ intersects $P_k$ if and only if $k=2r-4,2r-3,2r-2,(2r,2),2r+1,2r+2$.
\item The path $P_{(2r,3)}$ intersects $P_k$ if and only if $k=2r-2,2r-1,(2r,2),2r+2,2r+3,2r+4$.
\item The path $P_{(2r,2)}$ intersects $P_k$ if and only if $k=2r-4,2r-3,2r-2,2r-1,(2r,1),(2r,3)$, $2r+1,2r+2,2r+3,2r+4$.
\end{enumerate}
\end{proposition} 

\begin{proof}
By observing Figure \ref{Figure5.24}, we have that if $(m,n)$ is an admissible $(c,c)$-pair, and either $m\leq r-1$ or $m\geq r+3$, then we have either $m=n$ or $m=n+1$. Thus, it remains to consider the case where either $m=r$, $m=r+1$ or $m=r+2$. If $m=r$, then since the elementary chip corresponding to $E_{-(r-2)}$ appears before the elementary chip corresponding to $E_{-(r-1)}$ in $\overline{N}_{c,c}(\mathbf{i})$, it follows that we must have either $n=r$ or $n=r-1$. By a similar reasoning, it follows that if $m=r+1$, then we must have either $n=r+1$ or $n=r-1$. Finally, if $m=r+2$, then by a similar argument, we have $n=r+1,r$ or $r-1$, and this proves statement (1). 

Statement (2) follows from the definition of path weights, while statement (3) follows from the fact that the paths involved do not involve any ascents and descents. Statement (4) follow from observing Figures \ref{Figure5.6} and \ref{Figure5.11}, while modifying the labels that appear in Figure \ref{Figure5.6} to adapt to type $D$ where necessary. Statements (5)--(11) follow from observation. We will leave the details to the reader.
\end{proof}

Motivated by statements (3)--(11) of Proposition \ref{5.22}, let us define a graph $\mathcal{G}_D$, where the graph $\mathcal{G}_D$ is given as follows:

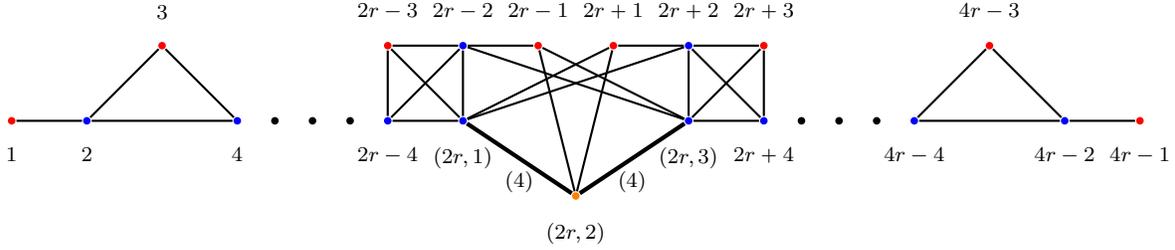
\begin{center}
\begin{tikzpicture}[
       thick,
			 font=\tiny,
       acteur/.style={
         circle,
         thick,
         inner sep=1pt,
         minimum size=0.1cm
       }
] 

\node (a1) at (0,0) [acteur,fill=red]{};
\node (a2) at (1,0) [acteur,fill=blue]{};
\node (a3) at (2,1) [acteur,fill=red]{};
\node (a4) at (3,0) [acteur,fill=blue]{};
\node (d1) at (3.5,0) [acteur,fill=black]{};
\node (d2) at (4,0) [acteur,fill=black]{};
\node (d3) at (4.5,0) [acteur,fill=black]{};
\node (a2rm4) at (5,0) [acteur,fill=blue]{};
\node (a2rm3) at (5,1) [acteur,fill=red]{};
\node (a2rc1) at (6,0) [acteur,fill=blue]{};
\node (a2rm2) at (6,1) [acteur,fill=blue]{};
\node (a2rm1) at (7,1) [acteur,fill=red]{};
\node (a2rc2) at (7.5,-1) [acteur,fill=orange]{};
\node (a2rp1) at (8,1) [acteur,fill=red]{};
\node (a2rc3) at (9,0) [acteur,fill=blue]{};
\node (a2rp2) at (9,1) [acteur,fill=blue]{};
\node (a2rp4) at (10,0) [acteur,fill=blue]{};
\node (a2rp3) at (10,1) [acteur,fill=red]{};
\node (d4) at (10.5,0) [acteur,fill=black]{};
\node (d5) at (11,0) [acteur,fill=black]{};
\node (d6) at (11.5,0) [acteur,fill=black]{};
\node (a4rm4) at (12,0) [acteur,fill=blue]{};
\node (a4rm3) at (13,1) [acteur,fill=red]{};
\node (a4rm2) at (14,0) [acteur,fill=blue]{};
\node (a4rm1) at (15,0) [acteur,fill=red]{};

\node (b1) at (0,-0.2) [below] {$1$};
\node (b2) at (1,-0.2) [below] {$2$};
\node (b3) at (2,1.2) [above] {$3$};
\node (b4) at (3,-0.2) [below] {$4$};
\node (b2rm4) at (5,-0.2) [below] {$2r-4$};
\node (b2rm3) at (5,1.2) [above] {$2r-3$};
\node (b2rc1) at (6,-0.2) [below] {$(2r,1)$};
\node (b2rm2) at (6,1.2) [above] {$2r-2$};
\node (b2rm1) at (7,1.2) [above] {$2r-1$};
\node (b2rc2) at (7.5,-1.2) [below] {$(2r,2)$};
\node (b2rp1) at (8,1.2) [above] {$2r+1$};
\node (b2rc3) at (9,-0.2) [below] {$(2r,3)$};
\node (b2rp2) at (9,1.2) [above] {$2r+2$};
\node (b2rp4) at (10,-0.2) [below] {$2r+4$};
\node (b2rp3) at (10,1.2) [above] {$2r+3$};
\node (b4rm4) at (12,-0.2) [below] {$4r-4$};
\node (b4rm3) at (13,1.2) [above] {$4r-3$};
\node (b4rm2) at (14,-0.2) [below] {$4r-2$};
\node (b4rm1) at (15,-0.2) [below] {$4r-1$};

\draw[-] (a1) to node {} (a2);
\draw[-] (a2) to node {} (a3);
\draw[-] (a2) to node {} (a4);
\draw[-] (a3) to node {} (a4);

\draw[-] (a2rm4) to node {} (a2rm3);
\draw[-] (a2rm4) to node {} (a2rm2);
\draw[-] (a2rm4) to node {} (a2rc1);
\draw[-] (a2rm3) to node {} (a2rm2);
\draw[-] (a2rm3) to node {} (a2rc1);
\draw[-] (a2rm2) to node {} (a2rc1);

\draw[-] (a2rm2) to node {} (a2rm1);
\draw[-] (a2rm2) to node {} (a2rc3);
\draw[-] (a2rc1) to node {} (a2rp1);
\draw[-] (a2rc1) to node {} (a2rp2);
\draw[-] (a2rm1) to node {} (a2rc3);
\draw[-] (a2rp1) to node {} (a2rp2);

\draw[-] (a2rp4) to node {} (a2rp3);
\draw[-] (a2rp4) to node {} (a2rp2);
\draw[-] (a2rp4) to node {} (a2rc3);
\draw[-] (a2rp3) to node {} (a2rp2);
\draw[-] (a2rp3) to node {} (a2rc3);
\draw[-] (a2rp2) to node {} (a2rc3);

\draw[-] (a4rm4) to node {} (a4rm3);
\draw[-] (a4rm4) to node {} (a4rm2);
\draw[-] (a4rm3) to node {} (a4rm2);
\draw[-] (a4rm2) to node {} (a4rm1);

\draw[-] (a2rc2) to node {} (a2rm1);
\draw[-] (a2rc2) to node {} (a2rp1);
\draw[ultra thick,-] (a2rc2) to node [below] {$(4)$} (a2rc1);
\draw[ultra thick,-] (a2rc2) to node [below] {$(4)$} (a2rc3);

\end{tikzpicture} 
\end{center}

Here, there are two complete graphs $K_4$ as subgraphs of $\mathcal{G}_D$ in the middle. The label $(4)$ on the two thick lines indicates that the vertex $(2r,2)$ connects to all $4$ vertices in both complete graphs $K_4$. With the undirected graph $\mathcal{G}_D$ given as above, it follows that for all distinct $j,k\in[1,2r-1]\cup[2r+1,4r-1]\cup\{(2r,j)\}_{j=1}^3$, we have that the paths $P_j$ and $P_k$ are non-intersecting if and only if vertices $j$ and $k$ are not adjacent to each other in $\mathcal{G}_D$. 

Similarly to types $A$--$C$, we may define a hard particle configuration on $\mathcal{G}_D$, as well as the weight $\wt(C)$ of a hard particle configuration $C$ on $\mathcal{G}_D$, and the set of hard particle configurations $\HPC(\mathcal{G}_D)$ on $\mathcal{G}_D$ analogously.

By \eqref{eq:5.37}, Definition \ref{5.7}, and Proposition \ref{5.22}, we recover the following formula for the conserved quantity $C_j=\rho_{\tau}^*(H_j^{c,c})$ of the normalized $D_r^{(1)}$ $Q$-system, or equivalently, the Coxeter-Toda Hamiltonian $H_j^{c,c}$, written in terms of the normalized $D_r^{(1)}$ $Q$-system variables:

\begin{theorem}\label{5.23}
Let $j\in[1,r-2]$. Then the $j$-th conserved quantity $C_j=\rho_{\tau}^*(H_j^{c,c})$ of the normalized $D_r^{(1)}$ $Q$-system is given by
\begin{equation}\label{eq:5.52}
C_j=\sum_{\substack{C\in\HPC(\mathcal{G}_D)\\ |C|=j}}\wt(C)+\sum_{\substack{C\in\HPC(\mathcal{G}_D)\\ |C|=j-2,(2r,2)\in C}}\wt(C).
\end{equation}
\end{theorem}

\subsection{The Coxeter-Toda Hamiltonians arising from spin representations of special orthogonal groups}\label{Section5.6}

In this subsection, our goal is to provide combinatorial formulas for, as well as network formulations of, the Coxeter-Toda Hamiltonians $H_r^{c,c}$ in the case where $G=SO_{2r+1}(\mathbb{C})$, as well as $H_{r-1}^{c,c}$ and $H_r^{c,c}$ in the case where $G=SO_{2r}(\mathbb{C})$. Our network formulations of these Coxeter-Toda Hamiltonians $H_r^{c,c}$ will be similar with the network formulations of the $F$-polynomials in cluster algebras of classical finite types arising from spin representations derived in \cite{Yang12}, and as such, we will use the notations and definitions developed therein for our network formulations of these Coxeter-Toda Hamiltonians for convenience. 

\subsubsection*{Type \textit{B}}\label{Section5.6.1}

Our first goal in this subsubsection is to give a combinatorial formula for $H_r^{c,c}$. To begin, let us first recall the spin representation $V(\omega_r)$ of $G=SO_{2r+1}(\mathbb{C})$. To this end, we let $\mathcal{C}$ denote the set of $r$-element subsets $I$ of $[1,r]\cup[r+2,2r+1]$ that satisfy $|I\cap\{k,2r+2-k\}|=1$ for all $k\in[1,r]$. The spin representation $V(\omega_r)$ has a basis $\{v_I:I\in\mathcal{C}\}$ of $\mathfrak{so}_{2r+1}$-weight vectors, indexed by elements of $\mathcal{C}$, with $v_{[1,r]}$ being the highest weight vector. For each $k\in[1,r-1]$, we define
\begin{align}
I_k^+&=\{I\in\mathcal{C}:k,2r+1-k\in I\},\quad I_k^-=\{I\in\mathcal{C}:k+1,2r+2-k\in I\}\label{eq:5.53}\\
I_r^+&=\{I\in\mathcal{C}:r\in I\},\quad I_r^-=\{I\in\mathcal{C}:r+2\in I\}.\label{eq:5.54}
\end{align}
We define a bijection $\phi_k:I_k^+\to I_k^-$ for all $k\in[1,r]$ by
\begin{equation}\label{eq:5.55}
\phi_k(I)=
\begin{cases}
(I\setminus\{k,2r+1-k\})\cup\{k+1,2r+2-k\}& \text{if }k\neq r,\\
(I\setminus\{r\})\cup\{r+2\}& \text{if }k=r.
\end{cases}
\end{equation}
The $\mathfrak{so}_{2r+1}$-action on $V(\omega_r)$ is then given by
{\allowdisplaybreaks
\begin{align}
e_k\cdot v_I&=
\begin{cases}
v_{\phi_k^{-1}(I)}& \text{if }I\in I_k^-,\\
0& \text{otherwise},
\end{cases}\label{eq:5.56}\\
e_{-k}\cdot v_I&=
\begin{cases}
v_{\phi_k(I)}& \text{if }I\in I_k^+,\\
0 & \text{otherwise}.
\end{cases}\label{eq:5.57}
\end{align}
In addition, for any $I\in\mathcal{C}$, we let 
}
\begin{equation}\label{eq:5.58}
R_I^{\pm}=\{k\in[1,r]:I\in I_k^{\pm}\}.
\end{equation}
Then it follows that as elements of $SL(V(\omega_r))$, we have
\begin{align}
E_k(a)&=I+a\sum_{J\in I_k^+}e_{J,\phi_k(J)},\label{eq:5.59}\\
E_{-k}(b)&=I+b\sum_{J\in I_k^+}e_{\phi_k(J),J},\label{eq:5.60}\\
D(t_1,\ldots,t_r)&=\sum_{I\in\mathcal{C}}t_I,\label{eq:5.61}
\end{align}
where
\begin{equation}\label{eq:5.62}
t_I=\frac{\prod_{i\in R_I^+}t_i}{\prod_{i\in R_I^-}t_i}.
\end{equation}
We are now ready to compute $H_r^{c,c}$. Firstly, we note that for a given sequence $\mathbf{J}=(J_0,\ldots,J_{2r})$ in $\mathcal{C}$, we have
\begin{equation*}
g_{\mathbf{J}}:=(E_{-1})_{J_0,J_1}\cdots (E_{-r})_{J_{r-1},J_r}D(t_1,\ldots,t_r)_{J_r,J_r}(E_1(c_1))_{J_r,J_{r+1}}\cdots (E_r(c_r))_{J_{2r-1},J_{2r}}\neq0
\end{equation*}
if and only if $J_k=J_{k-1}$ or $J_k=\phi_k^{-1}(J_{k-1})$, and $J_{r+k}=J_{r+k-1}$ or $J_{r+k}=\phi_k(J_{r+k-1})$ for all $k\in[1,r]$, and in this case, we have
\begin{equation}\label{eq:5.63}
g_{\mathbf{J}}=t_{J_r}\prod_{k\in[1,r]:J_{r+k}=\phi_k(J_{r+k-1})}c_k.
\end{equation}
Henceforth, we let $\mathcal{S}$ denote the set of sequences $\mathbf{J}=(J_0,J_1,\ldots,J_{2r})$ in $\mathcal{C}$ satisfying $J_0=J_{2r}$, $J_k=J_{k-1}$ or $J_k=\phi_k^{-1}(J_{k-1})$, and $J_{r+k}=J_{r+k-1}$ or $J_{r+k}=\phi_k(J_{r+k-1})$ for all $k\in[1,r]$. Then it follows that we have
\begin{equation}\label{eq:5.64}
H_r^{c,c}=\sum_{\mathbf{J}\in\mathcal{S}}g_{\mathbf{J}}.
\end{equation}
Next, we claim that for any $J\in\mathcal{C}$, there is a bijection between the set $\mathcal{S}_J$ of sequences $\mathbf{J}=(J_0,J_1,\ldots,J_{2r})$ in $\mathcal{S}$ with $J_r=J$, and the subsets of $R_J^+$. To begin, we first see from the definition of $I_k^{\pm}$ and $R_J^{\pm}$ that at most one of $i,i+1$ can appear in $R_J^{\pm}$ for all $i\in[1,r-1]$. This implies that for all distinct $i,j\in R_J^+$, we have $\phi_i(J)\in I_j^+$, $\phi_j(J)\in I_i^+$, and $(\phi_i\circ\phi_j)(J)=(\phi_j\circ\phi_i)(J)$. Based on this observation, it follows that for all $\mathbf{J}=(J_0,J_1,\ldots,J_{2r})$ in $\mathcal{S}$ with $J_r=J$, we have that if $J_{r+k}\neq J_{r+k-1}$, then $k\in R_J^+$. Hence, given a sequence $\mathbf{J}=(J_0,J_1,\ldots,J_{2r})$ in $\mathcal{S}$ with $J_r=J$, we have a subset $S_{\mathbf{J}}$ of $R_J^+$, given by $S_{\mathbf{J}}=\{k\in[1,r]:J_{r+k}\neq J_{r+k-1}\}$. 

Conversely, given a subset $S$ of $R_J^+$, we may define $\mathbf{J}_S=(J_0,J_1,\ldots,J_{2r})\in \mathcal{S}_J$ as follows: we first set $J_r=J$, and for all $i\in[1,r]$, we define $J_{i-1}=J_i$ if $i\notin S$, and $J_{i-1}=\phi_i(J_i)$ otherwise. Similarly, we define $J_{r+i}=J_{r+i-1}$ if $i\notin S$, and $J_{r+i}=\phi_i(J_{r+i-1})$ otherwise. Then $\mathbf{J}_S$ is a well-defined sequence in $\mathcal{C}$ with $J_r=J$ and $J_0=J_{2r}$, so we have $\mathbf{J}_S\in\mathcal{S}_J$. Now, it is easy to see that $\mathbf{J}_{S_{\mathbf{J}}}=\mathbf{J}$ and $S_{\mathbf{J}_S}=S$, so this establishes a bijection between $\mathcal{S}_J$, and the subsets of $R_J^+$. 

Now, using this established bijection, along with \eqref{eq:5.63}, we have
\begin{equation}\label{eq:5.65}
\sum_{\mathbf{J}\in\mathcal{S}_J}g_{\mathbf{J}}=t_J\prod_{k\in R_J^+}(1+c_k).
\end{equation}
Combining this with \eqref{eq:5.64}, we have
\begin{equation}\label{eq:5.66}
H_r^{c,c}
=\sum_{\mathbf{J}\in\mathcal{S}}g_{\mathbf{J}}
=\sum_{J\in\mathcal{C}}\sum_{\mathbf{J}\in\mathcal{S}_J}g_{\mathbf{J}}
=\sum_{J\in\mathcal{C}}t_J\prod_{k\in R_J^+}(1+c_k).
\end{equation}
From \eqref{eq:5.66}, we deduce that:

\begin{theorem}\label{5.24}
The $r$-th conserved quantity $C_r=\rho_{\tau}^*(H_r^{c,c})$ of the normalized $D_{r+1}^{(2)}$ $Q$-system is given by
\begin{equation}\label{eq:5.67}
C_r=\sum_{J\in\mathcal{C}}z_J\prod_{k\in R_J^+}(1+Y_{k,1}),
\end{equation}
where
\begin{equation}\label{eq:5.68}
z_J=\prod_{i\in R_J^+}\frac{R_{i,1}}{R_{i,0}}\prod_{i\in R_J^-}\frac{R_{i,0}}{R_{i,1}},\quad Y_{k,1}=\prod_{j=1}^rR_{j,1}^{-C_{j,k}}.
\end{equation}
\end{theorem}

Our next goal in this subsubsection is to give a network formulation of $H_r^{c,c}$, or equivalently, $C_r$. Following \cite{Yang12}, we say that a non-intersecting set $\overline{\mathcal{P}}$ of admissible $(c,c)$-triples with $|\overline{\mathcal{P}}|=r$ and $s(\overline{\mathcal{P}})=d(\overline{\mathcal{P}})$ is bundled if the following conditions are satisfied:
\begin{enumerate}
\item $s(\overline{\mathcal{P}})\in\mathcal{C}$,
\item None of the paths in $\overline{\mathcal{P}}$ intersect level $r+1$, and
\item For each $i\in[1,r-1]$, within each elementary chip that correspond to $E_{-i}$, either both of the diagonal edges belong to $\overline{\mathcal{P}}$, or neither of the diagonal edges belong to $\overline{\mathcal{P}}$, and within each elementary chip that correspond to $E_i(c_i)$, either both of the diagonal edges belong to $\overline{\mathcal{P}}$, or neither of the diagonal edges belong to $\overline{\mathcal{P}}$.
\end{enumerate}
In particular, it follows from the first condition that within the elementary chip that correspond to $E_{-r}$, none of the short diagonal edges belong to $\overline{\mathcal{P}}$, and within the elementary chip that correspond to $E_r(c_r)$, none of the short diagonal edges belong to $\overline{\mathcal{P}}$. By Corollary \ref{5.14}, it follows that $\overline{\mathcal{P}}$ is necessarily unmixed, and $(r+1,r+1,r+1),(r+1,r,r+1),(r+2,r+1,r+2)\notin\overline{\mathcal{P}}$. Moreover, it follows from the bundled conditions that we must have $v(\overline{\mathcal{P}})\in\mathcal{C}$. In light of these observations, we say that a non-intersecting set $\mathcal{P}$ of admissible $(c,c)$-pairs with $|\mathcal{P}|=r$ is bundled if $s(\mathcal{P}),v(\mathcal{P})\in\mathcal{C}$. With this definition, we define $\mathcal{P}_{\nonint,\bund}^{c,c}(J)$ for all $J\in\mathcal{C}$ as follows:
\begin{equation}\label{eq:5.69}
\mathcal{P}_{\nonint,\bund}^{c,c}(J)=\{\mathcal{P}\subseteq\mathcal{P}^{c,c}:\mathcal{P}\text{ is non-intersecting, }s(\mathcal{P})\in\mathcal{C}\text{ and }v(\mathcal{P})=J\}.\end{equation}
Now, it is easy to see that for any $J\in\mathcal{C}$, there is a natural bijection $\phi_J$ between $\mathcal{P}_{\nonint,\bund}^{c,c}(J)$ and $\mathcal{S}_J$. Moreover, due to the bundled conditions, it follows that for any $\mathcal{P}\in\mathcal{P}_{\nonint,\bund}^{c,c}(J)$, we have $\wt_{c,c}(\mathcal{P})=g_{\phi_J(\mathcal{P})}^2$. Consequently, it follows from \eqref{eq:5.64} that we have
\begin{equation}\label{eq:5.70}
H_r^{c,c}
=\sum_{J\in\mathcal{C}}\sum_{\mathbf{J}\in\mathcal{S}_J}g_{\mathbf{J}}
=\sum_{J\in\mathcal{C}}\sum_{\mathcal{P}_{\nonint,\bund}^{c,c}(J)}\sqrt{\wt_{c,c}(\mathcal{P})}.
\end{equation}
Finally, we may also give a combinatorial formula for $C_r$ in terms of certain hard particle configurations $\mathcal{G}_B$. To begin, we shall define the notion of a bundled hard particle configuration on $\mathcal{G}_B$:

\begin{definition}\label{5.25}
We say that a hard particle configuration $C$ on $\mathcal{G}_B$ is bundled if the following conditions are satisfied: 
\begin{enumerate}
\item $2r,2r+2,(2r+1,1)\notin C$,
\item Either $2k,4r+2-2k\in C$, or $2k,4r+2-2k\notin C$ for all $k\in[1,r-1]$, 
\item $|C\cap\{1,4r,4r+1\}|=1$, $|C\cap\{2k-2,2k-1,4r+2-2k,4r+3-2k\}|=1$ for all $k\in[2,r-1]$, $|C\cap\{2r-2,2r-1,(2r+1,2),2r+3\}|=1$, and
\item $|C\cap\{1,2,4r+1\}|=1$, $|C\cap\{2k-1,2k,4r+3-2k,4r+4-2k\}|=1$ for all $k\in[2,r-1]$, and $|C\cap\{2r-1,(2r+1,2),2r+3,2r+4\}|=1$.
\end{enumerate}
We denote the set of bundled hard particle configurations on $\mathcal{G}_B$ by $\HPC_{\bund}(\mathcal{G}_B)$.
\end{definition}

We claim that there is a natural bijection between $\HPC_{\bund}(\mathcal{G}_B)$ and $\bigcup_{J\in\mathcal{C}}\mathcal{P}_{\nonint,\bund}^{c,c}(J)$. Indeed, the first condition corresponds to the fact that $(r+1,r+1),(r+1,r),(r+2,r+1)\notin\mathcal{P}$, while the second condition corresponds to the fact that for each $i\in[1,r-1]$, within each elementary chip that correspond to $E_{-i}$, either both of the diagonal edges belong to $\mathcal{P}$, or neither of the diagonal edges belong to $\mathcal{P}$, and within each elementary chip that correspond to $E_i(c_i)$, either both of the diagonal edges belong to $\mathcal{P}$, or neither of the diagonal edges belong to $\mathcal{P}$. The third condition corresponds to the fact that since $s(\mathcal{P})\in\mathcal{C}$, we have $|s(\mathcal{P})\cap\{k,2r+2-k\}|=1$ for all $k\in[1,r]$, while the last condition corresponds to the fact that since $v(\mathcal{P})\in\mathcal{C}$, we have $|v(\mathcal{P})\cap\{k,2r+2-k\}|=1$ for all $k\in[1,r]$.

By \eqref{eq:5.70} and Definition \ref{5.25}, we recover another formula for the $r$-th conserved quantity $C_r=\rho_{\tau}^*(H_r^{c,c})$ of the normalized $D_{r+1}^{(2)}$ $Q$-system, or equivalently, the Coxeter-Toda Hamiltonian $H_r^{c,c}$, written in terms of the normalized $D_{r+1}^{(2)}$ $Q$-system variables:

\begin{theorem}\label{5.26}
The $r$-th conserved quantity $C_r=\rho_{\tau}^*(H_r^{c,c})$ of the normalized $D_{r+1}^{(2)}$ $Q$-system is given by
\begin{equation}\label{eq:5.71}
C_r=\sum_{C\in\HPC_{\bund}(\mathcal{G}_B)}\sqrt{\wt(C)}.
\end{equation}
\end{theorem}

\subsubsection*{Type \textit{D}}\label{Section5.6.2}

Our first goal in this subsubsection is to give a combinatorial formula for $H_{r-1}^{c,c}$ and $H_r^{c,c}$. To begin, let us first recall the spin representations $V(\omega_{r-1})$ and $V(\omega_r)$ of $G=SO_{2r}(\mathbb{C})$. To this end, we let $\mathcal{C}$ denote the set of $r$-element subsets $I$ of $[1,2r]$ that satisfy $|I\cap\{k,2r+1-k\}|=1$ for all $k\in[1,r]$. Then we have
\begin{equation}\label{eq:5.72}
\mathcal{C}=\mathcal{C}_{\odd}\sqcup\mathcal{C}_{\even},
\end{equation}
where 
\begin{equation}\label{eq:5.73}
\mathcal{C}_{\odd}=\{I\in\mathcal{C}:|I\cap[r+1,2r]|\text{ is odd}\},\quad\mathcal{C}_{\even}=\{I\in\mathcal{C}:|I\cap[r+1,2r]|\text{ is even}\}
\end{equation}
The spin representations $V(\omega_{r-1})$ and $V(\omega_r)$ has bases $\{v_I:I\in\mathcal{C}_{\odd}\}$ and $\{v_I:I\in\mathcal{C}_{\even}\}$ of $\mathfrak{so}_{2r+1}$-weight vectors, indexed by elements of $\mathcal{C}_{\odd}$ and $\mathcal{C}_{\odd}$ respectively, with $v_{[1,r-1]\cup\{r+1\}}$ and $v_{[1,r]}$ being the respective highest weight vectors. For each $k\in[1,r-1]$, we define
\begin{align}
I_k^+&=\{I\in\mathcal{C}:k,2r-k\in I\},\quad I_k^-=\{I\in\mathcal{C}:k+1,2r+1-k\in I\}\label{eq:5.74}\\
I_r^+&=\{I\in\mathcal{C}:r-1,r\in I\},\quad I_r^-=\{I\in\mathcal{C}:r+1,r+2\in I\}.\label{eq:5.75}
\end{align}
We define a bijection $\phi_k:I_k^+\to I_k^-$ for all $k\in[1,r]$ by
\begin{equation}\label{eq:5.76}
\phi_k(I)=
\begin{cases}
(I\setminus\{k,2r-k\})\cup\{k+1,2r+1-k\}& \text{if }k\neq r,\\
(I\setminus\{r-1,r\})\cup\{r+1,r+2\}& \text{if }k=r.
\end{cases}
\end{equation}
The $\mathfrak{so}_{2r}$-action on both $V(\omega_{r-1})$ and $V(\omega_r)$ is then given by formulas analogous to \eqref{eq:5.56} and \eqref{eq:5.57}. By defining $R_I^{\pm}$ for all $I\in\mathcal{C}$ by \eqref{eq:5.58}, we see that as elements of $SL(V(\omega_{r-1}))$, we have
\begin{align}
E_k(a)&=I+a\sum_{J\in I_k^+\cap\mathcal{C}_{\odd}}e_{J,\phi_k(J)},\label{eq:5.77}\\
E_{-k}(b)&=I+b\sum_{J\in I_k^+\cap\mathcal{C}_{\odd}}e_{\phi_k(J),J},\label{eq:5.78}\\
D(t_1,\ldots,t_r)&=\sum_{I\in\mathcal{C}_{\odd}}t_I,\label{eq:5.79}
\end{align}
and as elements of $SL(V(\omega_r))$, we have
{\allowdisplaybreaks
\begin{align}
E_k(a)&=I+a\sum_{J\in I_k^+\cap\mathcal{C}_{\even}}e_{J,\phi_k(J)},\label{eq:5.80}\\
E_{-k}(b)&=I+b\sum_{J\in I_k^+\cap\mathcal{C}_{\even}}e_{\phi_k(J),J},\label{eq:5.81}\\
D(t_1,\ldots,t_r)&=\sum_{I\in\mathcal{C}_{\even}}t_I,\label{eq:5.82}
\end{align}
where $t_I$ is defined by \eqref{eq:5.62}. 

We are now ready to compute $H_{r-1}^{c,c}$ and $H_r^{c,c}$. As in type $B$, we let $\mathcal{S}_{\odd}$ (respectively $\mathcal{S}_{\even}$) denote the set of sequences $\mathbf{J}=(J_0,J_1,\ldots,J_{2r})$ in $\mathcal{C}_{\odd}$ (respectively $\mathcal{C}_{\even}$) satisfying $J_0=J_{2r}$, $J_k=J_{k-1}$ or $J_k=\phi_k^{-1}(J_{k-1})$, and $J_{r+k}=J_{r+k-1}$ or $J_{r+k}=\phi_k(J_{r+k-1})$ for all $k\in[1,r]$. Then it follows that we have
}
\begin{equation}\label{eq:5.83}
H_{r-1}^{c,c}=\sum_{\mathbf{J}\in\mathcal{S}_{\odd}}g_{\mathbf{J}},\quad
H_r^{c,c}=\sum_{\mathbf{J}\in\mathcal{S}_{\even}}g_{\mathbf{J}},
\end{equation}
where $g_{\mathbf{J}}$ is defined by \eqref{eq:5.63}. As in type $B$, we have that for any $J\in\mathcal{C}_{\odd}$ (respectively $J\in\mathcal{C}_{\even}$), there is a bijection between the set $\mathcal{S}_J$ of sequences $\mathbf{J}=(J_0,J_1,\ldots,J_{2r})$ in $\mathcal{S}_{\odd}$ (respectively $\mathcal{S}_{\even}$) with $J_r=J$, and the subsets of $R_J^+$. Thus, by a similar argument as in type $B$, it follows that we have
\begin{equation}\label{eq:5.84}
H_{r-1}^{c,c}=\sum_{J\in\mathcal{C}_{\odd}}t_J\prod_{k\in R_J^+}(1+c_k),\quad
H_r^{c,c}=\sum_{J\in\mathcal{C}_{\even}}t_J\prod_{k\in R_J^+}(1+c_k)
\end{equation}
Thus, from \eqref{eq:5.84}, we deduce that:

\begin{theorem}\label{5.27}
The $(r-1)$-th and $r$-th conserved quantities $C_{r-1}=\rho_{\tau}^*(H_{r-1}^{c,c})$ and $C_r=\rho_{\tau}^*(H_r^{c,c})$ of the normalized $D_r^{(1)}$ $Q$-system are given by
\begin{equation}\label{eq:5.85}
C_{r-1}=\sum_{J\in\mathcal{C}_{\odd}}z_J\prod_{k\in R_J^+}(1+Y_{k,1}),\quad
C_r=\sum_{J\in\mathcal{C}_{\even}}z_J\prod_{k\in R_J^+}(1+Y_{k,1})
\end{equation}
where $z_J$ and $Y_{k,1}$ are given by \eqref{eq:5.68}.
\end{theorem}

Our next goal in this subsubsection is to give network formulations of $H_{r-1}^{c,c}$ and $H_r^{c,c}$, or equivalently, $C_{r-1}$ and $C_r$. Following \cite{Yang12}, we say that a non-intersecting set $\overline{\mathcal{P}}$ of admissible $(c,c)$-triples with $|\overline{\mathcal{P}}|=r$ and $s(\overline{\mathcal{P}})=d(\overline{\mathcal{P}})$ is bundled if the following conditions are satisfied:
\begin{enumerate}
\item $s(\overline{\mathcal{P}})\in\mathcal{C}$,
\item For each $i\in[1,r]$, within each elementary chip that correspond to $E_{-i}$, either both of the diagonal edges belong to $\overline{\mathcal{P}}$, or neither of the diagonal edges belong to $\overline{\mathcal{P}}$, and within each elementary chip that correspond to $E_i(c_i)$, either both of the diagonal edges belong to $\overline{\mathcal{P}}$, or neither of the diagonal edges belong to $\overline{\mathcal{P}}$.
\end{enumerate}
In particular, it follows from the bundled conditions, along with Corollary \ref{5.20} that $\overline{\mathcal{P}}$ is necessarily unmixed, $(r+2,r+1,r+2)\notin\overline{\mathcal{P}}$, and $v(\overline{\mathcal{P}})\in\mathcal{C}$. In light of these observations, we say that a non-intersecting set $\mathcal{P}$ of admissible $(c,c)$-pairs with $|\mathcal{P}|=r$ is bundled if $s(\mathcal{P}),v(\mathcal{P})\in\mathcal{C}$. With this definition, we define $\mathcal{P}_{\nonint,\bund}^{c,c}(J)$ analogously by \eqref{eq:5.69} for all $J\in\mathcal{C}$. 

Now, it is easy to see that for any $J\in\mathcal{C}$, there is a natural bijection $\phi_J$ between $\mathcal{P}_{\nonint,\bund}^{c,c}(J)$ and $\mathcal{S}_J$. Moreover, due to the bundled conditions, it follows that for any $\mathcal{P}\in\mathcal{P}_{\nonint,\bund}^{c,c}(J)$, we have $\wt_{c,c}(\mathcal{P})=g_{\phi_J(\mathcal{P})}^2$. Consequently, it follows from \eqref{eq:5.84} that we have
\begin{equation}\label{eq:5.86}
H_{r-1}^{c,c}=\sum_{J\in\mathcal{C}_{\odd}}\sum_{\mathcal{P}_{\nonint,\bund}^{c,c}(J)}\sqrt{\wt_{c,c}(\mathcal{P})},\quad
H_r^{c,c}=\sum_{J\in\mathcal{C}_{\even}}\sum_{\mathcal{P}_{\nonint,\bund}^{c,c}(J)}\sqrt{\wt_{c,c}(\mathcal{P})}.
\end{equation}
Finally, we may also give combinatorial formulas for $C_{r-1}$ and $C_r$ in terms of certain hard particle configurations $\mathcal{G}_D$. To begin, we shall define the notion of a bundled hard particle configuration on $\mathcal{G}_D$:

\begin{definition}\label{5.28}
We say that a hard particle configuration $C$ on $\mathcal{G}_D$ is bundled if the following conditions are satisfied: 
\begin{enumerate}
\item $(2r,2)\notin C$,
\item Either $2k,4r-2k\in C$, or $2k,4r-2k\notin C$ for all $k\in[1,r-1]$, and either $(2r,1),(2r,3)\in C$, $(2r,1),(2r,3)\notin C$,
\item $|C\cap\{1,4r-2,4r-1\}|=1$, $|C\cap\{2k-2,2k-1,4r-2k,4r+1-2k\}|=1$ for all $k\in[2,r-2]$, $|C\cap\{2r-4,2r-3,(2r,3),2r+2,2r+3\}|=1$, $|C\cap\{2r-2,2r-1,(2r,1),2r+1\}|=1$ and
\item $|C\cap\{1,2,4r-1\}|=1$, $|C\cap\{2k-1,2k,4r+1-2k,4r+2-2k\}|=1$ for all $k\in[2,r-2]$, $|C\cap\{2r-3,2r-2,(2r,1),2r+3,2r+4\}|=1$ and $|C\cap\{2r-1,(2r,3),2r+1,2r+2\}|=1$.
\end{enumerate}
We denote the set of bundled hard particle configurations $C$ on $\mathcal{G}_B$ for which $|C\cap([2r+1,4r-1]\cup\{(2r,1),(2r,3)\})|$ is odd by $\HPC_{\bund,\odd}(\mathcal{G}_D)$, and we denote the set of bundled hard particle configurations $C$ on $\mathcal{G}_B$ for which $|C\cap([2r+1,4r-1]\cup\{(2r,1),(2r,3)\})|$ is even by $\HPC_{\bund,\even}(\mathcal{G}_D)$.
\end{definition}

We claim that there is a natural bijection between $\HPC_{\bund,\odd}(\mathcal{G}_B)$ and $\bigcup_{J\in\mathcal{C}_{\odd}}\mathcal{P}_{\nonint,\bund}^{c,c}(J)$, as well as the sets $\HPC_{\bund,\even}(\mathcal{G}_B)$ and $\bigcup_{J\in\mathcal{C}_{\even}}\mathcal{P}_{\nonint,\bund}^{c,c}(J)$. Indeed, the first condition corresponds to the fact that $(r+2,r-1),\notin\mathcal{P}$, while the second condition corresponds to the fact that for each $i\in[1,r]$, within each elementary chip that correspond to $E_{-i}$, either both of the diagonal edges belong to $\mathcal{P}$, or neither of the diagonal edges belong to $\mathcal{P}$, and within each elementary chip that correspond to $E_i(c_i)$, either both of the diagonal edges belong to $\mathcal{P}$, or neither of the diagonal edges belong to $\mathcal{P}$. The third condition corresponds to the fact that since $s(\mathcal{P})\in\mathcal{C}$, we have $|s(\mathcal{P})\cap\{k,2r+1-k\}|=1$ for all $k\in[1,r]$, while the last condition corresponds to the fact that since $v(\mathcal{P})\in\mathcal{C}$, we have $|v(\mathcal{P})\cap\{k,2r+1-k\}|=1$ for all $k\in[1,r]$.

By \eqref{eq:5.86} and Definition \ref{5.28}, we record another formula for the $(r-1)$-th and $r$-th conserved quantities $C_{r-1}=\rho_{\tau}^*(H_{r-1}^{c,c})$ and $C_r=\rho_{\tau}^*(H_r^{c,c})$ of the normalized $D_r^{(1)}$ $Q$-system, or equivalently, the Coxeter-Toda Hamiltonians $H_{r-1}^{c,c}$ and $H_r^{c,c}$, written in terms of the normalized $D_r^{(1)}$ $Q$-system variables:

\begin{theorem}\label{5.29}
The $(r-1)$-th and $r$-th conserved quantities $C_{r-1}=\rho_{\tau}^*(H_{r-1}^{c,c})$ and $C_r=\rho_{\tau}^*(H_r^{c,c})$ of the normalized $D_r^{(1)}$ $Q$-system are given by
\begin{equation}\label{eq:5.87}
C_{r-1}=\sum_{C\in\HPC_{\bund,\odd}(\mathcal{G}_D)}\sqrt{\wt(C)},\quad
C_r=\sum_{C\in\HPC_{\bund,\even}(\mathcal{G}_D)}\sqrt{\wt(C)}.
\end{equation}
\end{theorem}
\section{Conclusion}

In this paper, we constructed generalized B\"{a}cklund-Darboux transformations between two conjugation quotient Coxeter double Bruhat cells in terms of cluster mutations, and shown that these transformations preserve Coxeter-Toda Hamiltonians arising from trace functions of representations of a simple Lie group \ref{4.4}, which implies that these generalized B\"{a}cklund-Darboux transformations preserve Hamiltonian flows generated by the aforementioned Coxeter-Toda Hamiltonians \ref{4.3}, thereby generalizing \cite[Theorem 6.1]{GSV11} to the other finite Lie types. In particular, this shows that the family of Coxeter-Toda systems on a simple Lie group forms a single cluster integrable system.

In addition, we have also developed network formulations of Coxeter-Toda Hamiltonians arising from trace functions of the fundamental representations or the exterior powers of the defining representation in the case where we have a simple classical Lie group, using the network formulations of elements that we have developed at the end of Section \ref{Section2}. By further examining the possible collections of non-intersecting paths that can arise in each classical Dynkin type, and using the connection between the discrete dynamics of the factorization mapping defined in \eqref{eq:4.37} and $Q$-system evolutions established in \cite{GSV11, Williams15}, we provide combinatorial formulas for Coxeter-Toda Hamiltonians of types $A_r, B_r, C_r$ and $D_r$ in terms of partition functions of hard particles on weighted graphs $\mathcal{G}_A,\mathcal{G}_B,\mathcal{G}_C,\mathcal{G}_D$ respectively, or equivalently, conserved quantities of $Q$-systems of affine Dynkin types $A_r^{(1)}, D_{r+1}^{(2)}, A_{2r-1}^{(2)}$ and $D_r^{(1)}$ respectively, given in Theorems \ref{5.8}, \ref{5.11}, \ref{5.17}, \ref{5.23}, \ref{5.24} and \ref{5.27}.

Here, we would like to mention a few directions for future work. Firstly, owing to the fact that there exists a natural Poisson structure on a given cluster algebra that is compatible with the cluster structure \cite{GSV03}, it follows that cluster algebras admit natural quantizations \cite{BZ05}, whose classical limit coincides with the aforementioned Poisson structure. Thus, it is a natural question to ask if we can extend the network formulations of the Coxeter-Toda Hamiltonians for a simple classical Lie group to the quantum case as well, and if we can derive combinatorial formulas for the quantum Coxeter-Toda Hamiltonians for a simple classical Lie group. As a first step in this direction, Lunsford \cite{Lunsford23} obtained network formulations of quantum Coxeter-Toda Hamiltonians of types $A$ and $C$, and identified them with the corresponding quantum difference Toda Hamiltonians of types $A$ and $C$ respectively obtained by Gonin and Tsymbaliuk \cite{GT21} using the Lax formalism.

Next, we would like to remark that cluster structures on double Bruhat cells of affine Kac-Moody groups were studied in \cite{Williams13}. In addition, a rigorous treatment of infinite-dimensional Poisson-Lie theory was developed in \cite{Williams13-2}, which allows one to define Poisson structures on affine Kac-Moody groups that is compatible with the underlying cluster structures, and subsequently construct affine Coxeter-Toda systems. Thus, it is natural to ask if there exist generalized B\"{a}cklund-Darboux transformations between two conjugation quotient Coxeter double Bruhat cells of an affine Kac-Moody group that preserve flows generated by an affine Coxeter-Toda Hamiltonian, and if we can extend the network formulations of the Coxeter-Toda Hamiltonians for a simple classical Lie group to the affine case as well.

Finally, we would like to mention that while the $Q$-systems of types $B_r^{(1)}$ and $C_r^{(1)}$ admit cluster algebraic formulations as well \cite{DFK09}, we did not consider these $Q$-systems in this paper, as we were unable to make a direct connection between the dynamics of these $Q$-system evolutions and that of factorization mappings. It is a natural question to ask if the conserved quantities of these $Q$-systems admit network formulations as well, and if any such network formulation of the conserved quantities of the $B_r^{(1)}$ $Q$-system has any connection with the formulation of the conserved quantities of the $B_r^{(1)}$ $Q$-system given in terms of partition functions of weights of perfect matchings on a weighted torus graph given in \cite{Vichitkunakorn18}.
\raggedbottom

\bibliographystyle{alpha}
\bibliography{references}

\sc{
\address{Singapore Institute of Manufacturing Technology (SIMTech), Agency for Science, Technology and Research (A*STAR), 5 Cleantech Loop, \#01-01 CleanTech Two Block B, Singapore 636732, Republic of Singapore}\\
Email address: \email{simon\_lin@simtech.a-star.edu.sg}
}

\end{document}